\documentclass[12pt, leqno]{amsart}
\pdfoutput=1
\usepackage{latexsym}
\usepackage{amssymb}
\usepackage{amsmath}
\usepackage[english]{babel}

\usepackage[colorlinks=true, linkcolor=red, citecolor=red, backref=page]{hyperref} 

\usepackage[top=1.25in, bottom=1.25in, left=1.25in, right=1.25in]{geometry}

\newtheorem{theorem}{Theorem}

\newtheorem{corollary}[theorem]{Corollary}

\newtheorem{definition}[theorem]{Definition}

\newtheorem{lemma}[theorem]{Lemma}

\newtheorem{proposition}[theorem]{Proposition}
\newtheorem{remark}[theorem]{Remark}

\newtheorem{mytheorem}{Theorem}

\newcommand{\rn}[1]{\mathbb{R}^{#1}}

\newcommand{\re}{ \mathbb{R}}

\newcommand{\beq}{\begin{equation}}
\newcommand{\bea}[1]{\begin{array}{#1} }
\newcommand{\eeq}{ \end{equation}}
\newcommand{\ea}{ \end{array}}
\newcommand{\ep}{\epsilon}
\newcommand{\es}{\emptyset}

\newcommand{\al}{\alpha}
\newcommand{\ga}{\gamma}

\newcommand{\de}{\delta}
\newcommand{\ds}{\displaystyle}
\newcommand{\ts}{\textstyle}
\newcommand{\rar }{\mbox{$\rightarrow$}}

\newcommand{\ran}{\rangle}
\newcommand{\lan}{\langle}
\newcommand{\Ga}{\Gamma}
\newcommand{\la}{\lambda}
\newcommand{\La}{\Lambda}
\newcommand{\ar}{\partial}
\newcommand{\si}{\sigma}

\newcommand{\om}{\omega}
\newcommand{\Om}{\Omega}
\newcommand{\be}{\beta}

\newcommand{\ph}{\phi}
\newcommand{\he}{\theta}
\newcommand{\He}{\Theta}
\newcommand{\Ph}{\Phi}
\newcommand{\hs}[1]{\mbox{$ \hspace{#1}$}}

\newcommand{\sem}{\setminus}
\newcommand{\ze}{\zeta}
\newcommand{\De}{\Delta}
\newcommand{\ti}{\tilde}
\newcommand{\noi}{\noindent}

\renewcommand*{\backref}[1]{}
\renewcommand*{\backrefalt}[4]{%
 \ifcase #1 (Not cited.)%
   \or        (Cited on page~#2.)%
    \else      (Cited on pages~#2.)%
    \fi}

\begin{document} 

\title[The  Brunn-Minkowski Inequality and A Minkowski problem]{The  Brunn-Minkowski Inequality and A Minkowski problem for Nonlinear Capacity}


\author[M. Akman]{Murat Akman}
\address{{\bf Murat Akman}\\
University of Connecticut, Storrs, CT 06269-3009} \email{murat.akman@uconn.edu}

\author[J. Gong]{Jasun Gong}
\address{{\bf Jasun Gong} \\ Mathematics Department \\ Fordham University \\ John Mulcahy Hall Bronx, NY 10458-5165}
\email{jgong7@fordham.edu}

\author[J. Hineman]{Jay Hineman}
\address{{\bf Jay Hineman} \\ Geometric Data Analytics, Durham, NC 27707}
\email{jay.hineman@geomdata.com}

\author[J. Lewis]{John Lewis}
\address{{\bf John Lewis} \\ Mathematics Department \\ University of Kentucky \\ Lexington, Kentucky, 40506}
\email{johnl@uky.edu}

\author[A. Vogel]{Andrew Vogel}
\address{{\bf Andrew Vogel}\\ Department of Mathematics, Syracuse University \\  Syracuse, New York 13244}
\email{alvogel@syracuse.edu}


\keywords{The Brunn-Minkowski inequality, Nonlinear capacities, Inequalities and extremum problems, Potentials and capacities, $\mathcal{A}$-harmonic PDEs, Minkowski problem, Variational formula, Hadamard variational formula}
\subjclass[2010]{35J60,31B15,39B62,52A40,35J20,52A20,35J92}
\begin{abstract}
In this article we study two classical potential-theoretic problems in convex geometry.  The first problem is  an inequality of  Brunn-Minkowski type for  a nonlinear capacity, $\mbox{Cap}_{\mathcal{A}},$ where  $\mathcal{A}$-capacity  is  associated with  a   nonlinear elliptic  PDE  whose  structure is modeled on 
the  $p$-Laplace equation and whose solutions in an open set are called $ \mathcal{A}$-harmonic.  

In the first part of this article, we prove the Brunn-Minkowski inequality for  this  capacity:
\[
\left[\mbox{Cap}_\mathcal{A}  ( \la  E_1 + (1-\la) E_2 )\right]^{\frac{1}{(n-p)}}  \geq  \la  \, 
\left[\mbox{Cap}_\mathcal{A}  (  E_1 )\right]^{\frac{1}{(n-p)}}  +  (1-\la)  \left[\mbox{Cap}_\mathcal{A}  (E_2 )\right]^{\frac{1}{(n-p)}}
\]
when  $1<p<n, 0 < \la < 1,   $ and $E_1, E_2$ are convex compact sets with positive  $\mathcal{A}$-capacity. 
Moreover, if equality holds  in  the  above inequality for some $E_1$ and $E_2, $ then under certain  regularity and structural assumptions on $\mathcal{A},$  we show that  these two sets are homothetic. 

In the second part of this article we study a Minkowski problem for a certain measure  
associated with a  compact convex set  $E$ with nonempty interior  and its    $ \mathcal{A}$-harmonic capacitary function in the complement of $E$. If  $   \mu_E $ denotes this measure, then the  Minkowski problem we consider in this setting is that; for a given finite Borel measure $\mu$ on $\mathbb{S}^{n-1}$, find necessary and sufficient conditions for which there exists $ E $ as above with $  \mu_E   =  \mu.  $   We show that  necessary and sufficient conditions for  existence under this setting are exactly the same conditions as in the classical Minkowski problem for volume as well as in the work of Jerison in \cite{J} for  electrostatic capacity. Using the Brunn-Minkowski inequality result from the first part, we also show that this problem has a unique solution up to translation  when $p\neq n- 1$ and translation and dilation when $p = n-1$.
\end{abstract}

\maketitle

\tableofcontents

\part{The Brunn-Minkowski inequality for nonlinear capacity}

 \section{Introduction} 
  The well-known Brunn-Minkowski inequality states  that 
\begin{align}
\label{BMVolume}
\left[\mbox{Vol}(\lambda E_1+(1-\lambda) E_2)\right]^{\frac{1}{n}} 
\geq\lambda\left[\mbox{Vol}(E_1)\right]^{\frac{1}{n}} +(1-\lambda) \left[\mbox{Vol}(E_2)\right]^{\frac{1}{n}}
\end{align}
whenever $E_1, E_2$ are compact convex
sets with nonempty interiors in $\mathbb{R}^{n}$ and $\lambda\in (0,1)$. 
Moreover, equality in \eqref{BMVolume} holds if and only if $E_1$ is a translation and dilation of $E_2$.  For  numerous  applications of this inequality  to  problems in  geometry and  analysis  see  the classical book by Schneider \cite{Sc} and  the survey paper by Gardner \cite{G}.
Here $\mbox{Vol}(\cdot)$ denotes the usual volume in $\mathbb{R}^{n}$ and the summation $(\lambda E_1+(1-\lambda) E_2)$ should be understood as a vector sum(called \textit{Minkowski addition}). \eqref{BMVolume} says that $\left[\mbox{Vol}(\cdot)\right]^{1/n}$ is a concave function with respect to  Minkowski addition.  
Inequalities of Brunn-Minkowski type  have also been proved for  other  homogeneous functionals. For example, one can replace  volume in \eqref{BMVolume} by ``capacity'' and in this case it was shown by Borell in \cite{B1} that
\begin{align}
\label{BMcapacity2}
\left[\mbox{Cap}_{2}(\lambda E_1+(1-\lambda) E_2)\right]^{\frac{1}{n-2}}\geq\lambda \left[\mbox{Cap}_{2}(E_{1})\right]^{\frac{1}{n-2}}+(1-\lambda)\left[\mbox{Cap}_{2}(E_2)\right]^{\frac{1}{n-2}}
\end{align}
whenever $E_1, E_2$ are compact convex sets with nonempty interiors in $\mathbb{R}^{n}$, $n\geq 3$. Here $\mbox{Cap}_{2}$ denotes the \textit{Newtonian capacity}.     The  exponents in  this inequality and \eqref{BMVolume} differ  as  $\mbox{Vol}(\cdot)$ is homogeneous of degree $n$ whereas $\mbox{Cap}_{2}(\cdot)$ is homogeneous of degree $n-2$. In \cite{B2}, Borell proved a Brunn-Minkowski type inequality for   \textit{logarithmic capacity}. The equality case in \eqref{BMcapacity2} was studied by Caffarelli, Jerison and Lieb in \cite{CJL} and it was shown that equality in \eqref{BMcapacity2} holds if and only if $E_2$ is a translate and dilate of $E_1$ when $n\geq 3$. Jerison in \cite{J} used  that result  to prove  uniqueness in the \textit{Minkowski problem} (see section \ref{section8} for the  Minkowski problem). In \cite{CS} Colesanti and Salani  proved the $p$-capacitary version of \eqref{BMcapacity2} for $1<p<n$. That is, 
\begin{align}
\label{BMcapacityp}
\left[\mbox{Cap}_{p}(\lambda E_1+(1-\lambda) E_2)\right]^{\frac{1}{n-p}}\geq\lambda \left[\mbox{Cap}_{p}(E_{1})\right]^{\frac{1}{n-p}}+(1-\lambda)\left[\mbox{Cap}_{p}(E_2)\right]^{\frac{1}{n-p}}
\end{align}
whenever $E_1, E_2$ are compact convex sets with nonempty interiors  in $\mathbb{R}^{n}$, and $\mbox{Cap}_{p}(\cdot)$ denotes the $p$-capacity of a set defined as
\[
\mbox{Cap}_{p}(E)=\inf\left\{\, \, \int_{\mathbb{R}^{n}} |\nabla v|^{p} dx:\, \, v\in C^{\infty}_{0}(\rn{n}),\, v(x)\geq  1\, \, \mbox{for}\, \,  x\in E \right\}.
\] 
It was also shown in the same paper that equality in \eqref{BMcapacityp} holds if and only if $E_2$ is a translate and dilate of $E_1$. In  \cite{CC},  Colesanti and Cuoghi  defined a   logarithmic capacity for  $ p = n, n \geq 3,  $  and  proved  a  Brunn-Minkowski type inequality for this capacity.  In \cite{CNSXYZ},  a Minkowski problem was studied for $p$-capacity, $1 < p < 2,$  using \eqref{BMcapacityp}. See \cite{C} for the torsional rigidity  and  first eigenvalue of  the Laplacian  versions of \eqref{BMVolume}. 

\setcounter{equation}{0} 
\setcounter{theorem}{0}
\section{Notation and statement of results}
\label{NSR}
\noindent Let $n\geq 2$ and points in Euclidean $ n
$-space $ \rn{n} $ be
denoted by $ y = ( y_1,
 \dots,  y_n) $.   $ \mathbb S^{n-1} $  will denote the unit sphere in $\mathbb R^n$. We write  $ e_m,  1 \leq m \leq n,  $  for the point in  $  \rn{n} $  with  1 in the $m$-th coordinate and  0  elsewhere. Let  $ \bar E,
\ar E, $  $\mbox{diam}(E), $ be the closure,
 boundary, diameter,  of the set $ E \subset
\mathbb R^{n} $  and  we define $ d ( y, E ) $   to be the distance from
 $  y \in \mathbb R^{n} $ to $ E$.  Given two sets, $ E, F  \subset \rn{n} $   let  
  \[ 
  d_{\mathcal{H}} ( E, F ) = \max  (  \sup   \{ d ( y, E ) : y \in F \},   \sup \{ d ( y, F ) : y \in E \} ) 
  \]  
  be the Hausdorff distance between the sets $ E, F \subset\mathbb R^n. $ %
 Also 
\[
 E  +  F  = \{ x + y:  x \in E, y \in F   \}
 \]  
 is  the  Minkowski sum of  $ E $ and $F.$ 
 We write  $  E  + x  $  for  $  E  +  \{x\} $ and set 
$  \rho E =  \{ \rho y : y \in E\}. $    Let  $   \lan \cdot ,  \cdot  \ran $  denote  the standard inner
product on $ \mathbb R^{n} $ and  let  $  | y | = \lan y, y \ran^{1/2} $ be
the  Euclidean norm of $ y. $   Put 
\[
B (z, r ) = \{ y \in \mathbb R^{n} : | z  -  y | < r \}\quad  \mbox{whenever}\, \,z\in \mathbb R^{n}, \, r>0, 
\]
and  $dy$ denote  Lebesgue $ n $-measure on    $ \mathbb R^{n} $.  
Let  $\mathcal{H}^{k},  0 <  k  \leq  n, $  denote $k$-dimensional  \textit{Hausdorff measure} on $ \rn{n}$ defined by 
\[
\mathcal{H}^{k}(E)=\lim_{\delta\to 0} \inf\left\{\sum_{j} r_j^{k}; \, \, E\subset\cup B(x_j, r_j), \, \, r_j\leq \delta\right\}
\] 
where infimum is taken over all possible cover $\{B(x_j, r_j)\}_{j}$ of set $E$. 
If $ O  \subset \mathbb R^{n} $ is open and $ 1  \leq  q  \leq  \infty, $ then by   $
W^{1 ,q} ( O ) $ we denote the space of equivalence classes of functions
$ h $ with distributional gradient $ \nabla h= ( h_{y_1},
 \dots, h_{y_n} ), $ both of which are $q$-th power integrable on $ O. $  Let  
 \[
 \| h \|_{1,q} = \| h \|_q +  \| \, | \nabla h | \, \|_{q}
 \]
be the  norm in $ W^{1,q} ( O ) $ where $ \| \cdot \|_q $ is
the usual  Lebesgue $ q $ norm  of functions in the Lebesgue space $ L^q(O).$  Next let $ C^\infty_0 (O )$ be
 the set of infinitely differentiable functions with compact support in $
O $ and let  $ W^{1,q}_0 ( O ) $ be the closure of $ C^\infty_0 ( O ) $
in the norm of $ W^{1,q} ( O  ). $  By $ \nabla \cdot $ we denote the divergence operator.
\begin{definition}  
\label{defn1.1}	
	Let $p,  \al \in (1,\infty) $ and 
	\[
	\mathcal{A}=(\mathcal{A}_1, \ldots, \mathcal{A}_n) \, : \, \rn{n}\sem \{0\}  \to \rn{n},
\]
such that $  
\mathcal{A}= \mathcal{A}(\eta)$  has  continuous  partial derivatives in $ \eta_k, 1  \leq k  \leq n, $  on $\rn{n}\setminus\{0\}.$  We say that the function $ \mathcal{A}$ belongs to the class
  $ M_p(\alpha)$ if the following conditions are satisfied whenever  $\xi\in\mathbb{R}^n$ and
$\eta\in\mathbb R^n\setminus\{0\}$:
	\begin{align*}
		(i)&\, \, \alpha^{-1}|\eta|^{p-2}|\xi|^2\leq \sum_{i,j=1}^n \frac{
		\partial  \mathcal{A}
_{i}}{\partial\eta_j}(\eta)\xi_i\xi_j   \mbox{ and }  \sum_{i,j=1}^n  \left| \frac{
		\partial  \mathcal{A}_i}{\ar\eta_ {j}} \right|   \leq\alpha |\eta|^{p-2},\\
		(ii)&\, \,  \mathcal{A} (\eta)=|\eta|^{p-1}  \mathcal{A}
(\eta/|\eta|).
	\end{align*}
	\end{definition}
We  put  $ \mathcal{A}(0) = 0 $  and note that  Definition \ref{defn1.1}  $(i), (ii) $ implies   
 \begin{align}  
 \label{eqn1.1}
 \begin{split}
 c^{-1}  (|\eta | + |\eta'|)^{p-2} \,   |\eta -\eta'|^2   \leq   \lan  \mathcal{A}(\eta) -& 
\mathcal{A}(\eta'), \eta - \eta' \ran \\
&\leq c |\eta-\eta'|^2   (|\eta|+|\eta'|)^{p-2}
\end{split}
\end{align}  
whenever $\eta, \eta'   \in  \rn{n} \sem \{0\}$.    
\begin{definition}
\label{defn1.2}                       
	Let $p\in (1,\infty)$ and let $ \mathcal{A}\in M_p(\alpha) $ for some $\alpha$. Given an  open set 
 $ O  $ we say that $ u $ is $  \mathcal{A}
$-harmonic in $ O $ provided $ u \in W^ {1,p} ( G ) $ for each open $ G $ with  $ \bar G \subset O $ and
	\begin{align}
	\label{eqn1.2}
		\int \lan    \mathcal{A}
(\nabla u(y)), \nabla \he ( y ) \ran \, dy = 0 \quad \mbox{whenever} \, \,\he \in W^{1, p}_0 ( G ).
			\end{align}
	 We say that $  u  $ is an  $\mathcal{A}$-subsolution ($\mathcal{A}$-supersolution) in $O$ if   $  u \in W^{1,p} (G) $ whenever $ G $ is as above and  \eqref{eqn1.2} holds with
$=$ replaced by $\leq$ ($\geq$) whenever $  \theta  \in W^{1,p}_{0} (G )$ with $\theta \geq 0$.  As a short notation for \eqref{eqn1.2} we write $\nabla \cdot \mathcal{A}(\nabla u)=0$ in $O$.  
\end{definition}   
 More about PDEs of this generalized type can be found in \cite[Chapter 5]{HKM} and 
\cite{A,ALV, LLN,LN4}. If $ \mathcal{A} ( \eta ) =   | \eta|^{p - 2} ( \eta_1, \dots, \eta_n ), $ and $u$ is a weak  solution relative to this  $ \mathcal{A} $ in $O, $   then $u$ is said to be $p$-harmonic in $ O. $
\begin{remark}    
\label{rmk1.3}
We remark  for  $  O, \mathcal{A}, p,  u,$  as in  Definition \ref{defn1.2}  that if   $F:\mathbb R^n\to \mathbb R^n$  is  the composition of
a translation, and a dilation
 then 
 \[
 \hat u(z)=u(F(z))\, \, \mbox{whenever}\, \, F(z)\in O\, \,  \mbox{is}\, \,  \mathcal{A}\mbox{-harmonic in} \, \,  F^{-1}(O).
 \]
 Moreover,  if   $ \ti  F:\mathbb R^n\to \mathbb R^n$  is  the  composition of
a translation,  a dilation, and a rotation 
   then 
   \[
   \ti  u (z) = u (\ti F(z) )\, \, \mbox{is}\, \,  \ti {\mathcal{A}}\mbox{-harmonic in}\, \,  \ti F^{-1}(O)\, \,  \mbox{and}
  \, \, \ti{\mathcal{A}}\in M_p(\alpha).
  \]  
\end{remark} 
 \noindent We shall use this remark numerous times in our proofs.

Let $E \subset  \mathbb{R}^n  $ be a  compact convex set  and let $\Omega= \rn{n} \setminus E$. 
Using \eqref{eqn1.1},  results in  \cite[Appendix 1]{HKM}, as well as Sobolev type  limiting  arguments,  we show in Lemma  \ref{lemma3.1} that  if  $  \mbox{Cap}_{\mathcal{A}} (E)  > 0 $,
or equivalently  $\mathcal{H}^{n-p} ( E ) = \infty,  $   then  there  exists a
 unique  continuous function $ u \not  \equiv 1,  0 < u \leq  1,  $ on $ \rn{n} $    satisfying 
\begin{align}
\label{eqn1.3}
\begin{split} 
(a)&\, \, u\, \, \mbox{is}\, \,  \mathcal{A}\mbox{-harmonic in}\, \,  \Om,  \\
(b)&\, \,  u \equiv 1\, \, \mbox{on}\, \,  E,     \\
 (c)&\, \,  | \nabla u | \in   L^p ( \rn{n} )\, \,  \mbox{and}\, \, u \in L^{p^*}  ( \rn{n}  )\, \, \mbox{for}\, \,  p^* =  \frac{np}{n-p}. 
  \end{split}
  \end{align}
 We put   \[ 
 \mbox{Cap}_\mathcal{A}  (E)  = \int_\Om \lan   \mathcal{A}  ( \nabla u ) , \nabla u  \ran \,  dy
 \]    
 and  call   $ \mbox{Cap}_\mathcal{A}  (E), $ the \textit{$\mathcal{A}$-capacity of $ E $} while  $u$ is the $ \mathcal{A}$-capacitary function corresponding to  $ E $ in  $\Om. $   We note that this definition is a slight extension of the usual definition of \emph{capacity}. However in case,  
 \[
 \mathcal{A}  ( \eta )  =  p^{-1} \nabla f (\eta )\quad   \mbox{on}\, \,  \rn{n} \sem \{0\}
 \]
 then from $ p  - 1 $ homogeneity in     Definition \ref{defn1.1} $(ii) $  it  follows that  
 \[
 f (t \eta )  = t^p  f (\eta ) \quad \mbox{whenever} \, \, t  > 0 \, \, \mbox{and}\, \, \eta \in \rn{n} \sem \{0\}.
 \]
In this case, using Euler's  formula,  one gets the usual definition of capacity relative to $f. $ That is,   
\[   
\mbox{Cap}_\mathcal{A}  ( E )  =   \inf \left\{  \int_{ \rn{n} }  f   (\nabla \psi (y) )   dy  \, :  \psi  \in C_0^\infty ( \rn{n} ) 
\mbox{ with }  \psi  \geq  1    \mbox{ on }   E \right\}. 
\]    
 See chapter 5 in  \cite{HKM} for   more about this definition of capacity in terms of such $f$. In   case  
 \[
 \mathcal{A}   (\eta)  =   | \eta|^{p - 2} ( \eta_1, \dots, \eta_n )
 \] 
 so  we have $ f (\eta) = p^{-1}  |\eta|^p $   then the  above capacity  will be denoted by  $\mbox{Cap}_p(E)$ and   called the $p$-capacity of $ E.$ Note  from  \eqref{eqn1.1}  with  $ \eta' = 0, $   that  
\begin{align}
\label{eqn1.4}    
c^{-1}  \mbox{Cap}_p (E)     \leq   \mbox{Cap}_\mathcal{A}  (E)    \leq    c \, \mbox{Cap}_p (E)
\end{align} 
where  $ c $ depends only on $ \al, p,  $ and $n. $  
     From  Remark \ref{rmk1.3} and uniqueness of  $ u $  in \eqref{eqn1.3},  we observe for  $ z  \in \rn{n}$ and  $\rho > 0, $   that if $ \ti  E   =  \rho E + z , $ then   
     \begin{align}
     \label{eqn1.5}
     \begin{split}
 (a')& \, \,   \mbox{Cap}_\mathcal{A}  (\rho E + z ) = \rho^{n-p} \,   \mbox{Cap}_\mathcal{A}  (E) ,  
 \\
 (b') &\, \,    \ti u  (  x  )   =  u (   ( x - z )/\rho  ),\,\, \mbox{for}\, \, x  \in \rn{n} \sem \ti E, \, \,    \mbox{is the  $ \mathcal{A}$-capacitary 
function for  $  \ti  E$}.
\end{split}
\end{align}  
Observe from \eqref{eqn1.5} $(a')$  that for $ z \in \rn{n}$ and  $R > 0, $ 
\begin{align}
\label{eqn1.6} 
\mbox{Cap}_\mathcal{A} (B(z, R ) ) = c_1  R^{n-p}
\end{align}
  where $c_1$ depends only on $ p, n, \al.$
  
In the first part of this article, we prove  the following  Brunn-Minkowski type theorem for $  \mathcal{A}$-capacities:
\begin{mytheorem}   
\label{theorem1.4}
 Let   $E_1,  E_2 $ be  compact convex sets  in $\rn{n} $  satisfying $ \mbox{Cap}_{\mathcal{A}} (E_i)   > 0 $ for  
$i  = 1, 2$.  If   $1 < p < n $ is fixed, $ \mathcal{A} $ is as in Definition \ref{defn1.1}, and   $ \la  \in  [0,1], $  then   
\begin{align}
\label{eqn1.7}
\left[\mbox{Cap}_\mathcal{A}  ( \la  E_1 + (1-\la) E_2 )\right]^{\frac{1}{(n-p)}}  \geq  \la  \, 
\left[\mbox{Cap}_\mathcal{A}  (  E_1 )\right]^{\frac{1}{(n-p)}}  +  (1-\la)  \left[\mbox{Cap}_\mathcal{A}  (E_2 )\right]^{\frac{1}{(n-p)}}.
\end{align} 
If  equality holds in  \eqref{eqn1.7} and      
\begin{align} 
\label{eqn1.8}
\begin{split}
(i) &\, \, \mbox{There exists} \, \, 1\leq  \La <\infty\ \, \mbox{such that}\, \, \left| \frac{ \ar \mathcal{A} _i }{\ar \eta_j}   ( \eta )  -  \frac{ \ar \mathcal{A} _{i} }{ \ar\eta'_j} ( \eta') \right|\leq  \La  \, | \eta  -   \eta' |    |\eta | ^{p-3}
 \\ 
 &\hs{.4in}   
 \mbox{whenever}\, \, 0 < {\ts \frac{1}{2}}\,  |\eta |   \leq   |\eta'|  \leq 2     |\eta |\, \, \mbox{and}\, \,   1 \leq  i ,j\leq n, \\
(ii)&\, \,  \mathcal{A}_i   (\eta)  =   \frac{\ar f}{\ar \eta_i} \, \,\mbox{for}\, \,  1 \leq i  \leq n  \, \, \mbox{where}\, \,   f (t \eta )  = t^p  f  (\eta ) \, \, \mbox{when}\, \,  t > 0\, \, \mbox{and}\, \, \eta \in \rn{n}  \sem \{0\}
\end{split}
\end{align}
    then $  E_2 $ is a translation and dilation of  $ E_1. $  
    \end{mytheorem} 

To briefly outline the proof of  Theorem  \ref{theorem1.4}, in section \ref{section2} we list some basic properties of $\mathcal{A}$-harmonic functions which will be used in the proof of  Theorem \ref{theorem1.4}.  We then use these properties in  sections \ref{section3}  and  \ref{section4}  to prove  inequality  \eqref{eqn1.7}.  The last sentence in Theorem \ref{theorem1.4}  regarding the  case of  equality in  \eqref{eqn1.7}  is proved in  section \ref{section5}. 
 
  As  for  the main steps in  our  proof,  after  the  preliminary  material,  we  show in  Lemma \ref{lemma3.4}  that  if  $ u $  is  a nontrivial  $  \mathcal{A}$-harmonic  capacitary function  for a  compact  convex set  $  E, $  then 
 $  \{ x : u (x)  > t \} $  is convex whenever  $ 0 < t < 1. $  The proof   uses  a  maximum  principle type argument  of Gabriel in \cite{Ga}  to  show that  if  $ u $  does not have levels  bounding  a  convex domain, then  a  certain function has  an  absolute maximum in  $  \rn{n}  \sem E, $  from which one  obtains a  contradiction.     This argument   was  later used  by  the fourth named author of this article in \cite{L}   in  the $p$-Laplace  setting  and also a   variant of  it   was used by Borell in \cite{B1}  (see \cite{BLS} for  recent  applications).      After  proving   Lemma  \ref{lemma3.1}  we use an  analogous  argument to prove  \eqref{eqn1.7}.   Our   proof  of  equality in  Theorem  \ref{theorem1.4}  is   inspired  by   the  proof  in  \cite{CS}  which  in  turn  uses  some ideas  of Longinetti in \cite{Lo}.   In particular, Lemma 2  in  \cite{CS}  plays an important role in our proof.  Unlike these authors though,  we  do not  convert the  PDE  for  $ u_1,  u
 _2  $  into  one  for  the  support  functions of  their   levels,   essentially because our  PDE  is  not rotationally  invariant.     The   arguments  we use require  a priori knowledge that the  levels  of $ u_1,  u_2  $    have  positive curvatures.     We  can show  this  near  $ \infty  $  when   $ \mathcal{A} 
 =  \nabla f , $  as in  Theorem \ref{theorem1.4}, by comparing  $ u_1, u_2 $ with their respective  ``fundamental solutions''  (see Lemma \ref{lemma5.1})  which can be calculated more or less directly.     A unique continuation argument then gives Theorem \ref{theorem1.4}.    This argument does not work  for a  general $ \mathcal{A}. $  In this  case   a method first used by Korevaar in \cite{K}  and  after that  by various authors  (see \cite{BGMX}) appears promising, although rather tedious and at the expense of assuming more regularity on $  \mathcal{A}$ for handling the case of  equality in   Theorem \ref{theorem1.4}. Finally, we  mention that  our main purpose in working on the Brunn-Minkowski inequality is to  prepare  a  background  for   our   investigation of  a Minkowski problem  when  $  \mathcal{A} =  \nabla f  $  and  $ 1 < p < n$ (see Theorem \ref{mink} in section \ref{section8}).   
 \setcounter{equation}{0} 
 \setcounter{theorem}{0}
\section{Basic estimates for $\mathcal{A}$-harmonic functions}
\label{section2}
 
In this section we  state  some fundamental estimates for 
   $\mathcal{A}$-harmonic functions. 
Concerning constants, unless otherwise stated, in this section, and throughout the paper,
$ c $ will denote a  positive constant  $ \geq 1$, not
necessarily the same at each occurrence, depending at most on
 $ p, n, \alpha,  \La $ 
   which sometimes we refer to as depending on the data.
In general,
$ c ( a_1, \dots, a_m ) $  denotes  a positive constant
$ \geq 1, $  which may depend at most on the data  and   $ a_1, \dots, a_m, $
not necessarily the same at each occurrence. If $B\approx C$ then $B/C$ is bounded from above and below by
constants which, unless otherwise stated, depend at most on  the data.
 Moreover, we let
$ { \ds \max_{F}   \ti u, \,  \min_{F} \ti  u } $ be the
  essential supremum and infimum  of $  \ti u $ on $ F$
whenever $ F \subset  \mathbb R^{n} $ and whenever $ \ti u$ is defined on $ F$.

\begin{lemma}
\label{lemma2.1}
Given $ p, 1 < p < n, $ assume that $\ti {\mathcal{A}} \in M_p(\alpha)$ for some
$\alpha  >  1.$    Let
 $ \ti u $  be  a
positive $\ti {\mathcal{A}}$-harmonic function in $B (w,4r)$, $r>0$.Then
	\begin{align}
	\label{eqn2.1}
	\begin{split}
		(i)&\, \, r^{ p - n} \,\int_{B ( w, r/2)} \, | \nabla \ti  u |^{ p } \, dy \, \leq \, c \, (\max_{B ( w,
r)} \ti  u)^p, \\
		(ii)&\, \, \max_{B ( w, r ) } \,\ti   u \, \leq c \min_{ B ( w, r )} \ti u.
		\end{split}
\end{align}	
	Furthermore, there exists $\ti \sigma=\ti \sigma(p,n,\alpha)\in(0,1)$ such that if $x, y \in B ( w, r )$, then
	\begin{align*}
		(iii)&\ \ | \ti  u ( x ) - \ti  u ( y ) | \leq c \left( \frac{ | x - y |}{r} \right)^{\ti \sigma} \, \max_{B (
w, 2 r )}  \, \ti  u.
	\end{align*}
   \end{lemma}
\begin{proof}
A proof of this lemma can be found in  \cite{S}.
\end{proof}
\begin{lemma}  
\label{lemma2.2}	
Let $p,n, \ti {\mathcal{A}}, \alpha, w,r, \ti u$ be as in Lemma \ref{lemma2.1}.
Then	 $ \ti u$ has a  representative locally in  $ W^{1, p} (B(w, 4r)),$    with H\"{o}lder
continuous partial derivatives in $B(w,4r) $  (also denoted $\ti u$), and there exists $ \ti \be \in (0,1],  c \geq 1 $,
depending only on $ p, n, \alpha$, such that if $ x, y \in B (  w,  r ), $    then  
	\begin{align}
	\label{eqn2.2}
	\begin{split}
&	(\hat a) \, \,  c^{ - 1} \, | \nabla \ti u ( x ) - \nabla \ti u ( y ) | \, \leq \,
 ( | x - y |/ r)^{\ti \be }\, \max_{B (  w , r )} \, | \nabla \ti u | \leq \, c \,  r^{ - 1} \, ( | x - y |/ r )^{\ti \be }\, \ti u (w). \\
 & (\hat b)  \, \,  \int_{B(w, r) }  \, \sum_{i,j = 1}^n  \,  |\nabla \ti u |^{p-2} \,  |\ti  u_{x_i x_j}|^2  dy 
\leq   c r^{(n-p-2)} \ti{u}(w).
  \end{split}
\end{align}
  	If 
  	\[ 
  	\ga\, r^{-1} \ti u   \leq |\nabla  \ti u | \leq  \ga^{-1} r^{-1}    \ti u  \quad \mbox{on}\quad  B ( w, 2r)  
\] 
for some $ \ga \in (0,1) $  and 
 \eqref{eqn1.8} $ (i)$ holds then 
 $\ti u $ has    H\"{o}lder
continuous second  partial derivatives in $B(w,r) $  and
 there exists $\ti \he \in (0,1), \bar c \geq 1, $ depending only on  the data  and  $ \ga $  such that
  \begin{align}  
  \label{eqn2.3}
  \begin{split}
  \left[ \sum_{i,j = 1}^n  \,   ( \ti  u_{x_i x_j } (x)      -  \ti  u_{y_i y_j } (y) ) ^2  \, \right]^{1/2} \, &\leq  \, \bar c  ( | x - y |/ r )^{\ti \he } \,  \max_{ B(w, r)}\,  \left(\sum_{i,j = 1}^n  \,   |\ti  u_{x_i x_j}| \right)     \\ 
  &\leq  \, \bar c^2  r^{-n/2}  ( | x - y |/ r )^{\ti \he } \,   \left(\sum_{i,j = 1}^n  \,   \int_{ B ( w, 2r) }   \ti  u_{x_i x_j}^2   dx \right) ^{1/2} \\
  &  \leq  \bar c^3 \,  r^{ - 2} \, ( | x - y |/ r )^{\ti \he }\, \ti u (w).
\end{split}  
  \end{align}
  whenever   $ x, y \in B (w,  r/2) $. 
   \end{lemma}    
   \begin{proof}
   A proof of \eqref{eqn2.2} can be found  in \cite{T}. Also, 
\eqref{eqn2.3}  follows from  \eqref{eqn2.2}, the added assumptions,  and Schauder type
estimates (see  \cite{GT}).
\end{proof}
\begin{lemma} 
\label{lemma2.3}
Fix  $ p, 1 < p < n, $ assume that $\ti {\mathcal{A}}\in M_p(\alpha),$  and let 
 $ \ti E \subset B (0, R) $, for some $R>0$,  be a compact  convex set with  $   \mbox{Cap}_{\mathcal{A}} ( \ti E ) > 0$.   
Let   $\zeta \in C_0^\infty ( B (0, 2 R)   ) $ with  $ \ze  \equiv 1 $  on  $  B (0, R). $  If $  0  \leq  \ti u $ is   $ \ti {\mathcal{A}}$-harmonic in  $ B (0, 4R) \sem \ti E,$    and  $  \ti u \ze \in  W_0^{1,p} (B (0,4R) \sem \ti E ),$    then  $  \ti u $  has a continuous  extension to  $ B (0, 4R) $  obtained by putting  $ \ti  u  \equiv 0 $ on  $   \ti  E. $  Moreover, if   $  0  < r  <    R$ and   $w \in \ar  \ti E$ then    
	 \begin{align}
	 \label{eqn2.4} 		
	 (i) \, \, r^{ p - n} \, \int\limits_{B ( w, r)} \, | \nabla \ti  u |^{ p } \, dy \, \leq \, c \, \left( \max_{B ( w, 2 r)}  \ti u \right)^p.   
	\end{align} 
	Furthermore,  there exists  
$\hat \sigma= \hat \sigma (p,n,\alpha, \ti E )\in(0,1)$  such that if $ x, y \in B ( w, r )$ and $ 0 < r <\mbox{diam}(\ti  E) $ then 
	\begin{align*}
		(ii)\, \, | \ti  u ( x ) - \ti  u ( y ) | \leq c \left(  \frac{ | x - y |}{r} \right)^{ \hat \sigma} \,  \max_{ B (w, 2 r )}\, \ti u.
	\end{align*}
 \end{lemma}  
 
\begin{proof} Here $(i) $ is a standard  Caccioppoli  inequality.  To prove   $(ii)$  we note that necessarily      $\mathcal{H}^{n-p} (\ti E ) = \infty, $ as  follows from  \eqref{eqn1.4}  and 
Theorem  2.27 in  \cite{HKM}.  From this note,  as well as convexity and compactness of  $ \ti E,  $   we  deduce that 
\[
\mathcal{H}^{l} ( B ( y, r )  \cap \ti E)  \approx  r^l 
\] 
 for some positive integer $ l  > n - p,  $  whenever  $ y  \in  \ti E$  and  $  0 <  r  <\mbox{diam}(\ti  E). $  Constants depend on $\ti E $ but are independent of $ r, z. $  Using this fact and metric properties of  certain capacities in  chapter 2  of  \cite{HKM},  it follows that  
\begin{align}
\label{eqn2.5}
  \mbox{Cap}_p ( B (y, r ) \cap \ti E )  \approx  r^{n-p} \quad \mbox{whenever} \, \, 0  < r  < \mbox{diam}(\ti  E) \mbox{ and } y  \in \ti E
  \end{align} 
  where constants depend  on $ \al, p, n $ and $ \ti E $.    Now $(ii) $  for  $ y \in  \ti E $   
 follows from \eqref{eqn2.5} and   essentially Theorem 6.18  in  \cite{HKM}.  Combining this fact with  
  \eqref{eqn2.1}  $ (iii)$ we now obtain  $ (ii). $   	    
\end{proof}
\begin{lemma} 
\label{lemma2.4}
Let  $   \ti {\mathcal{A}} , p, n, \ti E,  R, \ti u   $ be  as in  Lemma \ref{lemma2.3}. 
	 Then there exists a unique finite positive Borel measure $ \ti \mu $ on $ \mathbb
R^n$, with support contained in $ \ti E $ such that if   $\ph  \in C_0^\infty ( B(0,2R) )$ then
 \begin{align}
 \label{eqn2.6}
		(i)\, \, \int  \langle \ti {\mathcal{A}} (\nabla \ti u(y)),\nabla\ph(y)\rangle \, dy \, = \, -  \int \, \ph \, d
\ti{\mu}.
	     \end{align} 	 
Moreover, if  $  0 <  r  \leq  R$ and  $w \in \ar \ti E $  then  there exists $ c \geq 1, $ depending only on  the data  such that  
	 \begin{align*}
		(ii)\, \,  r^{ p - n} \ti \mu (B( w,  r ))\leq  c  \max_{B(w, 2r)} \ti u^{ p - 1}.
	\end{align*} 
\end{lemma} 
\begin{proof}
 For the proof  of  $(i)$ see  Theorem 21.2   in  \cite{HKM}.  $ (ii) $  follows from   \eqref{eqn1.1}  with  $ \eta'  = (0, \dots,0), $   H\"{o}lder's inequality, and  \eqref{eqn2.4} $ (i)$  using 
a  test function, $ \ph, $ with  $ \ph \equiv 1 $ on $ \ti E. $   
\end{proof}
\setcounter{equation}{0} 
\setcounter{theorem}{0}

\section{Preliminary reductions for the proof of  Theorem \ref{theorem1.4}}  
\label{section3}
Throughout this section we assume that  $ E $ is a  compact convex set with  $0 \in E$,  $ \mbox{diam}(E) = 1$, and $ \mbox{Cap}_{\mathcal{A}} (E)  >  0. $  
We begin with  
\begin{lemma} 
\label{lemma3.1}
For  fixed $p,  1 < p <n, \, $  there exists a unique locally H\"{o}lder continuous $ u $ on $ \rn{n} $  satisfying  \eqref{eqn1.3}.     
\end{lemma} 
\begin{proof}   
  Given  a positive integer $ m \geq 4 $, let $ u_ m$  be  the $  \mathcal{A}$-harmonic function in  $ B (0, m ) \sem E $   with  $u_m $ in  $ W_0^{1,p} ( B (0, m ) ) $ and  $ u_m = 1  $  on  $ E $ in the $ W^{1,p} $  Sobolev sense.   Existence of  $ u_m $ is proved in   \cite[Corollary 17.3, Appendix 1]{HKM}.    From   Lemma \ref{lemma2.3} with  
  \[
     \mathcal{\ti  A} ( \eta ) = - \mathcal{A} (- \eta ) \quad \mbox{whenever} \, \, \eta \in  \mathbb{R}^{n},  
     \]
      and  $ \ti u  = 1 -  u_m $   we see that   $ u_m  $  has  a  H\"{o}lder continuous extension to  $ B (0, m)$  with  $ u_m   = 1 $  on  $E$.   
From Sobolev's theorem,  \eqref{eqn1.1},   and  results for certain $p$  type capacities from   \cite{HKM}  we see  there exists $ c = c (p, n)  $ such that     
 if  $ p^*  = \frac{np}{n-p} $  then 
 \begin{align}
 \label{eqn3.1} 
  \| u_m \|_{p^*} \,  \leq \,  c  \, \| |\nabla u_m|  \|_p  \,   \leq \, c^2 
  \end{align} 
  where the norms are relative to  
$  L^q (\rn{n}), q \in  \{ p^*, p \}. $    
 From  Lemmas \ref{lemma2.1},  \ref{lemma2.3} we see that 
$ u_m $  is  locally  H\"{o}lder  continuous on compact subsets of   $ B (0, m ) $  with  exponent  and constant that is independent of  $ m $  while  $  \nabla u_m $ is  locally  H\"{o}lder continuous on compact subsets of  $  B (0, m )  \sem  E $  again with  exponent  and constant that is independent of  $ m. $  Using  these facts and  Ascoli's theorem we see there exists a subsequence  $ \{u_{m_k}\} $  of  $ \{u_m\}$  with
\begin{align*}
 \{u_{m_k},  \nabla  u_{m_k}\} \, \,  &\mbox{converging to}\, \,   \{u,  \nabla u \}  \, \, \, \mbox{as}\, \, m_k\to\infty \\
 &\mbox{uniformly on compact subsets of  $ \rn{n}, \rn{n}  \sem  E,  $  respectively.}
\end{align*}
 From this fact, the  Fatou's lemma,   and  Definition \ref{defn1.2}  we see that  $ u $ is continuous on $ \rn{n}, $   $ \mathcal{A}$-harmonic in 
$  \rn{n}  \sem  E $ with  $ u  \equiv  1 $ on  $  E , $  and   \eqref{eqn3.1} holds  with  $ u_m $  replaced  by 
$ u $.  Thus   $ u $  satisfies  \eqref{eqn1.3}.  

To prove uniqueness we note  from  Harnack's inequality in   \eqref{eqn2.1} $ (ii) $  that for 
$ | x | \geq  2, $ 
\begin{align}
\label{eqn3.2}   
u  (  x  )^{p^*}   \,    \leq  \, \ti c \,  | x |^{-n}   \, \int_{\rn{n}  \sem  B (x, |x|/2 ) } \, u^{p^*}  d y \,  \leq \, \ti c^2  \,  |x|^{-n}
\end{align} 
where $  \ti c $ has the same dependence as the constant  in \eqref{eqn3.1}.   If  $ v $ also  satisfies  \eqref{eqn1.3}, then  \eqref{eqn3.2} holds with $u$ replaced by  $ v $ so from the  usual  Sobolev type limiting arguments we see for each $ \ep > 0 $ that  
$\he  =   \max ( |u - v| - \ep,  0 ) $  can  be  used as  a  test function  in  \eqref{eqn1.2}  for  $ u$ and $v $.  Doing this and using  \eqref{eqn1.1}, it follows for some $ c \geq 1, $ depending only on the data that 
\begin{align}
\label{eqn3.3}
\begin{split}
  \int_{ \{ | u - v| > \ep \} } ( |\nabla u | + |\nabla v | )^{p-2}| \ \nabla u  - \nabla v |^2 dy 
&\leq  c  \int_{  \rn{n} \sem E }  \lan  \mathcal{A}   ( \nabla u )  -  \mathcal{A}  (\nabla v ),  \nabla \he \ran \, dy \\ &= 0.  
\end{split}
\end{align}       
Letting  $ \ep \to   0 $  we  conclude first from  \eqref{eqn3.3}  that  $ u  - v $ is constant on  $ \Om $ and then from  \eqref{eqn1.3} $(b)$  that  $ u  \equiv  v. $  
\end{proof} 

Throughout the rest of this section, we assume  $ u $ is the $\mathcal{A}$-capacitary function for $  E $  and a fixed $ p, 1  < p < n. $   Let  $ \mu $  be the measure associated with  $ \ti u  = 1 - u $, where  $\ti u$ is   $\ti {\mathcal{A}} ( \eta  )  = - \mathcal{A}  ( - \eta )$-harmonic,  as in Lemma \ref{lemma2.4}.  
 Next we prove  
\begin{lemma} 
\label{lemma3.2}
For $\mathcal{H}^1$ almost every  $ t \in (0,1) $
\begin{align}  
\label{eqn3.4}
\begin{split}
(a) \, \, &\mu ( E) =  \mbox{Cap}_\mathcal{A}  (E) = \int_{ \{u=t\} }
\lan  \mathcal{A}  ( \nabla u (y) ), \nabla u (y) /|\nabla u (y)|  \ran  \, d \mathcal{H}^{n-1}  \\  
& \hspace{3.1cm}= t^{-1} \int_{ \{u <t\} } \lan  \mathcal{A}  ( \nabla u (y) ), \nabla u (y) \ran  \, dy.
 \\
(b)\, \,  &  \mbox{There exists $ c \geq 1, $ depending only on $ p, n, \al,    $ so that} \\   &\hs{.5in}  c^{-1}  \mbox{ Cap}_\mathcal{A}  (E) |x|^{(p-n)}  \leq   u  ( x)^{p-1} 
 \leq  c |x|^{(p-n)} \mbox{Cap}_\mathcal{A}  (E)  \quad \mbox{whenever}\, \, |x| \geq 2.
\end{split}
\end{align}
  \end{lemma}    

  \begin{proof}  To prove \eqref{eqn3.4} $(a)$   fix $  R \geq  4 $   and  let  $ 0 \leq \psi \in  C_0^\infty ( B (0, 2R ) )$  with  $ \psi  \equiv 1 $ on $  B (0, R ) $  and  $  |\nabla \psi  |  \leq  c/R. $ Using   Sobolev type estimates we see that  $\ph =  u  \psi $   can be  used  as  a test function in  \eqref{eqn2.6} $(i)$ with  $ \ti u, \ti {\mathcal{A} } $ as above.  We get 

\begin{align}
\label{eqn3.5} 
\begin{split}
\mu (E )  &=  \int_{\rn{n}} \lan \mathcal{A}  ( \nabla u ), \nabla (u \psi) \ran dy\\
& =  \int_{\rn{n}} \lan \mathcal{A}  ( \nabla u ), \nabla u \ran \psi  dy + 
\int_{\rn{n}} u \lan \mathcal{A}  ( \nabla u) 
, \nabla  \psi \ran  dy  = T_1 + T_2. 
\end{split}
\end{align}
From the definition of $\mathcal{A}$-capacity we have  
\begin{align} 
\label{eqn3.6}  
T_1  \to  \mbox{Cap}_\mathcal{A}  (E) \quad  \mbox{as}\, \, R \to \infty.   
\end{align} 
Also, from \eqref{eqn1.1}  with  $ \eta'  = (0, \dots, 0) $ and  H\"{o}lder's inequality, we deduce that 
\begin{align}
\label{eqn3.7}
\begin{split}
  |T_2|  &\leq  c   R^{-1} \int_{ B (0, 2 R ) \sem B (0, R )} u  |\nabla u |^{p-1} dy  \\
  &\leq c^2    R^{-p} \int_{ B (0, 2 R ) \sem B (0, R )} u^p  dy +          
c^2  \int_{ B (0, 2 R ) \sem B (0, R )}   |\nabla u |^{p} dy.
\end{split}
\end{align}
Clearly, the last integral on the far right $  \to   0 $ as  $  R  \to  \infty, $ thanks to  \eqref{eqn1.3} $(c)$.  Moreover,  from  H\"{o}lder's  inequality  and \eqref{eqn3.1}  for $u, $  we have  for some $ \hat c = \hat c (p, n ), $ 
\begin{align}
\label{eqn3.8}
R^{-p} \int_{ B (0, 2 R ) \sem B (0, R )} u^p  dy   \leq   \hat c \,   \left[  \,\int_{ B (0, 2 R ) \sem B (0, R )}   u^{p^*}  dy \, \right]^{p/p^*}  \to  0  \mbox{ as }  R \to  \infty. 
\end{align}    
Using \eqref{eqn3.8}  in  \eqref{eqn3.7}     we see  first  that $ T_2 \to  0 $ as  $ R \to  \infty $ and then from   \eqref{eqn3.5}, \eqref{eqn3.6} that  
$  \mu  ( E )  =  \mbox{Cap}_\mathcal{A}  (E).   $    Next,  given  $ \ep > 0 $  and $ t > 4 \ep, $ let  $   k \geq 0  $ be infinitely differentiable on  $ \re $  with
\[
k(x)=\left\{
\begin{array}{ll}
1 &\mbox{when} \, \, x\in [t+\ep, \infty),\\
0 &\mbox{when} \, \, x\in(-\infty,  t -  \ep].
\end{array}
\right.
\]
Then using  $  k  \circ  u  $ as a test function  in  \eqref{eqn2.6} $(i)$  we find that 
\begin{align}
\label{eqn3.9} 
\begin{split}
\mu (E)   &=  \int_{\rn{n}}  \lan \mathcal{A}  (\nabla u ),  \nabla u \ran  (k' \circ u) dy \\
&=  \int_{t -  \ep}^{t+\ep} \left( \int_{\{u=s\} \cap  \{ |\nabla u | > 0 \} }  \lan \mathcal{A}  (\nabla u ),  \nabla u /|\nabla u | \ran d\mathcal{H}^{n-1} \right)  k' (s)  ds   
\end{split}
\end{align}
 where we have used the  coarea theorem (see \cite[Section 3, Theorem 1]{EG}) to get the last integral. Let   
\[   
I (s) =       \int_{\{u=s\} \cap  \{ |\nabla u | > 0 \} }  \lan \mathcal{A}  (\nabla u ),  \nabla u /|\nabla u | \ran d\mathcal{H}^{n-1}. 
\]  
Then \eqref{eqn3.9}  can be written as  
	\begin{align}     
	\label{eqn3.10}
	\mu ( E )  =   I  ( t )  +    \int_{t-\ep}^{t+\ep}   [I(s) - I(t)] k' (s) ds. 
	\end{align}     
	From  \eqref{eqn1.1}, \eqref{eqn1.3} $(c)$,  and the coarea  theorem once again we  see that  $I$ is integrable on  $ (0,1). $  Using this fact and the  Lebesgue differentiation theorem we find  that  the  integral in \eqref{eqn3.10} $ \to  0 $ as $ \ep \to  0 $  for almost every $ t \in (0,1). $  
It remains to prove the final inequality in   \eqref{eqn3.4} $(a)$.   To  accomplish this,  replace $ t  $ by  $  \tau $  in  the  far-right boundary integral in  \eqref{eqn3.4} $(a)$,  integrate from 0 to $ t $ and use the coarea theorem once again.   

  To prove   \eqref{eqn3.4} $(b)$  we note that if  $ a_1  \leq  u \leq  b_1  $ on $ \ar B (0, \rho )$ for $\rho \geq 4, $ then  from  \eqref{eqn1.6} and  \eqref{eqn1.4}    we deduce that 
 \begin{align}
 \label{eqn3.11}
 \begin{split}   
 c \rho^{n-p}  = \mbox{Cap}_p (B (0, \rho )) &\leq  a_1^{-p}  \int_{\rn{n} \sem B ( 0, \rho  ) } |\nabla u |^p dy  \\
& \leq c'  a_1^{-p }   \int_{\{ u  \leq   b_1\} }  \lan \mathcal{A} ( \nabla u), \nabla u \ran  dy. 
\end{split}
\end{align} 
 Using    \eqref{eqn3.4} $(a)$, \eqref{eqn3.11}, and Harnack's inequality we  see for almost every $ a_1, b_1 $  with 
 \[
  \min_{\ar B (0, \rho )} u  \leq 2 a_1 \quad  \mbox{and}  \quad   \max_{\ar B (0, \rho )} u \geq b_1/2
  \]
 we have 
\[    
\rho^{n-p} \leq  c_- \, a_1^{-p} b_1 \,  \mbox{Cap}_\mathcal{A}  (E) \leq  c_-^2 \,  b_1^{ 1 - p}   \mbox{Cap}_\mathcal{A}  (E)  
\] where $ c_- $ depends only on $ p, n, \al. $  This inequality implies   
the right-hand inequality in  \eqref{eqn3.4} $(b). $  To get the left-hand inequality in  \eqref{eqn3.4} $(b) $ for given $ x , |x| \geq 4, $  let  $  \psi $  be as in  \eqref{eqn3.5}  with $ R = |x|.  $  Using  $ \psi $ as a test function in  \eqref{eqn2.6}  $ (i) $  and  using  \eqref{eqn1.1},  H\"{o}lder's  inequality, Lemma \ref{lemma2.1} $(i), $         and  Harnack's  inequality we obtain   
\begin{align*}
 |x|^{p-n}   \mbox{Cap}_\mathcal{A}  (E) &= |x|^{p-n}  \mu (E)  \leq  c \, |x|^{p - n -1} \, {\ds  \int_{ \{ |x|  < y <  2|x| \} }   |\nabla u |^{p-1} dy } \\
& \leq  c^2  |x|^{(p -n)(1- 1/p) }
\left(  {\ds  \int_{\{ |x|  < y < 2 |x|\}  } } |\nabla u  |^p  dy  \right)^{1 - 1/p} \\
&\leq   c^3    \left( {\ds \max_{ \{ |x|/2  < |y|  <  4 |x|  \} } } u \right)^{p-1}  \leq  c^4  u ( x )^{p-1} 
\end{align*}
which yields the left-hand inequality in  \eqref{eqn3.4} $(b).$    
 The proof of  Lemma \ref{lemma3.2} is now complete. 
 \end{proof}   
For the following lemmas, let $\Omega=\mathbb{R}^{n}\setminus E$.
\begin{lemma}  
\label{lemma3.3}
If there exists $ r_0 > 0$ and  $z  \in E$ with  $  B (z, r_0) \subset E $  then there  is $ c_* \geq 1, $ depending only on $ p, n, \al, r_0 $ 
such that
  \begin{align}
\label{eqn3.12}  
\begin{split}
  &      (a)  \hs{.2in}  c_*   \lan \nabla u (x),  z-x \ran    
\geq u (x)  \quad \mbox{whenever} \, \,  x\in \Om,\\
& (b) \hs{.2in}       c_*^{-1}  \, |x|^{\frac{1-n}{p-1} }   \leq  |  \nabla u (x) |   \leq  c_*   \, |x|^{\frac{1-n}{p-1}} \quad \mbox{whenever} \, \, |x|\geq 4.
      \end{split} 
\end{align} 
\end{lemma}
\begin{proof}  
We may assume that  $  z  = 0 $ thanks to  Remark \ref{rmk1.3} and the fact that  \eqref{eqn3.12}  is invariant under translation.   Let  
\[
v ( x )  =  \, \frac{ u ( x ) -  u ( \la x ) }{\la - 1} - \frac{u (x)}{\breve{c}}
\]
when $x  \in \bar \Om$ and $1  <   \la  \leq  11/10$.   We claim that if   $  \breve c  = \breve c ( p, n, \al, r_0 ) \geq  1  $ is large  enough,  then 
\begin{align}  
\label{eqn3.14}
 v (x) \geq 0 \quad  \mbox{whenever} \quad x  \in   \bar \Om. 
\end{align} 
     From the maximum principle for $ \mathcal{A}$-harmonic functions, we see that it suffices to prove \eqref{eqn3.14} when $ x \in \ar E $  or equivalently that  
\begin{align}  
\label{eqn3.15} 
1 - u (\la  x  )  \geq    c^{-1}  ( \la - 1 ) |x| \quad \mbox{whenever} \quad x \in  \ar E
\end{align} 
for some $ c = c (p, n, \al, r_0) \geq 1$  since  $ r_0  \leq |x|  \leq 1$ and $x \in \ar E. $ 
We also note that    
\[        
d  ( \la x,  E )  \leq   (\la  - 1 ) |x|   \leq  c' d ( \la x, \ar E ) \quad \mbox{whenever} \quad  x  \in  \ar  E. 
\]  
Here $c'  = c' (p, n, \al, r_0 )$. Using this note  in \eqref{eqn3.15}  we conclude that to  prove  Lemma  \ref{lemma3.3}  it  suffices to show for some $ c'' = c'' (p, n, \al, r_0 ) \geq 1 $  that  
\begin{align}  
\label{eqn3.16}
u ( y  )  \leq  1  -   d (y, \ar E ) /c''\quad \mbox{whenever}\quad  0 <  d(y, \ar E )  <  1/10.
\end{align}
To  prove  \eqref{eqn3.16} choose   $  w  \in \ar E $ with  $ | y  -  w | = d ( y, \ar E ) $ and let  
$ \hat w  :=  w +    \frac{y  -  w}{|y-w|}. $  Then  $ w  \in \ar  B ( \hat w, 1)$,  $y \in B ( \hat w, 1 )$, and  $  E  \cap \bar   B (\hat w, 1) = w, $  thanks to the convexity of  $ E. $  From   Lemma  \ref{lemma3.2}  $(b) $   and  Harnack's  inequality for  $  1  - u,  $  we deduce  for  some  $  \de = \de (p, n, \al, r_0 )$ with  $0 <  \de  < 1$  that  
\begin{align}  
\label{eqn3.17}
(1 - u)(x)   \geq  \de   \quad \mbox{whenever}  \quad     x\in \bar B  ( \hat w,  1/2 ).
\end{align}
     Using  \eqref{eqn3.17} and  a  barrier-type argument as in  \cite[Section 2]{LLN} or  \cite[Section 4]{ALV},  it  now follows  that there exists $ c_+, 
$ depending only on $ \al, p, n, $ with   
\begin{align}
\label{eqn3.18}   
c_+  \, (1 - u) (x)  \geq     \de \,   d ( x, \ar B (\hat w, 1) ) \quad \mbox{whenever}\quad   x  \in  B (\hat w, 1 ).
\end{align}
For the readers convenience we outline  the  proof of  \eqref{eqn3.18}  in  the  Appendix \ref{appendix1}.  
Taking  $ y  =  x  $  in  \eqref{eqn3.18}  we get   \eqref{eqn3.16}.  Thus 
\eqref{eqn3.14} holds so letting $ \la \to  1 $ and using smoothness of  $  \nabla u $  (see Lemma  \ref{lemma2.2}) and  the chain rule   we obtain \eqref{eqn3.12} $(a)$.    

From   \eqref{eqn3.12} $(a), $   \eqref{eqn3.4} $(b)$,   and the fact  that  $\mbox{ Cap}_\mathcal{A}  (E) \approx  c (p, n, \al, r_0 ), $ we obtain, 
\[    
c^{-2}  |x|^{\frac{1-n}{p-1} }  \leq  c^{-1}   \, u  (x) / |x|   \leq     \lan \nabla u (x),  -x /|x| \ran    \leq | \nabla u  (x) |.    
\] 
The right-hand inequality in  \eqref{eqn3.12} $(b)$  follows from \eqref{eqn2.2}  $  (\hat a) $,   \eqref{eqn3.4} $(b)$ and the above fact. 
\end{proof}   
Before stating our last lemma in this section  recall from Lemma \ref{lemma3.1} that  $ u $ is continuous on $ \rn{n} $ with  $ u \equiv 1 $  on  $ E. $ 
\begin{lemma}
\label{lemma3.4}
For each $ t \in (0, 1)$, the set  $\{ x \in \rn{n} : u ( x )  >  t \} $ is convex. 
\end{lemma}     
 \begin{proof} 
We first prove  Lemma \ref{lemma3.4} under the added assumptions  that   
\begin{align}
\label{eqn3.19}
\mbox{$\mathcal{A}$  satisfies  \eqref{eqn1.8} $(i)$  and there exists  $r_0>0,\,  z\in  E $  with  $ B ( z, r_0 )  \subset  E. $} 
\end{align}
Assuming \eqref{eqn3.19} we note from Lemma \ref{lemma3.3} and \eqref{eqn2.3} that
\begin{align}
\label{eqn3.20}
\left\{ 
\begin{array}{l}
 |\nabla u | \neq 0 \,\, \mbox{in} \, \, \Omega,
 \\ \mbox{$u$  has H\"{o}lder  continuous second partials on compact subsets of  $ \Om.$}
 \end{array}
 \right.   
 \end{align}
Our proof of Lemma \ref{lemma3.4} is  by  contradiction.  We follow  the proof in \cite[section 4]{L},  
although  we shall modify it slightly  for later ease of use in  proving  \eqref{eqn1.7}.  We first define for $ \hat x \in \rn{n}, $   
 \[\mathbf{u}( \hat x)=\sup\left\{ 
\min\{u(\hat y), u(\hat z)\} \,;\, 
\begin{array}{l}
\hat x=\lambda \hat y+(1-\lambda) \hat z, \\
\lambda\in [0,1], \hat y,  \hat z   \in \mathbb{R}^n  \end{array} 
\right\}.
\]            
If   Lemma  \ref{lemma3.4} is false,  then from convexity of  $ E, $  continuity of  $ u, $  
and the fact that $ u (w) \to    0 $ as  $ w \to  \infty, $  we  see  there exists $ \la \in (0, 1),  \ep > 0, $ and 
$ x_0  \in  \Om  $  such that   
\begin{align}
\label{eqn3.21}
0 < \mathbf{u}^{1+\ep} (x_0) - u(x_0)  =  \max_{ \rn{n}} (\mathbf{u}^{1+\ep}  - u ). 
\end{align}
For ease of  writing we put  $ \mathbf{v}  :=  \mathbf{u}^{1+\ep} $  and  
$ v := u^{1+\ep}. $  With $ \la \in (0, 1)  $ now fixed  it follows from the definition of 
$ \mathbf{u}$  that there exists $ y_0, z_0 \in \Om \sem \{x_0\}  $ with 
\begin{align}
\label{eqn3.22}   
x_0 =  \la y_0 + (1-\la)  z_0\quad  \mbox{and}\quad  \mathbf{v}(x_0) = \min\{v(y_0), v(z_0) \}.
\end{align}
We first show that   
\begin{align}
\label{eqn3.23}
v (y_0) = v ( z_0).
\end{align}  
If, for example,  $ v (y_0) < v (z_0), $  then since $ \nabla u \neq 0$ in $ \Om, $  
we could choose  $ y'$ near  $ y_0 $ with $ v(y') > v(y_0)$ and then choose $ z' $  so that  
$ ( \la - 1) (z'  -  z_0)    =   \la   (y'- y_0) $ and  $ v ( z' ) > v (y' ). $  Then by similar triangles  or algebra, we see first  from  \eqref{eqn3.22} that $ x_0 = \la y' + (1-\la) z' $ and second   by  construction that   
\[  
\min\{v(y'), v(z') \}  > \mathbf{v} (x_0) \,\,  \mbox{which is a contradiction with \eqref{eqn3.21}}.  
\]
Thus   \eqref{eqn3.23} is true.   Next we prove that 
\begin{align} 
\label{samedirection}
\xi=\frac{\nabla v(y_0)}{|\nabla v(y_0)|}=\frac{\nabla v(z_0)}{|\nabla v(z_0)|}=\frac{\nabla u(x_0)}{|\nabla u(x_0)|}.
\end{align} 
Indeed,    
\[  
\frac{\nabla v(y_0)}{|\nabla v(y_0)|}=\frac{\nabla v(z_0)}{|\nabla v(z_0)|}
\] 
since  otherwise we could find  $ y', z'  $  as above with  $ v (y' ) > v (y_0), v (z')  > 
v (z_0 ). $  As previously, we then get a contradiction to \eqref{eqn3.21}.  Finally, armed with this knowledge we see that if  \eqref{samedirection} is false,  then   we could choose 
$  \nu  \in \rn{n},  |\nu |  $  small so that  $ \mathbf{v} $ is increasing at  $ y_0,  z_0 $ in the direction  $ \nu $  while   $ u $ is decreasing at $ x_0 $ in  this direction.  Choosing $ x', y', z' $ appropriately on  rays with direction $ \nu $  through  $ x_0, y_0, z_0 $, respectively we again arrive at a contradiction to \eqref{eqn3.21}. Hence  \eqref{samedirection} is valid.      

To simplify our  notation, let
\[
A=|\nabla v(y_0)|,\, \, B=|\nabla v(z_0)|,\, \, C=|\nabla u(x_0)|, \, \, a=|x_0-y_0|, \, \, b=|x_0-z_0|. 
\]
From  \eqref{eqn3.20}, we can write
\begin{align}
\label{taylorexp}
\begin{split}
v(y_0+\rho\eta)&=v(y_0)+A_1\rho+A_2\rho^{2}+o(\rho^{2}),\\
v(z_0+\rho\eta)&=v(z_0)+B_1\rho+B_2\rho^{2}+o(\rho^{2}),\\
u(x_0+\rho\eta)&=u(x_0)+C_1\rho+C_2\rho^{2}+o(\rho^{2})
\end{split}
\end{align}
as $\rho\to 0$ whenever $ \lan \xi, \, \eta \ran >0$ for a given $\eta\in\mathbb{S}^{n-1}$.  Also
\[
A_1/A=B_1/B=C_1/C= \lan \xi, \, \eta \ran
\] 
 where  the coefficients and $o(\rho^{2})$ depend on $\eta$. Given $\eta$ with $ \lan \xi, \, \eta \ran >0$ and $\rho_1$ sufficiently small we  see from \eqref{eqn3.20} that the inverse function theorem  can be used to obtain $\rho_2$ with
\[
v\left(y_0+\frac{\rho_1}{A}\eta\right)=v\left(z_0+\frac{\rho_2}{B}\eta\right).
\]
We  conclude as $\rho_1\to 0$ that
\begin{align}
\label{rho2rho1}
\rho_2=\rho_1+\frac{B}{B_1}\left(\frac{A_2}{A^{2}}-\frac{B_2}{B^{2}}\right) \rho_1^{2}+o(\rho^{2}_1).
\end{align}
Now from geometry we see that $  \la  = \frac{b}{a+b} $  so   
\[  x=x_0+\eta\frac{[\rho_1 \frac{b}{A}+\rho_2\frac{a}{B}]}{a+b}
=  \la  ( y_0+\frac{\rho_1}{A}\eta ) +  (1-\la) (z_0+\frac{\rho_2}{B}\eta).\]
From   this  equality, \eqref{rho2rho1}, and Taylor's theorem for second derivatives we have 
\begin{align}
\label{taylorforu}
\begin{split}
u(x)-u(x_0)
=& C_1 \, \left[\rho_1\frac{\la}{A}+\rho_2\frac{( 1 - \la ) }{B} \right]+C_2 \left[\rho_1\frac{\la}{A}+\rho_2\frac{(1-\la)}{B}\right]^{2}  \\
=& C_1\rho_1 \frac{ ( 1 - \la) A + \la B}{AB} + C_1 \, \frac{ ( 1 - \la ) }{B_1}\left(\frac{A_2}{A^2}-\frac{B_2}{B^2}\right)\rho_1^{2} \\
&+C_2\rho^{2}_1  \left(\frac{(1-\la) A+ \la B)}{AB} \right)^2               +o(\rho^{2}_1).
\end{split}
\end{align}

From   \eqref{eqn3.21} we also have
\[
v(y_0+\frac{\rho_1}{A}\eta)-u(x)\leq \mathbf{v}(x)-u(x)\leq \mathbf{v}(x_0)-u(x_0)=v(y_0)-u(x_0).
\]
Hence the mapping 
\[
\rho_1\to v(y_0+\frac{\rho_1}{A}\eta)-u(x)
\]
has a maximum at $\rho_1=0$. Using  the Taylor expansion for $v(y_0+\frac{\rho_1}{A}\eta)$ in \eqref{taylorexp} and $u(x)$ in \eqref{taylorforu} we  have
\begin{align*}
v(y_0+\frac{\rho_1}{A}\eta) -u(x)= &v(y_0)+\frac{A_1}{A}\rho_1+\frac{A_2}{A^{2}}\rho_1^{2}- u(x_0)\\ 
&-C_1\rho_1   \frac{(1 - \la) A+\la B}{AB}  - \frac{C_1}{a+b}\frac{a}{B_1}\left(\frac{A_2}{A^2}-\frac{B_2}{B^2}\right)\rho_1^{2}\\
&-C_2\rho^{2}_1       \left(\frac{(1 - \la) A+\la B}{AB} \right)^2 +o(\rho^{2}_1).
\end{align*}   
Now from the calculus second derivative test,  the  coefficient of $  \rho_1 $  should be zero and the coefficient of  $  \rho_1^2 $  should be non-positive. Hence combining terms we get  
\[  
\frac{A_1}{A}=C_1 \,  \frac{(1 - \la) A+\la B}{AB} 
\]   
so taking  $  \eta  = \xi $ we arrive first  at 
\begin{align}
\label{eqn3.28}
\frac{1}{C} \,  = \,  \frac{(1 - \la) A+\la B}{AB}    =   \frac{(1 - \la)}{B} +  \frac{\la}{A}. 
\end{align}
 Second, using  \eqref{eqn3.28}  in the  $ \rho_1^2 $  term   we find that 
\begin{align}
\label{eqn3.29}
0\geq \frac{A_2}{A^{2}}  -   C_1  \frac{ (1 - \la )}{B_1}\left(\frac{A_2}{A^2}-\frac{B_2}{B^2}\right)-  \frac{C_2}{C^2}.
\end{align}
Using  $  C_1/B_1 = C/B $  and doing some algebra in \eqref{eqn3.29}  we obtain  
\begin{align}
\label{subnewpde}
0\geq  (1-K)  \frac{A_2}{A^2}+   K \frac{B_2}{B^2}-   \frac{C_2 }{C^{2}}
\end{align}
where
\[
K=\frac{ (1- \la)  A}{ ( 1 -  \la) A+ \la B}<1. 
\]
We now focus on \eqref{subnewpde} by writing $A_1, B_1, C_1$ in terms of derivatives of $u$ and $v$;
\begin{align}
\label{vvu}
0\geq \sum\limits_{i,j=1}^{n} \left[\frac{(1-K)}{A^2} v_{x_i x_j}(y_0)+\frac{K}{B^2}v_{x_i x_j}(z_0)  -\frac{1}{C^2} u_{x_i x_j}(x_0)\right]\eta_i \eta_j.
\end{align} 
From symmetry and  continuity  considerations  we  observe that  \eqref{vvu}  holds whenever  $  \eta  \in  \mathbb{S}^{n-1} $.  Thus, if   
  \[ 
  w (x) = - \frac{(1-K)}{A^2} v (y_0 + x ) - \frac{K}{B^2}v(z_0 + x )\,     + 
\frac{1}{C^2}u(x_0 + x),  
\]  
then the  Hessian matrix of  $ w $ at $ x = 0 $ is positive semi-definite. That is,   $  ( w_{x_ix_j} (0) ) $  has  non-negative eigenvalues.   Also from  $(i)$ of  Definition \ref{defn1.1} we see that  if 
\[
 a_{ij} = \frac{1}{2}\left[ \frac{\partial \mathcal{A}_i}{\partial \eta_j}  (\xi ) )    
+   \frac{\partial \mathcal{A}_j}{\partial \eta_i}  (\xi ) )   \right] \quad \mbox{for}\, \, 1 \leq i, j \leq  n,
\]
then  $  (a_{ij}) $  is  positive definite.  From  these two observations we conclude that 
\begin{align}
\label{eqn3.32}   
\mbox{trace} \left( (  (a_{ij} )  \cdot ( w_{x_ix_j} (0) ) \right)   \geq  0.   
\end{align}   
To obtain a contradiction we observe from \eqref{eqn1.2},  the divergence theorem,  \eqref{subnewpde}, and   
  $  p - 2 $  homogeneity of  partial derivatives of  $ \mathcal{A}_i , $      that 
  \begin{align}
  \label{eqn3.33}    
  \sum\limits_{i,j=1}^{n}  a_{ij} \,  u_{x_jx_i} =   | \nabla u |^{2-p}  \sum\limits_{i,j=1}^{n} \frac{\partial \mathcal{A}_i}{\partial \eta_j}(\nabla u ) u_{x_jx_i} \, = \, 0 \quad \mbox{at}\,\, x_0, y_0, z_0. 
\end{align}  
Moreover,   from the definition of $v$ we have 
\begin{align}
\label{secondderiv}
\begin{split}
&v_{x_{i}}=(u^{1+\epsilon})_{x_{i}}= (1+\epsilon)u^{\epsilon} u_{x_{i}}, \\
&v_{x_{i}x_{j}}= (1+\epsilon)\epsilon u^{\epsilon-1} u_{x_{i}}u_{x_j}+(1+\epsilon)u^{\epsilon} u_{x_{i} x_{j}}.
\end{split}
\end{align}
Using Definition \ref{defn1.1}, \eqref{samedirection}, and $\mathcal{A}$-harmonicity of $u$ at those points, we find that 
\begin{align}
\label{fvupos}
\begin{split}
   |\nabla u |^{p-2} \, \sum\limits_{i,j=1}^{n}  & 
  a_{ij}  \,  v_{x_j x_i}  = \sum\limits_{i,j=1}^{n}  \frac{\partial \mathcal{A}_i}{\partial \eta_j}(\nabla u)  [(1+\epsilon)u^{\epsilon-1} u_{x_{j}}u_{x_i}+(1+\epsilon)u^{\epsilon} u_{x_{j} x_{i}}]\\
&=(1+\epsilon)\epsilon u^{\epsilon-1}\sum\limits_{i,j=1}^{n}  \frac{\partial \mathcal{A}_i}{\partial \eta_j}(\nabla u)  u_{x_{j}}u_{x_i} + (1+\epsilon)u^{\epsilon}   \,  \sum\limits_{i,j=1}^{n} \frac{\partial \mathcal{A}_i}{\partial \eta_j}(\nabla u )  u_{x_{j}x_i}\\
& \geq \alpha^{-1} (1+\epsilon)\epsilon u^{\epsilon-1}  |\nabla u|^{p-2} |\nabla u|^{2} +0>0
\end{split}
\end{align}
at points $y_0$ and $z_0$ ($\nabla u$ is also evaluated at these points). Using  \eqref{eqn3.33}, \eqref{fvupos}, we conclude that   
\begin{align} 
\label{eqn3.36}
\mbox{trace} \left(  (a_{ij} ) \cdot ( w_{x_ix_j} (0) ) \right)          =    \sum_{i,j=1}^n   a_{ij}   \,  w_{x_ix_j} (0)        < 0.
\end{align}
Now  \eqref{eqn3.36} and  \eqref{eqn3.32} contradict each other.  Thus Lemma \ref{lemma3.4} is true when   
\eqref{eqn3.19} holds.  

To  remove  assumption   \eqref{eqn3.19},  suppose   $ \left\{\mathcal{A}^{(l)}\right\},  l =  1,  2, \dots  \in  M_p ( \al/2 ), $         with    
\begin{align*}
\begin{split}
\left\{\mathcal{A}^{(l)} ,  
 \frac{\ar \mathcal{A}^{(l)}}{\ar \eta_k}\right\}  \to   &\, \left\{\mathcal{A}, \, \frac{\ar \mathcal{A}}{\ar \eta_k}\right\} \, \, \mbox{as} \,\, l\to\infty \quad \mbox{for each}\, \,  k = 1, 2, \dots, n,  \\
&\mbox{ uniformly on compact subsets of } \rn{n}  \sem  \{0\}.   
\end{split}
\end{align*}
Also assume that  \eqref{eqn1.8} $ (i) $ holds for each  $ l  $  where  $ \La  =  \La (l). $    Let  
\[
E_l  =  \{ x :  d  ( x,  E )  \leq  1/l \}, \, \,  l = 1, 2,  \dots, 
\]
and let $ u_l $ be the  $ \mathcal{A}^{(l)}$-capacitary function corresponding to   $ E_l. $     From  Lemmas 
\ref{lemma2.1}-\ref{lemma2.3} and  Lemma  \ref{lemma3.1}, we deduce that a subsequence of  $\{u_l\}$ say   $\{u'_l\}$  can be chosen so that  
\begin{align*}
\left\{u'_l, \nabla u'_l \right\}\to \left\{u,\nabla u\right\}\, \, &\mbox{converges uniformly as}\, \,  l\to\infty \\
& \mbox{on compact subsets of}\, \,   \rn{n} \, \, \mbox{and}\, \, \Om\, \, \mbox{respectively}.
\end{align*} 
Now  from our previous work  we see that 
Lemma  \ref{lemma3.4}  holds for   $ u_l $  so    
\[
E_l (t)  = \{  x: u'_l ( x )   >  t  \}\, \,\, \mbox{is convex for}\,\,  l  =  1, 2, \dots,\,\,  \mbox{and}\, \,t  \in (0, 1). 
\]
 Also from   Lemma  \ref{lemma3.2}  
 these sets are uniformly bounded for  a fixed $ t \in (0,1).$  Using  these facts,  it  is  easily seen  that 
\[
E (t)  = \{  x: u ( x )   >  t  \} \, \, \mbox{is convex}.
\]   
Indeed,  if  $ x, y  \in  E (t) $  and $ t > 4 \de > 0 $,   then  from  convexity of  $ E_l (t) $  and  uniform convergence of  $\{u'_l\}$  to  $ u $  we see that the line segment from  $ x $ to  $ y $ is  contained in  $ E ( t - \de ) $  whenever $ 2 \de < t. $ Letting $  \de \to  0 $  we get convexity of  $  E  (t). $ 

To prove existence of  $\{\mathcal{A}^{(l)}\}$    
let  
\[
  \psi (\eta )  =  \mathcal{A} (\eta/|\eta|) \quad \mbox{whenever}\, \, \eta \in \rn{n} \sem \{0\}. 
  \]
Given  $ \ep > 0, $ small  we also define 
\[
 \psi_\ep  (  \eta )   =  (\psi *\ph_\ep)(\eta)\quad \mbox{on}\quad  B (0, 2 ) \sem B (0, 1/2) 
 \]
where 
$*$ denotes convolution on $ \rn{n} $ with  each component of  $  \psi. $  Also, $\ph_\ep (\eta ) =  \ep^{-n} \ph (\eta /\ep), $  and  $ 0 \leq \ph    \in C_0^\infty (\rn{n} ) $  with  
 $  \int_{\rn{n}} \ph \, dx = 1. $   
Set
 \[   
 \mathcal{A}_\ep (\eta )   = |\eta|^{p-1} \psi_\ep (\eta /|\eta|) \quad \mbox{whenever}\quad \eta  \in \rn{n}  \sem \{0\}.
 \]
Then for $ \ep $ small enough  we deduce from Definition \ref{defn1.1} that  
$  \mathcal{A}_\ep \in  M_p (\al/2) $  and   \eqref{eqn1.8} $(i)$ holds for  $ \mathcal{A}_\ep. $ 
Letting $  \mathcal{A}^{(l)}   =  \mathcal{A}_{\ep_l}$ for sufficiently small $ \ep_l $ with
$ \ep_l \to  0 $  we get the above sequence.  The proof of  Lemma \ref{lemma3.4} is now complete. 
\end{proof}    

\setcounter{equation}{0} 
\setcounter{theorem}{0}    
\section{Proof of Theorem  \ref{theorem1.4}}  
\label{section4}
In the     proof  of   \eqref{eqn1.7}  we shall need the following lemma.     
\begin{lemma}  
\label{lemma4.1}
Given  $  \mathcal{A}  \in  M_p (\al), $ there exists  an  $ \mathcal{A}$-harmonic function  $G $  on $ \rn{n}  \sem \{0\} $  and  $ c  = c (p, n, \al)$ satisfying
\begin{align}
\label{eqn4.1} 
\begin{split}
&(a)\, \,  c^{-1} \,  | x |^{(p-n)/(p-1)}  \leq  G  (x)   \leq  c \, | x |^{(p-n)/(p-1)} \, \, \mbox{whenever}\quad x\in \rn{n} \sem \{0\}. \\
 &(b) \,\,  c^{-1} \,  | x |^{(1-n)/(p-1)}  \leq  | \nabla G |  \leq  c \, | x |^{(1-n)/(p-1)}\quad \mbox{whenever}\quad  x\in \rn{n} \sem \{0\}. \\
&(c)\,\, \mbox{If}\, \,  \he  \in  C_0^\infty (\rn{n} ) \, \, \mbox{then}\quad  
  \he (0)  = {\ds \int_{\rn{n}\sem \{0\}}  \lan  \mathcal{A} (\nabla G  ),  \nabla  \he  \ran \, dx}.  \\
&(d) \,\,    G\, \,   \mbox{is the unique}\, \, \mathcal{A}\mbox{-harmonic function on}\, \,   \rn{n} \sem  
\{0\}\, \, \mbox{satisfying  $ (a)$  and  $(c)$.} \\
&(e) \,\, G   (x)  =   | x  |^{(p-n)/(p-1)}   \, G  ( x / |x|)\quad \mbox{whenever}\quad  x  \in \rn{n} \sem \{0\}.
\end{split}
\end{align} 
\end{lemma}
\begin{proof}  
Let   $  \breve   u  $  be the  $ \mathcal{A}$-capacitary function for $  \bar {B} (0,1) $ and let  $ \breve  \mu $  be the corresponding capacitary  measure  for  $ \bar B (0, 1). $  Then from  
\eqref{eqn1.6} and  Lemma \ref{lemma3.2}  we have  
\begin{align}
\label{eqn4.2}
\breve  \mu ( \bar B (0, 1) )  =   \mbox{Cap}_{\mathcal{A}} (  \bar B (0,1) )=  c_1.   
\end{align}
 For $k=1,2,\ldots$, let 
\begin{align*} 
 &\breve u_k ( x ) : =   c_1^{- \frac{1}{p-1}} \,    k^{\frac{n - p}{p-1}}  \,  \breve  u (k x) \quad \mbox{whenever}\, \,  x \in \rn{n},\\
 & \breve  \mu_k  (  F )   := c_1^{-1} \,    k^{n - p }  \breve  \mu ( k F)\quad \mbox{whenever}\, \,  F  \subset \rn{n} \mbox{ is a Borel set.}
 \end{align*}
  Then from  Remark  \ref{rmk1.3}  and  Lemma \ref{lemma3.1}  we see that  $  \breve  u_k $ is  continuous on   $  \rn{n} $ and  $  \mathcal{A}$-harmonic in  $  \rn{n} \sem \bar B (0,1/k)$  with  
\[
  \breve   u_k \equiv    c_1^{- \frac{1}{p-1}}   k^{\frac{n - p}{p-1}}  \quad  \mbox{on}\, \,  \bar B (0, 1/k ). 
  \]  
  Also if   $  \ph   \in  C_0^\infty ( \rn{n} ) $ and  $ \ph_k ( x ) =  \ph (k x), $ then 
   from  \eqref{eqn2.6} $(i),  p -  1 $ homogeneity of  $  \mathcal{A},$ and the change of variables theorem,   we have   
	\begin{align}
	\label{eqn4.3} 
	\int_{\rn{n}} \lan  \mathcal{A} (\nabla  \breve  u_k ), \nabla \ph_k  \ran dx =  c_1^{-1}  \int_{\rn{n}} \lan \mathcal{A} (\nabla  \breve  u ), \nabla \ph  \ran dx  =  \int_{\rn{n}} \ph_k   d \breve  \mu_k.
	\end{align}
Thus    $  \breve  \mu_k $  is the measure corresponding to  $  \breve  u_k $  with  support  $ \subset  \bar B (0,1/k) $  and   $  \breve   \mu_k ( \bar B (0,1/k))  = 1 $ thanks to  \eqref{eqn4.2}.         
Also applying  \eqref{eqn3.4} $(b) $  and \eqref{eqn3.12} $(b) $  to  $  \breve   u $  we deduce that  
\begin{align}
\label{eqn4.4} 
\begin{split}
& (+) \hs{.3in}  c^{-1}_+ \,  |x|^{(p-n)/(p-1)}\,  \leq \breve   u_k (x)  \leq c_+  |x|^{(p-n)/(p-1)},    \\
 & (++) \hs{.18in}  c^{-1}_+ \,  |x|^{(1-n)/(p-1)}\,  \leq |\nabla \breve   u_k (x) | \leq c_+  |x|^{(1-n)/(p-1)}
\end{split}
\end{align}
whenever $|x| \geq 2/k$. Using   \eqref{eqn4.4},   Definition \ref{defn1.1},   and  H\"{o}lder's  inequality,  we see that for $  \rho   > 1/k, $   
\begin{align}
\label{eqn4.5}
\begin{split}
\int_{B(0,\rho)} |\mathcal{A} (\nabla  \breve  u_k ) | dx   &\leq  c \, k^{- n/p} \, \left(  \int_{B(0,\rho) \cap B (0, 2/k)}  |\nabla \breve  u_k |^p  dx \right)^{ 1- 1/p}  + c \, \int_{B(0,\rho) \sem  B (0, 2/k)}\,  |x|^{1-n} \,    dx \\
&\hs{.4in} \leq c^2 ( k^{-1} + \rho ). 
\end{split}
\end{align}
 If  $ \rho \leq 1/k, $ the  far right-hand integral is  0 so  \eqref{eqn4.5}  continues to hold.    
Using  Lemmas  \ref{lemma2.1}, \ref{lemma2.2},  we see there is  a  subsequence of   $\{ \breve  u_k\}$  say  $\{\breve   u_k' \} $ with  
\begin{align*}
 \left\{\breve   u_k' , \nabla  \breve  u_k'\right\}  \to  \left\{G, \nabla G\right\}\, \, & \mbox{converges uniformly as}\, \, k\to\infty \\
 &\hs{.4in}\mbox{on compact subsets of}\, \, \rn{n} \sem \{0\}.
\end{align*}  
It follows that  $ G $  is  $ \mathcal{A}$-harmonic in $ \rn{n}  \sem \{0\} $  and if $  \bar \mu  $ is the measure with mass $1$  and  support at  the origin, then 
\[
\breve  \mu_k \rightharpoonup \bar \mu\, \,  \mbox{weakly as measures as}\, \, k\to\infty.  
\]
Finally, \eqref{eqn4.5}   and the above  facts imply   the sequence  $  \{ |\mathcal{A} (\nabla \breve  u_k ) | \}_{k\geq 1} $  is uniformly integrable on  $ B (0, \rho), $ so using uniform convergence we get for  $ \he   \in  C_0^\infty (  B (0, \rho) )$ that 
\begin{align*}
\int_{\rn{n}}  \lan   \mathcal{A} (\nabla G ), \nabla \he \ran  dx  =  \lim_{k \to   \infty}   \, \int_{\rn{n}}  \lan  \mathcal{A} (\nabla  \breve   u_k ) , \nabla \he \ran  dx = \lim_{k\to  \infty} \int_{\rn{n}} \, \he \, d \breve  \mu_k 
=   \he (0)
\end{align*}   
where we have also used \eqref{eqn4.3}.

 To prove  uniqueness, suppose  $ v$ is  $ \mathcal{A}$-harmonic in $ \rn{n} \sem \{0\} $ and $(a), (c)$  of  \eqref{eqn4.1}  hold for $ v $  and some  constant $ \geq 1. $  Observe from  \eqref{eqn4.1} $(a)$ and  \eqref{eqn2.2} $ (\hat a)$ that 
 \begin{align}
 \label{eqn4.7}  
 | \nabla v (x) |  \leq  c^*  | x |^{(1-n)/(p-1)} \quad  \mbox{whenever} \, \, x  \in \rn{n} \sem \{0\}. 
  \end{align}
Given  $ \ga  > 0, $  let
\[
e (x):= G (x) - \ga  v(x)  \quad \mbox{whenever}\, \, x  \in \rn{n}   \sem  \{0\}. 
\] 
We note that if $  \vartheta,     \upsilon   \in \rn{n} \sem \{0\} $ then  
\begin{align}
\label{eqn4.8}  
\mathcal{A}_i ( \vartheta  )  -  \mathcal{A}_i ( \upsilon  )   = \sum_{j = 1 }^n
( \vartheta_j -\upsilon_j )\int\limits_0^1 \frac{\ar  \mathcal{A}_i}{\ar \eta_j } ( t \vartheta + (1-t) \upsilon  ) dt
\end{align}   
for $i\in\{1,..,n\}$.  
Using this  note  it  follows that  $ e $  is  a  weak  solution  to    
\begin{align*}
\hat{\mathcal{L}}e:=\sum_{i,j=1}^n\frac{ \ar }{ \ar y_i} \left(  \hat a_{i j }( y )  \frac{ \ar e}{ \ar y_j}\right)=0\mbox{ in
}  \rn{n}  \sem \{0\}
\end{align*}
where
\begin{align*}
\hat a_{ij} ( y )   =   \int_0^1   \frac{ \ar  \mathcal{A}_i}{\ar \eta_j }   \left( t \nabla G (y)  + \ga (1-t) \nabla v (y)\right) d t. 
\end{align*}
Moreover from  Definition \ref{defn1.1} $(i) $   we  see for  some  $ c = c (p, n, \al ) \geq 1  $ that 
\begin{align}  
\label{eqn4.11}
c^{-1} \si (y)  \,   | \xi |^2  \, \leq \,   \sum_{i,j=1}^n   \hat a_{ij} ( y  )  \xi_i\xi_j \quad \mbox{and}\quad  
\sum_{i,j=1}^n   |\hat a_{ij} ( y  ) |    
    \leq c \, \si (y) \,  \quad \mbox{whenever}\, \, \xi  \in \rn{n} \sem \{0\},
\end{align}
where 
\[  
\si (y)  =  \int_0^1  |t \nabla G (y)  + (1- t) \ga \nabla v (y) |^{p-2}  dt. 
\]  
Using  \eqref{eqn4.1} $(b)$  and \eqref{eqn4.7}     we  obtain for  $y  \in \rn{n} \sem \{0\}$ that
\begin{align}
\label{eqn4.12}    c ( \ga )^{-1}   \, |y|^{ \frac{(1-n)(p-2)}{p-1} }  \leq  
\si ( y  )  \approx  \left( |  \nabla G (y) | + \ga  |  \nabla v (y) |  \right)^{p-2} \leq  c (\ga)     |y|^{ \frac{(1-n)(p-2)}{p-1} }. 
\end{align}  
Here  $ c ( \ga ) $   depends only on  $ \ga $ and  those for $  G,  v $ in  \eqref{eqn4.1} $(a). $      
Thus  $  G  - \ga v $ is  a  solution to  a  linear uniformly elliptic PDE  on  $  B  ( x,  |x|/2 ) $  whenever  $ x  \in \rn{n} \sem \{0\} $  with ellipticity constants that are independent of  $ x  \in \rn{n}. $ 
Next  we  observe from   \eqref{eqn4.1} $(a)$  and the maximum principle for $  \mathcal{A}$-harmonic functions  that  for $ r > 0, $ 
\begin{align}
\label{eqn4.13}  
\max_{\rn{n}\sem B (0, r ) } \frac{G}{v}  =  \max_{\ar B (0, r)} \frac{G}{v} \quad    \mbox{ and } \quad \min_{\rn{n}\sem B (0, r ) } \frac{G}{v}  =  \min_{\ar B (0, r)} \frac{G}{v}.
\end{align}
To  continue   the  proof  of  \eqref{eqn4.1} $(d)$,   let  
\[
   \ga  =   \liminf_{ x \to  0}\,   \frac{ G (x)}{v(x)}.  
   \]   
Then from \eqref{eqn4.13} we see that 
$ G -  \ga v  \geq 0 $  in  $  \rn{n}  \sem \{0\} $  and  there exists  a  sequence  $\{z_m\}_{m\geq 1}$  with  
\[
\lim_{m \to  \infty} z_m = (0, \dots, 0)\, \, \mbox{and}\, \,   G ( z_m) - \ga  v (z_m)  = o ( v (z_m ) )\, \, \mbox{as}\, \,    m \to  \infty. 
\]  
Now from  Harnack's  inequality  for  linear  elliptic  PDE   and the usual  chaining-type argument in balls $  B  (x, r/2 ),  |x|= r, $ we see for some $c' \geq 1, $ independent of $x, $ that  
\[  
\max_{\ar B (0, r )}  ( G - \ga  v  )   \leq   c'  \min_{\ar B (0, r )}  ( G - \ga  v  ). 
\]   
Using this inequality with $ r = |z_m|,  $  the above facts, and Harnack's inequality for $ v, $  we deduce
 \[
 G ( x ) -  \ga  v  (  x )  =  o (v(x))\quad  \mbox{when}\,\,  | x |  =  |z_m|.
 \]
 This equality yields in view of  \eqref{eqn4.13} that    
first  
\[  
\limsup_{ x \to  0 }  \frac{G (x)}{v(x)} =   \ga  
\]
  and  second that   $G  =  \ga v. $  From   \eqref{eqn4.1} $(c)$  we have $ \ga   = 1 $ 
so    \eqref{eqn4.1} $ (d)$ is true.  

To  prove  \eqref{eqn4.1} $ (e) $  we  observe from Remark \ref{rmk1.3}   for fixed  $ t  > 0 $  that  
\[  
v( x )  =  t^{(n-p)/(p-1)  } \,  G ( t x ) 
\]
   is  $  \mathcal{A}$-harmonic in  $ \rn{n} \sem \{0\}. $  Also it is easily checked
 that  \eqref{eqn4.1} $ (a) - (c) $  are valid with  $  G  $  replaced by  $  v.  $   From  \eqref{eqn4.1} $ (d) $  it follows that $ G = v $  and thereupon  using  $ t =  |x|^{-1} $  that  
\eqref{eqn4.1} $(e) $ is valid.  
\end{proof}  

We call  $ G  $  the \emph{fundamental solution} or \emph{Green's function} for  $ \mathcal{A}$-harmonic functions with pole at  $ (0, \dots, 0). $  In  this section we assume only that  $  E \subset  \rn{n} $   is  a compact convex set  with  $ \mbox{Cap}_{\mathcal{A}} (E)  > 0, $  in contrast to  section \ref{section3}, where we also assumed that  $\mbox{diam}(E)=1$  and  $  0 \in  E.$    Using  Lemma \ref{lemma4.1} we  prove  
\begin{lemma} 
\label{lemma4.2}
If   $ u $ is  the  $ \mathcal{A}$-capacitary  function for $ E $ and $ G $ is as in Lemma \ref{lemma4.1}  then  
\begin{align*}   
\lim_{x \to  \infty}   \frac{  u (x)}{  G (x)} =    \mbox{Cap}_{\mathcal{A}} (E)^{\frac{1}{p-1}}. 
\end{align*} 
\end{lemma} 
\begin{proof}   
Translating and dilating  $E$ we see from  Remark  \ref{rmk1.3}  and  Lemma  \ref{lemma3.2}  that there exists, $ R_0=R_0 (E, p, n, \al ) > 100, $  such  that  
$  E   \subset   B (0,  R_0 ) $ and  
\begin{align*}   
c^{-1}  \, |x|^{(p-n)/(p-1)}  \leq u(x) \leq c  \, |x|^{(p-n)/(p-1)}\,   \,  \mbox{whenever}\, \,  |x| \geq  R_0 
\end{align*}
 where $ c = c (E, p, n, \al). $ Let  $  \{R_k\}_{k\geq 1}   $  be a sequence of  positive numbers  $ \geq R_0 $  with  $  {\ds \lim_{k\to  \infty} R_k = \infty}. $   Put 
\[  
\hat u_k  ( x )  =  R_k^{\left(\frac{n - p}{p-1}\right)} \, \mbox{Cap}_{\mathcal{A}}  (E)^{- \frac{1}{p-1}} \, \,  u ( R_k  x )\quad  \mbox{whenever} \, \, x \in \rn{n},  
\]  
and let  $ \hat \mu_k$ be the measure corresponding to 
$  1  -  \hat u_k. $  Then as in  \eqref{eqn4.3} we see that  $ \hat \mu_k  ( \rn{n} ) =1 $ and  the support of  $  \hat \mu _k $ is contained in  $  B (0, R_0/R_k ). $  Now arguing as in the proof of  \eqref{eqn4.1} $ (c) $  we  get  a subsequence  of  $ \{\hat u_k\}$ say  $\{\hat u_k'\}$ with 
\[
 \lim_{k\to  \infty} \hat u'_k = v \quad  \mbox{uniformly on compact subsets of $ \rn{n} \sem \{0\}$} 
 \] 
 where 
$ v $ is  $  \mathcal{A}$-harmonic in  $  \rn{n} \sem \{0\} $ and satisfies  $ (a), (c). $  Thus from  \eqref{eqn4.1} $ (d), \, v = G. $  Since every sequence has a  subsequence converging to $  G $  we see that 
    \begin{align*}
    \lim_{R\to  \infty}  R^{\left(\frac{n - p}{p-1}\right)} \, \mbox{Cap}_{\mathcal{A}}  (E)^{- \frac{1}{p-1}} &\, \,  u ( R  x ) = G(x)  \\
    &\mbox{uniformly on compact subsets of}\, \,  \ \rn{n} \sem \{0\}.
    \end{align*}
Equivalently  from  \eqref{eqn4.1} $ (e) $  that 
\[   
\lim_{R \to  \infty}  \frac{ u ( R x )}{G (R x )} =    \mbox{Cap}_{\mathcal{A}}  (E)^{ \frac{1}{p-1}}  \mbox{ uniformly on compact subsets of $ \rn{n} \sem \{0\} $}.
\] 
This completes the proof of  Lemma  \ref{lemma4.2}. 
\end{proof}

\subsection{Proof of (\ref{eqn1.7}) in Theorem \ref{theorem1.4}}      
\label{subsection4.1}

In this subsection we prove,  the Brunn-Minkowski inequality for  $\mbox{Cap}_{\mathcal{A}}$, \eqref{eqn1.7} in Theorem \ref{theorem1.4}. 
\begin{proof}[Proof of \eqref{eqn1.7}]       
Let  $  E_1,  E_2 $ be as in  Theorem \ref{theorem1.4}.  
Put  $  \Omega_i = \rn{n} \sem E_i$  
and let  $ u_i$  be the $\mathcal{A}$-capacitary 
function for $ E_i$  for $i=1,2$. Let $ u $  
be the $\mathcal{A}$-capacitary function for  $  E_1 + E_2. $  Following  \cite{B1,CS},  we note   that it suffices to prove  
\begin{align}
\label{eqn4.16}   
\mbox{Cap}_{\mathcal{A}}( E_1'   + E_2')^{\frac{1}{n-p}}     \geq    \mbox{Cap}_{\mathcal{A}} ( E_1')^{\frac{1}{n-p}}  \, + \,        \mbox{Cap}_{\mathcal{A}} ( E_2')^{\frac{1}{n-p}} 
\end{align}                                    
 whenever  $ E_i'$ for  $i =1, 2$  are convex sets with   $\mbox{Cap}_{\mathcal{A}}(E_i') > 0. $  To get \eqref{eqn1.7} from   \eqref{eqn4.16} put  
 \[
 E_1'  =  \la E_1 \quad \mbox{and} \quad E_2' =  (1-\la) E_2
 \]
and use \eqref{eqn1.5} $(a')$.  Also  to  prove  \eqref{eqn4.16}  it suffices to show,  for  all $  \la  \in  (0,1) $  that    
\begin{align}  
\label{eqn4.17}
\mbox{Cap}_{\mathcal{A}}(\la E_1''   + (1-\la)  E_2'')^{\frac{1}{n-p}}     \geq    \min  \left\{ \mbox{Cap}_{\mathcal{A}} ( E_1'')^{\frac{1}{n-p}}  \, ,  \,        \mbox{Cap}_{\mathcal{A}} ( E_2'')^{\frac{1}{n-p}}  \right\}
\end{align}                                    
 whenever  $ E_i''$ for $i =1, 2$  are convex sets with  $\mbox{Cap}_{\mathcal{A}}(E_i'') >0. $ 
 To  get   \eqref{eqn4.16} from \eqref{eqn4.17} let  
\[  
E_i''  =  \mbox{Cap}_{\mathcal{A}} ( E_i')^{ \frac{1}{p - n} }  \,  \,  E_i'  \quad \mbox{for}\, \,  i  = 1, 2
\]
and
\[
  \la  = \frac { \mbox{Cap}_{\mathcal{A}} ( E_1')^{\frac{1}{n-p}} }{ 
  \mbox{Cap}_{\mathcal{A}} ( E_1')^{\frac{1}{n-p}} +
  \mbox{Cap}_{\mathcal{A}} ( E_2')^{\frac{1}{n-p}}}
  \]   
 then use  \eqref{eqn1.5} $(a')$ and do some algebra.  Thus, we shall  only prove \eqref{eqn4.17} for $  E_1, E_2,  $  and  all  $  \la  \in (0, 1).$  Some of our proof is quite similar to  
the  proof  of  Lemma \ref{lemma3.4}.  For this reason  we first assume that \eqref{eqn3.19}  
holds for  $ E_1, E_2$ and $  \mathcal{A}. $ Fix $ \la  \in (0,1) $  and set   
   \[u^*(x)=\sup\left\{ 
\min\{u_1(y), u_2(z)\} \,;\, 
\begin{array}{l}x=\lambda y+(1-\lambda) z, \\
\lambda\in[0,1], y,  z   \in \mathbb{R}^n  \end{array} 
\right\}.
\]
We claim that 
\begin{align}  
\label{eqn4.18}
u^*(x)\leq u(x)\, \, \mbox{whenever}\,\, x\in\mathbb{R}^{n}.  
\end{align}
Once  \eqref{eqn4.18} is  proved  we   get  \eqref{eqn1.7}  under assumption  \eqref{eqn3.19} as follows.   From  \eqref{eqn4.18} and the definition of  $u^* $  we have  
\[
u(x) \,\geq \, u^*(x) \,\geq\, \min \{ u_1(x), u_2(x) \}\, 
\]  so from  Lemma  \ref{lemma4.2}   
\begin{align*}
\mbox{Cap}_{\mathcal{A}} ( \la E_1 + (1- \la) E_2 )^{\frac{1}{p-1} }&= \lim_{|x|  \to   \infty}  \frac{u(x)}{G(x)} \\
 & \geq \lim_{|x| \to   \infty } \frac{ \min\{u_1(x), u_2(x)\}}{ G ( x ) } \\
 &= \min  \left\{    \mbox{Cap}_{\mathcal{A}}(  E_1 )^{\frac{1}{p-1} } \, , \,  \mbox{Cap}_{\mathcal{A}}( E_2 )^{\frac{1}{p-1} }  \right\}.
\end{align*}
This finishes proof of \eqref{eqn4.17} which implies \eqref{eqn4.16} and from our earlier remarks this implies \eqref{eqn1.7} in Theorem \ref{theorem1.4} under the assumptions \eqref{eqn4.18} and \eqref{eqn3.19}.
\end{proof}
  
The proof of  \eqref{eqn4.18}  is essentially the same as the  proof after \eqref{eqn3.19} of  Lemma \ref{lemma3.4}. Therefore we shall  not give all  details.  From  \eqref{eqn3.19} we see that  $  \nabla \hat u  \not  =  0 $ and 
 $ \hat u $  has continuous second partials on  $  \rn{n} \sem \hat E $   whenever  $  \hat u  \in  \{ u, u_1, u_2 \} $  and   
$  \hat  E  \in  \{  \la  E_1  + (1-\la) E_2,  E_1,  E_2 \}.    $                       Assume that   \eqref{eqn4.18}  is  false.  Then  there exists  $     \ep > 0 $ and  $ x_0  \in \rn{n} $ 
such that if   
   
\[    
v_1 (x)  = u_1^{1+\ep}  (x),   \, \, v_2 (x)  = u_2^{1+\ep}  (x), \, \mbox{ and }  \, v^*  ( x  )  = (u^*)^{1+\ep}, 
\]  we have    
\begin{align}
\label{eqn4.19} 
0 < v^* (x_0) - u(x_0)  =  \max_{ \rn{n}} [ \, (u^*)^{1+\ep}  - u \, ]. 
\end{align}      
As in  \eqref{eqn3.21}, \eqref{eqn3.22},   there exists $ y_0 \in  \Om_1,  z_0 \in \Om_2  \, (  y_0  =  x_0 = z_0$  is now possible) with   
\[  
x_0 =  \la y_0 + (1-\la)  z_0   \quad \mbox{and}\quad  v^*(x_0) = v_1(y_0)  =  v_2(z_0). 
\] 
Also as in   \eqref{samedirection} we obtain  
\begin{align}
\label{eqn4.20}
\xi=\frac{\nabla v_1(y_0)}{|\nabla v_1(y_0)|}=\frac{\nabla v_2(z_0)}{|\nabla v_2(z_0)|}=\frac{\nabla u(x_0)}{|\nabla u(x_0)|}
\end{align}  so with 
\[
A=|\nabla v_1(y_0)|,\, \, B=|\nabla v_2(z_0)|,\, \, C=|\nabla u(x_0)|, \, \, a=|x_0-y_0|, \, \, b=|x_0-z_0|. 
\]  
we have 
\begin{align*}
\begin{split}
&v_1(y_0+\rho\eta)=v_1(y_0)+A_1\rho+A_2\rho^{2}+o(\rho^{2}),\\
&v_2 (z_0+\rho\eta)=v_2(z_0)+B_1\rho+B_2\rho^{2}+o(\rho^{2}),\\
&u(x_0+\rho\eta)=u(x_0)+C_1\rho+C_2\rho^{2}+o(\rho^{2})
\end{split}
\end{align*}
as $\rho\to 0$ whenever $\lan \xi, \, \eta \ran > 0$  and $\eta\in\mathbb{S}^{n-1}$.  

We  can now  essentially  copy the argument after  \eqref{taylorexp} through  \eqref{eqn3.36} to  eventually arrive  at a  contradiction to  \eqref{eqn4.18}.  
Assumption  \eqref{eqn3.19} for  $ E_1, E_2,  \mathcal{A} $  is  removed by the  same  argument as following  \eqref{eqn3.19}.  We omit the details.   

\setcounter{equation}{0} 
\setcounter{theorem}{0} 
\section{Final proof  of  Theorem \ref{theorem1.4}}    
\label{section5}
 To prove the statement on equality in  the  Brunn-Minkowski theorem  we shall need the following   lemma.   
\begin{lemma}  
\label{lemma5.1}
Let   $  \mathcal{A}  \in  M_p  (\al) $  satisfy  \eqref{eqn1.8} $(i)$   and  let $ E_1,  E_2,  \Om_1,  \Om_2,  G,  u, u_1, u_2   $  be as in  subsection \ref{subsection4.1}.
 If   
\begin{align}
\label{eqn5.1}  
- \frac{G_{\xi \xi} (x) }{ |\nabla G (x) |}  \geq  \tau  >  0 \, \, \mbox{whenever}\,\, \xi, x  \in \mathbb{S}^{n-1}\, \,  \mbox{with}\, \,   \lan \nabla G (x), \xi \ran  =  0
\end{align}     
then there exists $ R_1 =  R_1 (  \bar u,  \al, p, n ),    $  such that  if  $ \bar u   \in  \{ u, u_1, u_2  \}, $ then   
\begin{align*}
 - \frac{\bar u _{\ti\xi \ti \xi} (x) }{ |\nabla  \bar u (x) |}  \geq  \frac{\tau}{2 |x|}  >  0 \,\,  \mbox{whenever}\,\, \ti \xi  \in \mathbb{S}^{n-1}, \, \, |x| >  R_1,\, \,   \mbox{with}\, \,  \lan \nabla \bar u  (x), \ti \xi \ran  =  0.
\end{align*} 
 \end{lemma}        
 \begin{proof}   
 Let  $  \bar E  \in \{ E_1, E_2,  \la  E_1 + (1-\la) E_2 \}$  correspond to $  \bar u $ in     Lemma  \ref{lemma5.1}.   We note from  Lemma \ref{lemma3.4}  that    
 $  \{ x : \bar u ( x ) \geq  1/2  \} $ is convex  with nonempty interior  and  $  \min(2 \bar u , 1) $  is the capacitary function for this set.   Thus  we can  apply   Lemmas \ref{lemma3.2}, \ref{lemma3.3}  to  conclude the  existence of  $  R_0,$ and  $\bar c  \geq  1 $ depending on the data, $ \bar E,  \bar u $  with  $  \bar E \subset B (0, R_0/4)$  and    
\begin{align}
\label{eqn5.3}
\bar c^{-1}   |x|^{\frac{1-n}{p-1}}  \leq  - \lan x/|x|,  \nabla   \bar u \ran   \leq   |\nabla  \bar u  | (x)  \leq  \bar c  \,  |x|^{\frac{1-n}{p-1}}  
\end{align}
whenever $|x|>R_0$. We  note that  \eqref{eqn5.3}  also holds for    $  G $   with  $  \bar c $  replaced by  $ c $ provided  $  c  =  c (p, n, \al) $ is  large enough,  as  we see from  \eqref{eqn3.12}   and the construction of  $  G $ in  Lemma \ref{lemma4.1}. 
Set     
\[   
\bar e  (x)   :=   \bar u(x) -       \mbox{Cap}_{\mathcal{A}} ( \bar E)^{\frac{1}{p-1}}  \,  \,  G(x).  
\]
From  Lemma  \ref{lemma4.2} we see that   
\begin{align}
\label{eqn5.4}   
\bar  e  (x)  =   o( G (x) )  \, =  o( |x|^{\frac{p-n}{p-1} })\quad  \mbox{as } \,  \, |x|  \to   \infty.  
\end{align} 
  Also as in  \eqref{eqn4.8}-\eqref{eqn4.11} we deduce that  $  \bar  e  $  is  a  weak solution to the uniformly elliptic $ PDE $   
\begin{align}
\label{eqn5.5}
\bar{\mathcal{L}}\bar e:=\sum_{i,j=1}^n\frac{ \ar }{ \ar y_i} \biggl(  \bar a_{i j }( y )  \frac{ \ar \bar  e}{ \ar y_j}\biggr )=0  
\end{align}  
 on  $ B (x,  |x|/2)$  with  $|x|  \geq  R_0$    where  
\begin{align*}
\bar a_{ij} ( y )   =   \int_0^1   \frac{ \ar  \mathcal{A}_i}{\ar \eta_j }   \left( t  \nabla  \bar  u(y)  +  (1-t) \mbox{Cap}_{\mathcal{A}} ( \bar E)^{\frac{1}{p-1}}  \, \nabla G(y)   \right) d t.
\end{align*}
Moreover, for  some  $ c  \geq 1, $ independent of $ x, $  we also have
\begin{align}
\label{eqn5.7}   
c^{-1} \bar \si (y)  \,   | \xi |^2  \, \leq \,   \sum_{i,j=1}^n   \bar a_{ij} ( y  )  \xi_i  \xi_j  
\mbox{ and }     \sum_{i,j=1}^n   | \bar a_{ij} ( y  ) |   \leq c \, \bar \si (y) \,  
\end{align}
whenever $\xi  \in \rn{n} \sem \{0\}$ where  $\bar \si$ satisfies 
\begin{align}
\label{eqn5.8} 
\bar \si ( y  )  \approx  ( |  \nabla \bar u (y) | + \mbox{Cap}_{\mathcal{A}} (\bar E)^{\frac{1}{p-1}}    |  \nabla G (y) |  )^{p-2}  \approx |x|^{\frac{(1 - n)(p-2)}{p-1}} 
 \end{align}
for $|x|  \geq  R_0$. Constants depend  on various quantities but are independent of  $ x. $

From  well-known  results for  uniformly elliptic  PDE (see \cite{GT})  we see that   
\begin{align}   
\label{eqn5.9}
|x|^{-n/2}  \, \left(  \int_{ B ( x, |x|/4)} |\nabla \bar e |^2   \, dy \right)^{1/2}  \,    \leq c \, |x|^{-1}  \max_{B ( x, |x|/2)}  \, \bar e   =  o\left( |x|^{\frac{1-n}{p-1}} \right)\, \, \mbox{as}\, \, x \to  \infty,
\end{align}
where $ c $  as above  depends on various quantities but is  independent of  $ x. $     From \eqref{eqn5.9}, weak type estimates,   and Lemma \ref{lemma2.2} $(\hat a)$ for $ \bar u, G $ we also have 
\begin{align}  
\label{eqn5.10}
| \nabla \bar e (x)  | =   o\left( |x|^{\frac{1-n}{p-1}} \right)\,  \, \mbox{as}\, \, x \to  \infty. 
\end{align}     
Indeed,  given  $ \epsilon   > 0, $ we see from \eqref{eqn5.9}  that  there exists  $ \rho = \rho (\ep)  $     large,  such that  if  $ |x| \geq \rho,  $ then  
\[  
| \nabla  \bar e| \leq   \ep  |x|^{\frac{1-n}{p-1}}     \mbox{ on } \,  B ( x ,  |x|/2)
\]
except on  a  set  $ \Ga  \subset  B ( x, |x|/2) $ with  
\[
 \mathcal{H}^n (\Ga ) \leq  \ep^{n+1} |x|^n.    
 \] 
If   $ y \in \Ga $ and  $  \ep $ is small enough  there exists $ z \in B (x, |x|/2) \sem \Ga  $ with $ | z -  y | \leq \ep |x|. $   Then from \eqref{eqn2.2} $(\hat a)$  for $ \bar u, G $ we deduce    
\[  
| \nabla \bar e (y)|  \leq  \ep  |x|^{\frac{1-n}{p-1}}  + | \nabla \bar e (y) - \nabla \bar e (z) | \leq      \ep^{\be/2} \,   |x|^{\frac{1-n}{p-1}}    
\]
 for $ \ep $ small enough and $ |x| \geq \rho$. Since  $  \ep $  is  arbitrary   we  conclude the  validity of  \eqref{eqn5.10}.    

   We claim that also,
\begin{align}
\label{eqn5.11}   
\begin{split}
|x|^{-n/2}  \, \left(  \int_{ B ( x, |x|/4)}\,  \sum_{i, j =1}^n   \left|  \frac{ \ar^2  \bar e}{ \ar y_i  \ar y_j }  \right|^2   \, dy \right)^{1/2} &\leq c |x|^{-2}  \max_{B ( x, |x|/2)}  \, \bar e   \\
&= o \left( |x|^{\frac{2-n-p}{p-1}} \right) \mbox{ as } |x| \to  \infty. 
\end{split}
\end{align}
To prove  \eqref{eqn5.11} we first observe from  \eqref{eqn5.3}  for  $  \bar u,  G  $  that
\begin{align}
\label{eqn5.12}
\begin{split}
  &     \left(t  |\nabla  \bar  u(z) |  +  (1-t)\mbox{Cap}_{\mathcal{A}} ( \bar E)^{\frac{1}{p-1}}  \,| \nabla G(z)   |\right)   \\  
&\hspace{2cm} \leq    \left| t  \lan  \nabla  \bar  u(z) , z/|z| \ran  +  (1-t) \mbox{Cap}_{\mathcal{A}} ( \bar E)^{\frac{1}{p-1}}  \,\lan \nabla G(z), z/|z| \ran    \right|
\end{split}
\end{align}
when   $ z  \in  B (x,  |x|/2)  $  and  $ |x|   \geq  R_0. $ Using  \eqref{eqn5.12},  \eqref{eqn1.8} $(i),$  \eqref{eqn5.3}, and \eqref{eqn2.3} for  $ \bar u,  G $    we  deduce for some $ \breve{c}  \geq 1 $  and  $ \mathcal{H}^{n} $  almost every   $  \hat  x,  \hat y  \in B (x, |x|/2) $  with  $ | \hat x - \hat y | \leq  |x|/ \breve c $  that
\begin{align}
\label{eqn5.13}
\begin{split}   
|\bar a_{ij} (\hat x) -  \bar a_{ij}( \hat y ) | &\leq  \breve c \,     |  \hat x - \hat y |  \max_{B ( x, |x|/2)}   \left\{  \left( |\nabla \bar u(z) | +  | \nabla G(z) | \right)^{p-3}   \sum_{i,j= 1}^n  \, ( | u_{z_i z_j} (z) | + | G_{z_i z_j} (z) | )
  \right\} \\
& \leq  \breve c^2   | \hat x - \hat y |  \, |x|^{\frac{(2-p)n - 1}{p-1}} 
\end{split}
\end{align}
where $  \breve c $  is independent of  $ x, \hat x, \hat y $  subject to the above requirements.  Next we  use the  method  of  difference quotients.  Recall from the introduction that  $  e_m $ denotes the point with  $ x_l $ coordinate $ = 0, l \not = m, $ and $ x_m = 1.$  Let   
\[     
q_{h, m}    ( \hat y )  =    \frac{ q  (  \hat y  +  h e_m )  -    q (\hat y )}{h}   
\]  
whenever  $ q$ is defined at $\hat y$ where  $\hat y + h e_m \in   B (\hat x,  |x|/ \breve c  )$.  Let  $\ph$ be a non-negative functions satisfying 
\[
\ph \in  C_0^\infty ( B (\hat x,  |x|/(4 \breve c)) )\quad \mbox{with}\quad   \ph  \equiv  1 \, \,  \mbox{on}   \, \, B ( \hat x, |x|/(8 \breve c) ) \quad \mbox{and} \quad  |\nabla \ph  |  \leq    c_* |x|^{-1}. 
  \]
 Choosing appropriate test functions in   \eqref{eqn5.5} we see 
for $ 1 \leq  m  \leq  n $ that     
 \begin{align}    
 \label{eqn5.14}
 0  =     \int_{ B (\hat x,  |x|/(4\breve c) ) } \, \,   \sum_{i, j = 1}^n  ( \bar a_{ij}  \bar e_{ \hat y_i}  )_{h,m} \,   \,    (\bar e_{h,m} \, \ph^2 )_{\hat y_j}  d \hat y.
 \end{align} 
Using  \eqref{eqn5.7},  \eqref{eqn5.8}, \eqref{eqn5.10},  \eqref{eqn5.13}  to make estimates in  \eqref{eqn5.14}, as well as  Cauchy's inequality with  epsilon,  we find  for some $ c  \geq 1, $ independent of $ x, \hat x, $ 
\begin{align}
\label{eqn5.15}  
\begin{split}
|x|^{\frac{(1 - n)(p-2)}{p-1}} \int_{ B (\hat x,  |x|/(4\breve c)) }   & |  \nabla  \bar e_{h,m} |^2  \,  \ph^2  \,  d \hat y 
\leq  c  \int_{ B (\hat x,  |x|/(4\breve c) )}    \sum_{i, j =1}^n     a_{ij}       ( \bar e_{ \hat y_i}  )_{h,m} \,   \,    (\bar e_{\hat y_j} \, )_{h,m}  \,  \ph^2  \, d  \hat y \\
& \leq c^2  \int_{ B (\hat x, |x|/\breve{c} ) }  \sum_{i, j = 1}^n  | (\bar a_{ij})_{h,m} |  | \bar e_{ \hat y_i}  (\hat y+h) | \,   \,    | (\bar e_{h,m} \, \ph^2 )_{\hat y_j} | d \hat y
\\
 & \hspace{1cm}   + \frac{c^2}{|x|} \,  \int_{ B (\hat x,  |x|/(4\breve c)) }     \sum_{i, j =1}^n   | a_{ij}      ( \bar e_{ \hat y_i}  )_{h,m} \,   \,    \bar e_{\hat y_j} |\,   \, \ph  \, d  \hat y \\
& \leq    (1/2)   |x|^{\frac{(1 - n)(p-2)}{p-1}}  \int_{ B (\hat x,   |x|/(4\breve c)) }    |  \nabla  \bar e_{h,m} |^2  \,  \ph  \,  d \hat y   +   o\left( |x|^{\frac{2- n -p}{p-1}}\right).       
\end{split}
\end{align}
 It  follows from  \eqref{eqn5.15}   after some algebra that 
 \begin{align}
 \label{eqn5.16}
     |x|^{-n}   \,  \left(  \int_{ B (\hat x,  |x|/(8\breve c) ) }    |  \nabla  \bar e_{h,m} |^2  \,   \,  d \hat y \right)^{1/2}   =  
o \left( |x|^{\frac{2-n-p}{p-1}} \right)  \mbox{ as }  |x| \to   \infty.
\end{align} 
Letting  $ h  \to  0 $  in  \eqref{eqn5.16}  and covering    $ B ( x, |x|/2) $  by a finite number of balls of the form  $  B (\hat x, |x|/c ), $  we get  \eqref{eqn5.11}.  From \eqref{eqn5.11}, \eqref{eqn2.3}, and weak type estimates   it follows, as in the proof of \eqref{eqn5.10},    that      
\begin{align}
\label{eqn5.17}      
\sum_{i, j =1}^n   \left|  \frac{ \ar^2  \bar e}{ \ar x_i  \ar x_j }  \right| 
=  o\left( |x|^{\frac{2-n-p}{p-1}} \right)  \mbox{ as  } x \to  \infty. 
\end{align}
 We omit the details.  

We  now prove Lemma \ref{lemma5.1}.   Suppose  for some large $ x $ and $ \tilde \xi \in  
\mathbb{S}^{n-1}  $  that  $ \lan  \nabla  \bar u ( x),  \ti \xi  \ran  = 0. $  Then from 
\eqref{eqn5.10}  and   \eqref{eqn5.3}  for  $ G $  we see that  
\[   
\lan \nabla  G ( x), \xi \ran =  o \left( |x|^{(1-n)/(p-1)} \right) =  o \left( | \nabla G ( x ) |\right) 
\mbox{ as }  x  \to  \infty.    
\]
From this inequality we deduce that  $  \ti \xi =  \xi  + \la $  where  $ \xi $ is orthogonal to  $  \nabla G (x) $  and  
$ \la  $ points in the same direction as  $ \nabla G (x) $ with  $ |\la | = o(1) $ as 
$  x  \to  \infty.  $   Using these facts,  \eqref{eqn5.10}, \eqref{eqn5.17}, \eqref{eqn5.1}, \eqref{eqn2.3} for $ G,$  as well as  homogeneity of $ G $ and its derivatives,    we have  for large $ |x| $,   
\begin{align*}    
\frac{  \bar u_{\ti \xi \ti \xi}  (x)}{|\nabla \bar u (x) |} =
 ( 1 + o(1) )  \frac{  G_{\ti \xi \ti 
 \xi} (x) }{|\nabla G (x) |}   = o(1) |x|^{-1} 
+     \frac{ G_{\xi\xi}(x) }{|\nabla G (x) |}  \geq  (\tau /2) \, |x|^{-1} 
\end{align*}
for $ | x | \geq  R_0 $ provided $ R_0$ is large enough.    This finishes the proof of Lemma \ref{lemma5.1}. 
\end{proof} 

Next we state 
\begin{lemma} 
\label{lemma5.2}
 If  $ G $ is the Green's function for an    $ \mathcal{A} \in M_p (\al) $ satisfying   \eqref{eqn1.8} $(ii)$  then  \eqref{eqn5.1} is valid for some  $ \tau > 0. $ 
  \end{lemma}  
  
   The proof of  Lemma \ref{lemma5.2} is  given in  Appendix \ref{appendix2}. We continue  the proof equality in  Theorem \ref{theorem1.4} under the assumption that  Lemma \ref{lemma5.2} is valid.         Let  $ u, u_1, u_2,  E_1, E_2,  $  be  as in Lemma  \ref{lemma5.1}.    Following  \cite{CS}  we note  for  $ u^* $ as in  \eqref{eqn4.18} that  
          \begin{align}
          \label{eqn5.19}  
          \{ u^* ( x ) \geq t  \} =  \la \{ u_1 (y) \geq t \} + (1-\la) \{ u_2 ( z ) \geq t \} 
          \end{align}
          whenever $ t \in  (0, 1)$. Indeed  containment of the left-hand set in the right-hand set is  a  direct consequence of the definition of   $ u^*. $  Containment of the right-hand set in the left-hand set follows  from the fact that if 
     $ u^* (x)  = \min \left\{u_1 ( y), u_2( z)\right\}$  for some $ y \in E_1, z \in E_2, $ with 
     $ x =   \la y + ( 1 - \la ) z, $ then  $ u_1 ( y ) = u_2 (z). $   This fact is proved by the same argument as in the proof of \eqref{eqn3.23} or the display below  \eqref{eqn4.19}. 

 If  equality holds in \eqref{eqn1.7} in  Theorem \ref{theorem1.4} for  some $  \la  \in (0, 1),  $    we first observe from  Lemma  \ref{lemma3.1} and   \eqref{eqn3.4} $(a)$  that 
    for  almost every $ t \in (0,1), $   
\[   
\mbox{Cap}_{\mathcal{A}} ( \{ \hat u \geq  t\})  =   t^{1 - p }  \,  \mbox{Cap}_{\mathcal{A}} ( \{ \hat u \geq  1\})  \, \, \mbox{whenever}\, \,   \hat u \in \{u_1, u_2, u \}
\]   
and  second  that 
     \begin{align}
     \label{eqn5.20}  
     \mbox{Cap}_{\mathcal{A}} ( \{ u \geq  t\})^{\frac{1}{n-p}}  = \la   \mbox{Cap}_{\mathcal{A}} ( \{ u_1 \geq  t\})^{\frac{1}{n-p}}     +    (1-\la)  \mbox{Cap}_{\mathcal{A}} ( \{ u_2 \geq  t\})^{\frac{1}{n-p}}.  
\end{align}     
        On the other hand, using \eqref{eqn5.19}, convexity of  $  \{ u_i  \geq t \}, i = 1, 2, $ and  \eqref{eqn1.7}  we obtain 
     \begin{align}
     \label{eqn5.21}  
     \mbox{Cap}_{\mathcal{A}} ( \{ u^* \geq  t\})^{\frac{1}{n-p}}
     \geq  \la  \mbox{Cap}_{\mathcal{A}} ( \{ u_1 \geq  t\})^{\frac{1}{n-p}}      \
        + (1-\la) \mbox{Cap}_{\mathcal{A}} ( \{ u_2 \geq  t\})^{\frac{1}{n-p}}. 
        \end{align}  
 We conclude from     \eqref{eqn5.20}, \eqref{eqn5.21} that    for almost  every $ t \in  (0,1) $  
       \begin{align}
       \label{eqn5.22}  
       \mbox{Cap}_{\mathcal{A}} (  \{ u^* \geq t \} )    \geq    \mbox{Cap}_{\mathcal{A}} (  \{ u \geq t \} ).   
       \end{align}   
                            
Now from   \eqref{eqn4.18}  we see that   $ u^*  \leq  u  $   so   
       $ \{ u^*  \geq  t  \}  \subset  \{ u  \geq t  \}. $   This fact   and    \eqref{eqn5.22} imply
       for almost every $ t \in (0,1) $  that 
\begin{align} 
\label{eqn5.23}
 \{ u^* \geq t \} = \{ u \geq t \}. 
\end{align} 
To prove this statement let   $ U^*,  U $ be  the   corresponding  $ \mathcal{A}$-capacitary functions for these sets. Then from the maximum principle for  $ \mathcal{A}$-harmonic functions and  Lemma \ref{lemma3.1} we see that  $  U  -  U^*  \geq  0  $  in  $ \rn{n}. $   Moreover,  from  \eqref{eqn3.12} $(a)$  we deduce as in      \eqref{eqn4.8}-\eqref{eqn4.12}  that    $ U - U^* $  satisfies a uniformly elliptic PDE locally    in 
       $ \rn{n}  \sem \{ x :  U (x )  \geq 1 \}$  for which 
       non-negative  solutions satisfy a Harnack inequality.  It follows  from  Harnack's inequality and the usual chaining argument  that either  
       \[
       (+)\hspace{.5cm}  U  \equiv         U^*\, \,  \mbox{in}\, \,  \rn{n}  \sem \{ x :  U (x )  \geq 1 \}
        \]
        which implies  \eqref{eqn5.23},  or     
        \[
   \hspace{.4cm}     (++)\hspace{.5cm}  U - U^*  >  0\, \,  \mbox{in}\, \,  \rn{n}  \sem \{ x :  U (x )  \geq 1 \}.
        \]
         If $ (++) $  holds we see from a  continuity argument  that there exists  $  \rho > 0, \ga >1  $ for which  $ U, U^* $   are  $ \mathcal{A}$-harmonic in  $ \rn{n} \sem \bar B ( 0, \rho ) $ 
and   $ U/U^*  \geq  \ga $  on 
        $  \ar  B  (0, \rho ). $   Using the maximum principle for $ \mathcal{A}$-harmonic functions it would then follow that  
        \[
          U  \geq  \ga U^*\, \,  \mbox{in}\, \,  \rn{n} \sem \bar B (0, \rho ).  
          \]
         Dividing this inequality by $ G $ and taking limits as in 
        Lemma  \ref{lemma4.2}    we get, in contradiction to  \eqref{eqn5.22}, that  
    \[      
 \mbox{Cap}_{\mathcal{A}} (  \{ u \geq t \} )>   \mbox{Cap}_{\mathcal{A}} (  \{ u^* \geq t \} ).  
    \]     
    This proves  \eqref{eqn5.23}.  From continuity of  $ u, u^* $  we  conclude first that  \eqref{eqn5.23} holds for every $ t \in (0,1) $  and  second that  $ u^*   \equiv  u  $ in  $  \rn{n}.  $   Thus  \eqref{eqn5.19} is valid  with  $ u^* $  replaced by $ u.$   
    
    For fixed $ t \in (0,1)  $, let   $ h_i  (  \cdot, t ) $  be the support function for 
    $  \{ u_i \geq t \}$ for $i  = 1, 2,  $ and let  $ h ( \cdot, t ) $ be the support function for    
           $ \{ u \geq t \}. $  More specifically for  $ X \in \rn{n}$ and $t \in (0,1)$ 
           \[   
           h_i ( X, t  )  :=   \sup_{x \in \{ u_i  \geq t \} }  \lan X, x  \ran \, \, \mbox{for}\, \,  i = 1, 2  
          \quad \mbox{and}\quad
            h ( X, t  )  :=   \sup_{x \in \{ u  \geq t \} }  \lan X, x  \ran. 
            \]  
     From \eqref{eqn5.19} with $ u^*$  replaced by  $ u$  and the above definitions we see that 
     \begin{align}
     \label{eqn5.24}    
     h ( X, t )  =     \la  h_1 ( X, t )  +  (1- \la ) h_2 ( X, t ) \quad \mbox{for every} \, \, X \in 
     \rn{n}\, \, \mbox { and }\, \,  t  \in  (0, 1).
     \end{align}
 We  note from  \eqref{eqn2.3} and  Lemmas \ref{lemma3.3}, \ref{lemma3.4}, that  $   \nabla \bar  u  \not = 0 $  and  
   $ \bar u $ has  locally H\"{o}lder continuous second partials  in    $ \{ \bar u  <   1\}  $ whenever  $   \bar u   \in \{ u_1, u_2,  u   \}. $  From Lemmas \ref{lemma5.1} and \ref{lemma5.2}  we see there exists  $ t_0,  \tau_0  >0   $  small  and  $ R_0 $  large such that if $   \bar u   \in \{ u_1, u_2,  u   \}$ then 
 \begin{align}
 \label{eqn5.25}
\begin{split} 
 & (*)  \hs{.4in}  \{ \bar u \leq t\}   \subset  \rn{n} \sem \bar B (0, R_0) \,  \mbox{ for  } \, t \leq t_0 \leq 1/4,  \\ 
 &(**)  \hs{.23in} -  \frac{\bar u _{\ti\xi \ti \xi} (x) }{ |\nabla  \bar u (x) |}  \geq  \tau_0 \quad  \mbox{ whenever }\, \, \ti \xi  \in \mathbb{S}^{n-1} \, \, \mbox{and}\, \, |x| \geq   R_0\, \,  \mbox{with} \, \, \lan \nabla \bar u  (x), \ti \xi \ran  =  0. 
\end{split}
\end{align} 
  From  \eqref{eqn5.25}  we see that  the curvatures  at  points on   $ \{ \bar u = t \} $  are bounded away from  0 when  $  t  \leq  t_0 $.  Thus   
  \[
  -\frac{\nabla \bar u }{ | \nabla \bar u  | }\, \, \mbox{is  a  1-1  mapping    from} \, \,  \{ \bar u  = t \}\, \,  \mbox{onto}\, \,  \mathbb{S}^{n-1}
\]
  while     
  \[
   \left(-\frac{\nabla \bar u }{ | \nabla \bar u  | }, \bar u \right)\, \,  \mbox{is  a  1-1 mapping  from}\, \,   \{ u < t_0 \} \, \,  \mbox{onto} \, \, \mathbb{S}^{n-1}  \times  (0,t_0). 
   \]
From \eqref{eqn5.25},  elementary geometry,  and  the inverse function theorem it  follows  that if    $ \bar h $ is the support function corresponding to  $ \bar u \in \{ u, u_1, u_2 \} $ and  
$ 0 < t  <  t_0, $   then   $ \bar h $ has  H\"{o}lder continuous second partials and     
\begin{align} 
\label{eqn5.26}
\nabla_X \bar h ( X, t ) =  \bar x ( X, t )
\end{align}
where  $ \bar  x $ is the point in $ \{ \bar u = t \} $ with
\[
\frac{X}{|X|}  =-\frac{\nabla \bar u (\bar x)}{| \nabla \bar u (\bar x) |}.
\]
In  \eqref{eqn5.26},  $  \nabla_X $  denotes the gradient in the $ X $  variable only. Also  $ \bar h $ is homogeneous  of  degree one in the  $ X $  variable so  \eqref{eqn5.26} implies 
\[ 
 \bar h ( X, t ) = \lan X, \bar x ( X, t ) \ran \quad \mbox{and}  \quad  \frac{\ar \bar h}{\ar t } =  \lan  X,  \frac{\ar \bar x}{\ar t } \ran.    
 \]  
 Since  $ \bar u (\bar x) =  t$ and   $-  \nabla \bar u (\bar x) /
 |\nabla \bar u (\bar x) | = X/|X|$  we  get from the chain rule that 
\begin{align}
\label{eqn5.27}   
1 =   \lan \nabla \bar u,  \frac{\ar \bar x}{\ar t } \ran   = -  | \nabla u (\bar x) |  \, \, \frac{\ar \bar h}{\ar t }.
\end{align}   
 
Next since  $ 0 =  u^* - u  $  has an absolute maximum  at each  $ x   \in    \{ u  <  1\} $  we can repeat the  argument in  \eqref{eqn3.23}, \eqref{samedirection}  to deduce  that  there exists  $ y \in \{  u_1 <  1\}$, 
$z \in \{u_2 < 1 \}   $  with 
\begin{align}
\label{eqn5.28} 
x =  \la y  + (1-\la) z \quad \mbox{ and }  \quad  u (x) = u_1 (y) = u_2 (z). 
\end{align}
 Repeating the argument leading to \eqref{samedirection}  or  \eqref{eqn4.20} we find that 
\begin{align}
\label{eqn5.29}
\xi=\frac{\nabla u_1(y)}{|\nabla u_1(y)|}=\frac{\nabla u_2(z)}{|\nabla u_2 (z)|}=\frac{\nabla u(x)}{|\nabla u(x)|}
\end{align}  
and  after that  
\begin{align} 
\label{eqn5.30}
\begin{split} 
&u_1(y+\rho\eta)=u_1(y)+A_1\rho+A_2\rho^{2}+o(\rho^{2}),\\
&u_2 (z+\rho\eta)=u_2(z)+B_1\rho+B_2\rho^{2}+o(\rho^{2}),\\
&u(x+\rho\eta)=u(x)+C_1\rho+C_2\rho^{2}+o(\rho^{2})
\end{split}
\end{align}
as $\rho\to 0$ whenever $ \lan \xi, \, \eta \ran>0$  and $\eta\in\mathbb{S}^{n-1},$  
where     
\[
A=|\nabla u_1(y)|,\, \, B=|\nabla u_2(z)|,\, \, C=|\nabla u(x)|, \, \, \la  =  \frac{b}{a+b}. 
\] 
Using   \eqref{eqn5.30}  and  once again repeating the argument leading to  \eqref{vvu} we  first arrive at 
\begin{align}
\label{eqn5.31}
0\geq \sum\limits_{i,j=1}^{n} \left[\frac{(1-K)}{A^2} (u_1)_{x_i x_j}(y)+\frac{K}{B^2} (u_2)_{x_i x_j}(z)  - \frac{1}{C^2}  u_{x_i x_j}(x)\, \right]\eta_i \eta_j 
\end{align} 
where as  earlier,    
\begin{align}
\label{eqn5.32}
\frac{1}{C}  =\frac{(1 - \la) A+\la B}{AB} = \frac{ 1 - \la}{B} +    \frac{ \la}{A}\quad 
\mbox{and}\quad  K    =  \frac{ (1-\la) A }{ \la B  +  (1-\la) A }.
\end{align}
Using  \eqref{eqn5.31}  we  can argue as below  \eqref{vvu}    
to deduce first  that if  
\[
w (\hat x) = - \frac{(1-K)}{A^2} u_1 (y + \hat x ) - \frac{K}{B^2} u_2 (z + \hat x )\,     + 
\frac{1}{C^2}u (x + \hat x) ,  
\]  
then the  Hessian matrix of  $ w $ at $ \hat x = 0 $ is positive semi-definite.  Second if   
\[
 a_{ij} = \frac{1}{2}\left[\frac{\partial \mathcal{A}_i}{\partial \eta_j}  (\xi )   
+   \frac{\partial \mathcal{A}_j}{\partial \eta_i}  (\xi )  \right] \quad \mbox{for}\, \, 1 \leq i, j \leq n, 
\]    
then  $  (a_{ij}) $ is  positive definite and from  
$  \mathcal{A}$-harmonicity of  $ u, u_1, u_2, $  as well as   \eqref{eqn5.29},  
\[   
\mbox{trace} \left(  (a_{ij} )  \cdot ( w_{x_ix_j} (0)  \right)   =  0.   
\] 
From this equality   we conclude that  the Hessian of  $ w $ vanishes at  $ \hat x  = 0 $  so  by continuity,  equality  holds in \eqref{eqn5.31}  whenever   $ \eta \in  \mathbb{S}^{n-1}. $      
 
Using \eqref{eqn5.24}  we shall convert  this  equality into an  inequality involving support functions from which we  can make conclusions.   We shall need the following lemma  from   \cite{CS}: 
\begin{lemma}[\cite{CS}, Lemma 2]
\label{lemma5.3}
Let $H_1, H_2, $ be symmetric positive definite matrices  and let $ 0 < r,  s. $   Then for every $\lambda\in[0,1]$  the following inequality holds:
\[
(\lambda s+(1-\lambda) r)^{2} \mbox{trace}\left[ ( \la  H_1    +   (1-\la) H_2 ) ^{-1}\right] \leq  \lambda s^2\mbox{trace}\left[H_{1}^{-1}\right]+(1-\lambda)r^2\mbox{trace}\left[H_{2}^{-1}\right]
. \]
Equality holds if and only if 
\[
r H_1= s  H_2.
\]
\end{lemma} 
To  convert  \eqref{eqn5.24}  into an equality involving support functions we first assume that 
\begin{align}
\label{eqn5.33}
\xi   =   e_n  = (0, \dots, 0,1)\quad \mbox{and}\quad  u (x) = u_1 (y) = u_2 (z)  = t  
\end{align} 
in  \eqref{eqn5.28},  \eqref{eqn5.29}.    Then  $  X / |X|  =  e_n  $  and   from \eqref{eqn5.26}, as well as,  0-homogeneity of  the components of  $  \nabla_X \bar h $ we  see for fixed  $ t  \in (0, t_0 ) $  that    
\[
\bar h_{X_kX_n} = 0 \quad \mbox{for} \, \,  k=1,\ldots, n.   
\]
Also from the chain rule we  deduce   
for  $ 1 \leq i, j \leq n - 1 $  that 
\begin{align}
\label{eqn5.34} 
\de_{ij}   =   \sum_{i =1}^{n-1}    \,  \bar h_{X_i  X_k}  \frac{\ar X_k}{\ar  x_j} 
=  \sum_{i=1}^{ n-1} \bar  h_{X_i  X_k} \, \,  {  \frac{- \bar u_{  x_k   x_j}}{|\nabla \bar u | }  }
\end{align}
when $ X / |X| = e_n , $  where  $  \de_{ij} $ is the Kronecker $ \de $  and  partial derivatives of  $  \bar u $  are evaluated at $ x, y, z, $ depending on whether  $  \bar u = u, u_1, u_2 , $ respectively. For $1\leq i,j\leq n-1$, consider $(\bar h_{X_i X_j} )$ and $\left( \frac{- \bar u_{x_ix_j}}{|\nabla \bar u |} \right)$ as $(n-1)\times (n-1)$ matrices. Then \eqref{eqn5.34} implies (for $1 \leq i, j \leq n - 1$)
\begin{align} 
\label{eqn5.35}
(  \bar h_{X_i X_j} ) \mbox{ is the inverse of the positive definite matrix} 
\left( \frac{- \bar u_{x_ix_j}}{|\nabla \bar u |} \right). 
\end{align}
 For $1 \leq i, j \leq n - 1$, let
\[  
H_k  := ( (h_k)_{X_iX_j} ) \, \, \mbox{for}\,  \, k = 1, 2\quad  \mbox{and}\quad  H   := ( h_{X_iX_j} )   
\]
be the $(n-1)\times (n-1)$ matrices of second partials of the support functions corresponding to  $ h_1, h_2, h, $ respectively.  Using $ \eta = e_i$ for $1 \leq i  \leq n - 1, $  and multiplying each side of the   equality in  \eqref{eqn5.31} by  $   AB [ (1-\la) A  +  \la  B ]  $  we see in view of   \eqref{eqn5.32},  \eqref{eqn5.35}, after some algebra   that  the resulting equality 
  can be rewritten in terms of our new notation as     
\begin{align}
\label{eqn5.36}
 (\lambda B+(1-\lambda) A)^{2} H^{-1} = \lambda B^2 H_{1}^{-1}+(1-\lambda)A^2 H_{2}^{-1}. 
\end{align}
Now from  \eqref{eqn5.24}  we also have  $  H =  \la H_1 +  (1-\la ) H_2 $  so  obviously,  
 $  H^{-1} =   ( \la  H_1  +  (1-\la ) H_2  )^{-1}  $.  Using  this  equality in  \eqref{eqn5.36} we  conclude from Lemma \ref{lemma5.3}  with  $ A = r,  B = s $    that at $ ( X, t ), $ 
 $ A  H_1= B H_2$  and thereupon from  \eqref{eqn5.32}, \eqref{eqn5.24} that  
\begin{align}  
\label{eqn5.37}
A H_1 = B H_2  = C H    \mbox{ at $ (X, t ) $  when  \eqref{eqn5.33}  holds.} 
\end{align} 
We continue under assumption  \eqref{eqn5.33}.  Following \cite[page 470]{CS}, we will compute $\bar u_{ x_n x_n } (\bar x ) $ in terms of second   partial derivatives  of $ \bar h ( X, t ) $  where  $\bar  x = x, y, $ or $ z $ in  \eqref{eqn5.28} depending on whether $ \bar u = u, u_1 $ or $ u_2. $    
From  the chain rule, \eqref{eqn5.27},  and \eqref{eqn5.33}, 
\begin{align}
\label{eqn5.38}
\begin{split}
  -  \bar u_{x_n x_n}&=\frac{\partial }{\partial x_n}\left(\frac{1}{\bar h_t(X,t)}\right)=-\frac{1}{\bar h^{2}_{t}}\left[\sum\limits_{i=1}^{n} \frac{\partial \bar h_t}{\partial X_i }\frac{\partial X_i}{\partial x_n} + \bar h_{tt}\frac{\partial t}{\partial x_n}\right] \\
&=-\frac{1}{\bar h^{2}_{t}}\sum\limits_{i=1}^{n} \frac{\partial \bar h_t}{\partial X_i }\frac{\partial X_i}{\partial x_n} + \frac{1}{\bar h^{3}_{t}} \bar h_{tt}.
\end{split}
\end{align}
Taking derivatives in  \eqref{eqn5.26}   we also have   for $i=1, \ldots, n-1,$ 
\begin{align}
\label{eqn5.39}
 0 =\sum_{j=1}^{n-1} \bar h_{ X_i X_j} \,  \frac{\partial X_j}{\partial x_n} +   \bar h_{X_i t} \, \frac{\partial t}{\partial x_n}  =\sum\limits_{j=1}^{n-1} \bar h_{ X_i  X_j} \, \frac{\partial X_j}{\partial x_n} -  \frac{ \bar h_{X_i t}}{\bar h_t}.
 \end{align}
Using   \eqref{eqn5.39}   to  solve  for   $   \frac{\ar X}{\ar x_n} $  and then putting the result in   \eqref{eqn5.38} we obtain at  $ (\bar x, t ), $ 
\begin{align}
\label{uetaeta}
\begin{split}
- \bar u_{x_n x_n }& =  - \frac{1}{\bar h_t^{2}}\sum\limits_{i=1}^{n-1} \bar h_{tX_i} \frac{\partial X_i}{\partial x_n} +   
\frac{1}{\bar h_{t}^{3}} \bar h_{tt}\\
&=- \frac{1}{\bar h_t^{3}} \left[\langle   \, \nabla_{X} \bar h_t \, ( \bar h_{X_i X_j} )^{-1} , \nabla_{X} \bar h_t \rangle- \bar h_{tt}\right]
\end{split}
\end{align} 
where  $  \nabla_X \bar h $ is written as  a  $ 1  \times  n-1 $ row matrix.  Let  $ M $ denote the inverse of the matrix in   \eqref{eqn5.37}.   Note that $ M $ is positive definite and symmetric. 
Using  \eqref{eqn5.37},  \eqref{eqn5.27}, \eqref{uetaeta}, as  well as  the notation used  previously for  gradients of  $ u, u_1, u_2 $  at  $  x, y, z ,$    we see that       
\begin{align}
\label{eqn5.41}
\begin{split}
&-\frac{u_{x_nx_n}(x)}{C^2}   =     \, C^2  \, \lan  \nabla_X h_t   \,M,  \nabla_X h_t  \ran     - C h_{tt}, \\ 
&-\frac{(u_1)_{x_nx_n}(y)}{A^2}   =     A^2  \, \lan  \nabla_X (h_1)_t  \,M,  \nabla_X  (h_1)_t  \ran     - A \, (h_1)_{tt}, \\
&-\frac{(u_2)_{x_nx_n}(z)}{B^2}  =    B^2 \, \lan  \nabla_X (h_2)_t  \,M,  \nabla_X (h_2)_t  \ran     - B \, (h_2)_{tt}.
\end{split}
\end{align}
Using   \eqref{eqn5.41} in the equality in \eqref{eqn5.31}  with  $ \eta  = e_n, $  we find that 
\begin{align}
\label{eqn5.42}
\begin{split}
C^2  \lan  \nabla_X h_t  \,M,  \nabla_X h_t  \ran     - C h_{tt} 
     & =     \,    \frac{\la B }{\la B  +  (1-\la) A}  \, \,  [ A^2 \,  \lan  \nabla_X (h_1)_t  \,M,  \nabla_X (h_1)_t  \ran     - A (h_1)_{tt} ] \\
   & \hspace{.4cm}+   \frac{( 1 - \la) A}{\la B  +  (1-\la) A}    \, [ B^2 \lan  \nabla_X (h_2)_t  \,M,  \nabla_X (h_2)_t  \ran     - B (h_2)_{tt}].
\end{split}
\end{align}
Since  
\[
h  =  \la h_1  +  (1-\la) h_2 \quad  \mbox{and}\quad  C  =   \frac{ A B }{\la B  + (1-\la) A},
\]
 the terms involving two derivatives in  $ t $ on both sides of  \eqref{eqn5.42}  are equal so may be removed.   Doing this and using above identity involving $h$ and $C$ once again  we  arrive at       
\begin{align}
\label{eqn5.43} 
\begin{split}
\frac{ A^2  B^2  }{( \la B  + (1-\la) A)^2 } , \lan  ( \la \nabla_X (h_1)_t  &+ ( 1 - \la ) \nabla  (h_2)_t  )  \,M,  ( \la \nabla_X (h_1)_t   + ( 1 - \la ) \nabla  (h_2)_t  )   \ran     
     \\ &  =     \,    \frac{\la A^2 B }{\la B  +  (1-\la) A}  \, \,    \lan  \nabla_X  (h_1)_t   \,M,  \nabla_X  (h_1)_t   \ran    \\
     &      \hspace{.2cm} +  \frac{( 1 - \la) A B^2 }{\la B  +  (1-\la) A}    \, \lan  \nabla_X (h_2)_t  \,M,  \nabla_X  (h_2)_t  \ran.
\end{split}      
      \end{align}
      For ease of notation let
      \[
   \Upsilon:=\frac{ \la (1-\la)  }{( \la B  + (1-\la) A)^2 }.
      \]
Multiplying \eqref{eqn5.43} with this expression out, using partial fractions,     and  gathering terms in   $  \lan  \nabla_X (h_i)_t  M, \nabla_X (h_i)_t   \ran  $  for  $ i  = 1, 2, $ we see that 
  \begin{align*}
  2 \Upsilon A^2  B^2 & \lan  \nabla_X (h_1)_t  M, \nabla  (h_2)_t  \ran \\
  &=    \Upsilon A^3 B \lan  \nabla_X  (h_1)_t   \,M,  \nabla_X  (h_1)_t   \ran  +  \Upsilon A B^3 \lan  \nabla_X (h_2)_t  \,M,  \nabla_X  (h_2)_t  \ran.  
  \end{align*} 
This equality can be factored into 
\begin{align*}
\Upsilon \lan  ( A^{3/2} B^{1/2}  \nabla_X (h_1)_t & - B^{3/2} A^{1/2}  \nabla  (h_2)_t  )  \,M,   A^{3/2} B^{1/2}  \nabla_X (h_1)_t  - B^{3/2} A^{1/2}  \nabla  (h_2)_t     \ran\\
& = 0.
\end{align*}        
Since   $ M $ is positive definite  we conclude from this equality that 
\begin{align}
\label{eqn5.44}  
A   \nabla_X (h_1)_t  =   B  \nabla_{X}  (h_2)_t  \quad \mbox{or equivalently that}\quad   \nabla_X  \log \left( \frac{(h_1)_t}{(h_2)_t} \right) =  0. 
\end{align}
For $i=1,2$, let   $  \bar x_i ( X, t )$  be the parametrization of  $ \{ u_i = t  \}$   in  \eqref{eqn5.26} for $ t < t_0$ and $X   \in  \mathbb{S}^{n-1} $.   From  \eqref{eqn5.37}, \eqref{eqn5.44}, and  \eqref{eqn5.26}  we see that if  \eqref{eqn5.33} holds then 
\begin{align}
\label{eqn5.45}
  \frac{\ar}{\ar X_i } \left(\frac{| \nabla u_2 |  (\bar x_1) }{ | \nabla u_1 |  (\bar x_2)} 
\right) = 0 \quad   \mbox{ and  } \quad \frac{\ar}{\ar X_i }  \left(    \bar x_1    -     \bar  x_2  \,   
\frac{| \nabla u_2 |  (\bar x_1) }{ | \nabla u_1 |  (\bar x_2)}  
\right)   = 0   
\end{align}
when $1 \leq i  \leq n$ at $ (e_n, t ). $  

At this point,  we remove the assumption \eqref{eqn5.33}. If  \eqref{eqn5.33}
does not hold we can  introduce  a  new  coordinate system say  $  e_1' ,  \dots, e_n', $   with 
$ e_n'  =   \xi$  in \eqref{eqn5.29}. Then calculating   partial derivatives of  $   \bar u,  \bar h  $  in  this new coordinate system we  deduce that (for $1 \leq i, j \leq n - 1$)
\begin{align}   
\label{eqn5.46}
 (\bar h_{X'_i X'_j} ) \quad \mbox{is the inverse of the positive definite matrix} \quad  
\left( \frac{- \bar u_{x'_ix'_j}}{|\nabla \bar u |} \right). 
\end{align} 
Here   $  \bar h_{X'_i X'_j} , \, \,   \bar u_{x'_ix'_j}, \, $  denote second directional derivatives of   
$  \bar h,   \bar u $   in the  direction of  $ e_i' \, e_j'$ for $1 \leq i, j \leq n - 1$.  Using  \eqref{eqn5.46}   we can repeat the argument after  
\eqref{eqn5.35} with  $ x' ,  X' $ replacing  $ x,  X $    to  eventually  conclude  \eqref{eqn5.45} holds  at  
$  ( \xi, t ). $   

Hence we can continue with \eqref{eqn5.45}. Since $ \xi  \in  \mathbb{S}^{n-1}$ and $t   <   t_0 $  are arbitrary  and $ \bar x_1, \bar x_2 $ are  smooth we conclude for fixed $ t $  that  there exist $ a, b \in \re $ with  
\[
 \bar x_1 ( X, t )   =  a  \bar x_2 ( X, t )  +  b \quad  \mbox{whenever}  \quad X  \in  \mathbb{S}^{n-1}. 
 \]
 Using 
Remark  \ref{rmk1.3}  and the maximum principle for  $  \mathcal{A}$-harmonic functions we conclude that     
\[
 u_2  ( x )  =  u_1  ( a x + b ) \quad \mbox{whenever}\quad    u_2 ( x )  < t \, \, \mbox{and}\, \,  t < t_0. 
 \]   
 To finish the proof of  Theorem \ref{theorem1.4} let     $ v  ( x )  = u_2  ( x )  -  u_1  ( a x + b )$ for $x \in \rn{n}.$   Fix $ s  \in  (0, 1) $  and  let   
 \[    F_1  =  \{ x : u_1 ( a x + b )  \geq  s  \} \quad \mbox{and} \quad  F_2  =  \{ x :  u_2 (x)  \geq s  \} .   \]
 if $  F_1 \not =  F_2 ,   $   we    assume, as we may,  that   $  z $  lies  in  the  interior  of 
 $  F_2 \sem  F_1. $  Fix 
 $ y  $  in  the  interior of  $ F_1 $   and  draw  the  ray  from  $ y $  through $ z $  to  $ \infty. $  Let  $ w $  denote  the   point on  this  ray  in $ F_2 $   whose distance is  furtherest  from  $ y. $   Let  
$ l   $ denote the  part of  this ray  joining $ w $  to  $  \infty. $  If $ x  \in l,  x  \not = w ,  $ then  since  $ F_1,  F_2 $  are  convex and $ \mathcal{A}$-harmonic functions are invariant under  translation and  dilation,  it  follows from \eqref{eqn3.12} $(a)$  that  
\begin{align}
 \label{eqn5.47}   
\lan  \nabla u_2 (x),   x - w \ran   < 0\quad  \mbox{ and } \quad \lan  \nabla u_1 (ax+b) ,   x - w  \ran   < 0. 
\end{align}
From  arbitrariness of  $ x $  and  \eqref{eqn2.2}  $(\hat a)$ it  now follows  that there is  a  connected open set,  say  $ O,  $  containing all points  in 
$ l $  except possibly  $ w, $   for which   \eqref{eqn5.47} holds whenever  $ x  \in  O. $   Clearly this  inequality implies  that   
\begin{align}
  \label{eqn5.48}   
    |   \tau   \nabla u_2  (x)  +  ( 1 - \tau )  \nabla u_1 (ax + b)  |  \not  =  0   \quad   \mbox{when}  \, \, x \in O\, \,  \mbox{and}\, \,   \tau  \in [0, 1]. 
  \end{align}
  Using   the same  argument  as  in   either   \eqref{eqn5.4}-\eqref{eqn5.8}   or  \eqref{eqn4.8}-\eqref{eqn4.11} with $  \mathcal{A} =  \nabla f $  we deduce  that $ v $   is   a weak  solution   to  
  \[     \hat L  v  =  \sum_{i, j = 1}^n  \frac{\ar}{\ar x_i}  (  \hat a_{ij} v_{x_j} )  =   0  \mbox{ in } O     \]    where   
  \[  \hat a_{ij} (x)   =   \int_0^1   \frac{ \ar^2  f}{ \ar \eta_i  \ar \eta_j} (  \tau   \nabla u_2  (x)  +  ( 1 - \tau )  \nabla u_1 (ax + b) )   \,   d  \tau  \]    From   this  deduction,  Definition \ref{defn1.1},  \eqref{eqn1.8},  \eqref{eqn2.3},   and   \eqref{eqn5.48},  we  see that     $ v $  is    a  weak  solution to  a   locally  uniformly elliptic  PDE  in  divergence form  with  Lipschitz continuous and symmetric coefficients.  Since  $ v  \equiv 0 $  in  a  neighborhood  of  $ \infty $  it now  follows  from  a unique  continuation theorem  (see for example  \cite{GL}) that  
 $ v $  vanishes in  $ O. $   Then by continuity  $ v (w) = 0$  so  $ w $ is also in  $ F_1. $  We have reached a  contradiction.  Thus  $ F_1 = F_2 $  whenever  $ s > 0 $  and  consequently  
 $ v  \equiv  0  $  on  $  \rn{n}.  $  

\setcounter{equation}{0}
\setcounter{theorem}{0} 
 
\section{Appendix}
\label{appendix}

\subsection{Construction of a  barrier in (\ref{eqn3.18})}
\label{appendix1}

In  this  section we construct  a  barrier to   justify   display  \eqref{eqn3.18}  for  $ 1- u. $  Recall that  $u$ is  the  $  \mathcal{A}$-harmonic capacitary  function in Lemma \ref{lemma3.3}.  
  Let   $\hat{w}, \delta$ be as in \eqref{eqn3.17} and put $  \mathcal{\ti A} (\eta) =-\mathcal{A} (-\eta)$ whenever $\eta \in \rn{n}$.   Let $\epsilon>0$ be given and small. We define 
\[
\mathcal{\ti A}(\eta,\epsilon):=\int\limits_{\mathbb{R}^{n}} \mathcal{\ti A}(\eta-x)\theta_{\epsilon}(x)dx
\]
whenever $\eta\in\rn{n}$ and $\theta\in C_{0}^{\infty}(B(0,1))$ with
\[
\int\limits_{\rn{n}} \theta(x)dx=1 \, \, \mbox{and}\,\,\, \theta_{\epsilon}(x)=\epsilon^{-n} \theta(x/\epsilon) \, \, \,\mbox{for}\, \, x\in\rn{n}. 
\] 
From Definition \ref{defn1.1} and well-known properties of  approximations to  the identity, it follows that there exists $c=c(p,n)\geq 1$ such that 
\begin{align}
\label{eqn6.1}
(c\alpha)^{-1}(\epsilon+|\eta|)^{p-2}|\xi|^2\leq \sum_{i,j=1}^n \frac{\partial  \mathcal{\ti A}_{i}}{\partial\eta_j}(\eta,\epsilon)\xi_i\xi_j\leq c\alpha (\epsilon+|\eta|)^{p-2}|\xi|^2.
\end{align} 
Note that $\mathcal{\ti A}(\cdot, \epsilon)$ is infinitely differentiable for fixed $\epsilon>0$. Let $v(\cdot,\epsilon)$ be the solution to 
\[
\nabla\cdot \mathcal{\ti A}(\nabla v(z,\epsilon),\epsilon)=0
\]
with continuous boundary values equal to $1 - u$ on $\partial B(\hat{w},1)$. Let
\[
\mathcal{\ti A}^{*}_{ij}(z,\epsilon)=\frac{1}{2}(\epsilon+|\nabla v(z,\epsilon)|)^{2-p} \left[\frac{\partial \mathcal{\ti A}_i}{\partial \eta_j}(\nabla v(z,\epsilon),\epsilon) +\frac{\partial \mathcal{\ti A}_j}{\partial \eta_i}(\nabla v(z,\epsilon),\epsilon)\right]
\]
whenever $z\in B(\hat{w},1)$ and $1\leq i,j\leq n$. 
Note also that the  ellipticity constant for $\{\mathcal{\ti A}^{*}_{ij}(z,\epsilon)\}$ and the $L^{\infty}$-norm for $\mathcal{\ti A}^{*}_{ij}$, $1\leq i,j\leq n,$ in $B(\hat{w},1)$ depend only on $\alpha,p,n$.  
From \eqref{eqn6.1}  and Schauder type estimates we see that  $v (\cdot, \epsilon)$ is a classical solution to the non-divergence form uniformly elliptic equation,  
\[
\mathcal{L}^{*}v=\sum\limits_{i,j=1}^{n}\mathcal{ \ti A}^{*}_{ij}(z,\epsilon) v_{y_i y_j} (z)=0
\]
for $z\in B(\hat{w},1)$.   Moreover,  if we let
\[
\psi(z)=\frac{e^{-N|z-\hat{w}|^{2}}-e^{-N}}{e^{-N/4}-e^{-N}}
\]
whenever $z\in B(\hat{w},1)\setminus \bar B(\hat{w},1/2)$. Then $\mathcal{L}^{*}\psi \geq 0$ in $B(\hat{w},1)\setminus \bar B (\hat{w},1/2)$ if $N=N(\alpha, p,n)$ is sufficiently large, so  $  \psi $  is  a subsolution to $\mathcal{L}^* $  in   $B(\hat{w},1)\setminus \bar B (\hat{w},1/2)$. Also  by construction of $\psi$, we  have $\psi=1$ on $\partial B(\hat w,1/2)$ and $\psi=0$ on $\partial B(\hat w,1)$. Comparing boundary values of  $ v (  \cdot, \ep), \psi  $  and using the maximum principle for  $\mathcal{L}^*$  we conclude that  
\[
 v   \geq   ({\ds  \min_{\bar B ( \hat w, 1/2)} } v  )  \, \psi\quad \mbox{in}\, \,  B (\hat{w},1) \sem  \bar B(\hat{w},1/2).
\]
Moreover, it is   easily checked that for some $c=c(p,n,\alpha)\geq 1 $
\[
c\,\psi(z)\geq (1-|\hat{w}-z|) \, \, \mbox{\, \, whenever}\, \, z\in B(\hat{w},1)\setminus \bar B(\hat{w},1/2).
\]
Thus  
\begin{align}
\label{eqn6.2}
 \hat{c}  \,  v(z,\epsilon)) \, \geq \, (1-|\hat{w}-z| ) \, \min_{\bar B ( \hat w, 1/2)} v  \, \mbox{\, \, whenever}\, \, z\in B(\hat{w},1)\setminus \bar B(\hat{w}, 1/2)
\end{align}
for some $\hat{c}=\hat{c}(p,n,\alpha)\geq 1 $. We note  from  Lemmas \ref{lemma2.1},  \ref{lemma2.2}, that  
a subsequence of  $ \{1-  v ( \cdot, \ep) \} $  converges uniformly on compact subsets of $ B ( \hat w,  1 )  $  to  an  $ \mathcal{A}$-harmonic function in  $ B (\hat w, 1). $   Also by the same reasoning as in the proof  of  Lemma  \ref{lemma2.3} $(ii)$  one can derive  H\"{o}lder continuity estimates for  $ v $ near  $  \ar  B (\hat w, 1) $  which are independent of  $ \ep. $ Using these facts and letting  $\epsilon\to 0$  we see that a subsequence of  $ \{v ( \cdot, \ep )\}$ 
converges uniformly on  $ \bar B (\hat w, 1 ) $ to  $ 1 - u. $  In view of    \eqref{eqn6.2} and \eqref{eqn3.17} 
we have    
\[
c  (1-u(z))\geq \delta \,  (1-|\hat{w}-z|) \, = \,  \delta \, d(z, \partial B(\hat{w},1))
\] 
whenever $z\in B(\hat{w},1)\setminus \bar B(\hat{w},1/2)$ which is \eqref{eqn3.18}.   

\subsection{Curvature estimates for the levels of fundamental solutions}
\label{appendix2}
In this subsection  we prove Lemma  \ref{lemma5.2}  when  $   \mathcal{A}  \in  M_p ( \al ) $  can be written in the  form   (see \eqref{eqn1.8}): 
\begin{align}
\label{eqn6.3}
\mathcal{A}_{i}=\frac{\partial f}{\partial \eta_i}(\eta), \, \, 1\leq i\leq n, \, \mbox{where}\, \, f(t\eta)=t^{p} f(\eta)\, \, \mbox{when}\, \, t>0,\,\,\eta\in\mathbb{R}^{n}\setminus \{0\}
\end{align}
 and  $  f  $  has   continuous second partials on $ \rn{n} \sem\{0\}$. 
The proof  is  based on some ideas garnered from reading  \cite{CS1}.   To begin we write 
$f(\eta)=(k (-\eta))^{p}$ and note from  \eqref{eqn6.3} that  $ k ( \eta )$ for $\eta \in  \rn{n} \sem  \{0\}$ is  homogeneous of degree $1$ and has continuous second partials on $ \rn{n} \sem \{0\}. $   We claim that 
\begin{align}
\label{eqn6.4} 
  k^2  \mbox{ is strictly convex  on }  \mathbb{R}^n.   
  \end{align}  
  To prove   \eqref{eqn6.4}  
  let  $ \la  \in    \{ \eta  :  k ( \eta ) =  1  \}$ and put 
  \[
   \La  =  \{  \xi  \in   \mathbb{S}^{n-1}  :    \lan \nabla k (\la ),  \xi  \ran  = 0 \}. 
   \]  
   From  convexity of   $ f $  on  
  $ \mathbb{R}^n $ (see Definition \ref{defn1.1} $(i)$)  and  the definition of  $ k $  we see first  that    
  \[  
 f_{\eta_i \eta_j}(- \eta)= p (p-1) \, k_{\eta_i}(\eta) \, k_{\eta_j}(\eta)  k^{p-2}\,  + \, p \, k^{p-1}(\eta) \, k_{\eta_i \eta_j}(\eta)  
 \] 
for $ 1  \leq i, j \leq n$. Thereupon we conclude  for some $ c  \geq 1 $ depending only on the data that 
 if  $ \xi \in \La$ then
 \begin{align}  
 \label{eqn6.4a}   
 c^{-1}    \leq  f_{\xi \xi} (- \la )   =   p   \, k_{\xi \xi} (\la )  
 \leq c.  
 \end{align}      
 Next we   observe from  1-homogeneity of  $ k $  that  $ \la $  is  an  eigenvector corresponding to the eigenvalue 0 for  the Hessian of  $  k $  evaluated  at  $  \la.  $  
 Also 
 \[   
 \lan \nabla  k (\la ),  \la \ran   =  k (  \la  ) \approx  1  
 \] 
 so    we  can   write   
 \[ \tau  =    \nabla k ( \la )/ | \nabla k ( \la ) |  =  a \la + b  \,\xi   \]  where $ \xi \in 
 \La $  and  $ a  \approx  1. $  Again all  ratio constants depend only on the data.    We  conclude  from   
 \eqref{eqn6.4a}   and the above facts  that  
 \[   
 k_{\tau \tau }  \geq  b^2 \, k_{\xi \xi }  \geq  0.  
 \] 
 Thus $ k $ is  positive semidefinite and an  easy calculation  using the above facts now gives  \eqref{eqn6.4}.       
   
   From        \eqref{eqn6.4} we see  as in \eqref{eqn5.26}  that if $ X \in  \rn{n} \sem \{0\},$  then      
\[   
h  ( X )   =  \sup  \{  \lan \eta,  X  \ran :\, \,   \eta \in  \{k \leq  1\}  \}  
\] 
has continuous second partials  and  $ h $ is homogeneous of  degree $1$.   Moreover,  
\begin{align}
\label{eqn6.5} 
\nabla  h ( X ) =   \eta  ( X )\, \, \mbox{where}\, \,  \eta  \, \, \mbox{is the point in}\, \, \{  k = 1 \}\, \, \mbox{with}\, \,   \frac{X}{|X|}=  \frac{\nabla k (\eta  )}{| \nabla k  (\eta )|}.
\end{align}    
From  calculus and  Euler's formula for $1$-homogeneous functions it now follows that  if  $ X  \in   \mathbb{S}^{n-1}$ then   
  \begin{align*}
  h (X)  =   \lan  \eta ( X ),  X  \ran  =  |X|  \lan \, \eta (X),  \frac{ \nabla k (\eta)}{ | \nabla k (\eta) | } \, \ran  =  \frac{ |X|  }{ | \nabla k (\eta) |}.
  \end{align*}  
  Using this equality  we obtain first  
  \begin{align*}
\nabla k  (  \nabla h ( X ) )   \, =  \, \frac{ |  \nabla k (\eta)|  X}{ |X| } =  
\frac{ X }{ h ( X ) }
\end{align*}
and  second using  $1$-homogeneity of  $ k, h $ as well as 0-homogeneity of  $ \nabla k,  \nabla h, $  that 
\begin{align}
\label{eqn6.8} 
  k [ h (X)  \, \nabla h ( X ) ]    \, \,   \nabla k  [ h (X)  \nabla h ( X ) ]  =
h (X) \,  k  [ \nabla h (X)]   ( X/ h (X) )  = X. 
\end{align}
Thus  $ k  \, \nabla k  $  and  $ h \,  \nabla h $  are inverses of  each other on $  \rn{n} \sem \{0\}. $ 

For fixed $p$, $1<p<n$, let $\beta=(p-n)/(p-1) < 0$ and define
\[
\hat G(X)=h(X)^\beta \quad  \mbox{whenever}\, \,  X \in \rn{n} \sem \{0\}.
\]
We claim that  
$ \hat G$ is  a constant multiple of  the fundamental solution for the  $  \mathcal{A} $ in  \eqref{eqn6.3}. 
  Indeed,  if    $ X \in \mathbb{R}^{n}\setminus \{0\}, $  it  follows from \eqref{eqn6.5}-\eqref{eqn6.8} that    
\begin{align}
\label{eqn6.9}
\begin{split}
(\nabla f) (\nabla \hat G(X)) &=  -p\,k^{p-1} (-\nabla \hat G (X))\, \, (\nabla k (-\nabla \hat G(X)) ) \\
 & = \frac{-X}{h(X)} \, p  \, k^{p-1} (-\nabla \hat G (x) )  \\ 
  &=   X p (-\beta)^{p-1} \, h^{[( \beta - 1 )(p - 1) - 1 ] } (X) \, k^{p-1}  (\nabla  h (X))  \\
  &=   X \, p\,\, (-\beta)^{p-1}  h(X)^{-n}.
\end{split}
\end{align}
Now  $ X \mapsto   h ( X/|X| )^{-n} $ is homogeneous of  degree $0$  so   
\[
\lan  X,  \nabla [ h (X/|X| )^{-n} ] \ran  = 0
\]
by Euler's formula.  From this  observation and  \eqref{eqn6.9} we deduce    
\begin{align*}  
\begin{split}
p^{-1} (-\be)^{1-p}  & \nabla \cdot \left(   (\nabla f) (\nabla \hat G(X)) \right) 
\\&= h (X/|X| )^{-n} \, \, \nabla \cdot ( X |X|^{-n} )  +   |X|^{-n}  \lan X, \nabla [ h (X/|X| )^{-n} ]  \ran \\
    &= 0
    \end{split}
\end{align*}   
when  $  X  \in  \rn{n}  \sem  \{0\}. $  Hence  $  \hat  G  $ is  $  \mathcal{A}$-harmonic in 
$ \rn{n} \sem \{0\}. $  Now from $1$-homogeneity of  $ h$ and \eqref{eqn6.5} it is easily seen that   
\eqref{eqn4.1} $ (a), (b), $ are valid for $  \hat  G $  with constants that depend only on $ p, n, \al. $   Also  from  \eqref{eqn6.9}  we  note  that   
\[  
 | \nabla f(\nabla \hat G (X) ) | \approx  |  X  |^{1-n}\, \mbox{ on }\,   \rn{n}  \sem \{0\}.  
 \]   
 If   $ \he  \in  C_0^\infty ( \rn{n} ) $  then from the above display we deduce  that  the function $ X  \mapsto  
 \lan   \nabla f  (\nabla \hat G (X) ) ,  \nabla  \he (X) \ran $ is integrable on  $ \rn{ n}. $ Using this fact,  smoothness of  $  f,  h,  $  and an  integration by parts,  we get   
\begin{align}
\label{eqn6.11}
\begin{split}
\int_{\rn{n}}  \lan   \nabla f  (\nabla \hat G (X)) ,  \nabla  \he (X)  \ran dx  &=
 -  \lim\limits_{r \to  0}  \int_{\ar  B (0, r )} \,  \he (X) \,  \, \lan \nabla f ( \nabla \hat G (X) ), \frac{X}{|X|} \ran 
\, d \mathcal{H}^{n-1}    \\
&=  b \,  \he (0).
\end{split}
\end{align}
Using   \eqref{eqn6.9} once again   it follows that   
\begin{align*}
b&=  - \, \lim_{r\to  0 }    \int_{\ar B (0, r) } \, \lan \nabla f ( \nabla \hat G (X) ), X/|X| \ran 
\, d \mathcal{H}^{n-1}  \\
&= p (- \be )^{p-1}     \, \int_{\ar B (0, 1)}  h( X/|X| )^{-n}  \, d \mathcal{H}^{n-1}.   
\end{align*}
From  \eqref{eqn6.11} and  \eqref{eqn4.1} $ (a), (b), $  we conclude from  \eqref{eqn4.1} $(d)$ that  $  \hat G $  is  a constant multiple of the fundamental solution for    $  \mathcal{A} = \nabla f. $ 

To prove Lemma \ref{lemma5.2} which says that \eqref{eqn5.1} holds for $G$, we show that $\hat{G}$ satisfies \eqref{eqn5.1} which will finish the proof. To this end, recall that   $  k  \nabla k $  and  $ h  \nabla h $  are inverse functions.   Thus,  by the chain rule  the $ n \times n $  matrices  
\begin{align}
\label{eqn6.12}  
(k  \,   k_{\eta_i \eta_j} + k_{\eta_i} k_{\eta_j} ) \mbox{ and } 
(h  \,   h_{X_i X_j} + h_{X_i} h_{X_j} )  \mbox{ are inverses of each other.}
\end{align}
  From  \eqref{eqn6.4}  and   \eqref{eqn6.12}  we  conclude that    
$  (h  \,   h_{\eta_i \eta_j} + h_{\eta_i} h_{\eta_j} ) $ is  homogeneous of  degree 0 and   positive definite with eigenvalues bounded above and below by constants depending only on  $ p, n, \al. $   

To  prove    \eqref{eqn5.1}  for  $  \hat  G $,  suppose   $  X, \xi  \in \mathbb{S}^{n-1}$  and    $  \lan \nabla \hat G (X),  \xi  \ran  =   0 $. As $ \hat G = h^\be $ (where $\beta<0$) we also have  $ \lan \nabla h (X), \xi \ran
= 0 $  and   
\[ 
- \hat G_{\xi \xi} =  - \be h^{\be - 1} \,  h_{\xi \xi }  (X)   \geq   \tau'   > 0. 
\]
From this inequality and   \eqref{eqn4.1} $(b) $  or   \eqref{eqn6.5} we see that $ \tau' $  depends only on the data.  Thus   \eqref{eqn5.1} holds and  proof of  Lemma  \ref{lemma5.2} is complete.    
\begin{remark} 
\label{rmk7.1}
In view of \eqref{eqn6.11} and \eqref{eqn6.9}
\[
G(x)=b^{\frac{-1}{p-1}}\hat{G}(x)=b^{\frac{-1}{p-1}} h(x)^{\frac{p-n}{p-1}}
\]
is the fundamental solution in Lemma \ref{lemma4.1} where
\begin{align*}
b&=c \, \int_{\ar B (0, 1)}  h( X/|X| )^{-n}  \, d \mathcal{H}^{n-1} = p(-\beta)^{p-1}\, \int_{\ar B (0, 1)}  h( X/|X| )^{-n}  \, d \mathcal{H}^{n-1} \\
&=p\left(\frac{n-p}{p-1}\right)^{p-1}\, \int_{\mathbb{S}^{n-1}}  h(\omega)^{-n}  \, d\omega. 
\end{align*}
\end{remark}

\part{A Minkowski problem for nonlinear capacity}

\setcounter{equation}{0} 
\setcounter{theorem}{0}
\setcounter{section}{7}

\section{Introduction and  statement of  results}
\label{section8}
In  this  section  we   use our  work on  the  Brunn-Minkowski inequality to study  the   Minkowski problem  associated with  $ \mathcal{A}  =  \nabla f $-capacities  when    $ f $  is as in Theorem \ref{theorem1.4}.  To be more specific,  suppose  $ E  \subset  \rn{n} $  is  a  compact  convex set with nonempty interior.  Then for  $ \mathcal{H}^{n-1} $  almost every $ x \in \ar E,  $   there is  a  well defined  outer unit normal, $ \mathbf{g} ( x) $  to $ \ar E.  $   The function $ \mathbf{g}: \ar E  \to  \mathbb{S}^{n-1}$  (whenever  defined), is called the  Gauss map for $ \ar E.$     The  problem originally considered by Minkowski states:  given a  positive  finite Borel measure  $ \mu  $ on  $  \mathbb{S}^{n-1} $  satisfying 
\begin{align}  
\label{eqn7.1} 
\begin{split}
(i)&\, \,   {\ds \int_{ \mathbb{S}^{n-1}} } | \lan \he , \ze  \ran | \, d \mu ( \ze )  
>  0  \mbox{ for all }  \he \in \mathbb{S}^{n-1} ,\\  
(ii)& {\ds  \int_{ \mathbb{S}^{n-1}} } \ze   \, d \mu ( \ze )  = 0, 
\end{split}
\end{align}   
  show there exists up to  translation a  unique  compact convex set $  E $  with 
  nonempty interior  and   
\begin{align*}  
\mathcal{H}^{n - 1} ( \mathbf{g}^{-1} (K ) ) =  \mu ( K ) \mbox{ whenever  $ K \subset \mathbb{S}^{n-1} $ is a Borel set.} 
\end{align*} 
     Minkowski \cite{M1,M2}  proved  existence and uniqueness of  $ E $  when $ \mu $ is discrete or has a  continuous density.  The general  case  was  treated by  Alexandrov in \cite{A1,A2} and  Fenchel and Jessen in \cite{FJ}.  Note also that the conditions in \eqref{eqn7.1} are also necessary conditions for the existence and uniqueness of measure $\mu$. 
     
           In \cite{J}, a  similar problem was  considered for  electrostatic  capacity  when  $  E \subset \rn{n}, n \geq 3, $ is a compact convex set with nonempty interior and  $ u $ is  the  Newtonian or  $2$-capacitary function for  $ E $.  In this case,  
     $ u $ is harmonic in  $ \Om =  \rn{n}  \sem  E $  with boundary value $1$ on $\partial E$ and goes to zero as $|x|\to \infty$. Then a well-known work of  Dahlberg \cite{D} implies that 
 \begin{align}
     \label{eqn7.3}  
     \lim\limits_{\substack{y \to x \\ \, y \in \Gamma(x)}}  
     \nabla u  (y )  =  \nabla u (x)   \mbox{ exists for $ \mathcal{H}^{n-1} $  almost every $ x  \in  E. $ }  
\end{align} 
Here $\Gamma(x)$ is the non-tangential approach region in $\mathbb{R}^{n}\setminus E$.      Also,
     \[
\int_{\ar E }  | \nabla  u(x) |^2 d \mathcal{H}^{n-1}< \infty. 
      \]
        If           $  \mu $ is  a    positive  finite Borel measure  on  $  \mathbb{S}^{n-1}$ 
     satisfying  \eqref{eqn7.1}, it is shown by Jerison in  \cite[Theorem 0.8]{J}  that  there exists $  E $  a compact convex set with nonempty interior  and  corresponding  2-capacity function  $ u $  with            
\begin{align}
     \label{eqn7.4}   
      \int_{\mathbf{g}^{-1} (K )  }  | \nabla  u(x) |^2 \, d \mathcal{H}^{n-1}=  \mu  ( K)
\end{align}  
      whenever  $ K \subset \mathbb{S}^{n-1}$ is a Borel set and $  n  \geq 4$.   Moreover, $ E $ is  the  unique compact convex set with nonempty interior   up to translation  for which 
          \eqref{eqn7.4} holds.  If  $ n = 3, $  a less  precise result is available.

     Jerison's  result was generalized  in  \cite{CNSXYZ} as follows.   Given a  compact convex set  $  E $  with nonempty interior  and  $ p $  fixed,  $ 1 < p < n, $ let 
    $ u $  be the  $p$-capacitary function for  $  E. $  Then from \cite[Theorem 3]{LN} it follows that 
    \eqref{eqn7.3}  holds  for  $ u. $  Thus the Gauss map $\mathbf{g}$ can be defined for  $ \mathcal{H}^{n-1}$-almost every $ x \in  \ar E. $  
    If  $  \mu $ is  a    positive  finite Borel measure  on  $  \mathbb{S}^{n-1}$ having no antipodal point masses  (i.e.,  it is not true that  $ 0  <  \mu ( \{\xi\}) =  \mu (\{-\xi\}) $   for some  $ \xi \in \mathbb{S}^{n-1}$)      and if   \eqref{eqn7.1} holds,  then it is shown in \cite{CNSXYZ}  that for  $  1  <  p  <  2, $  there exists $  E $  a compact convex set with nonempty interior  and  corresponding  $p$-capacitary function  $ u $  with  
\begin{align}
          \label{eqn7.5}    
          \int_{\mathbf{g}^{-1} ( K)  }  | \nabla  u(x) |^p \, d \mathcal{H}^{n-1}=  \mu  (K )   
\end{align}
whenever  $ K\subset \mathbb{S}^{n-1}$ is a Borel set.      
Assuming the  existence of an  $  E  $ for which \eqref{eqn7.5}  holds when  $ p $ is fixed, $ 1 < p < n, $  it was also shown in \cite{CNSXYZ} that  $  E $  is unique up to  translation when $ p \not = n - 1 $  and unique up to translation  and  dilation when  $ p  =  n - 1. $   
       
          We  consider  an analogous  problem:         
             \begin{mytheorem}  
             \label{mink}  
             Let   $  \mu $ is  a  positive  finite Borel measure  on  $  \mathbb{S}^{n-1}$ satisfying  \eqref{eqn7.1}.  Let $p$ be fixed,  $1 < p < n$  and $ \mathcal{A} =  \nabla f $ as in \eqref{eqn1.7} in Theorem \ref{theorem1.4}.  
\begin{align}
  \label{eqn7.6}  
\begin{split}
\intertext{If $p\neq n-1$ then   there exists a  compact              convex set $ E $ with nonempty interior   and  corresponding  $  \mathcal{A}$-capacitary function $ u $    satisfying }
          & (a) \hs{.2in}  \eqref{eqn7.3} \mbox{ holds for $ u  $  and  } {\ds \int_{\ar E}  f ( \nabla u(x) ) \, d\mathcal{H}^{n-1}  <  \infty}.  \\
          & (b) \hs{.2in}     {\ds \int_{\mathbf{g}^{-1} ( K )  }  f ( \nabla  u(x)  )  \, d \mathcal{H}^{n-1}} =  \mu  (K)  \quad \mbox{whenever}\, \,   K \subset \mathbb{S}^{n-1}\, \,\mbox{is a Borel set}. \\
         & (c) \hs{.2in}     \mbox{$E$  is  the unique  set up to translation for which  $ (b) $  holds.} 
         \intertext{If $  p  =  n - 1  $ then  there exists a  compact              convex set $ E $ with nonempty interior,  a constant $b \in (0,\infty)$, and corresponding $  \mathcal{A}$-capacitary function $ u $ satisfying $(a)$  and}
 & (d) \hs{.2in}     {\ds   b  \int_{\mathbf{g}^{-1} ( K )  }  f ( \nabla  u  )  \, d \mathcal{H}^{n-1}} =  \, \mu  (K)  \mbox{ whenever  $ K \subset \mathbb{S}^{n-1}$ is a Borel set}. \\ 
& (e) \hs{.2in}     \mbox{$ E $  is  the unique  set up to translation  satisfying  $(d)$ and   $ \mbox{Cap}_{\mathcal{A}} (E)= 1. $  }  
\end{split} 
   \end{align}              
     \end{mytheorem}

    As a  broad outline of our proof  we  follow   \cite{CNSXYZ} (who in  turn  used ideas from \cite{J}).  However,  several important   arguments  in  \cite{CNSXYZ}  used  tools from  \cite{LN,LN1,LN2} for $p$-harmonic functions  vanishing on a portion  of  the boundary of  a Lipschitz domain.     To our knowledge similar results have not yet been proved 
  for $ \mathcal{A} = \nabla f$-harmonic functions  and the arguments  although often straight forward for the experts  are rather  subtle.     In  reviewing  these arguments   we   naturally  made editing decisions as  to which details to include and which to refer to.  Also  we attempted to clarify some details that were not  obvious to us even  in the  $p$-harmonic case  and our  proofs sometimes   use   later work of  the fourth named author and  Nystr{\"o}m  in   \cite{LN3,LN4}    when the authors  ``could see the forest for the trees''.  Thus the reader is advised to  have  the above papers  on hand.   These  preliminary  results for the proof of  Theorem \ref{mink} are given in sections  \ref{section9} and  \ref{section10}.   Our  work  in these sections gives  $(a)$ in  Theorem \ref{mink}. In  section  \ref{section11}  we consider   a sequence of  compact convex  sets, say 
$ \{E_m\}_{m\geq 1}$  with nonempty  interiors  which converge in the sense of  Hausdorff  distance  to   $  E $ a  compact convex set. 
If  $\{\mu_m\}_{m\geq 1}$ and $\mu$  denote the  corresponding measures  as  in  \eqref{eqn7.6}  we show that  
$  \{\mu_m\} $  converges weakly to  $ \mu $  on    $ \mathbb{S}^{n-1}. $      In section 
\ref{section12}  we first derive the Hadamard variational formula for
 $  \mathcal{A} = \nabla f$-capacitary functions in 
  compact convex sets with nonempty interior and smooth boundary.  Second using the results in 
section \ref{section11}  and   taking  limits   we get this formula  for  an arbitrary compact  convex  set  with nonempty interior.  Finally, in section  \ref{section13}  we  consider  a  minimum  problem  similar to the one considered in  \cite{J,CNSXYZ}.  However, unlike  \cite{CNSXYZ}, we  are able to  show that  compact convex sets of  dimension  $ k  \leq n - 1 $  (so with empty interior)  cannot be a  solution to our  minimum problem.   To rule out these possibilities  we use work in \cite{LN4} when $ k  < n - 1 $ while if  $ k = n - 1 $  we  use an   argument of  Venouziou and  Verchota  in \cite{VV}.   The solution to this minimum problem gives  existence of  $ E $  in  Theorem  \ref{mink} while uniqueness is proved using  Theorem \ref{theorem1.4}.  

  \setcounter{equation}{0} 
 \setcounter{theorem}{0}
    \section{Boundary behavior of $\mathcal{A}$-harmonic functions in  Lipschitz domains}   
    \label{section9}
We  begin this section with  several definitions.  Recall that     $ \ph :   K  \to \mathbb R $ is said to be Lipschitz on $ K  $ provided there
exists $  \hat b,  0 < \hat b  < \infty,  $ such that
\begin{align} 
\label{eqn8.14}   
| \ph ( z ) - \ph ( w ) |  \,  \leq \, \hat b   \, | z  - w | \quad  \mbox{whenever}\, \, z, w \in  K.  
\end{align}
The infimum of all  $ \hat b  $ such that  \eqref{eqn8.14} holds is called the
Lipschitz norm of $ \ph $ on $ K, $ denoted $ \| \ph  \hat  \|_{K}$.   
It is well-known that if $ K = \mathbb R^{n-1}, $ then   $ \ph $
is  differentiable almost everywhere on $ \mathbb R^{n-1} $
and  $ \| \ph  \hat \|_{\mathbb R^{n-1}} = \| \, | \nabla \ph |\,  \|_\infty. $
\begin{definition}[\bf Lipschitz Domain]
\label{defn8.2} 
A domain $D \subset\mathbb{R}^{n}$  is   called a bounded Lipschitz domain provided that  there exists a finite set of   
balls $\{B(x_i,r_i)\}$
 with $x_i\in\partial D$ and $r_i>0$, such that $\{B(x_i,r_i)\}$ constitutes a covering of an
open neighborhood of $\partial D$ and such that, for each $i$,
\begin{align*}   
D\cap B(x_i, 4 r_i)&=\{y=(y',y_n)\in\mathbb R^{n} : y_n >
  \ph_i ( y')\}\cap B(x_i, 4 r_i), \notag \\
  \partial D\cap B(x_i, 4 r_i)&=\{y=(y',y_n)\in\mathbb R^{n} : y_n=
  \ph_i ( y')\}\cap B(x_i, 4 r_i),
  \end{align*}   
  in an appropriate coordinate system and for a Lipschitz function $\phi_i$ on $ \mathbb{R}^{n-1}.$ The Lipschitz constant of $ D$ is defined 
  to be $M=\max_i\||\nabla\phi_i|\|_\infty$.  
  \end{definition}  
  If $ D$ is Lipschitz
and  $ r_0 = \min r_i, $  then  for each 
  $w\in\partial D$, $0<r<r_0$,  we can find points    
 \[
 a_r(w)\in D\cap B(w,r)\quad \mbox{with} \quad d(a_r(w),\ar D)\geq c^{-1}r
 \]
for a constant $c=c(M)$. In the following,    we let $a_r(w)$ denote one such point. We  also put  $  \De ( w, r )  = \ar D \cap  B (w, r)  $  when  $ w \in \ar D $  and  $  r  > 0. $    
    \begin{definition}[\bf Starlike Lipschitz domain]
    \label{defn8.3}  
    A bounded domain $ D\subset \rn{n} $   is said to be starlike Lipschitz
 with respect to $ z \in D$ provided  
 \begin{align*}
 \ar D = \{  z + \mathcal{R} ( \om  )  \om &: \om \in \ar B ( 0, 1 ) \} \\
 &\mbox{where}\, \, \log   \mathcal{R} :  \ar B ( 0, 1 )\rar\re\,\,  \mbox{is Lipschitz on}\,\, \ar B ( 0, 1 ).  
 \end{align*}
 \end{definition}   
 Under the above scenario  we say   that $ z  $  is the center of  $ D $ 
 and  $ \| \log \mathcal{R}  \hat \|_{\mathbb{S}^{n-1}}$ is the 
 starlike Lipschitz constant for  $  D.  $ 
 In the rest of this section  reference  to  the  ``data''  means  the constants in  Definition \ref{defn1.1},  \eqref{eqn3.8}  for  $ \mathcal{A} =  \nabla f, $ $ p, n, $ and the Lipschitz or starlike  Lipschitz constant whenever applicable.                  
We shall  need some lemmas similar to Lemmas \ref{lemma2.3}, \ref{lemma2.4}  for 
$  \mathcal A  = \nabla f$-harmonic functions vanishing on a portion of  a   Lipschitz  or starlike Lipschitz domain.     In the next two lemmas,   $ r_0' = r_0  $ when $ D $ is Lipschitz and $ r_0' =   | w - z |/100 $  when $  D $ is starlike Lipschitz with center at $ z. $   
\begin{lemma} 
\label{lemma8.4} 
Let $ D \subset\mathbb{R}^{n}$ be a bounded Lipschitz  or  starlike Lipschitz domain with center at $ z $  and  
 $p$  fixed, $1 <   p < n$. Let $w \in\partial D$, $0< 4 r<r_0'$, and suppose that $ v $ is a
positive $ \mathcal{A}$-harmonic function in $ D \cap B (w,4r)$ with  $ v \equiv 0 $ on 
$\ar D \cap  B ( w, 4r ) $  in the  $ W^{1,p}  $  Sobolev sense.  Then  $ v  $  has a representative in  $ W^{1,p} ( D \cap  B ( w,  s )), s < 4r,  $   which extends to  a  H\"{o}lder continuous function on   $ B  ( w, s) $  (denoted $ v $)  with $  v  \equiv  0 $  on  $  B ( w, s )  \sem D.$   Also,  there exists $ c \geq 1, $ depending only on the data,  
such that   if  $ \bar r  = r/c, $ then    
\begin{align}  
\label{eqn8.15}  
\bar  r^{ p - n}   \int\limits_{B ( w, \bar  r)}   | \nabla v |^{ p }  dx  \leq c    (v ( a_{2\bar r} (w)))^{p}.
   \end{align} 
Moreover, there exists $\ti \be\in(0,1),  $ depending only on the data,  such that if $x, y \in  B ( w, \bar r ), $ then
\begin{align}   
\label{eqn8.16}  | v ( x ) -  v ( y ) | \leq c \left( \frac{ | x - y |}{\bar r}\right)^{\ti \be}     v   (  a_{\bar r}  ( w ) ).    
 \end{align} 
\end{lemma}    
 
\begin{proof}
Here \eqref{eqn8.15}    with  $ v (a_{2 \bar r} (w) ) $ replaced by  $  \max_{B(w, 2 \bar r)} v $ is just a   standard Caccioppoli  inequality  while  $ \eqref{eqn8.16} $  with  $ v (a_{2 \bar r } (w) ) $ replaced by  $  \max_{B( w, 2 \bar r )} v $  follows as in Lemma  \ref{lemma2.3} from 
\eqref{eqn2.5} with $ E $  replaced by  $ \De (  w, \bar  r  ) $  and  Theorem 6.18 in   \cite{HKM}. The fact that   $  \max_{B( w, 2 \bar r)} v    \approx   v (a_ {2\bar r} ( w) )$  follows  from an argument often attributed to several authors (see in \cite[Lemma 2.2]{LN}).
\end{proof}

In  the sequel,  we  always assume  $ v $  as above  $  \equiv 0  $  on  $ B ( w, 4r)  \sem  D. $  
\begin{lemma} 
\label{lemma8.5}  
Let $  D ,  v, p,  r,  w  $  be as in  Lemma  \ref{lemma8.4}.    
  There
 exists a unique finite positive
Borel measure  $ \tau$ on $ \mathbb{R}^{n}$, with support contained in
$ \bar \Delta(w,r)$, such that if   $\ph \in C_0^\infty (  B(w,r) )$ 
then     
\begin{align}   
\label{eqn8.17}    
\int  \lan  \nabla f   ( \nabla v  ) ,   \nabla \ph \ran dx  =  -   \int   \ph \,  d \tau.
   \end{align} 

Moreover, there exists $ c \geq 1  $ depending only on the data such  that if $ \bar r = r/c,   $ then   
\begin{align}   
\label{eqn8.18} 
 c^{ - 1}  \bar r^{ p - n}   \tau ( \Delta (  w,  \bar r ))\leq (v  ( a_{\bar r} ( w ) ))^{ p - 1}\leq c  
\bar r^{ p - n }   \tau (\Delta (  w, \bar  r )).
\end{align} 
\end{lemma}  
\begin{proof}
See \cite[Lemma 3.1]{KZ} for a proof of  Lemma  \ref{lemma8.5}.
\end{proof}
We note that lemmas  similar  to  Lemmas   \ref{lemma8.6}-\ref{lemma8.7} and Proposition \ref{proposition8.9}  which follow  are  proved for   $p$-harmonic functions in  
\cite[Lemmas 2.5,  2.39, 2.45]{LN}.    

Throughout the remainder of  this paper, we assume that   $   \mathcal A =  \nabla   f  $  where  $  f$  is as in  Theorem \ref{theorem1.4}. 
    In order  to state the next  lemma,  we need some more notation.   Let  $ D $  be a   starlike  Lipschitz  domain with  center at $z$.   
     Given   $ x \in \ar D$ and $b > 1, $  let  
     \[  
     \Ga ( x ) = \{ y \in D   :  | y - x | <  \, b  \, d ( y, \ar D ) \}. 
\]  
If   $ w \in \ar D,$  $0 <   r  \leq  | w - z |/100,$  and  $   x   \in  \ar D  \cap  B ( w, r),  $  we note from elementary geometry that 
 if $ b  $ is large enough (depending on  the starlike Lipschitz constant
for $ D $) then 
 $ \Ga ( x  )  \cap  B  ( w, 8r )  $  contains the inside of  a truncated cone with
vertex $ x, $  axis parallel to $ z - x, $ 
angle opening $ \he = \he ( b ) > 0,  $ and height $ r.$  
 Fix $ b  $  so
that this property holds for all $ x \in \ar  D. $ 
 Given a measurable function  $ g $ on 
$ D \cap B ( w, 8 r ) $ define  the \textit{non-tangential maximal
function}
\[
\mathcal{N}_r (g)  : \ar  D \cap B ( w,  r ) \to \re 
\]
 of $ g $ relative to  $ D \cap B ( w, r) $ by  
\[     
\mathcal{N}_r(g)( x ) =  \sup_{y \in \Ga ( x ) \cap  B ( w, 8 r ) }  |g| ( y ) 
\quad  \mbox{whenever}\,\,  x \in \ar D \cap B ( w,  r ).  
\] 
 
Next we prove a reverse H\"{o}lder inequality.  
\begin{lemma}[\bf Reverse H\"{o}lder inequality]
\label{lemma8.6} 
Let  $ D $  be a starlike Lipschitz domain with center   $ z$ and let $w \in  \ar D$ with  $ 0 <   r <   | w - z |/100. $ Let    $  v, \tau, $   be  as in 
Lemma  \ref{lemma8.5} and  suppose that  for some  $  c_{\star}  \geq  1,  $

\begin{align} 
\label{eqn8.28}   
c_{\star}^{-1}  \frac{v ( x )}{d ( x, \ar   D ) } \leq    \, 
   \lan \, \frac{ z - x }{|z-x|} ,   \nabla  v (x) \,  \ran \leq 
|   \nabla v ( x ) |  \,  \leq  \, c_{\star}   \frac{v ( x )}{d ( x, \ar D )} 
 \end{align}  
 whenever   $  x  \in  B ( w,  4r) \cap  D. $    
There exists $ c \geq 1, $ depending only on $c_{\star}$  and  the data,  such that if  $ \ti  r = r/c, $  then   
\[
\frac{d\tau}{d\mathcal{H}^{n-1}} (y)  =  \, k^{ p - 1} (y) \quad  \mbox{for}   \,\, y  \in  \De  ( w,   \ti r ).
\]  
Also,   there exists $ q > p,  c_1$, and $c_2$   depending only on $ c_{\star}$  and the data with 
 \begin{align}  
 \label{eqn8.29} 
\begin{split}
 & (a)   \hs{.2in}    \int_{ \De ( w,     \ti r )  } \,  k^q \, d\mathcal{H}^{ n - 1 } 
\,   \leq \, c_1 \,  \ti  r^{  \frac{(n -  1)(p-1 - q) }{ p - 1} }    \left( \int_{\De ( w,   \ti r )  
 } \, k^{ p - 1} \, d\mathcal{H}^{ n - 1}   \quad \, \right)^{q/(p - 1)}.  
 \\
& (b) \hs{.2in}    \int_{ \De ( w,    \ti r )  } \,   \mathcal{N}_{ \ti r} (|\nabla v  |)^q \, d\mathcal{H}^{ n - 1 } 
\, \leq \, c_2 \,  \ti  r^{  \frac{(n -  1)(p-1 - q) }{ p - 1} }   \left( \int_{\De ( w,   \ti  r )  
 } \, k^{ p - 1} \, d \mathcal{H}^{ n - 1}    \quad 
 \, \right)^{q/(p - 1)}.
\end{split}
 \end{align} 
 \end{lemma}     
      \begin{proof}  
   Let  $   \ti  r  =   r/c  $  where   $ c \geq 100  $  is to be determined   and for fixed  $  s,   \ti r   <  s  <  2 \ti   r$ and  $t > 0 $ small,  let     
   \[
   D_1 =  B ( w, s ) \cap D \cap  \{v > t  \}.
   \]  
   Since 
$  \mathcal{A}$-harmonic functions are invariant under  translation, we assume as we may that 
$ z = 0. $    Note from \eqref{eqn2.3} that  $ \ar  D_1 \cap B (w, s)   $  is   smooth with outer normal  $ \nu  =  -  \nabla  v / | \nabla v | $   and also  that we can apply the divergence theorem to 
$ x   f  ( \nabla v (x) )  $  in   $ D_1$.   Doing this and using  $ \mathcal{A}$-harmonicity of  $  v $ in  $ D_1 , $ we arrive at   
\begin{align}
\label{eqn8.30}  
I =  \int_{ D_1 }   \nabla \cdot (  x  f ( \nabla  v) )  \, dx  =    \int_{\ar D_1}\lan x ,  \nu  \ran  f  ( \nabla  v)  d \mathcal{H}^{n-1}     
\end{align}
  and  
  \begin{align}  
  \label{eqn8.31}  
  \begin{split}
  I   &=   n  \int_{D_1 }  f ( \nabla  v)  dx   +    \sum_{k, j  = 1}^n  \int_{D_1}      x_k  f_{\eta_j } (\nabla  v)    v_{x_j x_k}  dx \\
    &= n  \int_{D_1}  f ( \nabla  v)  dx   + I_1.
\end{split}    
    \end{align}
Integrating $  I_1 $  by parts, using  $p$-homogeneity of  $ f $, as well as  $ \mathcal{A} =  \nabla  f$-harmonicity of  $  v $ in $ D_1  $  we  deduce that 
\begin{align}  
\label{eqn8.32}  
I_1  =  \int_{\ar D_1 } \lan x  ,  \nabla  v \ran  \, \lan \nabla f (\nabla    v) 
, \nu   \ran  \, d\mathcal{H}^{n-1}  -   p  \int_{D_1 }  f ( \nabla  v)  dx.   
\end{align}      
Combining  \eqref{eqn8.30}-\eqref{eqn8.32}   we  find after some juggling  that  
\begin{align}  
\label{eqn8.33}
\begin{split}    
 (n-p)    \int_{D_1}  f ( \nabla  v)  dx=    \int_{\ar  D_1 }  \lan x,  \nu  \ran  f  ( \nabla  v)  d \mathcal{H}^{n-1}     -  \int_{\ar D_1  } \lan x  ,  \nabla  v \ran  \, \lan \nabla f(\nabla v) , \nu   \ran  \, d \mathcal{H}^{n-1}.                                                            
\end{split}
\end{align}         
From  $p$-homogeneity  of  $ f $  we  obtain      
 \begin{align} 
 \label{eqn8.34} 
 \begin{split}        
\int_{\ar D_1 \cap B ( w, s ) }  \lan x ,  \nu  \ran 
 f  ( \nabla  v) d \mathcal{H}^{n-1}   &- \int_{ \ar D_1   \cap B ( w, s ) } \lan x ,  \nabla   v\ran  \, \lan \nabla  f(\nabla v) 
, \nu   \ran  \, d \mathcal{H}^{n-1}  \\                      
  &   =   (p - 1)  \int_{ \ar D_1 \cap B ( w, s ) }  \lan x ,  \nabla  v\ran \,   
 \frac{f  ( \nabla  v)}{|\nabla  v|} d \mathcal{H}^{n-1}\\
&   \leq  0.
\end{split}
  \end{align}           
Using  \eqref{eqn8.34} in \eqref{eqn8.33} and  \eqref{eqn8.31}, \eqref{eqn8.28},  we arrive after  some  more juggling at 
\begin{align} 
\label{eqn8.35}    
\begin{split} 
c^{-1}    \int_{ \ar  D_1 \cap  B ( w, s )} 
   |x|  f  ( \nabla v )   d \mathcal{H}^{n-1}  &\leq    - (p - 1)  \int_{\ar  D_1 \cap B ( w, s ) }  \lan x ,  \nabla  v\ran   
\frac{f  ( \nabla  v)}{|\nabla  v|}  d \mathcal{H}^{n-1}\\
&\leq  \,  F_1  
\end{split}
\end{align}   
where          
\begin{align} 
\label{eqn8.36}    
F_1 =     c   \int_{\ar D_1\cap\ar  B ( w, s)}  | x |  f(\nabla  v )  d \mathcal{H}^{n-1}.
\end{align}
 Here  $ c $ depends only on $c_{\star}$  and the data. Also  in getting  $ F_1 $  we have used the structure assumptions on  $ f $  in  Theorem \ref{theorem1.4}.               
 
       We note from  \eqref{eqn8.16}  that   
\begin{align} 
\label{8.36a}   
\bar D \cap \bar B ( w, 2r ) \cap \{ v \geq t \}  \to \bar D \cap \bar B ( w, 2r )
\end{align}
       in Hausdorff distance as $ t \to 0. $             
Also, from  \eqref{eqn8.28} we note  that if  $ v ( x ) = t$ and $  \om =  x/|x|$ then    $  \ti{\mathcal{R}}( \om ) =|x|  $ is  well-defined. Moreover, if   
\[   
\He = \{ \om \in \mathbb{S}^{n-1} :  \om = x/|x|\, \, \mbox{for some  $ x \in \bar B (w, 2r ) $ with  $ v ( x ) = t$}\}   
\]  
then   $  \log \ti{\mathcal{R}}$ is Lipschitz on  $ \He $ with Lipschitz constant   depending only on  the data and   $c_{\star}$.  Using the Whitney extension theorem (see \cite[Chapter VI, Section 1]{St}), we  can extend $ \log \ti{\mathcal{R}}$ to  a  Lipschitz function on   $ \mathbb{S}^{n-1}$  with Lipschitz constant depending only on $c_{\star}$  and the data.  

We next let  $  \ti v  =  \max  ( v  -  t, 0 )$  and if  $  t  >0$  is 
  sufficiently small then we  see from  \eqref{8.36a}  that   Lemmas  \ref{lemma8.4}, \ref{lemma8.5} can be applied to $ \ti v   $ in  $  D \cap B ( w,  2r  )  \cap \{ v > t \}. $   Let  $ \ti  \tau  $  be the measure corresponding to $  \ti v$.      Then from smoothness of  $  \ti v,  $  the divergence theorem,  and $ p$-homogeneity of  $ f $  we obtain  that  
   \begin{align*}   
   d \ti \tau ( x ) =    p \frac{f (  \nabla v ( x ))}{|\nabla v  ( x ) |}\, d\mathcal{H}^{n-1} \quad \mbox{ for }  \,\,  x  \in       B ( w,  2r )  \cap  \{ v = t \}.  
   \end{align*}
   
To estimate  $ F_1 $  in   \eqref{eqn8.35}   choose $ s \in  (\ti r, 2 \ti  r) $ so that  
 \[    
 \int_{ \ar B ( w , s ) \cap \ar D_1} 
\,    \, f( \nabla  v)  \, d \mathcal{H}^{ n - 1} \,  \leq  \,   
    2 \ti   r^{ - 1}  \, \int_{ B ( w,  2 \ti   r ) \cap D \cap \{  v > t \}   }    \,   
 \,  f( \nabla  v) \,   dx.   
 \]  
 This choice is possible from weak type estimates or Chebyshev's inequality.  Using this  inequality in    \eqref{eqn8.36}  and the above lemmas for $ \ti v $ we  obtain for $ t >0,$ small  and   $ c $ sufficiently large in the definition of $ \ti  r,  $  that  
\begin{align} 
\label{eqn8.38} 
\begin{split}
| w  -  z|^{-1}    F_1  &\leq   4 \ti   r^{ - 1}  \, {\ds\int_{ B ( w, 2 \ti   r ) \cap D \cap \{  v > t \}   }    \,   
 \,  f( \nabla  v) \,   dx\,  } \\
 &\leq \, \ti c  \ti r^{n - 1 -p}  \,   \ti v ( a_{4 \ti  r} (w) )^p \\
  & \leq \ti c^2  \ti  r^{  \frac{  1   - n }{ p - 1} }      
    \,     \left( \ti  \tau  ( B ( w,    2 \ti r   ) \right)^{p/(p-1)}  
    \end{split} 
    \end{align}  
 where  $ \ti c $  depends only on the data and  we have also used  Harnack's inequality for  $ \ti v $ to get the last inequality.  Putting  \eqref{eqn8.38}  into   \eqref{eqn8.35},   and using    \eqref{eqn8.15},    \eqref{eqn8.18},  we find that if   
\begin{align*} 
\ti k^{p - 1}  ( y )  =    p  \frac{f (  \nabla  v(y) )}{|\nabla v (y) |} \quad \mbox{for}\quad y  \in \{ v = t \}
\end{align*} 
then    for small  $ t >0, $  
\begin{align} 
\label{eqn8.40}  
\begin{split}
 \int_{ B ( w, \ti r ) \cap \{ \ti v = 0  \} } \,  \ti k^p \, d\mathcal{H}^{ n - 1 } \, &\leq \, c \, \ti r^{  \frac{1 - n }{ p - 1} }   \left( \int_{ B  ( w, 2 \ti r ) \cap \{\ti v = 0\} }    \, \ti k^{ p - 1} \, d \mathcal{H}^{ n - 1}    \, \right)^{p/(p - 1)} \\  
 &  \leq     c^2  \,  \ti r^{  \frac{1 - n }{ p - 1} }   \left( \int_{ B  ( w,  \ti r ) \cap \{\ti v = 0\} }    \, \ti k^{ p - 1} \, d \mathcal{H}^{ n - 1}   
 \, \right)^{p/(p - 1)}  
\end{split}
 \end{align}   
 where we have once again used  Harnack's  inequality for 
$ \ti v  $  in  Lemma  \ref{lemma8.5} to get the last inequality.  

 With $ \ti r $  fixed  we  now  let  $ t  \to 0 $  through a  decreasing sequence $ \{t_m\}. $  
Let  $  \tau_m  =  \ti \tau   $  when  $ t  =  t_m $  From  \eqref{eqn2.2} 
 and   Lemmas  \ref{lemma8.4}-\ref{lemma8.5}  we see that   
 \[
\mbox{$\tau_m$  converges weakly to $\tau$ as $ m  \to \infty  $}
\]
where $\tau$ is the measure associated with  $ v$.   
  Using  the change of variables formula    and  Lemma  \ref{lemma8.5}  we can
pull back   each $ \tau_m  $ to  a measure on  a subset of  $ \mathbb{S}^{n-1}$.  
 In view of \eqref{eqn8.40} we see
that the Radon-Nikodym derivative of each pullback measure
with respect to $  \mathcal{H}^{ n - 1} $ measure on $ \mathbb{S}^{n-1}$  
 satisfies a   $ L^{p/(p - 1)} $ reverse H\"{o}lder inequality on 
 \[
 \{ x/|x| : x \in   B ( w,  \ti  r ) \cap  \{ v = t_m      \} \}.  
 \]    
Moreover,  $ L^{p/(p - 1)} $  and H\"{o}lder constants
depend only on $c_{\star}$  and the  data.    Thus, any sequence of  these  derivatives has a subsequence which converges weakly in  $ L^{p/(p- 1)} $.  
 Using these observations  we  deduce  first that $  \tau $  viewed  as  a  measure on a subset of  $\mathbb{S}^{n-1}$  has  a  density  that  is  $p/(p-1) $  integrable and  second  that this  density  satisfies a  $ p/(p-1) $  reverse H\"{o}lder  inequality.  Transforming back  we  conclude that if 
   $ k^{p-1} $  denotes   the  Radon-Nikodym derivative of $ \tau $ on $  \De  ( w, \ti  r )  $ with respect   to   $  \mathcal{H}^{n-1}  $  then   \eqref{eqn8.29} is valid with   $ q $  replaced by  $ p. $  Now   $ r,  w $  can  obviously be replaced by  $ y,  \rho$ where   $y  \in  B ( w, r ) \cap \ar D$ and $0<  \rho  <   r/2 $  in  \eqref{eqn8.29} with $ q $ replaced by $p.$  Doing this we see from a  now  well-known theorem  that  the  resulting   reverse  H\"{o}lder inequality is self-improving, i.e, holds for some $ q > p, $  depending only on $c_{\star}$  and the data  (see   \cite[Theorem IV]{CF} for a proof of the self improving property).  This proves   \eqref{eqn8.29} $ (a).$  
   
 To prove  \eqref{eqn8.29} $(b)$  let  $ x  \in  D  \cap  B ( w,  \ti r), 
 y  \in  \Ga (x),  $  
 and suppose    
 \[  
\mathcal{N}_{\ti r} ( |\nabla v |) (x)  \,   \leq   2 | \nabla v (y) |  \leq  2   \mathcal{N}_{\ti r} (  | \nabla v  |) ( x)).  
\]   
If $  \ti r /100   \leq    d ( y, \ar D ), $   we deduce  from       \eqref{eqn8.28},  Lemma \ref{lemma8.5},  and  Harnack's inequality for $ v  $   that for some $ c' \geq 1  $ depending only on $ \si, b,  $  and the data,  
\[ 
\mathcal{N}_{\ti r} (|\nabla v |) (x)   \,  \leq  \, c'    \left(  \ti r^{ 1 - n} \,  
\int_{ \De  ( x, \ti r)  }   \, k^{p - 1} \,  d \mathcal{H}^{ n - 1} \, \right)^{\frac{1}{p-1}}.
\]
 Otherwise, 
  \[ 
   \mathcal{N}_{\ti r}  (|\nabla v |)  (x)  \leq  2 | \nabla  v  ( y) ) |  \leq  c  | x - y|^{-1}   v (y)  \,  \leq  c^2 \,  \mathcal{M}_1 k  (x)   
   \]   
   where 
  \begin{align}  
  \label{eqn8.41}   
  \mathcal{M}_1 k  (x)  :=  \left(  \sup_{0 < t < \ti r } \, t^{ 1 - n} \,  \int_{ \De  ( x, t
)  }  \, k^{p - 1} \,  d \mathcal{H}^{ n - 1} \, \right)^{\frac{1}{p-1}}.  
\end{align}
  Raising both sides of  either  inequality to the $ q$-th  power and integrating over 
$  x  \in   \De ( w, \ti r),  $   we deduce from the  Hardy-Littlewood  maximal  theorem  (see  \cite[Chapter 1]{St})  that  \eqref{eqn8.29} $(b)$ is true. 
This  completes the proof of  Lemma  \ref{lemma8.6}.    
\end{proof} 
 We use  Lemma \ref{lemma8.6}  to  prove the following localization lemma. 
\begin{lemma} 
\label{lemma8.7} 
Let  $  v, D,  z, c_{\star}, b, w, r, k$ be as in Lemma  \ref{lemma8.6}.  Let   $  \hat w  \in   
\De (w, 2 r ) ,   0 < s < r,   $  and let   $  \hat w'  
\in  D  $  be the point on the ray from  $ z $ through  $ \hat w $ with  $ | \hat w - 
\hat w' | = s/100. $  There exists $  c' , c''  \geq 1,$ depending only on $c_{\star}$  and the data,   
 such that if  $   s'   =  s/  c', $   then there is  a  starlike Lipschitz domain $ \ti  D   = \ti D ( \hat w,   s  )  
\subset B ( \hat w,  s  ) \cap   D $ with center at $ \hat w', $ 
\begin{align} 
\label{eqn8.42}  
\,   \frac{  \mathcal{H}^{ n - 1}   [    \ar \ti  D  \cap \De  ( \hat w, s' )  ] }{  \mathcal{H}^{ n - 1}   [    \De  ( \hat w, s' )  ] }  \geq  3/4 
\end{align}
and Lipschitz constant $ \leq c
    \, ( \| \log \mathcal{R}\hat \|_{\mathbb{S}^{n-1}} + 1 ) $ where $ \mathcal{R} $ is the graph function for $ D, $  and  $ c $ depends only on $ p, n.$  
Moreover, if $ x \in \ti   D,  $ then             
\begin{align} 
\label{eqn8.43}  
\frac{1}{c''}\, \frac{v (\hat w' )}{s}\,  \leq \,  | \nabla   v  ( x ) | \,  \leq
\,   c''\, \frac{v ( \hat w' )}{s}.
\end{align}  
\end{lemma} %
\begin{proof}       
From  starlike Lipschitzness of $  D $ and basic geometry we   deduce  the existence of 
  $   c' > 1  $   (depending only on the data) such that if   
 $   s'  =  s/  c', $  then for  any  $ x  \in  \De (\hat w, s' ), $  the  ice cream cone, say  $ C (x, \hat w') $   obtained by  drawing rays  from  $ x $  to all points  in    $\bar B (\hat w' ,  s')$  satisfies  
\begin{align}
 \label{eqn8.44}    
 \Ga ( x ) \cap  B ( \hat w,  s )   \supset    C(x, \hat w' ). 
 \end{align} for  $ b > 1 $  suitably large.  
    Put  
    \[ 
\mathcal{M}_2  k   ( x ) = \left( \inf_{0 < t <   s' } \,   t^{ 1 - n}
\, \int_{  \De ( x, t )  } k^{p-1}   d\mathcal {H}^{n-1} \, \right)^{1/(p-1)}  
\]  
for  $  x   \in \De (\hat w, s' ). $      We claim  there exists   a compact set  $  \hat K 
  \subset   \De ( \hat w,   s')  $  
and $ \hat c \geq 1$ (depending only on the data, as well as  $c_{\star}$ in \eqref{eqn8.28})  with     
\begin{align}  
\label{eqn8.64} 
 \hat c \, \mathcal{M}_2 k   >   \,  v ( a_{s' } ( \hat  w ) )/s'  \,  \,
\mbox{on}\,\, \hat K\quad  \mbox{and}   \quad \hat c \,  \mathcal{H}^{n-1}  ( \hat K
) \,  >    (s')^{n - 1}.
\end{align}   
 To see this  we temporarily allow  $ \hat c $ to vary.   We note that   if   
 \begin{align*}
& \ep =   ( 1/ \hat c ) \,     v ( a_{ s' } ( \hat w ) )/ s',  \\
& \Ph   = \{ x \in  \De  ( \hat w,  s' )  :  \mathcal{M}_2 k  ( x )  \leq   \ep   \}
 \end{align*}
   then by a standard covering argument  there exists $ \{ B ( x_i, r_i ) \} $ 
with $ x_i \in \Ph, \, 0 < r_i \leq  s',   \, \Ph \subset \bigcup_{i}  B ( x_i, r_i ) $ and 
$ \{ B ( x_i, r_i/10 ) \} $ pairwise disjoint.  Also, 
\[    
\int_{\De   ( x_i, r_i ) }\,  k^{ p - 1}  \, d
 \mathcal{H}^{ n - 1} \,   \leq  \, (2 \ep)^{p-1}   \, r_i^{n - 1}  \quad 
\mbox{for each} \,\, i.  
\]      
Using these facts and $ \mathcal{H}^{n-1}( B ( x_i, r_i/10) \cap \ar D) 
\approx r_i^{n-1}, $ we get  
 \begin{align}  
 \label{eqn8.65}   
\begin{split}
 \int_{\Ph} \, k^{ p - 1} \, d\mathcal{H}^{n-1}  \, &\leq \, 
     \,  \sum_i \, \int_{ \De  ( x_i, r_i )  } 
k^{ p - 1} \, d\mathcal{H}^{n-1}\\
&\leq   (2 \ep)^{p-1}   
 \sum_i   r_i^{ n - 1}  \, \\
 &\leq  \,    c  \, \ep^{p-1} \,   (s')^{ n -  1}. 
 \end{split}
 \end{align}
  On the other hand, if  $ \Psi  = \De  ( \hat w,  s')   \sem \Ph, $ then     
 from \eqref{eqn8.18}, \eqref{eqn8.28},  and \eqref{eqn8.29} with $ r $ replaced by $  s', $   
  the structure assumptions on  $ \mathcal{A}, $    and 
 H\"{o}lder's inequality, we get  for some $  c, $   depending only on the data and $c_{\star} $     
 in \eqref{eqn8.28}, 
 \begin{align}  
 \label{eqn8.66}   
\begin{split}
\int_{\Psi} \, k^{ p - 1} \, d\mathcal{H}^{ n - 1} &\leq     [\mathcal{H}
^{ n - 1} ( \Psi  )]^{ 1/p} 
 \left( \int_{ \De ( \hat w,  s')   } k^p  
d \mathcal{H}
^{ n - 1} \, \right)^{ 1 - 1/p} \\
 &\leq      c  \left[   (s')^{ 1 - n} \,  \mathcal{H}
^{ n - 1} ( \Psi  )  
\, \right]^{1/p}    
 \int_{\De  ( \hat  w,  s')   }   k^{ p - 1}   
d \mathcal{H}^{ n - 1}  \\
&\leq c^2  \,  \left[   (s')^{ 1 - n}  \mathcal{H}
^{ n - 1} (  \Psi  ) 
\right]^{1/p}  \,  (s')^{ n - p }  \, v ( a_{s'} ( \hat w ) )^{ p - 1}.
\end{split}
\end{align}     
Since
\[
 (s')^{ n - p}  \,  v ( a_{s'} ( \hat w ) )^{p - 1}   
   \approx  \int_{ \De ( \hat  w,
  s'  )  } k^{ p - 1} \, d \mathcal{H}
^{ n - 1}, \]
we  can add \eqref{eqn8.65},  \eqref{eqn8.66} to get after division by  
$  (s')^{ n - p}  \, v ( a_{ s' } ( \hat  w ) )^{p - 1} $ that for some $
c,  $  depending only on the data  and $c_{\star}$, 
\begin{align}  
\label{eqn8.67}  
c^{ - 1} \, \leq \, [ (s')^{ 1 - n}  \mathcal{H}
^{ n - 1} (  \Psi  ) ]^{1/p}   \, + \, 
(1/\hat c)^{p-1}.
\end{align} 
Clearly,  \eqref{eqn8.67}, implies  \eqref{eqn8.64} for $ \hat c $ large enough  with $ \hat K $ replaced by 
$ \Psi. $   A standard measure theory argument then shows that we can
replace $ \Psi $ by suitable $  \hat K $ compact, $ \hat K \subset \Psi. $
Thus   \eqref{eqn8.64} is valid for $ \hat c $ large enough.

Next   let $ K_1  =  \hat K $  and let  $  D_1 $  denote  the  interior  of the domain   obtained  from  drawing all  line segments   from  points in  $ \bar B (\hat w' ,  s' ) $  to   points in  $    K_1.  $  From   \eqref{eqn8.44}-\eqref{eqn8.64} and \eqref{eqn8.18}-\eqref{eqn8.28},    we   conclude for  some  $ \breve c  \geq 1, $  depending only  on  $c_{\star}$    and the data  that       
\begin{align} 
\label{eqn8.68}  
\,  \breve c   | \nabla   v  ( x ) | \,  \geq
\,    s^{ - 1} \,  \,  v ( \hat w' )\quad \mbox{whenever}\,\,  x \in D_1.
\end{align}  

 If   
\begin{align}
\label{eqn8.69}  
\,  \frac{  \mathcal{H}^{ n - 1}   (   \ar D_1 \cap \De (\hat  w, s')  )}{  \mathcal{H}^{ n - 1}  ( \De  ( \hat w, s' ) ) }  \leq   7/8 ,  
\end{align}
choose  $ c_1 >  2, $ depending only on the starlike Lipschitz constant for $ D,  $    so  that  if    $ s''  =   ( 1  -  2/c_1 ) s' , $         then   
\begin{align}
\label{eqn8.70}  
\,  \frac{  \mathcal{H}^{ n - 1}  ( \De  ( \hat w, s'' )) }{  \mathcal{H}^{ n - 1}  (   \De  ( \hat w, s' )) }  \geq   99/100. 
\end{align} 
Since    $ \ar D_1  \cap \De (\hat w,  s' )  =  K_1$    it follows  from  \eqref{eqn8.69}  \eqref{eqn8.70},  that  there exists $ y \in  \De (\hat w, s''  )  \sem  K_1. $      We  can  apply the argument  leading to \eqref{eqn8.64},   \eqref{eqn8.68}  with  $ \hat w, s $  replaced by  $y,  \tau  =  d (y, K_1)/c_1.$    If  $ 
\tau'  = \tau /c' , $   we obtain  a  compact set  
\[  \hat K (y ) \subset \De ( y, \tau ' )  \subset   \De (\hat w, s') \] 
  and  corresponding starlike Lipschitz domain 
$ \hat D (y) $  with center  at  $ y' $  where  $ y' $ is the point on the ray from  $ z $ to $ y $ with  
$ |y  -  y' |  =    \tau /100. $  Also    
 $  \hat  D (y) $  is the interior of  the  set obtained by drawing  all  rays  from   
$  \bar B (y' ,  \tau ' ) $    to  points in    $ \hat K (y) $ so that  
$   \ar \hat D (y)  \cap  \De (y, \tau ') = \hat K(y). $    Moreover,     
\begin{align}
\label{eqn8.71} 
\begin{split}
& (+) \hs{.4in}   \,    | \nabla   v  ( x ) | \,  \geq
\,    c_{\star\star}^{-1}  \,   d ( y, K_1 )^{ - 1} \,  v ( y' ) \quad \mbox{whenever}\,\, x   \in \hat D(y), \\
&  (++)    \hs{.29in}  \hat c \,    \mathcal{H}^{n-1} ( \ar \ti D (y) \cap \De ( y, \tau') )   \geq   (\tau ')^{n-1}   
\end{split}
\end{align}
where  $ \hat c $ is the constant in  \eqref{eqn8.64}  and $ c_{\star\star}\geq 1$ depends only on the data and $ c_{\star} $ in  \eqref{eqn8.28}.  We  now  use  a  Vitalli type covering argument  to  get 
  $  y_1, y_2,  \ldots, y_l, $  for some positive integer $l, $  satisfying the  above with $ y = y_i,  1 \leq i  \leq l,  $  and corresponding  $ \tau_i,  \tau_i' ,  y_i' ,   \hat K (y_i),   \hat D (y_i).  $  Then   \eqref{eqn8.71} holds with $ y $  replaced by  $ y_i,  1 \leq i  \leq l, $  and        
\begin{align}
\label{eqn8.72}  
\,  \frac{  \mathcal{H}^{ n - 1}   ( \bigcup_{i=1}^l  \hat  K (y_i) ) }{  \mathcal{H}^{ n - 1}   (  \De  ( \hat w, s' )) }  \geq   c^{-1} 
\end{align} 
for some $ c \geq 1 $ depending only on $ c_{\star}$ and the data. Let  $  D_2 $ be the  starlike Lipschitz domain with center at  $\hat  w' $  which is the interior of  the domain obtained by  drawing  all rays from points in  $  \bar B (\hat w', \hat s' )  $  to 
points in   
\[  
K_2  =   K_1  \cup  (  \bigcup_{i=1}^l \hat  K (y_i) ). 
\]  
We claim that 
\begin{align}
\label{eqn8.73} 
 \,    | \nabla   v  ( x ) | \,  \geq\,    \frac{1}{c_{_-}} \,  \frac{v ( \hat w' )}{s} \quad\mbox{whenever}\,\,   x  \in   D_2
\end{align}
 where $ c_-  $  depends only on the data. To  prove this claim, given  $ x \in D_2, 
$ let  $ \hat x  $  be the point in  $  \ar D_2  \cap \De ( \hat w, \hat s' ) $  which lies  on the line   from  $ \hat w' $  through   $ x $.  If  $ \hat x   \in K_1  $  it  follows  from  \eqref{eqn8.68}  that   \eqref{eqn8.73} is true for   suitably large $ c_-. $  Otherwise, suppose   $  \hat x \in  \hat K (y_j). $   If     
$  | \hat x  -  x |  \leq  \tau_j $  we observe from our construction  that there exists $ x^* \in 
\hat D (y_j ) $   with  
\begin{align} 
\label{eqn8.74}    
  | \hat x  -  x |  \approx    | \hat x  -  x^* |  \approx   |  x  -  x^* | 
  \end{align} 
where all  constants in the ratios depend only on $c_{\star}$  and the data. Using   
 \eqref{eqn8.74},   \eqref{eqn8.71} with $ y  = y_j, $     \eqref{eqn8.28},  and  Harnack's  inequality we deduce that     \eqref{eqn8.73}  holds.  If   $  | \hat x  -  x |   >  \tau_j $     
we  can choose  $  x^*  $ in  $  D_1 $   so  that   \eqref{eqn8.74}  is true.   Applying  
  \eqref{eqn8.68} and  arguing as  above  we get    \eqref{eqn8.73} once again.  This proves our claim in \eqref{eqn8.73}.

 From disjointness of   $ K_1 $  and   $ \cup_i  \hat K (y_i) $ as well as                       
\eqref{eqn8.72} it  follows that  
\begin{align} 
\label{eqn8.75}  
\,  \frac{  \mathcal{H}^{ n - 1}   ( K_2 )  }{  \mathcal{H}^{ n - 1} (    \De  ( \hat w, s' )) }  \geq   c^{-1}       \, + \,    \frac{  \mathcal{H}^{ n - 1}   ( K_1 )  }{  \mathcal{H}^{ n - 1}  (    \De  ( \hat w, s' ) ) }.
\end{align} 
Continuing this argument at most  $ N $  times, where 
$ N $  depends only on the data and  $c_{\star}$  we see from  \eqref{eqn8.75}   that we  eventually obtain   $ K_N $  a  compact set  
$ \subset  \De  ( \hat w', s' )  $   and  $ D_N $  a  starlike Lipschitz  domain with center at  $ \hat w' $ corresponding to  $ K_N $  for which         
\begin{align}
\label{eqn8.76}  
\,  \frac{  \mathcal{H}^{ n - 1}   ( K_N )  }{  \mathcal{H}^{ n - 1}   (\De  ( \hat w, s' ))}  \geq   7/8.   
\end{align}
Also  \eqref{eqn8.73} is valid for  large $ c_- $  with  $ D_2 $ replaced by  $  D_N.$  

To complete the construction of  $   \ti  D $,  we  need to  estimate $ |\nabla v |  $  from  above.
For this purpose let  $\mathcal{M}_1 k $ be as in  \eqref{eqn8.41} with $ \ti r $ replaced by $   s'.$  Once again we  use the Hardy-Littlewood  Maximal theorem    and also \eqref{eqn8.18},    \eqref{eqn8.28},     
  to find  $  K^* $ compact $  \subset      \De ( \hat w, s' ) $ and 
$   \bar c \geq 1 $ depending  on  $c_{\star}$  and the data  such that   
\begin{align}  
 \label{eqn8.77a}   
 \mathcal{M}_1 k \, \leq  \, \bar c \,   ( s' )^{ 1 - p} \,  v ( a_{s'} ( \hat w ) ) ^{ p - 1} \quad  \mbox{on}\quad K^*
\end{align}
and
\begin{align}
\label{eqn8.77b}
\mathcal{H}^{ n - 1} ( K^* ) \, \geq \, {\ts \frac{7}{8}} \, \,    \mathcal{H}^{ n - 1} (  \De  (\hat w, s' )   ). 
 \end{align}  
 Let    $  D^* $  be the interior of  the domain obtained  from  drawing  all rays from points in $ \bar B ( \hat w', s') $ to points   
in $ K^*. $ If  $ x \in   D^* $     then 
 from  \eqref{eqn8.77a},  \eqref{eqn8.77b}, \eqref{eqn8.28}, \eqref{eqn8.18},   and   Harnack's inequality for  $ v, $    we  find   for  some   $  \ti c  $ that  
\begin{align}  
\label{eqn8.78}  
|\nabla  v(x) |  \leq  \ti c  \, v (\hat w')/s \quad  \mbox{whenever} \quad x   \in  D^*. 
\end{align} 
   Let    $   \ti  D  =   D^*  \cap D_N. $   From   
\eqref{eqn8.44}  it is easily seen for $ c' $ large enough that $ \ti D $ is starlike 
Lipschitz with center at $ w' $ and starlike Lipschitz constant $ \leq c
    \, ( \| \log \mathcal{R} \hat \|_{\mathbb{S}^{n-1}} + 1 ). $  Also, from \eqref{eqn8.76}, \eqref{eqn8.73} with $ D_2 $ replaced by $ D_N$,  \eqref{eqn8.77b},  and \eqref{eqn8.78}  we see that 
 \eqref{eqn8.42},  \eqref{eqn8.43} are  valid.  The proof of  Lemma \ref{lemma8.7} is  now complete.
\end{proof}    

We use Lemmas \ref{lemma8.6} and \ref{lemma8.7}  to  prove 
\begin{proposition} 
\label{proposition8.9} 
Let $ D,  z,  c_{\star}, b, v, \tau, r,  k,   w,   $  be as in  Lemma  \ref{lemma8.7}. Then
\begin{align}
\label{eqn8.48}
{\ds \lim_{\substack{x\to y \\ x \in  \Ga ( y)\cap B(w,2r)}}}  \nabla v ( x ) \stackrel{def}{=}
\nabla v ( y ) \, \, \mbox{exists for  $ \mathcal{H}^{n-1}$-a.e $y
\in\Delta( w, 2 r )$.}
\end{align}    
Moreover, $ \De ( w, 2 r ) $ has a tangent plane
  for  $ \mathcal{H}^{n-1} $   almost every $ y  \in   \De ( w, 2 r )  $.   If
 $ \mathbf{n}( y ) $ denotes the unit normal to this tangent plane pointing
 into
 $ D \cap B ( w, 2 r ), $  then
\begin{align}   
\label{eqn8.49a}  
k (y )^{p-1}  =p  \frac{f (\nabla v ( y ) )}{|\nabla v ( y) |}  
\end{align}
and
\begin{align}
\label{eqn8.49b}  
\nabla v  ( y )
=  | \nabla v ( y )  | \, \mathbf{n} ( y ) \quad \mathcal{H}^{n-1}\mbox{-a.e.} \,\,  \mbox{on} \, \, \De ( w, 2 r ). 
\end{align}
   \end{proposition}   
\noindent 
\begin{proof}    
 In  the proof of       
Proposition   \ref{proposition8.9}  we argue  as in   \cite[Lemma 3.2]{LN3}.  The proof is by 
contradiction.  

Suppose there exists a  Borel set $ V \subset\Delta( w, 2 r )$ with  $ \mathcal{H}^{n-1}(V)>0$, such
that     Proposition   \ref{proposition8.9}  is false for each  $y\in V.$
Under this assumption, we let 
$ \hat w \in V $ be a point of density   for $ V $ with respect to 
$\mathcal{H}^{n-1} |_{\ar D} $. Then 
\begin{align*}
\frac{ \mathcal{H}^{n-1}(\Delta( \hat w, t)
	  \sem V )}{ \mathcal{H}^{n-1}(\Delta( \hat w, t) )
   } \to  0\quad \mbox{as}\quad t\to 0,
\end{align*}
and so  there exists $ c \geq 1$ depending only on $c_{\star}$ in  \eqref{eqn8.28} and the data such that
\begin{align*}
\,c \,  \mathcal{H}^{n-1} (\ar \ti  D \cap \Delta ( \hat w, s) \cap V)  \geq s^{ n - 1}
\end{align*}
provided  $ s > 0$ is small enough,
  where
 $ \ti  D  = \ti D ( \hat w,  s   ) \subset D $ is the
  starlike Lipschitz domain  defined in Lemma \ref{lemma8.7}.
  To get  a  contradiction, we show 
that  
\begin{align}
\label{eqn8.52} 
\mbox{Proposition \ref{proposition8.9} is true  for almost every $ y \in \ar \ti D \cap \De ( \hat w, s ).$ } 
 \end{align}
 To do this,   let  $ \Psi $ be the set of all $ y \in \ar \ti D \cap \De ( \hat w, s) $ satisfying
\begin{align}
 \label{eqn8.53}
 \begin{split}
   & (a) \hs{.2in}  \mbox{ $y$  is a point of  density for $ \Psi $ relative to  
$ \mathcal{H}^{n-1}|_{\ar D }, $  $ \mathcal{H}^{n-1}|_{\ar \ti D },     \tau $,} 
\\
   &(b) \hs{.2in} \mbox{ There is a tangent plane $T(y)$ to both $\ar D, \ar \ti D $ at  $y$,} 
   \\
  &(c) \hs{.2in}  \lim_{t \to 0} t^{1-n} \mathcal{H}^{n-1}  ( \ar D \cap B ( y, t) )   =
  \lim_{t \to 0} t^{1-n} \mathcal{H}^{n-1} ( \ar \ti D \cap B ( y, t)) =  b',
  \\
  &(d) \hs{.2in}\lim_{t \to 0} t^{1-n} \tau ( \ar D \cap B ( y, t) )   =  b'  \, k(y)^{p-1}.
 \end{split}
\end{align}  
   In  \eqref{eqn8.53}, $ b' $  denotes the Lebesgue $(n - 1)$-measure of the unit ball in $ \rn{n-1}. $   
   
   We claim that
    \begin{align}
    \label{eqn8.54}   
    \mathcal{H}^{n-1}  ( \ar \ti D\cap \De (\hat w , s ) \sem \Psi ) = 0. 
    \end{align}
    Indeed   $(a)$ of \eqref{eqn8.53}  for $ \mathcal{H}^{n-1}$-almost every $y$ is a consequence of the fact that  $ \mathcal{H}^{n-1}|_{\ar D}, $ $ \mathcal{H}^{n-1}|_{\ar \ti D}  $ are regular Borel measures and differentiation theory while  $ (a)$  of this display  for $ \tau $  and $ \mathcal{H}^{n-1}|_{\ar D}$ for almost every $y,$ follows from  the same observations and Lemma \ref{lemma8.6}.  
 \eqref{eqn8.53}  $ (b) $  follows from  the Lipschitz  character of   $ D,  \ti D, $ and    Rademacher's  theorem (\cite[Chapter 3]{EG}).   Finally $(c)$ and $ (d) $  of this display   are consequences  of   the Lebesgue differentiation theorem  and  Lemma  \ref{lemma8.6}.  Thus, $ \eqref{eqn8.54} $ is true. 
     
     We now use a blowup argument to complete the proof of  Proposition \ref{proposition8.9}. Let 
$ \Psi, s $ be as above   and
     $ y \in \Psi. $   Since  $ \mathcal{A}$-harmonic functions are invariant under translation we may assume that $ y = 0. $   Let  $ \{t_m\}_{m\geq 1}$ be a decreasing sequence of positive numbers
    with limit zero and $t_1 << s$.  Let
\begin{align*}
\begin{split}
D_m  &= \{x:\, \, t_m x \in   D \cap B ( \hat w, s ) \},  \\
\ti D_m    &= \{ x:\,\, t_m x \in \ti D \cap B ( \hat w, s )\}, \\
v_m (& x ) =  t_m^{-1}  \, v (t_m x) \quad\mbox{whenever} \quad  t_m x  \in  B (\hat w, s ).    
\end{split} 
\end{align*}

Fix  $ R >> 1. $ Then  for  $ m $ sufficiently large, say $ m \geq m_0, m_0 = m_0 (R), $  we note that    $ v_m $ is $ \mathcal{A}$-harmonic in  $ D_m \cap B ( 0, 2 R ) $ and continuous in $ B ( 0, 2 R ) $ with $ v_m \equiv 0 $ on $ B ( 0, 2 R ) \sem  D_m.$  
Let
  \begin{align} 
  \label{eqn8.56}  
  \nu_m ( J  ) =  t_m^{1-n}   \tau ( t_m J )\quad  \mbox{ whenever $J $ is a Borel subset of $ B ( 0, 2 R ).$} 
  \end{align}
  Then $ \nu_m $ is the measure corresponding to  $ v_m $ on $ B ( 0, 2R), $ as in
   Lemma \ref{lemma8.5} for
  $ m \geq m_0. $   Let   $  \xi   \in  \mathbb{S}^{n-1} $  be  a  normal to  $ T (0).  $   We assume as we may that
   $ H = \{ x: \lan x ,  \xi  \ran  > 0  \} $ contains  $ \hat w'. $      Then  from  Lipschitz  starlikeness of   
 $ D , \ti  D,  $  and    \eqref{eqn8.53} $(b)$   we deduce  that 
\begin{align}    
\label{eqn8.57}         
\begin{split}
d_{\mathcal{H}} (  D_m \cap B ( 0, R ),    &H  \cap  B ( 0, R ) ) \\
&+   d_{\mathcal{H}} ( \ti D_m \cap B ( 0, R ),    H  \cap  B ( 0, R ) ) \to 0 \quad \mbox{as} \quad
      m \to \infty,
      \end{split}
\end{align}  
where   $ d_{\mathcal{H}} $  as defined in  section \ref{NSR} denotes  Hausdorff   distance.    
 Let $ \eta = v  ( \hat w'  )/s $ and   
from \eqref{eqn8.43} we see that 
\begin{align}
\label{eqn8.58}   
|\nabla v_m |  \leq    c\eta   \quad \mbox{on} \,\, \ti  D_m. 
\end{align}
  Also, from   \eqref{eqn8.57}, \eqref{eqn8.58},    \eqref{eqn8.16} for  $ v_m ,$  and  
\eqref{eqn8.28}  we deduce that
\begin{align}
\label{eqn8.59}   
| v_m ( x ) | \leq c   \left( \frac{ d ( x, \ar D_m )}{ R} \right)^\be    \, \eta \, R \quad  \mbox{whenever} \,\, x \in D_m \cap B (0,R),
\end{align}
   where  $ \be $ is the H\"{o}lder exponent in 
\eqref{eqn8.16}.   From   \eqref{eqn8.58},  \eqref{eqn8.59},  and  \eqref{eqn2.2},    we see that a subsequence of $ \{v_m\}, $  say   $ \{v'_m ( x )\}$ where $v'_m ( x ) =   \,  v (t'_m x)/t'_m,  $ converges uniformly on
     compact subsets of  $ \rn{n} $ to a  H\"{o}lder continuous function $ \hat v $ with $ \hat v \equiv 0 $  in $ \rn{n} \sem H. $  Also $ \hat v \geq 0 $ is  $\mathcal{A}     = \nabla   f$-harmonic in $ H. $  
     
     We now apply    a boundary  Harnack inequality  in Theorem 1  of  \cite{LLN}  with    $ \Om, u $
       replaced by $ H,   \lan x,  \xi  \ran^+, $  respectively. Letting $ r \to \infty $ in this inequality, we get   
  $ \hat v ( x )  =  \ga  \lan x,  \xi  \ran^+ $ for some $ \ga \geq  0, $ where $ C^+ = \max(C, 0). $
          Let   $ \nu_m'  $ be   the measure  corresponding to $ v'_m $    
      and  observe from  \eqref{eqn8.18},  \eqref{eqn8.59}  that
    the sequence of measures, $\{\nu'_m\}_{m\geq 1}, $    corresponding to $ \{v'_m\}_{m\geq 1}, $  have  uniformly bounded total masses  on $ B ( 0, R ). $  Also from
    \eqref{eqn8.15}-\eqref{eqn8.17},  \eqref{eqn8.59}, we see that $ \{v'_m\}$ is uniformly 
bounded in
   $ W^{1,p} ( B ( 0, R )). $  Using these facts  and  \eqref{eqn2.2}   we obtain  that    
   \[
   \{\nu'_m\}\, \,  \mbox{converges weakly to }\, \,  \nu \quad \mbox{as} \, \, m\to \infty
   \]
   where $ \nu $ is the measure associated with
     $ \ga \, \lan x,  \xi \ran^+. $    Using  integration by parts and   the fact that  $ \lan x,  \xi \ran^+ $ is   $  \mathcal{A} = \nabla   f $-harmonic in  $  H $  we get 
\[ 
d \nu   = \ga^{p-1} \lan  \nabla  f (\xi ), \xi  \ran  d \mathcal{H}^{n-1} \, |_{\ar H}  =  p \ga^{p-1}   f (\xi )   d\mathcal{H}^{n-1} \, |_{\ar H}    
\]  
where we have also used  $ p$-homogeneity of  $ f. $  From  this computation, weak convergence, \eqref{eqn8.56}, and  \eqref{eqn8.53} $(d)$, we have 
\begin{align}
\label{eqn8.60}   
\begin{split}
p  \ga^{p-1} \,  f (\xi) \, b' \, R^{n-1}  &=   \lim_{m \to \infty}
       \nu'_m ( B (0,R)) \\
       &=  \lim_{m \to \infty}( t'_m)^{1-n} \nu ( B ( 0, R t'_m )) \\
       &=         b' \, R^{n- 1}  k^{p-1} ( 0 ).
       \end{split}
         \end{align}
Also, from  our earlier observations we see that  $ x \mapsto  t^{-1} v ( t x ) $  converges uniformly as $ t \to 0 $ to $ \ga  \lan  x,  \xi  \ran^+$  on  compact subsets of $ \rn{n} $  and $ x \mapsto  \nabla v  ( t x ) $ converges uniformly to $ \ga  \xi  $  as $ t  \to 0 $ when $ x $  lies in a compact subset of  $ H. $ Given $ 0 <  \he < 1, $  let  
  \[
  K_\he =           \{ x \in H:\,\, \lan x, \xi \ran   \geq \he | x | \}.
  \]
   In view of these remarks we conclude that
\begin{align}
\label{eqn8.61}  \lim_{t \to 0}  \nabla v  ( t  \om )  =  \ga   \xi 
\end{align}
           whenever $ 0 < \he < 1  $ is fixed and  $ \om \in  K_\he $ with $ |\om| = 1.$  It is easily seen  for given
           $0 < b < 1 $
           and $ t > 0 $ small that there exists $ \he > 0 $ such that $  \Ga (0) $ defined relative to  $ D $ and $ b $ satisfies 
           $ \Ga ( 0 ) \cap B ( 0, t ) \subset  K_\he. $
           From   this observation and \eqref{eqn8.61} we conclude the validity of  $  \eqref{eqn8.48} $ independently of  $ b. $  Then  $ \ga \xi   = \nabla v (0) $  by definition so using   
           \eqref{eqn8.60}   to solve for  $ k  (0 )  $  we  arrive at    \eqref{eqn8.49a} and \eqref{eqn8.49b}.  This completes the proof of   \eqref{eqn8.52}  which as mentioned above this display  gives a contradiction to our assumption that 
           Proposition \ref{proposition8.9} is false. 
           
\end{proof}


\setcounter{equation}{0} 
\setcounter{theorem}{0}
\section{Boundary Harnack inequalities}
\label{section10}
 In  this section we use our work in section \ref{section9}  to prove  boundary Harnack inequalities for the ratio 
of  two  $  \mathcal{A} = \nabla  f$-harmonic functions  $ \ti u,  \ti v , $  which are 
$ \mathcal{A}$-harmonic in  $ B ( w, 4r ) \cap D'$  and   continuous in  $B ( w, 4r)$  with    
$ \ti u   = \ti v  \equiv 0  $ on  
$  B ( w, 4r )\sem  D' $. Here $    D'   $  is  a  bounded Lipschitz domain.     

To  set the stage for these inequalities,   let  $ D  $  be a  starlike Lipschitz domain with center $ z $.   Let  
\[
 w \in \ar D    \mbox{ and }\, \, 0 < r  <  |w - z|/100.
 \]

Let  $  \breve{v}_i  > 0$, for $i = 1, 2,    $ be  $  \mathcal{A} = \nabla f$-harmonic  functions  satisfying \eqref{eqn8.28}  in 
 $  B ( w, 4r )  \cap D. $   Assume  also that  $  \breve{v}_i $ is continuous in  $  B ( w, 4r) $  with   $  \breve{v}_i   \equiv 0  $ on   $  B (w, 4r) \sem D, $  for $ i = 1, 2.$   
  
  Let   $ \hat  w  \in   \ar D \cap B (  w,  r  ) , 0 < s  < r , $  and  let     $ 
 \ti D_i =  \ti D_i  ( \hat w, s )   $  be   the  starlike Lipschitz domains in  Lemma \ref{lemma8.7} with center at  $ \hat w' $   defined  relative to  $ \breve{v}_i$ and $D $    
 for $ i  = 1, 2. $  Put   $ \ti D = \ti D_1\cap  \ti D_2 . $  
  
 From this lemma we see that 
 \begin{align} 
 \label{eqn8.63}   
 |\nabla \breve{v}_i  ( x )  |    \approx  \breve{v}_i  ( \hat w')/s  \,\, \mbox{ when } \, \,   x  \in  
  \ti D  \mbox{ and }  i = 1, 2,  
 \end{align} 
 where ratio constants depend only on the data and $c_{\star}. $  Also if 
  $ s' = s/c', $  then  
 \begin{align} 
 \label{eqn8.63a}
    \frac{  \mathcal{H}^{ n - 1}   [  \ar \ti  D  \cap  \ar D \cap  B ( \hat w, s' )  ] }{  \mathcal{H}^{ n - 1}   [  \ar D  \cap   B ( \hat w,  s')  ]  }  \geq   1/2.   
 \end{align}     
      
 Given  $    t_1, t_2  \geq 0, $ and for $ y \in    \ti D $ set   
            \begin{align*}
            d \ti \ga ( y )  :=     \left[d ( y, \ar   \ti D  ) \, 
  \max\limits_{x \in B ( y,  \frac{1}{4} d ( y, \ar  \ti D ) )  }   \,  \mathcal{M}(x) \right]\,dy
            \end{align*}
            where
  \[
 \mathcal{M}(x)= \left\{ \left[ t_1  | \nabla \breve v_1(x) | +  t_2  | \nabla  \breve v_2(x) |  \right]^{ 2p - 6}
 \, \sum_{i,j=1}^n   \,   \left[ t_1   |(\breve v_1(x))_{x_i x_j}|  + t_2   |(\breve v_2(x))_{x_i x_j}| \right]^2    \right\}.
  \]
We show    
 $ \ti \ga $ is a Carleson measure on $  \ti D. $  More specifically, we prove the following lemma.  
 \noindent 
 \begin{lemma}  
 \label{lemma8.11}  
 With the above  notation,   if     $ \hat x  \in \ar  \ti D  $  and   $ 0 < \rho   \leq\mbox{diam}(\ti D)$,  then 
 \begin{align*}  
 \ti \ga (  \ti D  \cap B ( \hat x , \rho  )   ) \, \leq \, c \,
 \left(t_1 \frac{\breve v_1(\hat w')}{s}  + t_2 \frac{\breve v_2 (\hat w')}{s}\right)^{2p - 4}      \rho^{ n - 1} \,  
 \end{align*} 
 where  $  c $  depends only on   $c_{\star}$  and the data. 
 \end{lemma}   
   \begin{proof}
   Observe from  \eqref{eqn8.63} and  \eqref{eqn2.3}   that 
      if   
      \[  
     \mathcal{I} = \left[\frac{(t_1 \breve v_1 (\hat w') +  t_2 \breve v_2 (\hat  w' ))}{s'}\right]^{ 2p- 6} 
      \] 
       then                
       \begin{align} 
       \label{eqn8.80} 
\begin{split}
       \ti \ga &(  \ti D  \cap B ( \hat x , \rho  )   ) = \int_{ \ti D  \cap B ( \hat x , \rho  ) } d \ti \ga ( y )\\
&  \leq  c  \mathcal{I}    \int_{ \ti D  \cap B ( \hat x, \rho )  } d ( y,  \ar  \ti D)   \max_{B ( y,  \frac{1}{4} d ( y, \ar \ti D ) )  }   \left\{    \sum_{i,j=1}^n  \,  ( t_1  |     (\breve v_1)_ {x_i x_j} |     +  t_2 | (\breve v_2)_ {x_i x_j} | )^2    \right\} \, dy  \\
   & \leq  c^2 \, \mathcal{I} \int_{   \ti D  \cap B ( \hat x, \rho  )  }    d ( y, \ar  \ti D )^{ 1 - n}  \,   
 \,  \left( \int_{B ( y,   \frac{3}{4}  d ( y,  \ar  \ti D ) )}   \, 
  \,  {\ds \sum_{i,j=1}^n } (t_1  |(\breve v_1)_{x_i x_j}|     +  t_2 | (\breve v_2)_ {x_i x_j} | )^2  dx  \right)dy \\                                      
      &\leq \, c^3  \mathcal{I} \, 
 \int_{  \ti D  \cap B ( \hat x,  \rho  ) }     d ( y, \ar   \ti D  ) \, 
    \,    \sum_{i,j=1}^n   \,   (t_1  |(\breve v_1)_{x_i x_j} (y) |     +  t_2 | (\breve v_2)_ {x_i x_j} (y)  | )^2      \, dy  =      \mathcal{II}, 
    \end{split}
    \end{align}
      where to get the last integral we have  interchanged the order of integration in the second integral. From  \eqref{eqn8.28} and  \eqref{eqn8.63}  we find for $ y \in  \ti D,  $  and $ i = 1, 2, $ 
  that      
  \begin{align}  
  \label{eqn8.81}  
  d (y,  \ar \ti D )  \leq  d ( y ,  \ar D )    \leq   c   ( s/ \breve v_i ( \hat w') )  \,  \breve v_i (y). 
  \end{align}   
  Using   \eqref{eqn8.81}   in     \eqref{eqn8.80},   it follows that     
    if    
    \[  
         \mathcal{II}_i   =   t_i^2 \, ( s/\breve v_i (\hat w' )) \, \mathcal{I} 
    \int_{\ti D  \cap  B ( \hat x,  \rho) } \, \breve v_i(y)   \, \sum_{j, k = 1}^n   |(\breve v_i)_{x_jx_k} (y) |^2 dy  \quad \mbox{for}\, \, i=1,2  
    \]
       then      
       \begin{align}  
       \label{eqn8.82}  \mathcal{II}  \leq \ti c  \, ( \mathcal{II}_1 +  \mathcal{II}_2 ) 
       \end{align} 
          where $  \ti c $ depends only on $c_{\star}$ in \eqref{eqn8.28} and the data. 
           
        To estimate $ \mathcal{II}_i$ for $i = 1, 2,  $   fix $ i   \in  \{1, 2\} $     and  for  small $\delta > 0, $ put 
    \[
     \vartheta=  \delta^{-1}  \nabla  \breve v_i ( x +\delta e_l ) \quad \mbox{and}\quad \upsilon=  \delta^{-1} \nabla  \breve v_i (x).
     \]
By repeating the argument from \eqref{eqn4.8} to \eqref{eqn4.12} and letting $  \delta \to 0 $ in this equality we deduce from     \eqref{eqn8.63}, \eqref{eqn2.3} that if       $ \zeta = (\breve v_i)_{x_l},  1 \leq l \leq n, $ then   $ \zeta $ is a weak solution in 
     $\ti D $ to   
\begin{align}  
\label{eqn8.83}  
\mathcal{L}_i  \zeta = \sum_{k,j=1}^n   \left(  (\breve b_i)_{kj}  \zeta_{x_j} \right)_{x_k}  = 0
\end{align} 
where 
\begin{align} 
\label{eqn8.84}  
(\breve b_i)_{kj} ( x )  =    f_{\eta_k \eta_j}  ( \nabla \breve v_i (x)) \quad \mbox{for}\, \, 1  \leq k, j   \leq n. 
\end{align}
Also  $ \breve v_i  $  is  a  solution to  \eqref{eqn8.83} as follows from    $ \mathcal{A}$-harmonicity of $ \breve v_i $   and  $p$-homogeneity of  $ f.  $          
   Using  \eqref{eqn8.83},   \eqref{eqn8.84}, the structure assumptions on $ \mathcal{A}, $ and     \eqref{eqn8.63} we deduce  for $ i = 1, 2, $   that   
    \begin{align} 
    \label{eqn8.85} 
    \begin{split}
     \mathcal{L}_i  ( | \nabla \breve v_i (x) |^2  )  &\geq   2  \sum_{k,j,l =1}^n     (\breve b_i)_{kj} ( x )  [(\breve v_i)_{x_k x_l }   (\breve v_i)_{x_j x_l } ]  \\ 
     & \geq c^{-1}  |\nabla \breve v_i(\hat w')|^{p-2} \sum_{k,j=1}^n  [(\breve v_i)_{x_k x_j} ]^2\\
     & \geq c^{-2}   ( \breve v_i (\hat w')/s)^{p-2}   \sum_{k,j=1}^n  [(\breve v_i)_{x_k x_j} ]^2 
\end{split}   
     \end{align}      
     weakly   in  $ \ti D $.   Given  $  t \in  (1/2, 1)$ and $ y \in  \ar \ti D $, let  $ y (t) $ be that point on the line segment from  $ \hat w' $ to $ y $ with  $ | y (t) - \hat w' |  =  t | y  - \hat w' |. $     Let  $ \ti D (t) $ be the union of all half open  line segments $ [\hat w', y (t) ) $   joining $ \hat w' $  to  $ y (t) $ when  $ y \in \ar \ti D$.   Using   starlike Lipschitzness  of   $ \ar \ti D (t) $, the fact that 
   \[  
   \breve{v_i}  \mathcal{L}_i ( |\nabla \breve v_i |^2 )  = \breve{v_i}  \mathcal{L}_i ( |\nabla \breve v_i |^2 ) - \mathcal{L}_i ( \breve{v_i} )  | \nabla  \breve v_i |^2
   \] 
   weakly in  $ \ti D (t), $       \eqref{eqn8.63}, \eqref{eqn8.85}, and  integration by parts   we  obtain for $ \mathcal{H}^1 $ almost every $ t \in (1/2, 1) $ that    
\begin{align} 
\label{eqn8.86}    
\begin{split}
( &\breve v_i (\hat w')/s)^{p-2}  \, { \ds \int_{ 
      \ti D (t)   \cap B ( \hat x, \rho )  }    \sum_{k,j=1}^n  } \, \breve v_i \,   |(\breve v_i) _{x_k x_j} |^2 \\
    &\leq \, c \,  \left | \int_{ \ar \, [  \ti D (t )  \cap  B( \hat x, \rho )   ]}  \, \,   \sum_{k, j = 1}^n    (\breve b_i)_{kj} [( \breve v_i \,     
    (|\nabla \breve v_i|^2)_{x_k}   - |\nabla \breve v_i |^2    (\breve{v_i})_{x_k}  )]   \nu(t)_j  d\mathcal{H}^{n-1} \right |
   \end{split}
   \end{align}
                        where $  \nu (t ) $ denotes the unit outer normal to  $ 
                      \ti D (t) \cap    B ( \hat x,  \rho)  $ and  $ c \geq 1 $ depends only on $c_{\star}$ and the data.    Using   once again  
     \eqref{eqn8.63} and  \eqref{eqn2.3} we can estimate the right-hand side of   
     \eqref{eqn8.86}.  Doing this and using the resulting estimate in \eqref{eqn8.86}   we deduce that 
\begin{align} 
\label{eqn8.87}      
\int_{\ti D (t) \cap  B ( \hat x,  \rho)   }    \sum_{k,j=1}^n  \, \breve v_i \,   |(\breve v_i) _{x_k x_j} |^2  \, dx   \, \leq  \breve c  \,  \left(\frac{\breve v_i (\hat w' )}{s}\right)^3   \rho^{n-1} 
\end{align}
   where $ \breve c $ depends only on  $c_{\star}$ and the data. Letting $ t \to 1 $ and using   Fatou's lemma we see that  \eqref{eqn8.87} remains valid with  
   $ \ti D ( t ) $  replaced by  $ \ti D. $        
     In view of  \eqref{eqn8.87} for $ t = 1 $   and  \eqref{eqn8.82} we conclude first  that    
     \begin{align*}
     \mathcal{II}  &\leq   \ti c \, \mathcal{I}    \left[  t_1^2  \left(\frac{\breve v_1 (\hat w' )}{s}\right)^2     +  t_2^2  \left(\frac{\breve v_2 (\hat w')}{s}\right)^2   \right]  \rho^{n-1} \\
     &\leq   \bar c  \left[  t_1  \frac{\breve v_1 (\hat w' )}{s}     +  t_2  \frac{\breve v_2 (w')}{s}\right]^{2p-4}  \rho^{n-1}  
     \end{align*}
     and thereupon from  \eqref{eqn8.80} and arbitrariness of  $ \hat x,  \rho $   that  Lemma   \ref{lemma8.11} is true.    
     \end{proof} 
      
      We continue  under the assumption that   $ \breve  v_i,  D,  t_i,  i = 1, 2, s,  \hat w, 
      \hat w' ,  \ti D $    are  as in Lemma \ref{lemma8.11}.  Let
      $  \breve v =     t_1 \breve v_1  - t_2 \breve v_2.$
             Using  \eqref{eqn4.8} with   $\vartheta  = t_1  \nabla \breve v_1$,  $\upsilon  =  t_2   \nabla \breve v_2,   \mathcal{A}$-harmonicity of $ \breve v_i, $   and  $p$-homogeneity of  $ f,  $ we deduce as in \eqref{eqn8.83}  that $ \breve v $ is a weak solution in  $ D $  to   
     \begin{align}
     \label{eqn8.88} 
  \breve{\mathcal{L}}\breve v  =  \sum_{k,j=1}^n   \left(  \ti  b_{kj}  \breve v_{x_j} \right)_{x_k}  = 0  
   \end{align} 
   where  at $x$ 
    \begin{align}
    \label{eqn8.89}  
    \ti b_{kj}(x)   =    \int_0^1 f_{\eta_k \eta_j}  ( s t_1 
    \nabla \breve v_1(x) +   (1-s) t_2  \nabla \breve v_2(x) ) ds  \quad  \mbox{for}\, \, 1  \leq k, j   \leq n.  
    \end{align}
Now, if
             \[
              \be (x)  =   ( t_1  |\nabla \breve v_1(x)|  + t_2 |\nabla \breve v_2(x) | )^{p-2}
              \]
               then   
\begin{align}
  \label{eqn8.90}   
  \sum_{k,j=1}^n   \ti b_{kj} (x) \xi_k \xi_j 
     \approx  \be (x)  |\xi |^{2} \quad \mbox{whenever}\quad  \xi \in \rn{n} \sem \{0\}.   
\end{align}     
       Ratio constants depend only on  the data. 
      We  note  from  \eqref{eqn8.43} for  $ \breve v_1,  \breve v_2 $   that  
\begin{align}
  \label{eqn8.91}  
      \be (x) \approx   ( t_1 \breve v_1 (\hat w')/s  + t_2  \breve v_2 (\hat w')/s)^{p-2} = \ph 
\end{align}
where $ \ph > 0 $  when  $ x  \in \ti D$.     Thus  $  (\ph^{-1}  \, \ti b_{kj}) $ is uniformly elliptic in   $ \ti D$ with ellipticity constant  $   \approx 1$.      It  is then  classical    (see  \cite[Theorem (6.1)]{LSW})  that Green's function for this operator with pole at  $ \hat w'  \in  \ti D $  exists as well as the corresponding elliptic measure, 
       $ \ti \om (  \cdot,  \hat w' ). $  Moreover, as  in  \cite[section 4]{CFMS}  there exists  $ \bar  c \geq 1 $ depending only on $c_{\star}$ and the data  such that 
   \begin{align} 
\label{eqn8.92}   
\bar c  \, \, \ti \om (  \ar \ti D  \cap B ( \hat w, s' ) , \hat  w'  ) \geq  1.
\end{align}

  Using  Lemma \ref{lemma8.11}  we see that Theorem 2.6  in   \cite{KP}  can be applied  to conclude that $ \ti \om ( \cdot, \hat w'  ) $  is  an     $ A^\infty$-weight on $ \ar \ti D $ in the following sense.  There exists $  \ti c_+  \geq  1 $ depending only on 
$c_{\star}$  and the data such that  if $ \hat  x   \in \ar \ti D,  0 < \rho  < $   diam 
$ \ti  D, $   and    $ K \subset \ar \ti D  \cap B ( \hat x, \rho  )  $ is a Borel set  then 
\begin{align}  
\label{eqn8.93}  
 \frac{\mathcal{H}^{n-1} (K)}{\mathcal{H}^{n-1} ( \ar \ti D  \cap B ( \hat x, \rho)  )} \geq  1/4 \quad  \Rightarrow \quad  \frac{\ti \om (K, \hat w'   )}{\ti \om (  \ar \ti D  \cap B ( \hat x, \rho  ) , \hat  w'  ) }  \geq  c^{-1} _+. 
 \end{align}

To   use this  result  and avoid existence questions for  elliptic measure as well as  the Green's function defined relative  to $ (\ph^{-1}  \ti b_{kj}) $   in   $  D $  we temporarily assume that 
\begin{align}  
\label{eqn8.94}  
\mathcal{R} \in C^\infty ( \mathbb{S}^{n-1})  \quad \mbox{where}\quad  \ar    D  =   \{ z+  \mathcal{R} ( \ze ) \,  \ze :\,\,   \ze  \in   \mathbb{S}^{n-1}  \}   
\end{align} 
and $\mathcal{R} $ is as in Definition \ref{defn8.2}. 
Then from  \cite[Theorem 1]{Li}  it  follows that  $ \nabla \breve v_i, i  = 1, 2,  $  has a   nonzero  locally  H\"{o}lder  continuous  extension  to   $ \bar D  \cap  B (  w, 4r). $  
 From this theorem  and  \eqref{eqn8.90}  we  deduce  that      if    $ y \in 
  D  \cap  B ( w, 2r)   $  and  $  0  <  s <  r,  $   then     Green's function for  $ \mathcal{\breve L}$   with pole
at  $ y$   and   the  corresponding elliptic measure    $  \om  ( \cdot, y ) $, exist 
relative  to  $  D  \cap   B  ( w, 2r) . $      Then from \eqref{eqn8.63a},  \eqref{eqn8.92}, 
  \eqref{eqn8.93} with  $ \hat x  = \hat w,   \rho = s',  K  =  \ar \ti D  \cap  \ar D   \cap B ( \hat w, s' )  $ and the  weak  maximum principle for $ \mathcal{\breve L}$ we find  $  c_{++} \geq 1, $ depending only on $c_{\star}$ and the data such that          
\begin{align} 
\label{eqn8.95}  
\begin{split}
\om ( \ar D  \cap  B (\hat w ,  s' )  , \hat w' )   &\geq  \ti \om (\ar D \cap \ar  \ti  D  \cap  B (\hat w,  s'   ) , \hat w') \\
 &\geq  c_{++}^{-1}.
 \end{split}
 \end{align}
From  arbitrariness of  $ \hat w,  s, $     we deduce from   \eqref{eqn8.95} and   a  covering argument  that  if  
    $  \hat  x     \in   \ar D \cap  B ( w, r ), $  $  0 <  t     \leq  r/2,  $  
         then  
  \begin{align}
   \label{eqn8.95a}  \om ( \ar D \cap  [ B (  \hat x  ,  7 t /8  ) \sem  B ( \hat x , 5 t /8 )] ,   \cdot     ) \geq c^{-1} \quad \mbox{ on } \, \,  \ar B ( \hat x , 3 t /4 ) \cap D 
   \end{align}
where $ c $  depends only on  $ c_{\star} $  and the  data.    
Let   $ g ( \cdot,  \ze  )  $  denote   Green's  function  for   $  \mathcal{\breve L} $  in   
$ D \cap    B (\hat x ,  t )  $ with pole at  $ \ze $ in  $  D \cap   B (\hat x ,  t ) . $        Let
 $  \hat  x'  $   denote  the  point  on the ray  from  $ z $  to  $ \hat x $   with   
$  | \hat x -   \hat   x' |    =    t/100 . $  
 Let  $c'  $   be  as  in  Lemma  \ref{lemma8.7}.  We  assume as we may that  if  $ t' = t/c', $   then   $ B ( \hat x'  ,   2   t' ) \subset  
D \cap  B  ( \hat x,  t  ). $     
 Let    $  a   = {\ds \max_{ \ar B ( \hat x'  ,   t') } }  g ( \cdot,  \hat x' ) . $  
    Let $ t''  = t/ \hat c^+ ,   c'  < < 
  \hat  c^+ ,  $  and  suppose  
$  y \in D \cap  B ( \hat x,  t'' ). $  Choose  $  \bar y   \in  \ar D \cap  B ( \hat x,  2t'' ) $ 
with   $ | \bar y  - y  |  =  d  ( y,  \ar D ). $   Then  using the  iteration argument  in   
 \cite{LN}  from  (3.16) to  (3.26) it  follows  that   for $ \hat c^+   \geq 100 c', $  large enough,   depending only on the data  and $ c_{\star}, $   that 
\[
\om (  \ar D  \cap  [ B ( \bar y , t ) \sem  B (\bar  y, t/2 )] ,    y      )  \leq \,  
  \hat c^+   \,            \frac{g (  y,  \hat x'  ) }{ a } \]  
 and  
\[ \ar D \cap  [ B ( \hat x , 7t/8 ) \sem  B (\hat x, 5t/8 )]    \subset 
   \ar D \cap  [ B ( \bar y , t ) \sem  B (\bar  y, t/2 )]   \]  
whenever $ y  \in   \ar D \cap  B ( \hat x, t'' ).  $  

Thus     from the maximum principle for  solutions to $  \breve{\mathcal{L}}$, 
\begin{align}
 \label{eqn8.96} 
\om (   \ar D  \cap  [ B ( \hat x , 7t/8 ) \sem  B (\hat x , 5t/8 )] ,   \cdot       )  \leq \,  
  \hat c^+   \,            \frac{g (  \cdot ,  \hat x'  ) }{ a }  \quad \mbox{on}\, \,  D  \cap  B ( \hat  x,   t'' ).
  \end{align}
Using \eqref{eqn8.96} we prove   
\begin{lemma} 
 \label{lemma8.12}  
 Let  $ \breve v_i$, for $i = 1, 2,  D,  r, w  $ be  
  as  introduced above   Lemma \ref{lemma8.11} and  suppose   \eqref{eqn8.94}  holds.   
  Let $   \hat x   \in    \ar  D  \cap  B  ( w,   r  ),  0   <   t   \leq     r/2, $  and define  $ \hat x' $ relative to $ \hat x $ as in   \eqref{eqn8.96}. 
  Let           $  {\mathcal{\breve L}}$    be as  in   \eqref{eqn8.88}  and let   $ h_1,  h_2 $  be positive  weak solutions to $ {\mathcal{\breve L}} h_i = 0 $ for $i = 1, 2,  $  in    $  D  \cap  B ( \hat x , t ). $  Suppose also that  $ h_1, h_2  $   are  continuous in   $ \bar  D   \cap  \bar B  ( \hat x , t )$ 
  with boundary value 0 on  
 $  \ar D  \cap  B ( \hat x , t). $  Then   for some  $ \ti c  \geq 1 $  depending only on $c_{\star}$  and the data,   \begin{align} 
\label{eqn8.98} 
\ti c^{-1}   \frac{ h_1 ( \hat x' ) }{ {\ds \max_{\bar D \cap \ar B ( \hat x, t  ) } } h_2 }   \leq   \frac{ h_1(x)}{h_2 (x)}  \leq  \ti c  \frac{ {\ds  \max_{\bar D \cap \ar B ( \hat x, t  )} } h_1 }{
  h_2  (  \hat x'  ) }  
 \end{align}
whenever  $x \in  D  \cap   B ( \hat x,   t''  ).$   If  $  t_1 = 1,   \breve v_2   \equiv  0, $  then  there  exists $  c_1'  \geq  1,  \he  \in  (0, 1),   $  depending only on the data  and  $ c_* $ 
such that    if $t''' = t''/c'_1, $  then 
\begin{align}
\label{eqn10.3} 
\left|    \frac{ h_1 (x ) }{ h_2 ( x ) } -     \frac{ h_1(y)}{h_2 (y)}  \right|   \leq   c \left(  \frac{ |x - y| }{ t } \right)^{\he}   \frac{ h_1 ( x)  }{ 
 h_2 ( x)}  \quad \mbox{ whenever }   x, y  \in  B ( \hat  x,  t''') . 
\end{align}
  \end{lemma}   
   \begin{proof}  We note from  \eqref{eqn8.95a}  and  the 
   maximum principle for  solutions to $  \breve{\mathcal{L}}$, 
   that 
 \[  h_i    \leq  c \left( {\ds \max_{\bar D \cap \ar B ( \hat x, t  ) } } h_i  \right) \,  \, \om (   \ar D  \cap  [ B ( \hat x , 7t/8 ) \sem  B (\hat x , 5t/8 )] ,   \cdot       )\quad   \mbox { in }  D  \cap B ( \hat x,  3t/4 )
  \]   
  and 
   \[  
   h_i  \geq  c^{-1}  h_i ( \hat x' )   g ( \cdot,  \hat x' ) /a \quad   \mbox{ in }   D  \cap  B ( \hat x, t )  \sem   B ( \hat x' ,  t' )     
   \]   
   when $ i  = 1, 2 $ for  some $   c   \geq  1, $  depending only on the data and 
 $ c_\star$.    These inequalities and    \eqref{eqn8.96} give  
   \eqref{eqn8.98}.    To prove  \eqref{eqn10.3}  we  note  that  $  \breve v_1 $  is now a  solution to  $  \mathcal{\breve L}$   in  $  D  \cap  B ( w,  4r) $  so  from  \eqref{eqn8.16} and   \eqref{eqn8.98}  with 
   $ h_2   =   \breve v $    we  have       
\begin{align} 
 \label{eqn10.4}  h_1 ( x )   \leq    c \, 
   {\ds  \left( \max_{\bar D \cap \ar B ( \hat x, t  )}  h_1  \right) }   \,  \,  ( | x  -  \hat x |/ t  )^{\ti \be} \quad      \mbox{for}\, \,  x \in  B  (  \hat  x,  t''  ).  
     \end{align} 
Now     \eqref{eqn10.4}, arbitrariness of  $  \hat x,  t $    and the    local Harnack inequality implied by 
 \eqref{eqn8.90} for  $ h_1 $    
 yield  as  in  \cite{CFMS} that  for some $  \be > 0,  c'_1 \geq 1 $  depending only on the data and $ c_\star, $  that if  $ t''' = t''/c'_1, $ then 
\begin{align}
 \label{eqn10.5}    
    h_1 (x)   \leq  c'_1  ( \rho/t )^{\be }   h (a_\rho  (y) ) \quad \mbox{for}\, \,  x  \in  D \cap  B (y, \rho),  
\end{align}  
whenever $ 0 <   \rho  \leq t''', $   $   x  \in  D  \cap  B ( y,  \rho ), $  and   $ y  
\in   \ar  D \cap B ( w,  r  ). $  \eqref{eqn8.98} and \eqref{eqn10.5}   imply that   
\begin{align} 
 \label{eqn10.6}  
  \frac{ h_1 (  x ) }{ h_2 (x) }   \approx   \frac{ h_1( a_{t'''} (\hat x ) )}{h_2 (a_{t'''} (\hat x) )}
\quad \mbox{for}\, \,  x \in D \cap B ( \hat x,  t''')     
\end{align}   
where once again all constants depend only on the data and $c_{\star}.$  
   
Next   if     $ \ze \in  \ar D \cap  B ( \hat  x,  \rho )$, $ 0 <  \rho   \leq  t , $   we  let
\[ 
M ( \rho )  = \sup_{B ( \ze , \rho ) } \frac{h_1}{h_2} \quad 
 \mbox{and}\quad  m ( \rho ) = \inf_{B ( \ze , \rho ) } \frac{ h_1 }{h_2 }
\] 
   Also put  
\[
  \mbox{osc}(\rho):= M ( \rho ) - m ( \rho )\quad \mbox{for} \quad 0 < \rho <
t.
\] 
Then, if $ \rho $ is fixed we  see that $ h_1  -  h_2 m ( \rho)  $   and  $  h_2  M ( \rho )  -  h_1 $  are positive solutions to  $  \mathcal{ \breve L}  $  in   $ D \cap B ( \hat  x,  \rho ) $  so   we   can use   \eqref{eqn10.6}     with    $h_1$  replaced  by  $ h_1  -  m (\rho)  h_2  $  to  get  that 
if $  \rho''' /\rho  = t'''/t, $ then  
  \[  M (   \rho''' )  -  m (\rho ) \leq  c^* ( m ( \rho''' ) - m ( \rho ) ).   \]  Likewise  using 
  \eqref{eqn10.6}     with    $ M ( \rho ) h_2   - h_1   $   replacing  $ h_1 ,$   we obtain  that     
  \[   
  M ( \rho ) - m (   \rho''')  \leq  c^*  ( M ( \rho ) -  M (  \rho''' ) ).  
  \]   
  Adding these inequalities we obtain after some arithmetic that
    \begin{align}
    \label{eqn10.33} 
    \mbox{osc} (  \rho''' ) \, \leq \,
\frac{ c^* - 1}{ c^* + 1 }  \, \mbox{osc} ( \rho )
\end{align}
 where $ c^* $  depends only  on $c_{\star}$  and the data.  
 Iterating  \eqref{eqn10.33}  we conclude  for some $ c  \geq 1,  \al' \in (0,1), $ depending only on the data and $c_{\star}$ 
that
\begin{align}
\label{eqn10.34}
\mbox{osc} (  s ) \, \leq \,
c  ( s / t )^{\al'}  \,  \mbox{osc} ( t ) \quad \mbox{whenever}\quad 0 < s  \leq  t   \leq   r .
\end{align}
Lemma  \ref{lemma8.12}  follows from this inequality,  arbitrariness of  $ \ze, $        
    and the interior H\"{o}lder continuity-Harnack inequalities  for solutions  to  
    $  \mathcal{ \breve L}$.  \end{proof}

      We use  Lemma \ref{lemma8.12}   to  prove  
\begin{lemma} 
\label{lemma8.13}  Let  $ \breve v_i$, for $i = 1, 2$, $w,  r, D, $ be as introduced  above   Lemma \ref{lemma8.11}. Suppose  \eqref{eqn8.94}  holds.  There exists  $ \bar c_{+}  \geq 1  $  depending only on $ c_{\star} $ and the data  such that  if $ r^+ =  r/\bar c_+ $  then   
 \begin{align}  
 \label{eqn10.24} 
 \bar c_{+}^{-1} \frac{  \breve v_1 (a_{ r^{+}} (w))  }{  \breve v_2 ( a_{r^{+}} (w)) }
\leq\frac{  \breve v_1(y) }{  \breve v_2 ( y ) } 
 \leq \bar c_{+} \frac{  \breve v_1 ( a_{r^{+}} (w) ) }{  \breve  v_2 ( a_{ r^{+}} (w) ) },
\end{align} 
whenever $ y \in  D  \cap B (w,  r^{+} )$.

\end{lemma} 
\begin{proof}  
 Our proof is similar to the  proof of Lemma 4.9 in  \cite{LN4}.   
   To  prove the left-hand inequality in  \eqref{eqn10.24}   we set
 \[  
 t_1  =    \frac{  T \, }{   \breve v_1 (a_{ \bar r}(w)) } \quad  \mbox{and} \quad t_2 =  \frac{ 1 }{\breve v_2 (a_{ \bar r}(w))}
  \]    where  $ \bar r $ is as in \eqref{eqn8.16}.  Let  $ r^+  =  \bar r ''  $  where  $  \bar r '' $ is as in Lemma  \ref{lemma8.12} with $ \bar r = t.   $     
We also let 
  \[ 
  \breve v     = t_1  \breve v_1 - t_2  \breve v_2 \quad \mbox{in}\quad   D  \cap  B ( w, 4  r)
  \]  
  where $ T $  is to be determined  so that  $ \breve v  \geq 0 $ in   $  D \cap B ( w,  r^+ )$.  Let  $ \mathcal{\breve L}$ be as in  \eqref{eqn8.88} relative to   $  \breve v $   and  let 
   $ h_1, h_2 $  be weak solutions  to  $  \mathcal{\breve L} $  in  $  D \cap B ( w,  \bar  r ) $  with continuous boundary values             
   \[
    h_i (x) = \frac{\breve v_i(x)}{\breve v_i ( a_{\bar r} (w) )} \quad \mbox{whenever}\, \,  
   x\in \ar [D \cap  B ( w, \bar r)]\, \, \mbox{for} \, \, i  = 1, 2.
   \]  
    
   From     \eqref{eqn8.16} we  see for $ i = 1, 2, $ that 
\[    \breve{v_i} ( a_{\bar r} (w) ) \approx    \max_{D \cap B ( w, \bar r )}
\breve v_i.     \]    Using  this inequality and  \eqref{eqn8.98} we see that 
 if     $  \bar w' $  denotes the point on the line segment from $ z $  to  $ w $  with   $  |w -   \bar w' | =  \bar r/100,  $  then 
\begin{align}
\label{eqn10.25}    
\frac{h_1}{h_2}   \geq \ti c^{-1}  h_1 ( \bar w'  )    = 
  T^{-1}  \quad \mbox{on}\,\,     D \cap B ( w,  r^{+}). 
 \end{align}   
    
 Thus, if   $ T $ is as in  \eqref{eqn10.25}  then  $  T  h_1 -  h_2 \geq 0 $  in  $  D \cap B ( w, r^{+} ). $ 
 Also  $  T h_1 - h_2$ and $ \breve v $  are   weak solutions  to  $  \mathcal{\breve L}$  in  
 $  D  \cap B ( w, \bar r) $  and these functions have the same boundary values, so from the maximum principle for this PDE  we have 
\[
 \breve v = T h_1 - h_2 \quad \mbox{in}\, \,  D \cap  B (w, \bar r). 
 \] 
 Thus to complete the proof of the left-hand inequality in    \eqref{eqn10.24} it suffices to show that  
\begin{align}
\label{eqn10.26}    
h_1 ( a_{r^{+}}(w) )  \approx  h_1  (  \bar w'  ) )  \approx 1  \quad \mbox{and}  \quad
h_2 ( a_{ r^+}(w) ) \approx 
 1       
\end{align}     
where ratio constants  depend only on $c_{\star}$  and the data.  
To  do  this  let   $ \ti w  $  be  the  point  on the  line segment  from $ z $ to $ w $ which also lies on 
$ \ar B ( w, \bar r).$   Then  
\[
 \breve v_i  \approx  \breve v_i (\ti  w ) \quad \mbox{in}\quad  B (\ti w,  d  (\ti  w,  \ar D ) / 8 ) \quad \mbox{and}\quad    \breve v_i  (\ti  w ) \approx  {\ds   \breve v_i (a_{\bar r} (w) )}.
\]
     From   \eqref{eqn8.90}, the structure assumptions on $ \mathcal{A}  $ in   Definition \ref{defn1.1}, and  Lemma  \ref{lemma2.2}   we see that  $ \be (\ti w )^{-1}  \, \breve{\mathcal{L}}$  is  uniformly  elliptic in    
$ B (\ti  w,  d  (\ti  w,  \ar D)/ 8 )$ with ellipticity constant $ \approx 1. $  Using these facts we can apply  estimates for  elliptic measure from \cite{CFMS}  to  conclude first that  $  h_i  ( \ti  w  ) \approx  h_i  ( w^* ), i = 1, 2,   $   where $  w^* $ lies on the line segment from $ \ti w  $  to  $ w $ with   
$ d ( w^* ,   \ar [D \cap   B ( w, \bar r)])  \approx \bar r   $.   We  can then use  Harnack's  inequality in a  chain of disks connecting 
$  w^* $  to   $ a_{r^+} (w),    \bar w' ,    $  to  eventually conclude \eqref{eqn10.26}.  This proves  
the  left-hand inequality in  \eqref{eqn10.24}.  To get the right-hand inequality in  \eqref{eqn10.24} we argue as above with $ \breve v_1,  \breve v_2 $ interchanged.    
Thus, \eqref{eqn10.24} is valid. 
\end{proof} 

   Our goal  now  is to show that  Lemmas \ref{lemma8.12}, \ref{lemma8.13},  remain valid  without assumption 
  \eqref{eqn8.94} and  \eqref{eqn8.28}  for certain  $  \breve v_1, \breve v_2.  $   To do this we first  prove a    lemma on  
    the  ``Green's function''  for $ \mathcal{A}$-harmonic functions  in a bounded domain $ O $  with pole at $ w \in  O$.  In  this  lemma  $ G$ denotes the fundamental solution for $ \mathcal{A}$-harmonic functions with pole at 0   from   Lemma \ref{lemma4.1}.                
\begin{lemma} 
  \label{lemma9.1} 
  Given  a bounded  connected   open set  $ O $  and       $   w  \in O $  there exists  a  function   $ \mathcal{G}  $ on  $ O \sem\{w\}$   satisfying 
     \begin{align} 
     \label{eqn10.36}  
\begin{split}
&     (a) \hs{.2in}       \mathcal{G} \, \, \mbox{is}\, \,   \mathcal{A} = \nabla  f\mbox{-harmonic in}\, \,  O'\, \, \mbox{whenver}\, \, O'  \, \, \mbox{is open with } \bar O'  \subset O  \sem \{w\}.   
\\
&  (b)  \hs{.2in}    \mathcal{G}\, \,    \mbox{has boundary value    $0$ on  $\ar O$  in the  $ W^{1,p} $  Sobolev sense. }  \\   
&   (c)  \hs{.2in}   \mbox{If}\, \,  F ( x ) = G( x - w ),  \, \, \mbox{for}\, \, x  \in \rn{n} \sem \{w\}, \, \,  \mbox{then}\, \, \mathcal{G}(x)\leq  F (x) \, \,  \mbox{whenever} \,\, x \in  O.    \\    
&( d) \hs{.2in}  \int_O  \lan  \nabla  f  ( \nabla  \mathcal{G} ) ,  \nabla  \he  \ran  \, dx    =  \he (w) \quad \mbox{whenever} \, \,  \he  \in  C_0^\infty ( O ). \\
& (e)\hs{.2in}  \mbox{  $\ze =  F  - \mathcal{G}$ extends   to  a   locally H\"{o}lder   continuous function  in  $ O $  and  if  $ O'$  is an  } \\  
&\hs{.45in}  \mbox{open set  with $  \bar O'  \subset  O, $  then   ${\ds \min_{\ar O' }  \ze   \leq     \ze (x)   \leq     \max_{\ar O' } \ze  } $  for $x$ in  $O' $.}  \\ 
& (f) \hs{.2in} \mbox{There exists  $ c \geq 1$ and $\de \in (0,1),$ depending only on $  \al, p, n $, and $\Lambda$  in Theorem \ref{theorem1.4}  }  \\ 
 &\hs{.43in}  \mbox{ such that   $ | \nabla \ze (x)  |  \leq  c  | x - w |^{\de - 1}   $          for  $ x \in B (w, d( w, \ar O)/c). $ }   \\           
 &       (g) \hs{.2in}  \mbox{ $ \mathcal{G} $  is   the unique  function satisfying  $(a)-(d). $}
\end{split}
       \end{align} 
       \end{lemma}  

\begin{proof}   
We note from  Lemma \ref{lemma3.3}   that  if  
    $ B ( w, 2/m )  \subset  O $  and  $  \psi_m    $ is the $ \mathcal{A} = \nabla f$-capacitary function for          
    $  \bar B (w, 1/m ) $  with corresponding measure  $ \mu_m, $   then 
\begin{align}
 \label{eqn10.37}    
  \mu_m (  \bar B (w, 1/m) )^{-1/(p-1)}  \psi_m    \leq  c  |x-w|^{(p-n)/(p-1)}    
  \end{align}
for $  x  \in  \rn{n}  \sem \bar B ( w, 2/m) $    where   $ c $  depends only on the  data.     Also   as in  Lemma \ref{lemma4.1} we deduce that the   sequence, 
\begin{align*}
 \left\{ \mu_m (  \bar B (w, 1/m) )^{-1/(p-1)}   \psi_m   \right\}&\, \, \mbox{converges to}\, \, F \quad \mbox{as} \, \, m\to\infty\\
   &\mbox{uniformly  on compact subsets of}\, \, \rn{n}  \sem \{w\}
\end{align*}
   where $F$ is as  in \eqref{eqn10.36} $(c)$.   To  construct  $\mathcal{G}$,  let  $ \{\ti  \psi_m \}_{m\geq 1} $ be a sequence of    continuous  $ \mathcal{A}$-super harmonic  functions in       $ O $  with  $ \ti  \psi_m   \equiv 1 $  on  $  \bar B (w, 1/m), $  while $  \ti  \psi_m    $  is      $  \mathcal{A}$-harmonic  in  $  O \sem  \bar B ( w, 1/m) $  with  boundary value      $ 0 $   on   $  \ar  O $ in the  $ W^{1,p}  $  Sobolev sense.     
     Let  $  \ti  \mu_m  $ denote the measure corresponding to  $ \ti  \psi_m.  $   Then  
      from the  definition of   $  \mathcal{A}$-harmonic  capacity we see that  
      \[
      \ti \mu_m  ( B ( w, 1/m) )  \geq  \mu_m  ( ( B (w, 1/m ) ) 
      \]
        so from \eqref{eqn10.37}  we  have
 \begin{align}
 \label{eqn10.38} 
 \begin{split}
  \ti \mu_m (  \bar B (w, 1/m) )^{-1/(p-1)}   \ti  \psi_m  (x)  &\leq  \ti c    \mu_m (  \bar B (w, 1/m) )^{-1/(p-1)}   \psi_m (x) \\
  & \leq  \ti c^2   |x-w|^{(p-n)/(p-1)} 
\end{split}
  \end{align}
  for $  x  \in  O  \sem \bar B ( w, 2/m). $  
       
                  Now from  \eqref{eqn10.38} and the basic estimates in section \ref{section2} we see that 
       a  subsequence of  $ \{\ti \mu_m (  \bar B (w, 1/m) )^{-1/(p-1)} \ti \psi_m (x) \}_{m\geq 1}$   
     and  the corresponding sequence of  gradients,  converges uniformly on compact subsets of  $ \rn{n}  \sem  \{w\}  $   to an  $ \mathcal{A}$-harmonic function  in  $ O\sem \{w\} $ and its  gradient which we now denote by  $  \mathcal{G},  \nabla \mathcal{G}$.  Clearly,  $\mathcal{G} $ satisfies  \eqref{eqn10.36}  
$  (a), (b). $    Also, since   
\[
  \mu_m ( \bar B (w, 1/m) )^{-1/(p-1)}   \psi_m    \to  F \quad   \mbox{as} \quad m \to \infty 
  \]
    we see from  \eqref{eqn10.38}   that   \eqref{eqn10.36}   $(c)$ is true.         
      
           From   \eqref{eqn10.38}  and  Lemma \ref{lemma2.2}  it follows that  
\begin{align}
\label{eqn10.39}   
| \nabla   \ti  \psi_m   (x) | \, \leq  \,  \hat c 
     \frac{\ti  \psi_m  (x) }{|x - w| } \leq  \hat c^2 
   \ti \mu ( \bar B (w, 1/m)^{1/(p-1)}  \,   |x-w|^{(1-n)/(p-1) }  
   \end{align}
    for $  x  \in  O  \sem \bar B ( w, 4/m). $  
           From \eqref{eqn10.39}   we conclude    for fixed  $ q  <  n (p-1)/(n-1) $   and  $ m \geq l,  $  that the sequence 
\begin{align}
 \label{eqn10.40} 
 \begin{split}
\{( \ti \mu (\bar B (w, 1/m) )^{-1/(p-1)} & | \nabla   \ti  \psi_m   | \}\, \,  \mbox{is  uniformly bounded}\\
   & \mbox{ in  $ L^q  ( O  \sem  B ( w,  4/l) ) $   independent of $l.$ }     
    \end{split}
 \end{align}   
Now \eqref{eqn10.40}  and  uniform convergence  of  a subsequence of   $  \ti \mu ( \bar B (w, 1/m) )^{-1/(p-1)}   \nabla   \ti  \psi_m  $     on compact  subsets     of  $  O \sem \{w\} $  imply   that    this  subsequence also converges strongly in     $ L^q  ( O  \sem  \{w\} )  $ to  $  \nabla \mathcal{G} $  whenever  $ q  < n (p-1)/(n-1). $              Using this fact and  writing out  the integral identities  involving  $ \ti  \psi_m   , \ti \mu_m, $  we conclude after taking  limits, that  \eqref{eqn10.36} $ (d) $  is also valid.   

 To  prove  \eqref{eqn10.36} $(e)  $   we  note  from  the estimate in remark \ref{rmk7.1}  that   
\begin{align}  
\label{eqn10.41}  
|  \nabla  F ( x )  |  \approx   \lan \nabla F ( x ),    
\frac{w - x}{|w-x|}\ran \approx  | x  -   w|^{(1-n)/(p-1)}  \approx  F ( x ) / | x -  w |   
\end{align}
  whenever  $   x  \in \rn{n} \sem \{w \}  $ where constants in the ratios depend only on the data.   It follows from  \eqref{eqn10.41}   
       as in the derivations of  \eqref{eqn4.8}-\eqref{eqn4.11},   \eqref{eqn5.5}-\eqref{eqn5.8}, \eqref{eqn8.88}-\eqref{eqn8.90},    that   $ \ze =  F  -  \mathcal{G}$ is  a  weak  solution to  a  locally uniformly  elliptic  PDE  in $ O \sem \{w\}$ of the  form,  
\begin{align}
 \label{eqn10.42}  
 \sum_{i, j  = 1}^n   \frac{  \ar }{\ar x_i}  \left( b_{ij}  \frac{ \ar  \ze }{\ar x_j}\right)  =  0 
 \end{align}
 where 
 \[
 b_{ij} ( x )  =    \int_0^1  f_{ \eta_i \eta_j }  ( t   \nabla  F (x)   + ( 1 - t )  \nabla\mathcal{G} (x) ) dt \quad \mbox{for}\, \, 1 \leq i, j \leq n. 
 \]
Also  if  $  B (w, 2 r ) \subset  O, $   then for      some  $ c = c(p, n, \al, \La ) \geq 1, $     
\begin{align}
 \label{eqn10.43} 
 c^{-1}   |\xi |^2   |x-w|^{ \frac{(p-2) (1-n)}{p-1}  }   \leq   \sum_{i,j=1}^n    b_{ij} (x)   \xi_i \xi_j     \leq    c   |\xi |^2          |x-w|^{ \frac{(p-2) (1-n)}{p-1} } 
 \end{align}
   whenever $  x \in   B ( w, r) \sem \{w\}$.  Comparing boundary values of  $ \mathcal{G},  F  $  we observe from the  maximum principle for  $ \mathcal{A}$-harmonic functions  and  elliptic  regularity theory that it suffices to  prove  \eqref{eqn10.36} $ (e) $  when  $   O'  =  B ( w,  r ).  $     To  this  end,  let  
\[  
m (s)  =   \min_{\ar B  ( w, s ) }  \ze  \quad  \mbox{and}  \quad      M (s)  =   \max_{\ar B  ( w, s ) } \ze    \quad \mbox{when}\quad  0 <  s  \leq r.    
\]         
Let   
\[
  \xi   =   \liminf_{s \to  0} m(s) \quad  \mbox{and}  \quad  \be   =   \limsup_{s \to  0} M(s). 
  \]   
  We  claim  that          
\begin{align}   
\label{eqn10.44}    
m ( r  )  \leq \xi    =  \be   \leq  M (  r  ). 
\end{align} 
     To  establish  \eqref{eqn10.44},   first suppose  $ \xi    >  M(r). $      In  this  case,  given  $  0   <   N  <  \xi   -   M(r)   $,    we  let   
\[
\he(x)  =
\left\{
\begin{array}{ll}
\min[\max ( \ze(x)  - M(r) ,  0 ), N] & \mbox{when} \, \, x \in B(w,r),\\
0 &\mbox{elsewhere in}\, \, O.
\end{array}
\right.
\]     
Then  $  \he  = N$  in a neighborhood of  $ w $  and  vanishes outside of  $  B (w, r ) $  so  approximating  $ \he  $  by smooth functions which are constant in a  ball about $ w $  and taking a limit we see that  $  \he  $  can be  used as a test function in  \eqref{eqn10.36} $(d)$ for  both $  \mathcal{G} $ and   $   F.  $ Doing this  and  using the structure assumptions on $ f $    in  Theorem  \ref{theorem1.4}  it follows that 
\begin{align}
 \label{eqn10.45} 
 \begin{split}    
 c'  \int_{\{ M( r  )  +  N   <   \ze  \} }   ( |\nabla \mathcal{G}  |  +  |\nabla  F  |)^{p-2}     | \nabla  \ze  |^2   dx  \, & \leq 
      \int_O   \lan \nabla f (\nabla F  )  -   \nabla f (\nabla \mathcal{G} ),  
      \nabla \he \ran \, dx     \\
      & = 0. 
      \end{split}
      \end{align}  
 From  \eqref{eqn10.45} we  see that   $  \ze  \leq  M(r) +N   $   almost everywhere in  
 $  B ( w, r ) $  which contradicts our assumption that $ \xi    > M ( r).  $  Thus  
 $ \xi    \leq  M ( r ).  $   Next  choose    a  decreasing   sequence  $ \{r_l\}_{l\geq 1} $  with  $ r_1  = r/2 $    and  $  \lim_{l \to \infty}  m (r_l )   = \xi.  $   Applying the  minimum  principle  for  $ \mathcal{A}$-harmonic functions  in   $  B (w, r_k ) \sem  B  ( w, r_l )$ for $l > k $  and  letting  $  l  \to \infty $  we see that  
 \[
 \ze   \geq  \min ( m ( r_k ), \xi   )  =: \xi_k\quad   \mbox{in}\,\,     B ( w,  r_k ).
 \]
  Now using  Harnack's   inequality  in  balls  $  B  ( y,    s/2) $  whenever  $  y \in  \ar B ( w,  s )$ and $0  <  s  <  r_k/2  $  we deduce that    
 \[  
 M (s)  -  \xi_k   \leq     c'  ( m (s)  - \xi_k) . 
 \]
 Applying this inequality with  $  s  =  r_l, $  when  $  r_l  <  r_k/2 $  and letting 
 first  $  l  \to \infty  $   and then $ k  \to  \infty $   we find that  
 \[  
 \liminf_{ s \to 0}  M ( s )  =   \xi.  
 \]  
 Now  applying the maximum principle once again in a certain  sequence of  shells with inner  radius tending   to  zero  we  conclude that  $ \xi   = \be. $  
 Finally, if    $ \xi   <  m (r),  $   let  $  0   <  N   < m (r) -  \xi  $   and  set   
 \[
   \he(x)  =
\left\{
   \begin{array}{ll}
   \min  [ \max  ( m(r)  -   \ze(x),   0  ),  N   ] & \mbox{whenever} \, \,  x\in  B (w, r ), \\
   0 &\mbox{otherwise in} \, \, O.
\end{array}  
\right.
   \]
  Then  $  \he \equiv  N   $  in  a neighborhood  of  $ w $  since $ \xi   = \be. $  Arguing as in the case $ \xi    >  M ( r ), $  we  arrive at a  contradiction.  Thus  \eqref{eqn10.44}  is  valid.  
  Note from \eqref{eqn10.44} and arbitrariness of  $ r $  with  $  B (w, 2r)  \subset  O $ that 
  $ M (\cdot ) $ is  increasing   and   $ m (  \cdot ) $   decreasing  on    $  (0, r_0 )$  if   $  B  ( w, 2 r_0  )  \subset  O.  $    Using  this fact  and  arguing  as in the derivation of  \eqref{eqn10.34} we get   for some  $\de   \in  (0, 1) $  depending only on the data  that
\begin{align}
 \label{eqn10.46}   
  M(t)  - m(t)   \leq   ( t/s)^{ \de}  ( M (s)  -  m (s)  )  \quad \mbox{for}\quad    0 <  t  \leq  s    \leq  r_0. 
  \end{align}   
  It follows from  \eqref{eqn10.46},  \eqref{eqn10.42}-\eqref{eqn10.43},  and elliptic regularity theory that  
  $  \ze $  is  H{\"o}lder continuous in  $   B (w, r_0). $    This completes the proof of 
    \eqref{eqn10.36} $ (e).  $  

To prove  \eqref{eqn10.36} $(f)$  we note  from 
    \eqref{eqn10.36} $(e)$ that 
  \[ 
  0  < \ze (x)  \leq  \max_{ B ( w,  d ( w, \ar O) )}  F \quad \mbox{whenever} \quad x    \in  B ( w,  d ( w, \ar O) ). 
  \]  
  This note,  \eqref{eqn10.41}, and Lemma \ref{lemma9.0}  with  $  \hat u_1  = F$,    $\hat u_2  = \mathcal{G}$,  and $ O  = O \sem \{w\}, $ 
  imply    the existence of  $ c^* \geq 1 $  such that 
\begin{align}
   \label{eqn10.47} 
    |  \nabla  \mathcal{G} ( x )  |    \approx      \lan {\ts \frac{w  - x}{|w-x|}},  \nabla  \mathcal{G}( x ) \ran  
\approx  | x  - w |^{(1-n)/(p-1)}  \approx  \frac{\mathcal{G}(x)}{| x - w  |}  
\end{align} 
in $B ( w,  d ( w, \ar O )/ c^* )$  where  $ c^* $  and the  constants in the ratio all depend only on 
  $ p, n, \al, \La. $   From    \eqref{eqn10.47}   and    \eqref{eqn2.3}   it now follows that  
\begin{align} 
\label{eqn10.48}
   |\nabla b_{ij} (x) |  \leq  c   |x-w|^{ \frac{ - 1  -  n(p-2) }{p-1}}    
  \end{align}
whenever $x \in   B ( w,  d( w, \ar O )/ c^* )$  where $ (b_{ij} )$ are as in      
\eqref{eqn10.42}.   Finally,  \eqref{eqn10.48},   Lemma   \ref{lemma9.1}  $(e)$,  and elliptic regularity theory  imply  Lemma   \ref{lemma9.1}   $(f)$.          
     \end{proof}    Next we prove   

 \begin{lemma} 
 \label{lemma9.2}
  Let  $  D  $  be a starlike Lipschitz domain with  center $ z  $  and  let  $ \mathcal{G} $  be the  $ \mathcal{A}$-harmonic Green's function for 
 $ D $  with pole at  $ z.  $    if   
 \[
   d^* ( x )  =   \min \{d ( x, \ar D), | x  -   z | \}
   \]
     then  
  there exists  $ c    \geq 1 $  depending only on the data such that 
\begin{align} 
\label{eqn9.14}  
\begin{split}
&(\al) \hs{.2in}     
     0 < | \nabla \mathcal{G}  ( x ) |    \, \leq \, c \,    \lan  \frac{ z -
x}{ | z - x |} \,   ,  \, \nabla \mathcal{G}  ( x ) \ran    
    \quad  \mbox{whenever} \,\,  x \in   D \sem \{z\}.    \\
&  (\be) \hs{.2in}  
c^{-1}  \,    \frac{\mathcal{G}  ( x )}{d^* ( x, \ar   D ) }\leq    \, 
|   \nabla \mathcal{G} ( x ) |  \,  \leq  \, c  \frac{\mathcal{G}  ( x )}{d^* ( x, \ar
 D )} \quad \mbox{for}\,\,   x \in  \bar D  \sem \{z\}.
\end{split}
 \end{align}  
\end{lemma} 
 \begin{proof} 
\noindent  Since $  \mathcal{A}$-harmonic functions are invariant under
translation and dilation  and \eqref{eqn9.14} is also invariant under translation and dilation,  we assume, as we may, that 
\begin{align*}
z = 0 \quad  \mbox{and} \quad \mbox{diam}(D)= 1.
\end{align*}  
 Let  $  F, \mathcal{G}$  be as in   Lemma  \ref{lemma9.1}  with  $  w  = 0,  O = D. $   
Using \eqref{eqn10.47}, starlikeness   of $ \ar D, $  the maximum principle for $  \mathcal{A}$-harmonic functions, 
and  comparing boundary values we see for some  $ \ti  c  \geq  1 $   and $  \ga  >  1 $   near 1, that 
\[ 
\frac{\mathcal{G} (  x ) - \mathcal{G} ( \ga x )}{\ga  -  1}  \geq   \frac{\mathcal{G}(x)}{\ti{c}} \quad \mbox{whenever} \quad x  \in D \sem\{0\}  
\]
where $ \ti c $   depends  only on  the data. Letting  $ \ga \to 1 $  and using  Lemma  \ref{lemma2.2}  we obtain 
\begin{align}
\label{eqn9.16}   
- \ti c \lan \nabla \mathcal{G}  ( x ),  x \ran  \geq  \mathcal{G}  ( x )  \quad \mbox{when}\quad    x  \in   D   \sem \{0 \}.
  \end{align}
      Let    
      \[
        \mathcal{P} ( x )  = -  \lan  \nabla \mathcal{G}  (x), x   \ran \quad \mbox{whenever} \quad   x \in D \sem \{0\}.
        \]   
        From   \eqref{eqn9.16}, \eqref{eqn2.2},   and the same argument as in   \eqref{eqn8.83}
we deduce that  
$  \ph =  \mathcal{G} _{x_i},  1 \leq i \leq n, $ or  $  \ph  =  \mathcal{P}$   
      are weak solutions in    $ D  \sem \{ 0  \} $  to  
\begin{align} 
\label{eqn9.17}   
\sum_{i, j = 1}^n   \frac{\ar}{ \ar x_i }  ( \hat  b_{ij} \ph_{x_j} )  = 0     
\end{align}
   where  
  \begin{align}  
  \label{eqn9.18}   
  \hat b_{ij} ( x )  =   f_{\eta_i \eta_j} ( \nabla \mathcal{G} (x) ) \quad \mbox{whenever} \quad  x \in   D \sem \{0\}.  
\end{align}
      
We  temporarily assume that 
\begin{align}  
\label{eqn9.19}  
\mathcal{R} \in C^\infty ( \rn{n} )  
\end{align}
where $\mathcal{R}$ is as in  Definition \ref{defn8.3}.   
 Then  as in  \eqref{eqn8.94} we deduce that 
 $ \mathcal{P} $ and  $ \mathcal{G}_{x_i}, 1 \leq  j  \leq n,   $  have   continuous  extensions to 
 $ \bar D \sem \{0\}. $   We  also  have $ d ( z, \ar D) \approx  \mbox{diam}(D) $   where constants in the ratio depend only on the starlike Lipschitz constant for $ D. $   
   Using  \eqref{eqn10.47},    Lipschitz   starlikeness of  $  D,  $  and  
 \eqref{eqn9.16}  we find for some  $  \breve{c} \geq 1 $   depending only on the data that   
  \[      
  \breve{c}\,  \mathcal{P} (x) \, \geq \pm   \mathcal{G}_{x_i}  (x) \quad \mbox{on}\quad  \ar  D   \cup  B ( 0,  1/\breve{c} ) \sem \{0\}  
 \]   
  when $ 1 \leq i  \leq n.  $   From  this inequality  and the boundary maximum principle for the  PDE in   \eqref{eqn9.17}, 
  we  conclude that   \eqref{eqn9.14} $(\al)$ is valid when  \eqref{eqn9.19} holds with constants  depending only on the data.    To prove    \eqref{eqn9.14} $(\be)$ note  
  from  \eqref{eqn10.47}  that this inequality is  valid  in  $  B ( 0,  {\ts \frac{1}{2}} d ( 0, \ar D ) )  \sem \{0\}. $   Also from   Lemma  \ref{eqn2.2} $(\hat a)$  we deduce that the right-hand inequality in \eqref{eqn9.14} $(\be)$ holds  when  $ x  \in  D \sem B ( 0,  {\ts \frac{1}{2}} d ( 0, \ar D ) ) . $     Thus we prove only the left-hand inequality in 
  \eqref{eqn9.14} $(\be)$.    To do this we first use \eqref{eqn9.14} $(\al)$  and  \eqref{eqn9.17}, \eqref{eqn9.18} for $ \mathcal{P},     $   once   again,  to  deduce  that    Moser iteration can be  applied to powers of $ \mathcal{P} $ in order to obtain,  
\begin{align} 
\label{eqn9.20}   
\max_{B ( w, s ) }  \mathcal{P}    \leq  c \min_{B ( w, s ) } \mathcal{P}  \quad \mbox{whenever}\quad   B ( w, 2 s ) \subset  D \sem \{0\}.
\end{align}
      
     If   $ x  \in  D \sem B ( 0,  {\ts \frac{1}{2}} d ( 0, \ar D ) ) , $ 
                we   draw a ray $ l $  from 0  through  $ x $  to  a
point   in $  \ar D.$  Let $ y  $ be the first point on  $ l $  (starting from $x$)
with    
$  \mathcal{G} ( y ) =  \mathcal{G} ( x )/2. $  Then from the mean value theorem of
elementary calculus there exists $ \hat w  $ on the part of $ l $
between $ x, y $ with 
\begin{align} 
\label{eqn9.21}   
\mathcal{G}( x )/2   =   \mathcal{G}( x ) -  \mathcal{G} ( y )   \, 
\leq  \, | \nabla  \mathcal{G}( \hat w ) | \,  | y -  x |.
\end{align}
From  \eqref{eqn8.16}  with $   v  =  \mathcal{G} ,  r =  2  d  ( x,  \ar D),    $   and  $  x = a_{2r} ( w)$, we deduce the existence of $ c \geq 1 $ depending only on the data with 
\begin{align}  
\label{eqn9.22}   
y, \hat w \in B [ x,  (  1 - c^{ - 1} )  d ( x, \ar \hat D ) ]. 
\end{align}  
      Using \eqref{eqn9.22},    the  Harnack inequality in \eqref{eqn9.20},  and    \eqref{eqn9.14} $(\al),$   it follows 
 for some $ c', $  depending only on  the data, that   
 \[    
 | \nabla \mathcal{G}  ( \hat w  ) | \, \leq \, c' \, | \nabla\mathcal{G} ( x ) |
 \]   
   and thereupon  from    \eqref{eqn9.21}  that  
 \[   
 \mathcal{G} ( x ) \, \leq   \, c \, | \nabla \mathcal{G} (x  )  | \,  d ( x,  \ar  \hat  D   ).  
\]
   Thus  the left-hand inequality   in  \eqref{eqn9.14} $(\be)$ is valid when  $ x  \in  D \sem B ( 0,  {\ts \frac{1}{2}} d ( 0, \ar D ) )  $ for $ c $ suitably large   and the proof of   Lemma  \ref{lemma9.2}  is  complete under assumption   \eqref{eqn9.19}. 
    
    To complete the proof of   Lemma  \ref{lemma9.2}  
we show that  \eqref{eqn9.19}  is unnecessary.  For this purpose let $ \mathcal{R}_m  \in C^\infty 
( \rn{n}  )   $ for $ m = 1, 2, \dots, $ with   
\begin{align}  
\label{eqn9.22a}  
 \|  \log   \mathcal{R}_m   \hat \|_{\mathbb{S}^{n-1}}   \leq c    \|  \log   \mathcal{R} \hat \|_{\mathbb{S}^{n-1}}
\end{align}
 and 
$  \mathcal{R}_m \to  \mathcal{R} $ as $ m \to \infty $ uniformly on $  \mathbb{S}^{n-1}$.   Here $ c $ depends only on $n.$ 
  Let $  D_m,  \mathcal{G}_m $ be the  corresponding  starlike  Lipschitz domain 
and $  \mathcal{A}$-harmonic  Green's function for $ D_m $ with pole at 0.  Applying  Lemma \ref{lemma9.2} to  $ \mathcal{G}_m $,  using   Lemmas  \ref{lemma2.2}, \ref{lemma8.4}, and arguing as in the proof of  \eqref{eqn10.36} $(d)$ we see that 
\begin{align*}
\{\mathcal{G}_m, \nabla \mathcal{G}_m\}& \, \, \mbox{converge to}\, \, \{\mathcal{G}, \nabla \mathcal{G}\} \\
&\mbox{ uniformly on compact subsets of}\, \,  D \sem \{0\}.
\end{align*}
 Since the constants in this lemma 
 are independent of 
 $ m $    we conclude  upon taking limits  that Lemma \ref{lemma9.2} also holds for $ \mathcal{G} $  without hypothesis  \eqref{eqn9.19}.   The proof of Lemma  \ref{lemma9.2}  is now complete. 
\end{proof}    
Before proceeding further we note the following consequences of  Lemma  \ref{lemma9.2}.     
\begin{corollary}  
\label{corollary9.3}  
Let    $   D_1,   D_2  $  be  starlike Lipschitz domains with center at  $ z $  and  let  $   \mathcal{G}_1,  \mathcal{G}_2 $  be the  corresponding $  \mathcal{A}$-harmonic  Green's  functions with pole at  $ z. $  Suppose      $  w \in \ar D_1  \cap  \ar D_2, $        $  0   <    r   \leq  |w  -  z|/100,  $  and     \[   D_1  \cap  B  ( w,  4r  )  =   D_2  \cap  B ( w, 4r).   \]     Then  Lemma    \ref{lemma8.13}   is valid  with $ \breve v_i  =  \mathcal{G}_i$ for  $i = 1, 2,   $  without  assumption   \eqref{eqn8.94}.    Moreover,  $ c_{\star}$   in this   lemma  and so also  constants   depend only on the data. 
\end{corollary}   
\begin{proof}    
From Lemma   \ref{lemma9.2}  we see that     
Lemmas  \ref{lemma8.12},   \ref{lemma8.13}   are valid  with $ \breve v_i  =  \mathcal{G}_i ,  i = 1, 2,   $ under   assumption   \eqref{eqn8.94}.   Also from  Lemma  \ref{lemma9.2}  we deduce  that $ c_{\star}$  in these lemmas for $ \mathcal{G}_1, \mathcal{G}_2, $ depends only on the data.   To show assumption   \eqref{eqn8.94}  is unnecessary  in   Lemma \ref{lemma8.13} let   $   \mathcal{R}_1,  \mathcal{R}_2, $  
be the  graph  functions  for  $ D_1,  D_2 $  and let    $  \mathcal{R}_{1, m},  \mathcal{R}_{2, m},  
m = 1, 2, \dots $ be  approximating  graph functions to  $ \mathcal{R}_1,  \mathcal{R}_2 $  satisfying    \eqref{eqn9.22a}  with   $ \mathcal{R}_i $   replacing  $  \mathcal{R} $  for $ i  = 1, 2. $  Also   $ \mathcal{R}_{i,m}     \to   \mathcal{R}_i, $  as $ m \to \infty, $   uniformly  on  $ \mathbb{S}^{n-1}. $   Finally  we choose this sequence so that      
\[  
 \mathcal{R}_{1, m }    =  \mathcal{R}_{2, m } \quad  \mbox{on} \quad   \{ \om  \in  \mathbb{S}^{n-1} : 
\mathcal{R}_1 ( \om )  =  \mathcal{R}_{2 } ( \om )  \in B ( w, 4r) \}.  
 \]
Let  $  D_{i, m } $  be the  corresponding starlike Lipschitz  domains with center at  $ z $  and  let  $  
\mathcal{G}_{i,m}  $  be  the  $ \mathcal{A}$-harmonic  Green's  functions  for  $  D_{i,m} $  with pole  at  $ z$  for $  i  = 1, 2.  $       Applying  Lemma  \ref{lemma8.13}  to   $  \mathcal{G}_{1, m},  \mathcal{G}_{2, m}$  in  $ D_{1,m} $  we see that 
constants in   \eqref{eqn10.24}  depend  only on the data. Using   Lemma \ref{lemma9.1}   and  taking limits  as $ m \rar \infty,  $    we  get    \eqref{eqn10.24}  for   $ \mathcal{G}_1, \mathcal{G}_2. $ 
\end{proof}    

Next we prove,    
\begin{lemma} 
\label{lemma9.4}  
Let $ D $  be a  starlike  Lipschitz domain  with  center  $z,  w   \in \ar  D, $  and  
$ 0   <   r     \leq  |w-z|/100.  $    
Given $ p, 1 <   p <n$,  suppose
that $ \ti u,  \ti v   $ are positive $ \mathcal{A}$-harmonic functions in $ D\cap B ( w, 4r )$ and  that 
$ \ti u,  \ti v $ are continuous in  $ B ( w, 4r ) \sem D $, with 
 $ \ti u,  \ti v = 0 $ on $ B (w, 4r ) \sem D $.  Then there exists $  \ti c_1, 1 \leq \ti c_1 < \infty, $ depending only
on the  data  such that if $ r_1  = r/ \ti c_1,  $  then 
\[ 
 \frac{ \ti u ( y ) }{ \ti v ( y ) } \, \leq \, \ti  c_1  \frac{ \ti u (a_{r_1}(w) )}{\ti v (a_{r_1}(w) ) }    \quad  \mbox{whenever}\,\,  
 y \in  D \cap B (w, r_1 ).  
 \]
 \end{lemma} 
\begin{proof}                           
Let  $  \ti w    $   denote  the  point on  the  ray  from  $ z $  to  $ w $  with  $ | w  -  \ti w |=   \ti  r
< < r . $ To  prove  Lemma   \ref{lemma9.4},  we  assume as we may that $  \ti u (  \ti w )  = 1   
 =  \ti v (  \ti w  ),  $   since  $ \mathcal{A}$-harmonic functions are invariant under multiplication by positive constants. Also  from  \eqref{eqn8.16} we  see that  $  \ti c  \geq 1 $ can be chosen, depending only on the  data  so  that if   
  $ \ti r = r/ \ti c, $   then    
\begin{align}
 \label{eqn9.23a}      \max_{B ( w, \ti r )}   \ti u  \approx   \ti u  (  \ti  w ) = 1 \quad \mbox{ and } \quad   \max_{B ( w, \ti r )}   \ti v  \approx   \ti v  (  \ti w ) = 1. 
 \end{align}    
 Second  let  $ \ti r' = \ti r/c $   and  let   $ D_1 $  denote  the interior  of  the  domain  obtained  from  drawing  all  line  segments from points  in $ \ar  D  \cap   \bar B ( w, \ti r' ) $   to   $  \bar B (\ti   w,  \ti r'  ). $ 
If  $ c  > 10000  $ is large enough  (depending on the data),  and   $ 
\ti  r'   = \ti r/ c ,$   we  deduce as in   \eqref{eqn8.44}  that  $ D_1   $ is  starlike   Lipschitz  with  
center at $ \ti  w . $   Also  the       starlike  Lipschitz  constant  for  $ D_1 $   can  be estimated  in terms of the starlike Lipschitz constant for  $ D $ as in  Lemma    \ref{lemma8.7}.  Finally  there  exists $ \hat c > > c  , $   depending only on  the  data  such  that  if $  \hat r  = r/  \hat c  <   <  \ti r' , $  then   
\begin{align}
    \label{eqn9.23b}    
    D \cap  B ( w,  32  \hat r )   =    D_1  \cap  B ( w,  32 \hat r ).   
    \end{align}
 Let $ \mathcal{G}_1 $ be the $ \mathcal{A}$-harmonic Green's  function for $ D_1 $ with pole  at  $  \ti w. $  
  Using  Harnack's inequality,   the  maximum principle for $  \mathcal{A}$-harmonic functions,    the fact that  $ D_1   \subset  D, $  \eqref{eqn10.36} $(e),$ and \eqref{eqn10.41}   we  obtain that     
\begin{align}  
\label{eqn9.23} 
 c \min (  \ti u,  \ti v )   \geq  r^{(n-p)/(p-1)}  \mathcal{G}_1   \quad \mbox{in} \quad D_1 \sem  B ( \ti w,  
\ti r' / 4 ).
 \end{align} 
Let $ \mathcal{R},   \mathcal{R}_1,  $  denote the graph functions for  $ D,  D_1 $ and  set   
\[ 
\bea{l}  
 K_i := \left\{\om  =   \frac{ y - \ti w }{| y - \ti w |} : y \in B ( w, 2^{ 4- i}  
\hat r   )  \, \right\} \quad \mbox{for}\quad  i = 0, 1, 2,  
\\
L: = {\ds \sup_{K_0} \mathcal{R}_1}.    
\ea 
\]
   Choosing   $ \hat  c $  still larger if  necessary but with the  same dependence we  see that in addition to   \eqref{eqn9.23b}   we  may  also  assume  that    
\begin{align}
 \label{eqn9.24a}  \{  \mathcal{R}_1 ( \om ) \,   \om   :   \om \in K_0   \}    \subset   
 \ar D  \cap  \ar  D_1  \,.
 \end{align}
 From   our  construction  we  deduce  that  there exists   
  $ c_- $ (depending
on $ p, n, $ and the Lipschitz constant for $ \mathcal{R}_1 $)   with   
\begin{align}
\label{eqn9.24} 
 \min  \{  d ( K_2,  \mathbb{S}^{n-1} \sem K_1 ),  
   d ( K_1,  \mathbb{S}^{n-1} \sem K_0 ) \} \geq c_-^{ - 1} 
   \end{align}
Let  $ 0 \leq  \vartheta  \leq 1$ with $\vartheta  \in C_0^\infty  ( \rn{n} ), $  and $ \vartheta  \equiv 1 $ on
$ K_2 $  while  $ \vartheta \equiv 0 $ on $  \mathbb{S}^{n-1}  \sem K_1. $  
 Moreover,  thanks to  \eqref{eqn9.24}, we can choose $ \vartheta$ so that 
\begin{align}
 \label{eqn9.25} 
 | \nabla \vartheta | \, \leq  \, \bar c_-^{ -1 }  \mbox{ where  $\bar  c_-$ has the same dependence as  $ c_-$ . }   \,  
 \end{align} 
  Let  
  \[  
  \log \mathcal{R}_2 ( \om ) := 
  \left\{ 
  \bea{ll}  
  \vartheta \, \log \mathcal{R}_1 
\, + ( 1 - \vartheta ) \, \log (2 L) & \mbox{when} \quad \om \in K_0, \\
\log(2L) & \mbox{when}\quad  \om \in  \mathbb{S}^{n-1} \sem K_0. 
\ea 
\right.  
\]
Using \eqref{eqn9.25}, it is easily shown that 
\[   
 \| \log \mathcal{R}_2  \hat \|_{\mathbb{S}^{n-1}} 
 \, \leq \, c \, ( \|  \log\mathcal{R}  \hat \|_{\ar K_0}   +  1 ). 
 \]   
Let $  D_2  $ be the starlike Lipschitz domain with  center at $
 \ti w  $ and graph function $\mathcal{R}_2. $  Also let $ \mathcal{G}_2 $ be the $ \mathcal{A}$-harmonic Green's  function  for  $  D_2  $   with pole at  $ \ti w. $   Then from our construction,  the fact that $ L  \approx   
\ti  r, $ 
      Harnack's inequality,   and once again 
   \eqref{eqn10.36} $(e),$ \eqref{eqn10.41},    we deduce  first  
          that    
          \[
     c \, r^{(n-p)/(p-1) }  \mathcal{G}_2 \geq    1 \quad \mbox{on}\quad  \{ \ti w  + t   \mathcal{R}_1 ( \om ) \,  \om  : \om \in  K_0 \sem K_1,  0  \leq  t   \leq 1. \}            
     \]  
      and  second  from  \eqref{eqn9.23a},  \eqref{eqn9.24a}, and  the  maximum  principle  for $ \mathcal{A}$-harmonic  functions   that   
\begin{align}
  \label{eqn9.25a}     \max(u, v )  \leq c \, r^{(n-p)/(p-1) }  \mathcal{G}_2   \mbox{ in }  
 \{ \ti w  + t   \mathcal{R}_1 ( \om ) \,  \om  : \om \in  K_0 ,  0  <   t   \leq 1. \} 
 \end{align}
Now  from    \eqref{eqn9.24a},  \eqref{eqn9.23b},    and the  definition of  $ \mathcal{R}_2, $   it  follows that     
\begin{align}
 \label{eqn9.25b}   D\cap B ( w, 4 \hat r) = D_1  \cap  B ( w,  4 \hat r )  = D_2 \cap  B ( w, 4 \hat r )  .
 \end{align}
  Using  this display,  \eqref{eqn9.25a},     
    Corollary  \ref{corollary9.3}  with   $  r $  replaced by $  \hat  r, $     and  once again Harnack's  inequality  we deduce the  validity  of  Lemma   \ref{lemma9.4}    since    $ r_1 $  can be chosen so  that    
    \[ 
 \frac{ \mathcal{G}_1 ( a_{r_1} (w)) }{ \mathcal{G}_2 ( a_{r_1} (w))} \, \approx 1    \approx  
   \ti u (a_{r_1}(w) )   \approx \ti v (a_{r_1}(w) ).  
\]     \end{proof} 

  In order to finish the   proof  of   our  boundary Harnack  inequalities      we  need    a  lemma whose proof  requires  only     Lemmas  \ref{lemma2.1}-\ref{lemma2.2}. 
     \begin{lemma} 
         \label{lemma9.0} 
Let $O\subset\mathbb{R}^n$ be  an open set,           $ p $ fixed, 
 $1 <   p <n$ and  $\mathcal{A} \in M_p(\alpha)$. 
Also, suppose  that
 $ \hat v_1, \hat v_2 $ are non-negative $\mathcal{A}$-harmonic  functions in $O$. 

 Let  $ \tilde a \geq 1$ , $y \in O$,  $ \eta \in \mathbb{S}^{n-1}$,  and assume that
 \[
\frac 1 {\tilde a} \frac { \hat v_1( y)}{d(y,\partial O )}\leq \lan \nabla \hat v_1 ( y ), \eta \ran  \leq  |\nabla \hat v_1(y)|\leq \tilde a \frac {\hat
v_1(y)}{d(y,\partial O )}.
\]
Let  $ \ti \ep^{- 1 }  =  (c\tilde a)^{(1 + \ti \he )/\ti \he }$ where $ \ti \he  $ is as in
\eqref{eqn2.3}  of  Lemma \ref{lemma2.2}.  If
\[  
( 1 - \ti \ep) \hat L \leq \frac{ \hat v_2  }{\hat v_1}
\leq ( 1 +  \ti \ep ) \hat L \quad \mbox{ in }\,\,  B ( y, {\ts \frac{1}{100}} d(y,\partial O )) 
\]
for some $ \hat L, 0 < \hat L  < \infty,$ then
for $ c  = c(p,n,\alpha)  $  suitably large, 
\begin{align*}
\frac{1}{c\, \tilde a}
   \,   \frac {\hat v_2(  y )}{d(y,\partial O )}
\, \leq  \lan \nabla \hat v_2( y ), \eta \ran \, \leq 
| \nabla \hat v_2(  y) |
 \leq   \,  c\,  \tilde a      \frac {\hat v_2( y )}{d(y,\partial O )}.
\end{align*}
\end{lemma}

For the  proof  of similar  lemmas  see Lemma 3.18 in \cite{LLN}  and   Lemma 5.4     in   \cite{LN1}.   

Using    Lemma \ref{lemma9.0}   and  Lemma \ref{lemma9.4}          we   prove    
\begin{lemma} 
\label{lemma9.5a}  
Let $ D,  D_1, \mathcal{G}_1,     \ti u,  \ti v,  w, \ti w,  \ti r,  \ti r',  \hat r,       c_1,  r_1   $  be as in   Lemma   \ref{lemma9.4}.    Then there exists $  c_2, 1 \leq c_1 < c_2  <\infty,  $  and $   \he   \in  (0,1) , $ depending only
on the  data , such that if  $ r_2 = r/c_2, $  then    
\[  
\left| \frac{ \ti u ( y ) }{ \ti v ( y ) } -  \frac{ \ti u (x ) }{ \ti v ( x ) } \, \right|  \leq \,  c_2            \left( \frac{| x  - y |}{ r} \right)^{ \he}   \frac{ \ti u ( y) }{\ti v (y)  }          
\]
whenever   $  x, y \in  D \cap B (w, r_2)$.    
   \end{lemma}

   \begin{proof}    
   To prove  Lemma     \ref{lemma9.5a}  we may assume, as is easily shown using  Lemma 
\ref{lemma9.4}.  that
\begin{align}   
\label{eqn9.26}   
 \ti u (a_{r_1}(w))   =  \ti  v(a_{r_1} (w) )  = 1 \quad \mbox{and}\quad  \ti u   =   \frac{\mathcal{G}_1}{\mathcal{G}_1 (  a_{r_1} (w) )}    \mbox{ in }  D \cap B ( w, r_1) .    
 \end{align}  We also  temporarily assume  that  
\begin{align}
  \label{eqn9.26a} \mathcal{R}   \in  C^\infty ( \mathbb{S}^{n-1} ). 
  \end{align}  
 From  Lemma \ref{lemma9.4} we  see  that
  \begin{align} 
  \label{eqn9.27}  
  c_1^{-1}    \leq  \frac{ \ti  u(y)}{\ti  v(y)}   \leq  c_1  \quad \mbox{in}\quad  D \cap B ( w, r_1)     \end{align}  
   where $ c_1 \geq 1 $  depends only on the data.     Hence if   $ \bar  u  =  2 c_1  \ti  u, $   then   
  \begin{align}  
  \label{eqn9.28}  
  \ti  v   \leq  \bar u /2   \leq   c_1^2 \ti  v     \mbox{ in }  D \cap  B ( w,  r_1 ).  \end{align}
   Let   $   \{ u ( \cdot, t)\}$, $0 \leq  t \leq 1, $ be the sequence of    $\mathcal{A}$-harmonic functions  in   $  D \cap B  ( w,  r_1/2 ) $  with continuous boundary values,
\begin{align*}  
u ( y, t ) = t  \bar u(y) + ( 1 - t ) \ti v  (y),  \quad \mbox{for} \, \, 0 \leq  t \leq 1.    
\end{align*}
 Existence  of   $ u (  \cdot, t ) ,  t   \in (0,1),  $  is proved in   
 \cite{HKM}.    Checking  boundary values  and  using
 the maximum principle for $\mathcal{A}$-harmonic functions,
    as well as   \eqref{eqn9.27}-\eqref{eqn9.28},    we find  for some $ \ti  c, $ depending only on the data,    that
\begin{align}   
\label{eqn9.30}  
\frac{u ( \cdot, t_1 )}{\ti c} \,    \leq \,   \frac{ u ( \cdot, t_2 ) - u (
\cdot, t_1) }{ t_2 - t_1}  \leq
\ti c\,   u( \cdot,  t_1 )  
\end{align} 
on $ D \cap  B ( w, r_1/2) $ whenever $ 0 \leq t_1 < t_2 \leq 1.$

 Let $ \ep_0=\ti \ep $ where $\ti \ep$ is as in
Lemma \ref{lemma9.0}.  From   \eqref{eqn9.30}  we find the existence of   $ \ep_0', 0 <
\ep_0' \leq \ep_0,  $ with the same dependence as $ \ep_0,$    such
that if  $ | t_2 - t_1 | \leq \ep_0', $ then
\begin{align*} 
1 - \ep_0/2  \leq \frac{u ( \cdot, t_2 )}{ u (\cdot, t_1 ) } \leq 1 + \ep_0/2  \quad \mbox{in}\quad D \cap B (w, r_1/2).
\end{align*}
Let $ \xi_1 = 0 <  \xi_2 <  \ldots < \xi_l = 1 $ and consider $[0,1]$ as divided into  $ \{ [ \xi_k,
\xi_{k+1} ] \}, 1 \leq k \leq l - 1$. We assume that all of these intervals have a length of
$ \ep_0'/2 $ with the possible exception of the interval containing $\xi_l=1$  which is of length
$  \leq  \ep_0'/2.$  Using    Lemma  \ref{lemma9.2} with  $  \mathcal{G} =  \mathcal{G}_1 , z = \ti w, $ the fact that    $ u ( \cdot, \xi_1 ) =  \bar u = 2 c_1  \ti  u,  $ 
\eqref{eqn9.26},  and   \eqref{eqn9.25b},  we see that    Lemma \ref{lemma9.0} can be applied with $ \hat v_1 =  u  ( \cdot, \xi_1) $ and  $ \hat
v_2 =u ( \cdot, t ), $  for $ 0 <  t  \leq  \xi_2. $     Doing this we   find, for some $ c_- \geq 1 $ depending only on   the data,
that
\begin{align}
\label{eqn9.32} 
c_-^{-1} \frac{u ( y, t )}{d(y,\ar D )}\leq |\nabla u ( y,t )|\leq    c_-  \lan \frac{\ti w - y }{|y-\ti w|},  \nabla    u ( y, t ) \ran  
\, \leq   c_-^2   \frac{ u ( y,  t )}{d(y, \ar D)},
\end{align} 
whenever $y\in D  \cap  B ( w,    r_1 / 4 )$.  
From    \eqref{eqn9.32}  and  \eqref{eqn9.30}     we  see  as in  \eqref{eqn10.42}    that if    for fixed $ t \in  [0, \xi_2] $ we define 
\[ U ( \cdot ,    \tau   )  =  \frac{  u  ( \cdot,  t )  -   u  ( \cdot, \tau) }{ t - \tau } \quad \mbox{on}\, \, D \cap B ( w, r_1/4)   
\]  
 for  $  0  \leq  \tau < t \leq \xi_2,  $  then  $   \mathcal{L}^{\tau}_* U ( \cdot ,  \tau  )  = 0 $  weakly  in  
$  D  \cap  B ( w, r_1/4) $  and   $ U  ( \cdot,  \tau ) $  has   continuous boundary value  0  on  $ \ar D \cap  B ( w, r_1/4), $ where   
\begin{align}
 \label{eqn9.32a}  
 \mathcal{L}^{\tau}_* U ( x, \tau )  = 
 \sum_{i, j  = 1}^n   \frac{  \ar }{\ar x_i}  \left( b^*_{ij} (x, \tau )  \frac{ \ar  U ( x, \tau ) }{\ar x_j}\right)  =  0 
\end{align}
 and  
\begin{align}
 \label{eqn9.32b}
 b^*_{ij} ( x , \tau )  =    \int_0^1  f_{ \eta_i \eta_j }  ( s   \nabla  u (x, t )   + ( 1 - s )  \nabla u ( x, \tau ) ) \, ds \quad \mbox{for}\, \, 1 \leq i, j \leq n. 
\end{align}
Also  for      some  $ c \geq 1, $ depending only on the data,      
\begin{align}
 \label{eqn9.32c} 
   \sum_{i,j=1}^n    b^*_{ij} (x, \tau)   \xi_i \xi_j     \approx    |\xi |^2    \, ( | \nabla u ( x, t ) | + 
   | \nabla u  ( x,  \tau ) |  )^{p-2}  
 \end{align}
   whenever $  x \in   D \cap B ( w, r_1/4) $.   Moreover   from \eqref{eqn9.30},   $ 0  \leq  U ( \cdot, \tau )   \approx u ( \cdot, \tau)   $  so    from      \eqref{eqn8.15}, \eqref{eqn8.16},   elliptic PDE theory, and  Ascoli's theorem   it  follows that  a subsequence of  $ U  ( \cdot, \tau )$  converges  uniformly  as $  \tau   \rar t $   to   $ V  ( \cdot ) $   with continuous boundary value 0  on   $ \ar D  \cap B ( w, r_1/8) $   satisfying  $ \mathcal{L}^t_*   V  ( \cdot ) = 0  $  weakly in  $ D  \cap B ( w, r_1/8). $   
   Since $ u ( \cdot, t  ) $  is also  a weak  solution  to  this  pde  we  can  now apply     \eqref{eqn10.3}    of  Lemma   \ref{lemma8.12}  with $  \breve v_1  =  u ( \cdot, t ),  r =  r_1/32 ,  $   $  h_2  =   u (   \cdot, t), $  
   $ h_1  = V ( \cdot ),  $   $ \hat x  =  \ti w ,$  to  conclude  for  some 
   $ c   \geq  100, \he  \in [0, 1],   $  depending  only on the data  that  if   $ s_1  = r_1 / c , $  and  $ t \in   [\xi_1, \xi_2].$  then 
 \begin{align}
  \label{eqn9.32d}   
\left| \frac{  V ( y  ) }{ u ( y, t ) } -  \frac{  V (x  ) }{ u ( x, t ) } \, \right|  \leq \,  c            \left( \frac{| x  - y |}{ r} \right)^{\he }   \frac{  V (y) }{u ( y,  t  )}  \leq \,  c^2            \left( \frac{| x  - y |}{ r} \right)^{\he }           
\end{align}
whenever   $  x, y \in  D \cap B (w, s_1).$ To  avoid  confusion we now write  $ V ( \cdot, t )$ for   $  V ( \cdot ). $ 

   From   \eqref{eqn9.30} we see for fixed $ x  \in  D \cap  B ( w,  r_1/2) $  
   that  $  t  \rar  u ( x, t )  $  is  Lipschitz   on  [0, 1].  Hence   
   $ u_t ( x, t)  $  exists for  almost  every  $ t  \in  [0, 1]  $   and is  absolutely continuous on  
   $ [0, 1].$    
   Choose  a countable dense sequence  $ ( x_{\nu})$  of  $ D \cap B ( w, r_1/2) $  and  a set 
   $ W \subset [0, 1] $  such  that  $ \mathcal{H}^1 ( [0, 1] \sem W ) = 0 $  and  
    $ u_t ( x_{\nu}, t ) $  exists for  all $ t \in W $  and $ \nu = 1, 2,   \dots  $    Then  clearly  
   \[   
   u_t ( x_{\nu}, t )  = V ( x_{\nu }, t)  \quad \mbox{for}\, \,  t  \in W \, \,   \mbox{and}\, \,  \nu = 1, 2, \dots 
    \] 
	   Using this  equality   in     \eqref{eqn9.32d}     we  deduce that
\begin{align}
 \label{eqn9.33a} 
\biggl|  \log \biggl( \frac{  u ( x_m,  \xi_2    ) }{  u ( x_m,
 \xi_1     ) }\biggr)   \,
 -   \log\biggl(\frac{  u ( x_k,  \xi_2  ) }{  u ( x_k , \xi_1  ) }\biggr) \biggr|\leq
\int\limits_{\xi_1}^{\xi_2}
\biggl|  \frac{  u_\tau  ( x_m, \tau  ) }{  u ( x_m, \tau  ) }   \,
 -    \frac{  u_\tau  ( x_k, \tau  ) }{  u ( x_k , \tau  ) } \biggr|
\, d \tau   \, \leq
\,    c
  \biggl(  \frac{ | x_m  - x_k |}{ r } \biggr)^\he 
  \end{align}
 whenever  $x_m,x_k\in (x_\nu)$ and $ x_m, x_k   \in  D \cap B (w,  s_1  )$.  As $(x_\nu)$ is a dense sequence
 in $ D \cap B (w, s_1 )$ it follows  from continuity that   \eqref{eqn9.33a}  is  valid with  $ x_m =  x,   x_k  = y,  $  whenever $  x, y  \in  B ( w,  s_1 ). $  This  validity  and \eqref{eqn9.28}   imply    for   $ x, y  \in D   \cap  B  ( w,  s_1 ). $   that   
\begin{align}
 \label{eqn9.33b}   
\left| \frac{  u ( y , \xi_2 ) }{ u ( y, \xi_1 ) } -  \frac{  u (x, \xi_2   ) }{ u ( x, \xi_1 ) } \, \right|  \leq \,  \,   c'            \left( \frac{| x  - y |}{ r} \right)^{\he }   \frac{ \ti u (y, \xi_2) }{u ( y,  \xi_1  )}          
\end{align}
Using  \eqref{eqn9.33b}   we see there exists    $  \ti c_1'     \geq 1, $  depending only on the data  so that  if  $   s_1'  =    s_1/ \ti c_1' , $  then   for  fixed  $  y  \in  D  \cap  B ( w, s_1' ), $     we   have 
\[  
 ( 1  -     \ep_0/4)   \frac{ u ( y, \xi_2 )}{ u ( y, \xi_1) }  \leq      \frac{ u ( x, \xi_2 )}{ u ( x, \xi_1) }  \leq     ( 1  +  \ep_0 /4 )       \frac{ u ( y, \xi_2 )}{ u ( y, \xi_1) }    \quad \mbox{ whenever }   x        \in   D \cap B  ( w,  s_1'  ) .   \]  
  If  $  t  \in  [\xi_2,  \xi_3] ,  $ it  follows  from our choice of  $  \ep_0' ,  s_1'  , $     that  if $    x        \in D \cap   B  ( w,  s_1'  ), $ then   
\[ 
 \frac{ u (x, t )}{u ( x,  \xi_1)}  =     \frac{ u (x, t)}{u ( x,  \xi_2)}   \cdot  \frac{ u (x, \xi_2 )}{u ( x,  \xi_1)}   \leq   ( 1 + \ep_0/2)( 1 + \ep_0/4) \frac{ u (y, \xi_2 )}{u ( y,  \xi_1)}              <   ( 1 + \ep_0) \frac{ u (y, \xi_2 )}{u ( y,  \xi_1)}  .
  \]    
   for   $ \ep_0 >  0 $  small enough. 
 Similarly   
 \[ (  1  -  \ep_0 )  \frac{ u (y, \xi_2 )}{u ( y,  \xi_1)}   \leq     
\frac{ u (x, t )}{u ( x,  \xi_1)}     \mbox{ whenever }       x        \in   D \cap B  ( w,  s_1'  ).      
\] 
From  these inequalities we  obtain  that  Lemma  \ref{lemma9.0}   can  be  applied  with  
$ \hat v_1  = u  ( \cdot,  \xi_1 ),  \hat v_2  =   u  (  \cdot,  t ) $   and  $ L  =   
\frac{ u (y, \xi_2 )}{u ( y,  \xi_1)} $   in  $ D   \cap   B ( w,  s_1' ). $  We get that  
  \eqref{eqn9.32}  holds  with  constants depending only on the data in   $ D \cap B  ( w,  s_1'/2 ) $  whenever   $  t  \in [ \xi_2, \xi_3] .   $ We then can argue  as  above  to eventually 
  conclude  that  \eqref{eqn9.33a},  \eqref{eqn9.33b} are valid  with  $  \xi_1,  \xi_2 $  
  replaced  by   $  \xi_2,   \xi_3 $   in  $   D   \cap  B (  w, s_2 ),   s_2  < < s_1' . $  
  
  We  now continue by induction, as in the proof of (4.24) - (4.27) in  Theorem 2 of \cite{LN1}.  We  eventually  obtain      (see \cite[Lemma 4.28]{LN1}) that
 \eqref{eqn9.32}   holds  for some constant depending  only on  the  data  in   
 $  D   \cap  B ( w, s_l) $ whenever   $  t  \in   [ \xi_{l-1},   \xi_l] .   $    Also $ r/s_l $   depends only  on  the data. Using  this  fact  and  arguing as above  we finally obtain for some $ c \geq 100 $ depending only on the data  that if   $ s_l' =   s_l/c, $   then  \eqref{eqn9.33a},  \eqref{eqn9.33b} are valid  with  $  \xi_1,  \xi_2 $  
  replaced  by   $  \xi_{l-1},   \xi_l $   in  $   D   \cap  B (  w, s'_{l}  ),    s_l' < s_l <  s'_{l-1} < 
  \dots  < s_1'   <  s_1.    $  
     Since $ l $ also depends only on the data it follows from the  $ l  $  inequalities  obtained and the triangle  inequality  that  Lemma   \ref{lemma9.5a} is true under  assumption 
       \eqref{eqn9.26a}.

    Finally, we show the assumption \eqref{eqn9.26a} is unnecessary. Let  $ \{
    \mathcal{R}_k\}_{k\geq 1}$  be  a  sequence of  $ C^{\infty} ( \mathbb{S}^{n-1} )$  functions 
with  
\[
 \mathcal{R}  \leq  \mathcal{R}_k \quad\mbox{and}\quad  \| \mathcal{R}_k \hat \|  \leq  c   \| \mathcal{R} \hat \|
 \]
 satisfying $\mathcal{R}_k  \to  \mathcal{R} $ uniformly on compact subsets of  
$ \mathbb{S}^{n-1} . $  Let  $  D_k,  k  = 1, 2,  \dots $  be   Lipschitz starlike domains with  center  $ z $  and  graph function  $\mathcal{R}_k. $  Extend  $  \ti u, \ti v $  to  continuous functions on $ B (w,4 r )  \cap D_k $ by  putting  $ \ti u = \ti v  = 0 $  in    
$   B (w,4 r )  \cap (D_k \sem D ). $

Let  $ \ti u_k,  \ti v_k $  be  $\mathcal{A}$-harmonic functions  in  $  D_k \cap  B ( w, 3r) $  with   $   \ti u_k  = \ti u ,  \ti v_k  = \ti v    $  on  $ \ar (D_k \cap  B ( w, 3r)). $ 
Using   Lemma  \ref{lemma2.2},  we see that  
\[
\{\ti u_k\}_{k\geq 1}\, \, \mbox{and}\, \, \{\ti v_k \}_{k\geq 1}\,\,  \mbox{converge uniformly  to}\, \,    \ti u\, \, \mbox{and} \, \, \ti v\, \,   \mbox{on}\, \,   B (w, 3 r ).
\]
 Moreover,    $ \{   \nabla \ti u_k\}_{k\geq 1}$ and $\{\nabla \ti v_k\}_{k\geq 1}$ converge  uniformly to $\nabla  \ti u$ and $\nabla \ti v$ on  compact subsets of $D \cap B  (w, 3 r)$. 
 Also, $ \ti u_k,  \ti v_k $  satisfy the hypotheses of  Lemma   \ref{lemma9.5a}  with $   3r,  D_k, $  replacing  $ 4r,  D. $  We  apply this lemma  to  $ \ti u_k, \ti v_k. $  Since  the  constants in this inequality   depend  only on the data, we can then take limits to get   Lemma 
\ref{lemma9.5a} for  $ \ti u,  \ti v . $   
\end{proof}  

   As  a  corollary to Lemma \ref{lemma9.5a}   we note that the proof  outlined above gives 
 \begin{corollary} 
 \label{corollary9.6a}   
   Let $ D,    p,   \mathcal{A},  \ti v,  w,  r,  r_2 $  be as in Lemma \ref{lemma9.5a}.  Then    
 \begin{align} 
  \label{eqn9.34}  
c^{-1} \frac{\ti v ( y )}{d(y, \ar D )}\leq |\nabla \ti v ( y )|\leq    c  \, \lan  \frac{ \ti w - y}{ | \ti w -   y | }  ,  \nabla    \ti v ( y
 ) \ran \leq   c^2   \frac{ \ti v ( y )}{d(y,\ar D)}  
\end{align}  
whenever $y  \in   D \cap B ( w,  r_2 )$.
\end{corollary}  
\begin{proof}  
As noted in the proof  of Lemma \ref{lemma9.5a}  ,
	an  induction type argument eventually gives  \eqref{eqn9.34} for $ u ( \cdot, \xi_l) = \ti v $ in the smooth  case.   Taking limits  as  previously,  we then get  Corollary \ref{corollary9.6a}  in general.  
\end{proof} 

 Finally in  this  section  we use  our  results on  starlike  Lipschitz domains to  prove 
   \begin{lemma} 
\label{lemma9.5}  
Let $ \hat  D $  be a   Lipschitz  domain      and suppose that   
\[
\hat D  \cap  B ( w, 4r)  = \{ y = ( y', y_n )  \in \rn{n} :  y_n > \ph ( y')   \} \cap  B  ( w, 4r)
    \]  where $ \ph $ is Lipschitz on $ \rn{n - 1},$  $w\in\partial \hat D,$ and $r>0$. Given $ p, 1 <   p <n$,  suppose
that $ \ti u,  \ti v   $ are positive $ \mathcal{A}  =  \nabla  f$-harmonic functions in $ \hat D\cap B ( w, 4r )$ and  that 
$ \ti u,  \ti v $ are continuous in  $ B ( w, 4r ) \sem \hat D $, with 
 $ \ti u,  \ti v = 0 $ on $ B (w, 4r ) \sem \hat D $.  Then there exists $  c_3,  c_4,   1 \leq c_3, c_4  < \infty, $ depending only on  $ p, n,  \al,  \La, $  and  $  \|  \ph  \hat \| $    such that if 
 $ r_3  = r/ c_3,  $  then 
\begin{align}
 \label{eqn9.34a}
\left| \frac{ \ti u ( y ) }{ \ti v ( y ) } -  \frac{ \ti u (x ) }{ \ti v ( x ) } \, \right|  \leq \,  c_3            \left( \frac{| x  - y |}{ r} \right)^{ \he}   \frac{ \ti u ( y) }{\ti v (y)  }          
\end{align}
and if  $ u^*  =  \ti u $  or $ \ti v,  $ then 
\begin{align}
 \label{eqn9.34b}  c_4^{-1} \frac{u^* ( y )}{d(y,\ar \hat D )}\leq |\nabla u^* ( y  )|\leq    c_4  \,  \lan e_n ,  \nabla    u^* ( y ) \ran  
\, \leq   c_4^2   \frac{ u^* ( y )}{d(y, \ar \hat D)} ,
\end{align}
whenever   $  x, y \in  D \cap B (w, r_3)$.    
 \end{lemma} 
\begin{proof}                           
Let  $ \bar w = w + \frac{r}{4} e_n. $  
 As in Lemma  \ref{lemma8.7} we   observe that  if $ \bar c $ is large enough (depending on $ p, n $
and the  Lipschitz norm of $ \ph $),  then the domain 
$ D  \subset \hat D\cap B(w,4r) $ obtained from drawing all open line segments  from points in $  \ar \hat D  \cap B ( w, r/ \bar c)  $ to points in $  B (  \bar  w, r/ \bar c), $ is starlike  Lipschitz
with center $ \bar  w  $  and   starlike  Lipschitz constant  bounded above by $   c  (   \| | \nabla \ph |
\|_{\infty} + 1 ), $ where $ c $ depends only on $ n. $ 
From  this  observation,   we conclude that  \eqref{eqn9.34a}   follows  from  Lemma     \ref{lemma9.5a}  while \eqref{eqn9.34b} follows from this observation, Corollary  \ref{corollary9.6a}, and basic geometry.     

\end{proof} 

   \setcounter{equation}{0} 
   \setcounter{theorem}{0}
       \section{Weak convergence of  certain  measures on  $ \mathbb{S}^{n-1}$}    
\label{section11}       
In this section, we use the results in  sections \ref{section9} and \ref{section10} to  fill in some of the details  outlined  in  section \ref{section8},  regarding the pullback  of   a   certain measure  under the Gauss map  on  the boundary of  a  convex domain.  To begin, suppose that  $   E $ is  a compact convex set with $ 0 $ in the interior of    $ E $  and  let $ u $   be  the   $ \mathcal{A}  = \nabla f$-capacitary function for $ E. $  Let     
 $  \mathcal{\ti A} =   \nabla \ti f, $  where  
   $  \ti f ( \eta )$  
  $ =  f (- \eta ),  \eta  \in \rn{n} \sem \{0\}.   $  
From  convexity of  $  E $,  we see that $ \ar  E $ is  Lipschitz so Corollary \ref{corollary9.6a}  can be  applied to $  1 - u $ with  
   $ \mathcal{A} = \nabla  f $  replaced  by   
   $  \mathcal{\ti A} =  \nabla \ti f .$  More specifically, from  this   corollary  and basic geometry we see as in   
 Lemma  \ref{lemma8.7}  that  
  if  $ w \in  \ar  E $   and  $ 0 < r_4   \leq r_3 $ is small enough, depending only on  the  data (i.e., the Lipschitz constant for $ E, p, n,   \al $   in  Definition \ref{defn1.1}  and  $ \La $  in \eqref{eqn1.8}) that   there is  a  starlike  Lipschitz domain, say  $ \ti \Om   
  \subset   \rn{n} \sem E $   with  center at  $   z \in \rn{n} \sem E,   | w - z  | \approx r_4  \approx  d ( z,  E )$,    and  
  \[
  \ti \Om \cap B (w, r_4)  =   ( \rn{n} \sem E ) \cap B ( w, r_4 ). 
  \] 
  Moreover, the starlike Lipschitz constant for $ \ti \Om $ can be estimated in terms of the Lipschitz constant for $ E $ as in  Lemma  \ref{lemma8.7}.   Using  these facts  and   Corollary 
  \ref{corollary9.6a}  we see that $ v = 1 - u $  satisfies  \eqref{eqn8.28} with $ f $ replaced by  $  \ti f $  and  $  D $ by  $ \ti  \Om. $   Thus Lemma \ref{lemma8.6}  and Proposition   \ref{proposition8.9}  hold for  $ v . $    It follows  that for  $ \mathcal{H}^{n-1}$-almost every $y
\in   \ar  E, $  
\begin{align}  
\label{eqn11.1a}    
 \lim\limits_{\substack{x \to y\\ x \in \Ga (y) }}  \nabla u ( x ) =\nabla u ( y ) \, \, \mbox{exists} 
\end{align}
and
\begin{align}
\label{eqn11.1b}
\frac{\nabla u(y) }{|\nabla u (y) |}\, \, \mbox{is the unit inner normal to}\, \, E.
\end{align}      
Here  $  \Ga (y) $ is  a  non-tangential  approach region  $  \subset  \rn{n} \sem E $ defined below   \eqref{eqn8.18}.  
Also,  if   $     \De ( w,  r  )  =  \ar E  \cap B ( w, r ), $   and     $ \tau $ is the measure corresponding to  $ 1 - u $  as in   \eqref{eqn8.17},    then 
  $ \tau $ is absolutely continuous with respect to  $ \mathcal{H}^{n-1}  $ on $  \ar E $  
and    
\[
d \tau(y)  =  p   \frac{f (   \nabla u(y)  )} {| \nabla  u(y)  |}d\mathcal{H}^{n-1}\quad \mbox{for}\, \, \mathcal{H}^{n-1}\mbox{-a.e.}\, \,   y\in\ar  E.  
\]
Using the above  facts  and  Lemma  \ref{lemma8.5} we observe for $ 0 < r \leq r_4 $ that      
\begin{align}  
\label{eqn11.2}  
p   r^{p-n}   \int_{ \De ( w, r ) }   \frac{ f (  \nabla u )}{|\nabla u |}    d \mathcal{H}^{n-1}  =   r^{p-n} \tau ( \De ( w, r  ) ) \approx  \left(1 - u  ( a_{2r} (w)  ) \right)^{p-1} 
\end{align}      
where proportionality constants depend only on the data.  Finally, from  \eqref{eqn11.2},   
\eqref{eqn8.29}, and  \eqref{eqn8.49a}-\eqref{eqn8.49b}, and H{\"o}lder's inequality   we have  for $ 0 < r  \leq r_4, $ 
\begin{align}   
 \label{eqn11.3} 
\begin{split}
&(a)\hs{.2in}  \int_{\De ( w, r )}  \left( \frac{ f (  \nabla u )}{|\nabla u |} \right)^t  d \mathcal{H}^{n-1} \,   \leq  \,  c_*  r^{(n-1)(   1  - t )}   \left( \int_{\De ( w, r )}   \frac{ f (  \nabla u )}{|\nabla u |}    d \mathcal{H}^{n-1} \right)^{t} \\ 
 &(b) \hs{.2in}    \int_{\De ( w, r )}   \mathcal{N}_r  ( | \nabla u |)^{t(p-1)}  d \mathcal{H}^{n-1} \,   \leq  \,  c_*  r^{(n-1) (  1 - t ) }    \left( \int_{\De ( w, r )}  \frac{ f (  \nabla u )}{|\nabla u |}  d \mathcal{H}^{n-1} \right)^{t}  
\end{split}
 \end{align}
for some  $ t >   p/(p-1) $ and $ c_* $  depending only on the data.   We  note that   the non-tangential maximal function  $\mathcal{N}_r (\cdot) $ was defined above  Lemma \ref{lemma8.6}.      	

Let  $\mathbf{g}_{E}( x )=\mathbf{g}: \partial E\to \mathbb{S}^{n-1}$ be defined by
\[
\mathbf{g}_{E}( x ) = -\frac{\nabla u (x)}{ |\nabla u  (x) |}
\]
which is well-defined  on a set  $ \He    \subset   \ar E   $ with   $  \mathcal{H}^{n-1} ( \ar E \sem \He  ) =  0. $   From   \eqref{eqn11.1a} and \eqref{eqn11.1b} we  see that  if   $  F  \subset  \mathbb{S}^{n-1}$  is  a  Borel  set,  then  $\mathbf{g}^{-1} (F) $ is 
 $ \mathcal{H}^{n-1} $  measurable.    Define  a measure  
$ \mu(\cdot)= \mu_{E,f} ( \cdot )   $  on  $  \mathbb{S}^{n-1}$  by    
\begin{align} 
\label{eqn11.4}  
\mu ( F )  :=   \int_{ \He   \cap \mathbf{g}^{-1} ( F  )}  f ( \nabla u ) \,  d \mathcal{H}^{n-1}  \, \, \mbox{whenever $ F  \subset  \mathbb{S}^{n-1}$ is Borel set}. 
\end{align}   

Next suppose that $ \{E_m \}_{m\geq 1}$        is  a  sequence of  compact convex  sets with nonempty interiors  which converge to  $  E $ 
 in  the  sense of  Hausdorff distance. That is,  $ d_{\mathcal{H}} ( E_m, E ) \rar 0 $ as 
 $ m \to \infty $ where   $  d_{\mathcal{H}} $  was defined at  the beginning of section \ref{NSR}.  Let     $ u_m  $  be   the  corresponding  $ \mathcal{A} = \nabla f$-capacitary function for $E_m$   and  $m=1, 2, \ldots$   Then for $ m $ large enough say  $ m  \geq m_0 $  we see that  \eqref{eqn11.1a}-\eqref{eqn11.3} hold  with $ u,  E $  replaced by   $ u_m,  E_m,  $ when $ m  \geq m_0 $ with  constants depending only on the data for  $ E. $  
 Let  $  \mu_m  $  be the measure on  $  \mathbb{S}^{n-1}$   defined as in   \eqref{eqn11.4}  with  $ u, E  $ replaced by  $ u_m, E_m. $  From    \eqref{eqn11.2}, \eqref{eqn11.3},  we see for some 
$ q  > p $  that   
\begin{align} 
\label{eqn11.5} 
\int_{\ar E }   |\nabla u |^q   d \mathcal{H}^{n-1}  +   \int_{\ar E_m }   |\nabla u_m |^q    d \mathcal{H}^{n-1}    \leq   T   <   \infty 
\end{align}     
for  $ m \geq m_0 $ where $ T $  depends   on the data and the number of  balls of radius $ r_4 $  needed to cover  $  \ar  E. $    From $ p$-homogeneity of  $ f $,   \eqref{eqn11.5},  and H\"{o}lder's  inequality we see that each of the above measures has finite total mass  $  \leq  \hat  T  $  where $ \hat T $  has the same dependence as  $ T$  above.         
  We prove  
  \begin{proposition} 
  \label{proposition11.1}   
  Let $ \{\mu_m\}_{m\geq 1}$ and $\mu$ be measures  corresponding to  $\{E_m\}_{m\geq 1}$ and $E $ as 
  in  \eqref{eqn11.4}. Then
   \[  
   \mu_m \rightharpoonup  \mu \quad \mbox{weakly as} \,\,   m \to \infty.
   \] 
   \end{proposition}   

Armed with our  work in  sections  \ref{section9} and  \ref{section10},  we could   follow  \cite[section 4]{CNSXYZ} which in turn  was inspired by the argument  in \cite[section 3]{J}.  However,  this approach  would  require that  we  first  prove  some  preliminary results that were available in the $p$-harmonic setting.    Thus, we give another argument which makes use of the major ideas in  \cite{J}  but which for us  was considerably more straight forward.   To this end, we first need the following lemma  (see \cite[Lemma 3.3]{J}).
\begin{lemma}[{\cite[Lemma 3.3]{J}}]
    \label{lemma11.2}   
	    For any $ \ep > 0 $  there exists  a positive integer  $  m_1 = m_1 (\ep) > m_0, $ and  a  finite collection of  disjoint closed  balls   $  \bar B ( x_j, r_j )$, $1\leq j \leq N,  $   with  $ r_j \leq \ep$,   $ x_j  \in  \ar E, $  and  
\begin{align*}  
\mathcal{H}^{n-1}    (   \ar E  \sem \bigcup_{j=1}^N    \bar B ( x_j,  r_j ) )  \leq \ep. 
   \end{align*} 
   Moreover, for every  $ j \in \{1, \dots, N\} $ and $ m \geq m_1, $  there exists    a  rotation and translation,  say  $  M_j $ of  $ \rn{n},  $ for which $ M_j ( x_j ) = 0, $ 
   \[
   \bea{c} 
    M_j (  E  \cap B ( x_j, r_j/\ep ))    =   \{  ( x' , x_n )  :  x_n  >  \ph ( x' ) \} \cap B (0, r_j/\ep ),    \\ 
    M_j (  E_m  \cap B ( x_j, r_j/\ep ) )= \{  ( x' , x_n )  :  x_n  >  \ph_m ( x' ) \}   \cap B ( 0, r_j/\ep ),  
   \ea 
   \]   
   where  $ \ph$ and $\ph_m$ are  Lipschitz  functions on $ \mathbb{R}^{n-1}$  satisfying   
\begin{align} 
\label{eqn11.7}    
\| \nabla \ph \|_\infty +  \| \nabla \ph_m \|_\infty  \leq \ep. 
\end{align}    
   \end{lemma} 
   \begin{proof}[Proof of Proposition \ref{proposition11.1}]        
        Let  
        \[
        \Ph  :=  \ar E  \sem \bigcup_{j=1}^N 
   \bar B ( x_j,  r_j )\quad \mbox{and} \quad 
   \Ph_m  : =    \ar E_m  \sem \bigcup_{j=1}^N 
   \bar B ( x_j,  r_j )  \quad \mbox{for}\,\,  m  \geq  m_1. 
   \]  
   Let   $  \rho_1,    \rho_2, $  denote respectively the radius of  the largest  ball  contained in the interior of  $ E, $  and the smallest ball containing $ E, $ both with center at  the origin.   We note that the radial  projections from  $  E, E_m,  $ onto  $  B (0, \rho_1/2) $  are   bilipschitz mappings whose  bilipschitz constants can  be estimated independently of $ m $  ($ m  \geq m_1) $  and in   terms  of  $ \rho_2/\rho_1,   n. $    Using this fact and comparing the projections of   $ \Ph,   \Ph_m $  onto  $  \mathbb{S}^{n-1}$   we see from Lemma \ref{lemma11.2} that 
\begin{align} 
\label{eqn11.8}  
\mathcal{H}^{n-1} (\Ph_m)    \leq  \kappa   \ep \quad \mbox{for}\,\,   m \geq m_2 \geq m_1, 
\end{align}
   where $ \kappa $ is a positive constant independent of $ m, $  depending only on the ratio of    $ \rho_2$  to  $ \rho_1 $  and $ n.$ 
   
Next  for fixed  $ j, 1 \leq j  \leq N, $ and $ M_j $ as in Lemma \ref{lemma11.2}  we  let  
\[
   \hat  u_m (x) :=         u_m  (M_j^{-1} x  )\quad \mbox{and}\quad   \hat u ( x ) := u  (M_j^{-1} x )
   \]
   when  $ x \in  \mathbb{R}^n$. We also put     
   \[
    \hat E_m :=  M_j  (E_m)\quad \mbox{and}\quad \hat E :=  M_j (E).  
    \]
        We note that          $ \hat u$ and $\hat u_m $  are   $ \mathcal{ \hat  A } =  \nabla  \hat f$-harmonic in 
         $ \mathbb{R}^n \sem  \hat E$ and $\mathbb{R}^n \sem  \hat E_m$  where 
         $  \hat  f $  satisfies the  same  structure and smoothness conditions as $f$.                 Let
         \[
          C   =   \{ ( x' , x_n ) \in  \mathbb{R}^n :  | x' |  \leq   r_j, - r_j <  x_n < r_j \}.
          \]
We also put
\[
 D' :=   C \cap (\mathbb{R}^n \sem \hat E ) \quad \mbox{and}\quad D'_m :=   C \cap (\mathbb{R}^n \sem \hat E_m ) \quad \mbox{for}\,\,  m  \geq m_1.   
 \]
Given  $ \eta > 0 $ small,  we claim that if  $ \ep  > 0 $ is small enough,   then   
\begin{align}  
\label{eqn11.9}   
| \nabla \hat u_m ( x' , x_n ) |  \leq c (\ep)  (x_n - \ph_m (x'))^{- \eta}  
      \end{align} 
      when $x  \in    B (0, 2 r_j) \sem \hat E_m$ and where  $ c (\ep)  $  depends on $ \ep $  and the data for  $ E, u, $   but is independent of  $ m $ and $ j $  for $ m $ large enough, say  $  m  \geq m_3 \geq m_2. $   The proof of this claim will be given  after  we use it to prove  Proposition \ref{proposition11.1}. 
         
 We next let $e  =  \frac{r_j}{2} e_n$ and        using  the  Gauss-Green Theorem in certain approximating domains to  $ D'_m $,  \eqref{eqn11.1a}-\eqref{eqn11.3} to take limits  we deduce as in  
       \eqref{eqn8.30}-\eqref{eqn8.34}  that  
       \begin{align}  
       \label{eqn11.10}   
L_m =    (p - 1)  \int_{ \ar D'_m \cap \ar  \hat E_m  }  \lan e - x  ,  \nabla \hat u_m \ran \,     \frac{ \hat f  ( \nabla  \hat u_m )}{|\nabla  \hat u_m |} d \mathcal{H}^{n-1}         = J_m + K_m 
       \end{align}  
       where    
         \[   
         J_m  =     (n-p)        \int_{D'_m}  \hat f ( \nabla \hat u_m )  dx  
          \]
         and 
         \begin{align*}
         K_m   =    \int_{\ar  D'_m\sem \ar E_m }  \lan  & e - x,  
     \nu_m \ran  \hat f  ( \nabla  \hat u_m )  d \mathcal{H}^{n-1}   \\
       & +  \int_{\ar D'_m \sem \ar E_m } \lan x  - e ,  \nabla  \hat u_m  \ran  \, \lan \nabla \hat f(\nabla \hat u_m ) , 
      \nu_m  \ran  \, d \mathcal{H}^{n-1}.                                                                   
 \end{align*}   
From   Lemmas \ref{lemma2.1} and  \ref{lemma2.2}                                            as well as uniqueness of the  $ \mathcal{\hat A}$-capacitary function for a  compact convex set with interior, we see that  $\{ \hat u_m\}_{m\geq 1}$ converges uniformly to  $ \hat u $ in  $ \mathbb{R}^n  $  while  $ \{\nabla  \hat u_m \}_{m\geq 1}$ converges uniformly to  $ \nabla  \hat u $  on compact subsets of $ \mathbb{R}^n\sem E. $   Using these facts and  claim \eqref{eqn11.9}  we find from uniform integrability type estimates that 
\begin{align}     
\label{eqn11.11}    
\lim_{m \to \infty}    J_m  =     (n-p)        \int_{D'}  \hat f ( \nabla \hat u )  dx   
\end{align}
and
\begin{align}
\label{eqn11.12} 
\begin{split}
{ \ds \lim_{m \to \infty}   K_m  
                         =    \int_{\ar  D' \sem \ar E }  \lan  e - x,  
      \nu \ran  \hat f  ( \nabla  \hat u )   d \mathcal{H}^{n-1} }   \\
     \hs{.5in}  +  {\ds \int_{\ar D' \sem \ar E } \lan x  - e ,  \nabla  \hat u  \ran  \, \lan \nabla \hat f(\nabla \hat u ), 
         \nu   \ran  \, d \mathcal{H}^{n-1}. }                                                            
\end{split}
\end{align}
 where $ \nu $ is the outer unit normal to $D'.$  Now \eqref{eqn11.10} also holds  with  $ D'_m $  replaced by $ D' $ and $ \hat u_m $ by  $ \hat u $  Using this fact and   \eqref{eqn11.11}-\eqref{eqn11.12}  we conclude that     
\begin{align}
 \label{eqn11.13}  
\lim_{m\to \infty}     \int_{ \ar D'_m \cap \ar  \hat E_m  }  \lan e - x  ,  \nabla \hat u_m \ran \,      \frac{ \hat f  ( \nabla  \hat u_m )}{|\nabla  \hat u_m |} d \mathcal{H}^{n-1}  =   \int_{ \ar D' \cap \ar  \hat E  }  \lan e - x  ,  \nabla \hat u \ran \,      \frac{ \hat f  ( \nabla  \hat u )}{|\nabla  \hat u |} d \mathcal{H}^{n-1}.
\end{align}                                         
Next we note that  we can compute the inner  normal to   $  \ar \hat  E_m $ in two ways:   $(\hat a)$  using  \eqref{eqn11.1a}-\eqref{eqn11.1b} with  
   $ \hat u, \hat E,  $ replaced by  $ \hat u_m, \hat E_m $ 
 or $(\hat b)$  using \eqref{eqn11.7}  of  Lemma  \ref{lemma11.2} and calculus.    Doing this and using the resulting computation in     \eqref{eqn11.13}, we obtain      
\begin{align}
 \label{eqn11.14}  
 \begin{split}
 \limsup_{m\to \infty}  &\left|   \int_{ \ar D'_m \cap \ar  \hat E_m  }   
         \hat f  ( \nabla  \hat u_m ) d \mathcal{H}^{n-1}     
 -   \int_{ \ar D' \cap \ar  \hat E  }   \hat f  ( \nabla  \hat u ) d \mathcal{H}^{n-1} \right|          
 \\  
 &\leq  c^*  \ep  
        \left[ \limsup_{m\to \infty}     \int_{ \ar D'_m \cap \ar  \hat E_m  }    \hat f  ( \nabla  \hat u_m ) d \mathcal{H}^{n-1}  +   \int_{ \ar D' \cap \ar  \hat E  }   \hat f  ( \nabla  \hat u ) d \mathcal{H}^{n-1} \right] 
\end{split}
        \end{align}   
 where                                            $ c^*$  depends only on the data for $ E, u. $  As noted earlier,  both integrals on the right-hand side of   \eqref{eqn11.14} are  $  \leq  \hat T  < \infty. $   We assume as we may that 
  $ c^* \ep \leq 1/4.$      Then from   \eqref{eqn11.14} we easily deduce that  
\begin{align}
\label{eqn11.15} 
\begin{split}
& (\al)\, \,
	       \limsup_{m\to \infty}      \int_{ \ar D'_m \cap \ar  \hat E_m  }
	        \hat f  ( \nabla  \hat u_m ) d \mathcal{H}^{n-1}    \leq  
        2  \int_{ \ar D' \cap \ar  \hat E  }     \hat f  ( \nabla  \hat u ) d \mathcal{H}^{n-1},\\  
	&(\be) \, \, \limsup_{m\to \infty}  \left|   \int_{ \ar D'_m \cap \ar  \hat E_m }     \hat f  ( \nabla  \hat u_m ) d \mathcal{H}^{n-1}     
 -   \int_{ \ar D' \cap \ar  \hat E  }   \hat f  ( \nabla  \hat u ) d \mathcal{H}^{n-1} \right|   \leq              3  c^*  \ep  
     \int_{ \ar D' \cap \ar  \hat E  }   \hat f  ( \nabla  \hat u ) d \mathcal{H}^{n-1}.   
\end{split}
\end{align}                   
    Transferring back to our original scenario using  $ M_j, M_j^{-1}, $ allowing  $ j $ to vary, and summing from 1 to  $ N $  we obtain from  \eqref{eqn11.14}, \eqref{eqn11.15},  that 
  \begin{align*}
         \limsup_{m\to \infty}   \left|   \int_{   \partial E_m \sem \Ph_m  }     f  ( \nabla   u_m ) d \mathcal{H}^{n-1}  -   \int_{  \partial E \sem \Ph  }    f  ( \nabla   u ) d \mathcal{H}^{n-1} \right|   \leq   c^{**}  \ep \hat T, 
        \end{align*}   
      where $ \hat T $ was defined above Proposition \ref{proposition11.1} and $ 
      c^{**} $  has the same dependence as $ c^*. $
                                              
       To continue the proof  of  Proposition   \ref{proposition11.1} (under the assumption that claim \eqref{eqn11.9} is true), let $ \he  $  be  a  continuous function on  $ \mathbb{S}^{n-1}$   and let  
\[  
\psi (\rho)  = \sup \{ | \he ( \om ) -  \he (\om') | :\, \, \om, \om'  \in  \mathbb{S}^{n-1}, \, \,   | \om -  \om'| \leq \rho\} \quad \mbox{whenever}\, \,   \rho \in (0, \infty) 
\] 
be the modulus of continuity of  $  \he. $   Let  $ \mathbf{g}_m  =  - | \nabla \hat u_m |^{-1} \nabla  \hat u_m   $  be the Gauss map defined on a set  $ \He_m  \subset         \ar E_m $  with   $  \mathcal{H}^{n-1} ( \mathbb{S}^{n-1} \sem \He_m ) = 0.$  Define   $ \mu_m $ as  in \eqref{eqn11.4} relative to $ \mathbf{g}_m, E_m, u_m $ and set 
        \[  
        \bea{l} 
        \mu_{j,m} (F)  :=  {\ds \int_{\He_m \cap \bar B ( x_j, r_j ) \cap \mathbf{g}_m^{-1} (F)} f ( \nabla u_m ) d  \mathcal{H}^{n-1} } \quad \mbox{for}\, \,   j = 1, \dots, N,
        \\
        \mu_{0, m} ( F ) := {\ds \int_{\He_m \cap \Ph_m \cap \mathbf{g}_m^{-1} (F)} f ( \nabla u_m ) d  \mathcal{H}^{n-1}}  
        \ea 
        \] 
        whenever  $ F \subset  \mathbb{S}^{n-1}$ is a Borel set.  Define  $  \mu_{j}$ for $0 \leq j \leq N, $ similarly with $ u_m, \mathbf{g}_m, E_m, \He_m, $ replaced by  $ u, \mathbf{g}, E, \He. $  We note that     
\begin{align*}
 \mu_m  :=  \sum_{ j=0}^N \mu_{j,m} \quad \mbox{and} \quad \mu:= \sum_{ j=0}^N \mu_{j}.
\end{align*}          
Let $ j $ be fixed, $ 1 \leq j \leq N, $ 
          and  let 
          \[
           x, x' \in  \ar E \cap \He \cap B ( x_j, r_j ) \quad \mbox{and} \quad y, y'  \in \ar E_m \cap \He_m \cap B ( x_j, r_j).
           \] 
Then from    from  Lemma  \ref{lemma11.2}  we see for  $ m \geq m_1 $ that  
\begin{align}
\label{eqn11.18}  
| \mathbf{g} ( x ) -  \mathbf{g} ( x' ) |  + | \mathbf{g}_m ( y ) - \mathbf{g}_m ( y') |  + | \mathbf{g}(x) - \mathbf{g}_m ( y )  | \leq  c \ep 
\end{align}
where $ c = c ( n). $  From            \eqref{eqn11.18} and \eqref{eqn11.15} we deduce that   
\begin{align} 
\label{eqn11.19} 
\begin{split}
 \limsup_{m\to \infty}   & \left| \int  \he  d \mu_{j,m}  -   \int \he d\mu_{j} \right| \\
 &\leq\| \he \|_\infty  \limsup_{m \to \infty}   \,  |  \mu_{j, m}  ( \mathbb{S}^{n-1})  -    \mu_j ( \mathbb{S}^{n-1})  | +  \psi (  c \ep)\,  [ {\ds  \limsup_{m \to \infty}    \mu_{j, m}  ( \mathbb{S}^{n-1})  +    \mu_j ( \mathbb{S}^{n-1}) } ] \\
 &\leq 3 ( c^* \ep \| \he \|_\infty +  \psi ( c \ep)  ) \,   \mu_j ( \mathbb{S}^{n-1}).
           \end{split}
           \end{align}                                                                                                                                                                                            
           Also  from     \eqref{eqn11.8}, Lemma  \ref{lemma11.2},  H\"{o}lder's  inequality,  and \eqref{eqn11.5},  we deduce  for  $ m \geq m_2 $  that   
\begin{align}
\label{eqn11.20} 
\begin{split}
\mu_m (\Ph_m)  +  \mu (\Ph )  &\leq   \hat c  \, 
         [ \mathcal{H}^{n-1} ( \Ph_m \cup \Ph ) ]^{1 - p/q}  \, \,     T^{p/q}  \\
         &\leq \hat c^2  \ep^{ 1 - p/q } \,  T^{p/q} 
\end{split}         
         \end{align}   
           where $  \hat c $  depends only on the data for $ E, u. $  
                                                                      Summing \eqref{eqn11.19} over  $ 1 \leq j  \leq N $  we get in view of   \eqref{eqn11.20},   that  
           \begin{align*}
\begin{split}
             \limsup_{m\to \infty}   \left| \int_{\mathbb{S}^{n-1}}   \he  d \mu_{m}  -   \int_{\mathbb{S}^{n-1}}  \he d\mu \right| \leq 
            3 ( \psi ( c \,   \ep)\, + c^*  \ep  \| \he \|_\infty  )   \hat  T +  \hat c^2  \| \he  \|_\infty  \, \ep^{ 1 - p/q }  \, T^{p/q}.
\end{split}
             \end{align*}
   Since  $  \ep $ can be arbitrarily small  and  $ \he $ is an arbitrary continuous function on $ \mathbb{S}^{n-1}$   we conclude that  Proposition \ref{proposition11.1} is true under claim   \eqref{eqn11.9}.
 \end{proof}%
 \begin{proof}[Proof of claim \eqref{eqn11.9}]      
 Our  proof is quite similar to the proof of Lemma 5.28  in  \cite{LN}.  Let  $ \hat f $ be as 
  defined above  \eqref{eqn11.9} and let    
  \[
   K = \{ x: x_n/|x|   >   \cos  \hat \he  \} 
   \]
    be the open  spherical cone  in  $ \mathbb{R}^n $ with vertex at the origin,  angle opening  $ \hat \he,  \pi/2 <
\hat \he  <  \pi, $  and  axis parallel to  $ e_n.  $  Let   $  \bar u_l,$ be the  $ \mathcal{\hat A} =  \nabla \hat f $-capacitary function for  $  \bar B (0, l ) \sem K $  and put 
   \[
    \bar v_l   =    \frac{ 1 -  \bar u_l }{ 1 - \bar u_l ( e_n) } \quad   \mbox{for} \,\, l = 1, 2, \dots.
    \]
      Then 
   $ \bar v_l $  is   $   \mathcal{\breve{A}}  =  \nabla \breve f$-harmonic in  the  complement of   $ B ( 0, 1 ) \sem  K  $
        where 
   $ \breve f (  \eta ) =  \hat f  ( - \eta ) ,  \eta \in  \mathbb{R}^n. $  Also  $  \bar v_l $  has continuous boundary values with  $ \bar v_l  \equiv 0 $ on  $  B (0, l ) \sem K $ and 
   $ \bar v_l ( e_n )  = 1. $   Using Lemmas \ref{lemma3.1}-\ref{lemma3.3} and taking limits of  a  certain subsequence of  $ \{\bar v_l\}_{l\geq 1} $, we  see there exists $ v \geq 0,$ a continuous function on $  \mathbb{R}^n $  which is                  
     $  \mathcal{\breve  A}  =  \nabla \breve  f$-harmonic in   $ K $
     with  $ v \equiv 0  $  on  $  \mathbb{R}^n  \sem K  $  
     and     $ v ( e_n )  = 1. $           
     
     We assert that 
\begin{align} 
\label{eqn11.22}      
v ( t x ) =  v ( t  e_n )  \,  v ( x )  \quad \mbox{whenever} \, \,  x \in \mathbb{R}^n\, \, \mbox{and}\, \,  t  \in (0, \infty).
\end{align}
        To see this, we note that  $ v ( t x ),  x  \in \mathbb{R}^n $ is also  $ \mathcal{\breve{A}}$-harmonic  in   $ K $ as follows from $p$-homogeneity of  $ \breve f. $  Also $ v ( t x ) = 0 $  when $  x \not \in K. $  From these  observations we see that  Lemma \ref{lemma9.5} can be used with  $ \ti u, \ti v $  replaced by    $ v ( x), v ( t x ), $ for  $ x  \in  K \cap B ( 0, R )$ when $R\geq 1$.   Moreover, from Harnack's inequality we see that  the ratio of  $  v ( Re_n )$ to $  v ( t R e_n) $ is  bounded above and below by constants independent of  $ R $ when $ R \geq 1. $   Using these facts and  letting  $  R \to \infty $ we deduce from  Lemma \ref{lemma9.5}  that  $ v ( t x )$ is a constant multiple of  $ v ( x )  $  whenever $ x \in K. $ From   this statement  and   $ v ( e_n ) = 1, $ we obtain \eqref{eqn11.22}.   Differentiating  \eqref{eqn11.22} with respect to $ t $  and evaluating at $ t =  1 $  we see that   
      \[ 
        \lan x,  \nabla v ( x ) \ran =   \lan e_n,  \nabla v ( e_n ) \ran  v ( x )\quad \mbox{whenever} \quad x  \in K. 
        \]    
If we let $ r =  | x |,  \om =  x / | x |$  in this identity we obtain that  
\[
r  v_r  ( r \om ) =  \lan e_n,  \nabla v ( e_n )\ran v ( r  \om ).  
\]
Dividing this equality by  $ r v  ( r \om )  $   and integrating with respect to $ r, $  we find  that   
        $  v (  \rho \om )  =  \rho^{\ga } v ( \om ) $  whenever $ \om \in \mathbb{S}^{n-1}$  where  $ \ga  =  \lan e_n,  \nabla v ( e_n ) \ran .$  Next we assert that  
   \begin{align} 
   \label{eqn11.23}   
   \ga = \ga (\hat  \he ) \to 1 \quad \mbox{as}\quad  \hat \he \to \pi/2.  
   \end{align}     
          To verify this assertion, we write $ v = v ( \cdot, \hat \he ), \ga =  \ga ( \hat \he ) $ for the above $ v,  \ga. $   If  $ \pi/2  <  \hat \he_k  \to  \pi/2 $ as $ k  \to  \infty $  and a   subsequence of  $  \ga ( \hat \he_k ) \to \ga  $  as  $ k \to \infty$.  Then, also a certain  subsequence of  $ v ( \cdot, \hat \he_k ) $ converges uniformly  to $ v $  an 
        $ \mathcal{\breve A}$-harmonic function in  $  \{ x : x_n > 0 \} $  with  $ v ( x ) = 0 $ 
        when $ x_n \leq 0. $    Since  $ x_n $ is also $ \mathcal{\breve  A}$-harmonic we conclude as above that $ v $ is a constant multiple of $ x_n $  and thereupon that  \eqref{eqn11.23} is true.

   Now  given  $  \eta >  0  $ as  in   \eqref{eqn11.9}  it follows from   \eqref{eqn11.22}  and  \eqref{eqn11.23}      that there exists $    \hat \he$ such that
   \[
\pi/2 <  \hat \he < \pi\quad \mbox{with}\, \,  \ga (\hat \he)   \geq  1  -  \eta.
    \]
         With  $   \hat \he  $ is now fixed,  we find  from  Lemma 
   \ref{lemma11.2}  that  there exists $ \ep = \ep (\hat \he )  > 0 $  for which the following is true: 
    Given  $ \hat x  \in \ar  \hat E_m \cap B (0, 2 r_j )$ we have   $ B ( \hat x, r_j ) \sem    ( K + \hat x )    \subset  \hat E_m . $     
      Let   $ v^*   $  be  an   $ \mathcal{\breve  A}$-harmonic  function with   $ v^* \equiv 1 - \hat u_m $ on  $   \ar [B  ( \hat x,  r_j) \sem ( K  +  \hat x )].  $ Existence of $ v^* $  follows as in  
\cite{HKM}.  Comparing boundary values and using the maximum principle for $  \mathcal{\breve  A}$-harmonic functions we have 
      $ 1 -\hat u_m \leq   v^* $ in   $   B  ( \hat x,  r_j) \sem ( K  +  \hat x ).  $   
        Since    
      $ x  \to   v (  x -  \hat x   )$  is  $  \mathcal{\breve  A}$-harmonic in   $ K  +  \hat x $ we conclude from Lemma   \ref{lemma8.13}, construction of  $ v, v^* \leq 1, $   and our  choice of $   \hat \he  $  that     
\begin{align}
 \label{eqn11.24}   
\begin{split} 
 ( 1 -  \hat  u_m  ) (   x  )& \leq     v^* ( x )    \\
 &\leq   c   \frac{ v^*(  \hat x  +  r_j e_n/2) }{ v (  r_j e_n/2) }  v ( x - \hat x)  \\
       &     \leq c ( | x  -   \hat x |/r_j )^{ 1  - \eta}
\end{split}
\end{align}
in $B ( \hat x, r_j )  \sem  \hat E_m$ where $ c $  depends only on the data for   $  m  \geq m_2$.   
From  \eqref{eqn11.24} with  $ x - \hat x $  a  multiple of  $ e_n $    and  \eqref{eqn2.2} we get   claim \eqref{eqn11.9}.     In view of our earlier remarks, this finishes the proof of Proposition \ref{proposition11.1}. 
\end{proof}

\setcounter{equation}{0} 
\setcounter{theorem}{0}
\setcounter{equation}{0} 
\setcounter{theorem}{0}

\section{The Hadamard variational formula for nonlinear capacity} 
\label{section12}
Let $ E_1$ and $E_2$ be   compact convex sets  and suppose  $ 0 $ is in the interior of  $ E_1 \cap E_2. $ Let   $ u ( \cdot, t ) $   be the $ \mathcal{A} = \nabla f$-capacitary function  for  $  E_1  + t  E_2 $  when $ t  \geq 0. $  Also  let   
$  \mu_{E_1 + t E_2} $ be the measure defined   in   \eqref{eqn7.6}  relative to   $ u (  \cdot,  t ) . $   In this section  we  prove 
\begin{proposition} 
\label{proposition12.1}  
With  the  above  notation let  $  h_1,  h_2 $   be  the  support   functions for $ E_1, E_2, $  respectively and let  $\mathbf{g}( \cdot,  E_1 + t E_2 ) $  be the  Gauss map for $  \ar  (  E_1 + t E_2 ) .$  Then   for  $  t_2   \geq  0 $ we  have   
\begin{align}  
\label{eqn12.1}   
\left. \frac{d}{dt}    \mbox{Cap}_{\mathcal{A}} (  E_1 + t E_2 )  \,  \right|_{t=t_2} =   (p-1)\int_{\mathbb{S}^{n-1}}  h_2 ( \mathbf{g}(z,  E_1 + t_2 E_2 ) )   \,  d  \mu_{ E_1 + t_2 E_2}(z). 
\end{align} 
\end{proposition}  

\begin{proof}   
In  the proof  of  \eqref{eqn12.1} we first assume for $ i = 1, 2 $  that 
\begin{align} 
\label{eqn12.2} 
\begin{split}
\mbox{$ \ar E_i $ is  locally the graph of}&\mbox{ an infinitely differentiable} \\
& \mbox{and strictly convex  function on  $  \mathbb{R}^{n-1}. $ } 
\end{split}
\end{align}   
 We  note from Lemma  \ref{lemma2.2}   that   $ u ( \cdot,  t_i) $, for $i=1,2$,   has  H\"{o}lder  continuous  second  partials  in    $\mathbb{R}^n \sem (E_1 + t_2 E_2).   $   Moreover, as in  \eqref{eqn4.8}-\eqref{eqn4.12}  we see that  if $ 0   < t_1  <   t_2, $ then   
 \[     
 \zeta (x, t_1)  =   \frac{ u ( x, t_1 )  -  u ( x, t_2)}{ t_1 - t_2} \quad \mbox{whenever} \,\, x  \in  \mathbb{R}^n,   
 \] 
 is  a positive weak  solution to    
\begin{align}  
\label{eqn12.3}  
 \sum_{i, j = 1}^n  \frac{ \ar}{ \ar x_i} (\bar{d}_{ij}  \ze_{x_j} ) 
= 0   
\end{align}  
in    $ \mathbb{R}^n  \sem  ( E_1 + t_2  E_2)  $ where  
\begin{align*}   
\bar{d}_{ij} (x)  =  \int_0^1 f_{\eta_i \eta_j} ( s  \nabla u ( x, t_1 ) + (1-s) \nabla u ( x, t_2 ) ) ds. 
\end{align*} 
Also,   
\begin{align}  
\label{eqn12.5}   
c^{-1} \bar \si (x)  \,   | \xi |^2  \, \leq \,   \sum_{i,j=1}^n   \bar d_{ij} ( x  )  \xi_i  \xi_j   \leq c \, \bar \si (x) \, |\xi|^2  
\end{align}
whenever $\xi  \in \mathbb{R}^n \sem \{0\},$ $ x   \in  \mathbb{R}^n  \sem  ( E_1 + t_2  E_2)$ and 
\begin{align}
\label{eqn12.6} 
\bar \si ( x  )  \approx  ( |  \nabla  u (x, t_2 ) | + |  \nabla  u (x, t_1 ) |  )^{p-2} . 
\end{align}
Constant $c$  in \eqref{eqn12.5} depends only on $p$ and $n$ and constants in \eqref{eqn12.6} depend only on the structure constants for $ \mathcal{A},  p,$  and $  n.  $    
  Also  from  Lemmas   \ref{lemma2.2},   \ref{lemma3.2}, and  the  Theorem  in \cite[Theorem 1]{Li}  mentioned earlier,  we see that   $ \nabla  u (\cdot, t_i  )$ for $i = 1, 2, $      extend   to   H\"{o}lder continuous  functions in the closure of  $  \mathbb{R}^n \sem  (E  _1+ t_i E_2)$.  More  specifically, if  
  \[
   t_2/2 \leq t_1 <  t_2 \quad \mbox{and} \quad \rho  = 2 t_2 \left(\mbox{diam}(E_1) +  \mbox{diam} (E_2) \right)
   \]
    then there  exist $  \be \in (0,1)$ and  $C^{\star}  \geq 1, $ independent of $ t_1, $  such that for $ i = 1, 2, $ 
\begin{align}
\label{eqn12.7} 
\begin{split}
&(a)  \hs{.2in} \, | \nabla  u ( x, t_i  ) - \nabla  u ( y     ,  t_i ) | \, \leq \,
      C^{\star}     | x - y |^{ \be },
      \\  
  &    (b) \hs{.2in} (C^{\star})^{-1}  \leq    | \nabla u (x, t_i )  |  \leq C^{\star}  
\end{split}
      \end{align}
 whenever $x,y$ are  in the closure of   $B ( 0, \rho )  \sem ( E_ 1 + t_i E_2 )$.    From \eqref{eqn12.7} $(b)$ and the  mean value theorem from  calculus we see that   $ \ze $  is  bounded  on  $  \ar  (E_1 + t_2  E_2 )$   by a constant  independent of  $ t_1 $  when  $ t_2/2 \leq t_1 < t_2. $   Using this fact,  \eqref{eqn3.4} $(b)$,   \eqref{eqn3.12} $(b)$   and  a   weak   maximum  principle type argument  we see for $C^{\star} $ sufficiently large,  independent of  $ t_1 \in [ t_2/2,  t_2) $     that   
\begin{align}  
\label{eqn12.8}   
\ze \leq  C^{\star}    u ( \cdot, t_2 ) \quad \mbox{on}\quad  \mathbb{R}^n \sem( E_1 + t_2 E_2 ).  
\end{align}  
  Also  from   \eqref{eqn12.7} $(a)$,   \eqref{eqn12.7}  $(b)$,  and the Lemmas mentioned above we deduce that \begin{align}
\label{eqn12.9}
\nabla u  ( \cdot, t_1 ) \to  \nabla u ( \cdot, t_2 ) \, \, \mbox{uniformly on  the  closure of}\, \,  \mathbb{R}^n \sem (E_1 + t_2 E_2)\, \,  \mbox{as}\, \, t_1 \to t_2.   
\end{align} 
From  \eqref{eqn12.9}   and  \eqref{eqn12.3}-\eqref{eqn12.6}  we see that  $ \ze $ is  locally  a  bounded solution  
to  a     divergence form uniformly elliptic PDE  with   coefficients that  have  local   H\"{o}lder, say $  \hat \be$-norm independent of  $  t_1   \in   [t_2/2,  t_2).  $        
From  these  facts,  Lemma  \ref{lemma2.2},  and Caccioppoli type estimates for locally uniformly elliptic  PDE, we deduce that if  $ t_1 \to  t_2  $  through  an increasing   sequence,  then a  subsequence of the  functions corresponding to  these values  converges  uniformly  on    
compact subsets of  $ \mathbb{R}^n   \sem ( E_1 + t_2 E_2 ) $  to  a locally   H\"{o}lder $ \hat \be$  continuous    function,   say  $ \breve \ze,  $  with  
$  \breve  \ze  \leq C^{\star}    u ( \cdot,  t_1 ). $  Moreover,  this  subsequence also converges    to   $  
\breve \ze $ locally  weakly in  $  W^{1,2} $ of   $ \mathbb{R}^n   \sem ( E_1 + t_2 E_2 ). $ Finally,     
\begin{align}
\label{eqn12.10}   
\sum_{i, j = 1}^n  \frac{ \ar}{ \ar x_i} ( f_{\eta_i \eta_j} ( \nabla u ( x, t_2 ) ) \,  \breve \ze_{x_j} (x))  
= 0  \quad \mbox{for}\quad  x   \in   \mathbb{R}^n  \sem  ( E_1 + t_2  E_2)
\end{align} 
locally in the weak sense in  $ \mathbb{R}^n   \sem ( E_1 + t_2 E_2 ). $

Next  we  show that  $ \breve  \ze $  is  independent of  the choice of  sequence.   To do this, for $k = 1, 2$ we let      
\begin{align*}  
x_k  ( Z) =     \nabla  h_k  (  Z  )\quad \mbox{whenever} \, \, Z  \in  \mathbb{S}^{n-1}.
\end{align*}    
We fix  $ X, Y   \in  \mathbb{S}^{n-1} , $    write  $ x, y $  for   $ x_1 ( X ) + t_2  x_2 (X),   x_1 (Y)  +  t_2   x_2 (Y)$  respectively and note that    
\[
x  =  \mathbf{g}^{-1} ( X,  E_1  + t_2 E_2 )\in   \ar ( E_1 + t_2 E_2) \quad \mbox{and}\quad  y   =  \mathbf{g}^{-1} ( Y,  E_1  + t_2 E_2 )\in   \ar ( E_1 + t_2 E_2).
\]   
 We  consider two cases.  First if  
\[
|  x  - y |  \leq   d ( x,  E_1 +  t_1 E_2 )/2
\]  
then from  \eqref{eqn12.7} $(a)$ and the mean value theorem of   calculus    we have    
\begin{align}
\label{eqn12.12}   | \ze (x)  -   \ze (y)  -  \lan \nabla  \ze (x),  x -  y \ran  |  \leq   \hat C
| x - y |^\be  
\end{align} 
where  $ \hat C $ is independent of  $ t_1. $   Second,   if   
\[
  | x - y | >  d ( x,  E_1 + t_1 E_2   )/2
  \]   
  then   using $ u ( \cdot , t_1  )  \equiv 1 $ on 
 $ \ar ( E_1 + t_1 E_2) $  and the  same strategy as  above we see that 
\begin{align}   
\label{eqn12.13} 
\begin{split}
  | \ze (x)  +    \lan \nabla  & u ( x_1 ( X ) + t_1 x_2 (X),  t_1 ) ,   x_2  (X) \ran  |  \\
   &+  | \ze (y)  +     \lan \nabla  u (x_1 (Y ) +  t_1  x_2 (Y) , t_1 ),  y_2 (Y) \ran  |  \\
   & \hs{.8in} \leq   \hat  C | x - y |^\be.  
\end{split}
\end{align}  Now   $ h_1  + t_1 h_2 $ is the support  function  for  $  E_1 + t_1  E_2 $  and so   from \eqref{eqn5.26} we have
\begin{align}
\label{eqn12.14} 
\begin{split}
 \lan  \nabla u ( x_1 ( X ) + t_1 (x_2 (X)),  t_1 ) ,   x_2  (X) \ran &=   | \nabla u ( x_1 ( X ) + t_1 (x_2 (X)),  t_1 ) | \lan X,  x_2 ( X ) \ran \\
&=    | \nabla u ( x ,  t_1 ) | h_2 ( \mathbf{g} (x, E_1 + t_2 E_2 ) )   +  \la (x)  \\
&=III_1+ \la (x)
\end{split}
\end{align}  
where   $  \la (x) \leq  \bar C   | x - y |^\be $  and  $  \bar C $ is independent of   $ t_1. $   Similarly,  
\begin{align}  
\label{eqn12.15} 
\begin{split}
 \lan \nabla  u ( x_1 ( Y ) + t_1 x_2 (Y),  t_1 ) ,   x_2  (Y) \ran &=   | \nabla u ( y,  t_1 ) |   h_2 ( \mathbf{g} (y,  E_1   + t_2  E_2  )  )  +  \bar \la (y) \\
 &=III_2+\bar \la (y) 
\end{split}
\end{align} 
where $ \bar \la (y) $  satisfies the same inequality  as  $  \la (x). $ 
From  \eqref{eqn12.13}-\eqref{eqn12.15} and the triangle  inequality we find that  
\begin{align} 
\label{eqn12.15n} 
\begin{split}
 | \ze ( x, t_1) -  \ze (y, t_1) | &  \leq  \left| III_1-III_2\right| +\la ( x ) + \bar \la (y)  
 \\ &\leq  \ti C | x  - y |^{\be}
\end{split}
 \end{align} 
 where $  \ti C $ is independent of $ t_1 \in  [t_2/2, t_2)$.
  From  \eqref{eqn12.12}, \eqref{eqn12.15n},  we deduce that     $ \ze (\cdot, t_1)  $ is  H\"{o}lder  $  \be $ continuous  on  $ \ar ( E_1 + t_2 E_2)$ with  H\"{o}lder norm    
  bounded  by  a  constant independent of  $ t_1 \in   [t_2/2, t_2 ). $    
 
We note from  \eqref{eqn12.10} that   $  \ze $ is the  solution to  
    a divergence form uniformly elliptic  PDE with Lipschitz  coefficients.  Using  this note  and  
    H\"{o}lder $ \be $   continuity  of $ \ze $   on  $ \ar E_1+ t_2 E_2 $    it   follows  first  (see \cite[Theorem 8.29]{GT})  that   $ \ze $ is  H\"{o}lder     $ \he  $  continuous for some $ \he > 0$  independent of $ t_1 $ in  $[t_2/2 , t_2] $ with  H\"{o}lder norm on the closure of   $ B (0,  \rho)  \sem \ar ( E_1 + t_2 E_2)$    bounded  by  a  constant that is also  independent of  $ t_1 \in   [t_2/2, t_2 ). $  Second,    
\begin{align}
\label{eqn12.16a}  |\nabla \ze (x) |  \leq   \bar C   d ( x,  \ar (E_1 + t_2 E_2)^{\he - 1} \quad \mbox{whenever}\quad  x \in   B (0,  \rho)  \sem \ar ( E_1 + t_2 E_2). 
\end{align}
Taking limits we conclude from  Ascoli's  theorem  that  $ \breve  \ze $ is  
the uniform limit in  the closure of   $ B (0, \rho )  \sem (E_1  +   t_2 E_2) $ of   a  certain subsequence of  $  \ze ( \cdot,  t_1 )  $ and thus is also  
H\"{o}lder $\he$  continuous  in the closure of    $ B (0, \rho )  \sem (E_1  +   t_2 E_2). $  Moreover,    \eqref{eqn12.16a}  holds for $ \breve \ze . $  Finally,   arguing as  in    \eqref{eqn12.13}-\eqref{eqn12.15n}  we  find that   
\begin{align}  
\label{eqn12.16}   
\ze (x, t_1)    \to   | \nabla u ( x ,  t_2 ) | h_2 ( \mathbf{g} (x,  E_1   + t_2  E_2  ) ) \quad  \mbox{as}\quad   t_1   \to   t_2 
  \end{align}
 whenever  $  x  \in  \ar  (E_1 + t_2 E_2)$.  From \eqref{eqn12.16} we see that  every convergent subsequence of   $  \{ \ze ( \cdot, t_1 ) \} $ converges to  a  weak solution of      
\eqref{eqn12.10} with continuous boundary values
 \[ 
\left.  |  \nabla  u  ( \cdot, t_2  ) \right|  h_2 (\mathbf{g} ( \cdot,  E_1  + t_2  E_2 )  )  \quad   \mbox{on}\quad   \ar ( E_1 + t_2 E_2 ).
  \]     
From this deduction,  \eqref{eqn12.8},  \eqref{eqn12.10},     \eqref{eqn3.4} $(b)$,   \eqref{eqn3.12} $(b)$ and  a   weak   maximum  principle  argument  we conclude that  
\begin{align}  
\label{eqn12.17}   
u_t ( x, t) |_{t=t_2}  =    \lim_{t_1 \to t_2}   \ze (x, t_1)   =  \breve \ze (x) 
\end{align}
whenever  $  x $  is  in  the closure of  $ \mathbb{R}^n \sem   (E_1 + t_2 E_2). $ 
To  begin the  proof  of  Proposition \ref{proposition12.1}  in the smooth case and when $ t_2/2  \leq t_1 <  
t_2 $  we write  
\begin{align}
\label{eqn12.18}   
( t_1 - t_2)^{-1} [ \mbox{Cap}_{\mathcal{A}} ( E_1 + t_1 E_2)    -   \mbox{Cap}_{\mathcal{A}} ( E_1 + t_2 E_2) ] =T_1+T_2
\end{align}
where
\[
T_1=( t_1 - t_2)^{-1} {\ds  \int_{\mathbb{R}^n  \sem ( E_1 + t_2 E_2 ) } } (  f (  \nabla  u  ( x,  t_1 ) )  -   f (  \nabla  u  ( x,  t_2 ) ) )  dx
\]
and
\[
T_2= ( t_1 - t_2)^{-1} {\ds  \int_{  (E_1 + t_2 E_2)  \sem ( E_1 + t_1 E_2) } }   f (  \nabla  u  ( x,  t_1 ) )  dx.
\]
  We note from   \eqref{eqn12.16}  that   
\begin{align}
\label{eqn12.19} 
\begin{split}
T_2   &=   p^{-1} \, ( t_1 - t_2)^{-1} {\ds  \int_{  (E_1 + t_2 E_2) 
  \sem ( E_1 + t_1 E_2) } }  \nabla \cdot  [  (u  ( x,  t_1 ) - 1 )   \, \nabla f (  \nabla  u  ( x ,  t_1 ) ) ]  dx  \\
  &= - p^{-1}     {\ds  \int_{  \ar  ( E_1 + t_2 E_2) }   \ze (x, t_1 )  | \nabla u ( x, t_2 )  |^{-1}   
\lan \nabla   f( \nabla u  ( x,  t_1 ) ),   \, \nabla u  ( x,  t_2 )  \ran \, d \mathcal{H}^{n-1}   }  
\\
  & \to  - {\ds  \int_{  \ar  ( E_1 + t_2 E_2) } }  h_2  ( \mathbf{g}(x ,  E_1 + t_2 E_2 ) )   f( \nabla u  ( x,  t_2 ) ) \, d  \mathcal{H}^{n-1}      
\end{split}
\end{align}
   as   $  t_1 \to t_2  .$   Next, we claim that   
\begin{align}
\label{eqn12.20} 
          \lim_{t_1 \to t_2 }  T_1   =   \int_{ \mathbb{R}^n   \sem (E_1 + t_2 E_2 ) }  \lan   (\nabla  f)   ( \nabla u  ( x,  t_2 ) ),   \nabla u_{t_2} (x)  \ran dx.
\end{align} 
To   prove this claim, first observe that 
\begin{align}
\label{eqn12.21}  
\begin{split}         
 T_1   &=   (t_1 - t_2 )^{-1}  \,  \int_{ \mathbb{R}^n   \sem (E_1 + t_2 E_2 ) } 
   \int_0^1  \frac{d}{ds}  \left[ f    ( s  \nabla u ( x,  t_1 ) + (1-s)  \nabla u (x, t_2 ) ) \right]ds   dx  \\
   &  =  \int_{ \mathbb{R}^n   \sem (E_1 + t_2 E_2 ) } 
    \int_0^1    \lan   ( \nabla f )  ( s  \nabla u ( x,  t_1 ) + (1-s)  \nabla u (x, t_2 ) ) ,   \nabla  \ze    ( x ,  t_1 )   \ran  ds   dx.
    \end{split}
    \end{align}
    From    local  weak  convergence in 
$ W^{1,2} $  of  $  \ze  ( \cdot, t_1 ) $  to  $ u_{t_2} $,  we  have  
\[ 
 \int_{ K } 
    \int_0^1   \lan   ( \nabla f )  ( s  \nabla u ( x,  t_1 ) + (1-s)  \nabla u (x, t_2 ) ) ,   \nabla  \ze    ( x ,  t_1 )   \ran  ds dx  \to  
\int_{ K  }  \lan   (\nabla  f)   ( \nabla u  ( x,  t_2 ) ),   \nabla u_{t_2} (x )  \ran dx 
\]  
 when $ t_1 \to t_2$ for each compact   $ K  \subset  \mathbb{R}^n   \sem (E_1 + t_2 E_2 )$. Thus to prove 
\eqref{eqn12.20} in view of  \eqref{eqn12.21}  it suffices to show for  given $  \ep > 0 $ that if  
\[   
\mathcal{K}(x, s, t_1 )  =  |  \nabla f   ( s  \nabla u ( x,  t_1 ) + (1-s)  \nabla u ( x, t_2 ) ) |  |   \nabla  \ze    ( x  ,  t_1 ) |   
  \]   
  for  $ s \in [0,1], t_1  \in [t_2/2, t_2 ), x  \in  \mathbb{R}^n   \sem (E_1 + t_2 E_2 ), $  then  there exists $  \de > 0 $ small  and $ R > 0 $ large such that 
  \begin{align}  
  \label{eqn12.22} 
\int_{\mathbb{R}^n \sem B ( 0, R ) }  \mathcal{K} (x, s, t_1 )  dx  + \int_{\{x:\, \,   d ( x,  \ar (E_1 + t_2 E_2 ) ) \leq \de \}}  \mathcal{K}(x, s, t_1 )  dx  \leq \ep. 
\end{align} 
Indeed, from  \eqref{eqn3.12},   \eqref{eqn12.7},    \eqref{eqn12.8},  and   Cacccioppoli type estimates for  uniformly elliptic PDE       
in  divergence  form,  we see that if   $ E_1 + t_2 E_2  \subset   B (0, R/2)$,  then      
  \begin{align}
 \label{eqn12.23} 
 \begin{split}
\int_{\mathbb{R}^n \sem B ( 0, R ) }  \mathcal{K} (x, s, t_1 )  dx  &\leq  \ti C   \int_{R}^\infty   r^{ (1-n)/(p-1)}\, dr\\
&= ({\ts  \frac{n-p}{p-1}})  \ti C  R^{(p-n)/(p-1)}    \\
& \leq   \ep/2  
\end{split}
\end{align}
 for  $  R \geq R_0 $ where $ R_0 ,  \ti  C, $ is independent of  $ s, t_1 $ in the above intervals.   Also from    \eqref{eqn12.16a}    we see that
\begin{align}
  \label{eqn12.24} 
\int_{ \{x:\,\,   d ( x,  \ar (E_1 + t_2 E_2 ) ) \leq \de  \} }  \mathcal{K}(x, s, t_1 )  dx \leq  \ti C  \de^{\he} \leq \ep/2
\end{align}
for  $ \de \leq \de_0 $  where  $  \de_0 $  is  independent of $ s, t_1 $  in the above intervals.   From  \eqref{eqn12.21}-\eqref{eqn12.24}  we conclude        \eqref{eqn12.20}.                        
 
From  \eqref{eqn12.21} and   \eqref{eqn12.16a}  in the closure of  $ \mathbb{R}^n \sem ( E_1 + t_2 E_2 ),   \mathcal{A} = \nabla f$-harmonicity of   $ u ( \cdot, t_1)$, $p$-homogeneity of  $ f, $ and   \eqref{eqn12.16}, \eqref{eqn12.17},  we deduce that 
\begin{align}
\label{eqn12.25} 
 \lim_{t_1 \to t_2 }  T_1   =  p  \int_{\ar (E_1 + t_2 E_2 ) }  h_2 (\mathbf{g}( x,  E_1 + t_2  E_2  ) )       f  ( \nabla  u  (x,  t_2 )  )  d\mathcal{H}^{n-1}.
 \end{align} 
Combining  \eqref{eqn12.25} and \eqref{eqn12.19}, we conclude  from  \eqref{eqn12.18} that 
\begin{align}
 \label{eqn12.26} 
 \begin{split}
 \frac{d}{dt}    \mbox{Cap}_{\mathcal{A}} (  E_1 + t E_2 )  \,  |_{t=t_2}&=\lim_{t_1  \to t_2} \frac{\mbox{Cap}_{\mathcal{A}} ( E_1 + t_2 E_2)  -  \mbox{Cap}_{\mathcal{A}} ( E_1 + t_1 E_2) }{t_2 - t_1}\\ 
  &= (p-1)   \int_{\ar  (E_1 + t_2 E_2 ) }   h_2 (\mathbf{g} ( x,  E_1 + t_2  E_2  ) )       f  ( \nabla  u  ( x,  t_2 )  )  \,  d\mathcal{H}^{n-1}.
\end{split}
\end{align}
 Now,  if  $ 0 < s <  1 $  and  $ t  =  s/ ( 1 - s ), $  then     
\begin{align}
 \label{eqn12.27}    
 \begin{split}
 \mbox{Cap}_{\mathcal{A}} ( E_1 + t E_2) &= (1-s)^{p-n}    [  \mbox{Cap}_{\mathcal{A}} ( (1-s) E_1 + s E_2)  ] \\
 &= (1-s)^{p-n} \ph (s)^{n-p} .   
\end{split} 
 \end{align}
 where  $ \ph $ is concave on [0,1] thanks to  Theorem \ref{theorem1.4} so Lipschitz and differentiable off  a  countable set.  From this observation,  the chain rule,  and  \eqref{eqn12.26}-\eqref{eqn12.27} we see that  Proposition  \ref{proposition12.1} is valid  under assumption \eqref{eqn12.2}  except for at most a countable set of  $ t \in [0, \infty)$.     Also,  from  Proposition \ref{proposition11.1} and properties of support functions  we  observe that 
\begin{align}  
 \label{eqn12.28} 
  \psi (t) =(p-1)   \int_{\ar (E_1 + t  E_2)}   h_2 ( \mathbf{g}(x,  E_1 + t E_2) )\,    f(\nabla u (x, t ) )  d\mathcal{H}^{n-1}   
\end{align}  
   is   continuous as  a   function of $ t $ on $[0, \infty]$.   From the Lebesgue differentiation  theorem we conclude that   Proposition 
 \ref{proposition12.1} is valid under assumption \eqref{eqn12.2}.

     We next remove the assumption \eqref{eqn12.2}. To this end, choose  sequences of uniformly bounded convex domains  $\{E_1^{(k)}\}_{k\geq 1}$ and $\{E_2^{(k)}\}_{k\geq 1}$  with  $ E_i \subset  E_i^{(k)} $  for $ i = 1, 2 $  and $k=1,2,\ldots,$ satisfying   \eqref{eqn12.2}   with $ \ar  E_i $  replaced by  $  \ar E_i^{(k)} , i = 1, 2$ and $k=1,2,\ldots,$.  We also choose these sequences so  that  $  E_i^{(k)} $  converges to  $ E_i $ in the sense of Hausdorff distance as  $ k  \to \infty$. Let  $  \psi_k ( t ) $ denote the function in 
 \eqref{eqn12.28} with  $ E_i, $ replaced by $ E_i^{(k)}$.   Given  $ 0 < a  < \infty $ we claim   there exists  $ l =  l (a), M =  M (a),   $  such that   for  $ k  \geq  l, $  we have  
 \begin{align}   
 \label{eqn12.29}   
 0 <  \psi_k ( t )   \leq   M   \quad \mbox{for} \quad  t  \in  [0,a].  
 \end{align} 
 To verify this  assertion   fix  $k,  t,  $  let    
 \[
 E_0   =  E_1^{(k)}  + t  E^{(k)}_2
 \]
   and let  $ h_0, \mathbf{g}_0, u_0  $  be the support, Gauss map,   and  $  \mathcal{A}  =   \nabla  f$-capacitary functions corresponding to  $ E_0. $   Applying  Proposition \ref{proposition12.1}  in this  case  with $ E_1,  E_2,  t,  $  replaced by  $ E_0,  E_0, 0, $  and  using the fact  that  
\[    
\mbox{Cap}_{\mathcal{A}} ( (1+t)  E_0)  =   (1+t)^{n-p}  \mbox{Cap}_{\mathcal{A}} (  E_0) 
\] 
we get   
\begin{align}
 \label{eqn12.30}  
 \mbox{Cap}_{\mathcal{A}} (  E_0)    =  \frac{p  - 1}{n-p}  \int_{\ar E_0 }  
 h_0 (\mathbf{g}_0(x,E_0))  \,   f ( \nabla u_0(x)) \, d\mathcal{H}^{n-1}.  
\end{align} 
Since  $  E_0 $  is  uniformly bounded and  $ h_0  \geq \min_{\mathbb{S}^{n-1}}  h_1  >  0, $  it follows from \eqref{eqn12.30} and properties of capacity, support functions,  that   \eqref{eqn12.29}  is  true.     From   \eqref{eqn12.29}, \eqref{eqn12.30}, Proposition 
\ref{proposition11.1}, Proposition 
\ref{proposition12.1} in the smooth case,  and  the  Lebesgue dominated  convergence  theorem  we conclude  that 
\begin{align}
\label{eqn12.31}   
\begin{split}
\mbox{Cap}_{\mathcal{A}} ( E_1 + t  E_2)  -   \mbox{Cap}_{\mathcal{A}} (  E_1 )  
&=  {\ds  \lim_{k\rar \infty}   [ \mbox{Cap}_{\mathcal{A}} ( E^{(k)}_1 + t  E^{(k)}_2)  -   \mbox{Cap}_{\mathcal{A}} (  E^{(k)}_1 ) ] } 
\\ 
 & = \lim_{k \rar \infty}     \int_0^t \psi_k (s) ds=   \int_0^t \psi (s) ds.
\end{split}
\end{align}
  Also $\psi $ is  continuous on $ [0, \infty)$   by   Proposition  \ref{proposition11.1} 
so  \eqref{eqn12.31} and  the Lebesgue differentiation theorem yield  Proposition 
\ref{proposition12.1} without assumption  \eqref{eqn12.2}.   
     \end{proof} 
          
     \begin{remark}  
\label{remark12.2}
Finally,  we remark   that  Proposition \ref{proposition12.1} remains valid   for  $ t_2 > 0 $  if    we  assume only  that  $0 \in E_1,  $  rather than  
$0$  is in the interior of  $ E_1 $  (so  $ \mathcal{H}^n(E_1)=0 $  is possible but from the definition of $E_2$ we still have $0$ in the interior of $E_2$).  To handle  this case  we    put  $ E'_1 = E_1 + t_2 E_2$ and 
 $ E'_2   = E_2 .$   Then    $ E_1' ,  E_2' $  are compact convex  sets and  $0$ is in the interior of  $ E_1'  \cap E_2' .$       Applying  Proposition \ref{proposition12.1}  with  $ E_1, E_2 $  replaced by  $  E_1'  ,  E_2'  $  respectively     and   at  $ t_2=0 $  we obtain  the above generalization  of     
Proposition \ref{proposition12.1}. 

We  also note  that if  $ E_1 $ has  interior  points and  
$ E_2  =  B (0, 1 ), $  then from the  Brunn-Minkowski inequality we  have  for fixed 
$ p,  1< p   < n, $  and  $ t \in (0, 1) $ that  
\[  
\mbox{Cap}_{\mathcal{A}} ( E_1 + t  E_2)^{1/(n-p)}  -   \mbox{Cap}_{\mathcal{A}} (  E_1 )^{1/(n-p)}  \geq  t  \mbox{Cap}_{\mathcal{A}} (  E_2 )^{1/(n-p)} .
\] 
Dividing  this  inequality by $ t $ and letting $ t \to 0 $  we get from Proposition \ref{proposition12.1} and the  chain rule that   
\begin{align}  
\label{eqn12.32}     
\mu_{ E_1 } ( \mathbb{S}^{n-1} ) \geq  c^{-1}  \mbox{Cap}_{\mathcal{A}} (  E_1 )^{\frac{ n - 1  - p  }{n-p} }
\end{align}
  where $ c \geq 1 $ depends only on the  data.       

\end{remark}   
\setcounter{equation}{0} 
\setcounter{theorem}{0}
\section{ Proof of  Theorem \ref{mink}}  
\label{section13}
For  use  in  proving  Theorem \ref{mink}  we  shall  need the following Lemma. 
 \begin{lemma} 
 \label{lemma13.1}  
 Let  $ \hat E $ and   $  \hat E_l, l = 1, 2, \dots,  $  be  a  sequence of  uniformly bounded  compact convex sets with  $  \hat E_l   \to  \hat  E  $ in the Hausdorff distance sense as $ l  \to \infty.  $  Then   
\begin{align}
    \label{eqn13.1}  
    \lim_{l\to \infty}    \mbox{Cap}_{\mathcal{A}}  (\hat E_l) 
   =  \mbox{Cap}_{\mathcal{A}}  (\hat E) .  
   \end{align} 
   \end{lemma} 
   \begin{proof}   
   Let  $   \hat u_l $  be the capacitary function      for $ \hat E_l $   and  let  $ \hat \nu_l $  be the corresponding capacitary  measure  for  $ \hat E_l,   l = 1, 2, \dots $  defined as in \eqref{eqn2.6} $(i)$  relative to  $ 1 - u_l. $  From   \eqref{eqn3.4} $(a)$ we deduce  that     
\begin{align}
 \label{eqn13.2} 
 \hat  \nu_l (\hat E_l ) = \mbox{Cap}_{\mathcal{A}} ( \hat{E_l} )\quad     \mbox{ for } \quad l = 1, 2, \dots.  
     \end{align}      
    From   Lemmas  \ref{lemma2.2}, \ref{lemma2.3},    \eqref{eqn3.4} $(b),$  and  Ascoli's  theorem  we  see that  a  subsequence of  $\{( \hat u_l ),  ( \nabla \hat u_l ) \}$ converges uniformly on compact subsets of  $  \mathbb{R}^n  \sem \hat E  $  and  locally in $  W^{1,p} ( \mathbb{R}^n )  $    to   $\{\hat u, \nabla \hat u\}$  with    $ \hat u $  an  $ \mathcal{A}$-harmonic function   in 
$  \mathbb{R}^n  \sem \hat E  $. Moreover, both $\hat u$ and  $\nabla \hat u $ are locally in 
    $  W^{1,p} ( \mathbb{R}^n ).  $   On the other hand, taking a   subsequence of  the subsequence we used  to get $\{\hat u, \nabla \hat u \}$  if necessary  and using local  uniform boundedness in  $ W^{1,p} $  of this sequence   we may also  assume that  $\{\hat \nu_l\}_{l\geq 1}$  converges weakly to  a  measure  $ \hat \nu $   with support in  $ \bar B (0, \rho ). $  Replacing $ \ti u  $ in  \eqref{eqn2.6} $(i)$ by   $  \hat u_l  $   and  taking limits in  \eqref{eqn2.6}  $(i) $  of  the above subsequence,  we then get    
\begin{align}
 \label{eqn13.3} 
   \int   \lan  \mathcal{A} ( \nabla \hat u ),  \nabla \ph \ran  dx =
 \int \ph \,  d \hat \nu\quad \mbox{ whenever }  \quad \ph \in C_0^\infty  (  \mathbb{R}^n ).
 \end{align}  
    
To  prove  Lemma  \ref{lemma13.1} we consider  two  cases.  

\noindent {\bf Case 1:} If  $\mathcal{H}^{n-p} ( \hat E )< \infty$   then  $  \mbox{Cap}_p ( \hat E ) = 0  $  so   $1-  \hat u $  extends to a non-negative  
        $ \mathcal{A}$-harmonic function in  $ \mathbb{R}^n $  (see \cite[Chapter 2]{HKM}) with
        \[
          1 -  \hat  u ( x ) \to  0\quad \mbox{uniformly as}\quad   |x|  \to\infty. 
          \]  
          Using the maximum principle it then follows that 
    $ \hat u  \equiv 1. $    So  from   \eqref{eqn13.3} we see that  
    $ \hat \nu  (  \mathbb{R}^n ) = 0. $   Since  every  subsequence of  capacitary measures  contains a  subsequence  converging weakly to the  0 measure  we  conclude from \eqref{eqn13.2}  that  
    \eqref{eqn13.1}  is true in this case and 
     the  limit is zero. 
     
\noindent {\bf Case 2:} If   $\mathcal{H}^{n-p} (\hat E )  =  \infty$ then using  Hausdorff  convergence of $ \hat  E_l $ to $ \hat E $  we see that the constants in  \eqref{eqn2.5}  with $  \ti E $ replaced by  $ \hat E_l $   can be chosen independent of $ \hat  E_l $ provided  $ l \geq l_0  $ and  $l_0 $ is large enough.   From this observation we deduce that $  \hat \si, c,  $   in  \eqref{eqn2.4} $(ii)$ applied to $1-  \hat u_l $  are  independent of  $ l $ for $ l  \geq l_0. $  Thus,  $  \{\hat u_l\}_{l\geq l_0}^\infty  $ is locally 
    H\"{o}lder continuous with uniform constants on $ \mathbb{R}^n. $  Using  this fact we find  that  $ \hat u \equiv 1 $ on $  \hat E. $    We conclude  that  
    $ \hat u $ is the capacitary function for $ \hat E $  and  $ \hat \nu  $  the corresponding capacitary measure.   Since  every subsequence of  measures  has  a  subsequence  converging weakly to the  capacitary measure for $  \hat E $,  we  deduce from \eqref{eqn3.4} $(a)$  for 
$ \hat \nu $  and    \eqref{eqn13.2}  that   Lemma  \ref{lemma13.1} is true in  this case also.  
\end{proof}      
\subsection{Proof of  existence in    Theorem \ref{mink} in the  discrete case}  
Finally we are in  a  position to start the proof of existence of the measure in Theorem \ref{mink} in the discrete case.  
     Let   $  c_1,  c_2,  \dots,  c_m  $  be positive numbers  and  $   \xi_i \in  
\mathbb{S}^{n-1}$,  for  $   1 \leq  i   \leq  m . $     Assume that  $ \xi_i \not = \xi_j,  i \not = j, $  and  let  $ \de_{\xi_i} $  denote the  measure with  a point mass at  $ \xi_i. $  
Let  $ \mu $  be a measure on   $ \mathbb{S}^{n-1}$  with    
\[   
\mu ( K )  = \sum_{i=1}^m   c_i   \de_{\xi_i} (K) \quad \mbox{whenever} \quad K\subset \mathbb{S}^{n-1} \, \,  \mbox{is a Borel set}.   
\]   
We also  assume  $ \mu $  satisfies  \eqref{eqn7.1} $ (i)$ and $(ii). $   That is,  
\begin{align} 
\label{eqn13.4}   
\sum_{i=1}^m   c_i  \,  |  \lan \he,  \xi_i \ran |   >   0  \quad \mbox{for all}  \quad \he \in  \mathbb{S}^{n-1}     
\end{align}
and
\begin{align}
\label{eqn13.5} 
\sum_{i=1}^m   c_i  \,  \xi_i  =  0. 
\end{align}    
For technical reasons and following  either \cite{J} or \cite{CNSXYZ} we also assume 
\begin{align}
\label{eqn13.6}  
 \mbox{either}\, \, \, \mu  ( \{\xi\} ) \, \, \, \mbox{or} \, \, \, \mu (\{- \xi\} )= 0  \quad \mbox{whenever} \quad \xi \in \mathbb{S}^{n-1}.
 \end{align}   
This condition will be removed   in  our  proof  of   the general case of  Theorem \ref{mink}.    For   $  \mu $  as above and $ p \not = n - 1, $   we    show    there is  a   compact   convex  polyhedron  $E$  with  $0$  in the interior of  $ E $   and  
\[
\mu (K)  = \int_{\mathbf{g}^{-1}(K)}  f(\nabla U ) \, d\mathcal{H}^{n-1} \quad \mbox{whenever} \quad K\subset\mathbb{S}^{n-1} \, \, \mbox{is a Borel set}
\]
where   $ \mathbf{g}$  is the  Gauss map for  $ \ar E $  and    $U$ is the $ \mathcal{A}$-capacitary function for  $  \rn{n}\sem E. $  
 Thus, if  $  F_i $  denotes the  face of  $ \ar E $ with outer normal  $ \xi_i,  1 \leq i  \leq  m,  $  then  $\mathbf{g}(F_i) = \xi_i  $  and  
\begin{align}  
\label{eqn13.7}     
\mu ( \{\xi_i\} )  = c_i   =  \int_{F_i} f(\nabla U ) \, d\mathcal{H}^{n-1}  \quad \mbox{for}\quad  1 \leq i   \leq m.
\end{align} 
If $ p  = n - 1 $ our results are less precise.   For given  $\mu $  as above we show the existence of  $ E $  as above  with $ \mbox{Cap}_{\mathcal{A}} ( E ) = 1 $   and corresponding  capacitary function $ U$  such that  for some $ b \in (0, \infty), $   \eqref{eqn13.7}  holds with  $ f $ replaced by $ b f.$  
   
To set up the minimization problem that will  eventually lead to  \eqref{eqn13.7}  (as in  
\cite{J,CNSXYZ})  fix   $ m,  (\xi_i)_1^m ,  (c_i)_1^m, $ and  let 
$ q=(q_1,\ldots, q_m)\in \rn{m} $ with  $ q_i   \geq 0,  1 \leq i  \leq m. $ Let  
\[  
E(q):=  \bigcap_{i=1}^m  \{x:\, \,  \lan x,  \xi_i \ran  \leq q_i  \}    \quad \mbox{and}\quad  \He:=   \{E (q):\,\,   \mbox{Cap}_{\mathcal{A}}  (E(q)) \geq  1 \}.
 \]
We also set     
\[ 
\ga (q) =  \sum_{i=1}^m  c_i  \, q_i \quad \mbox{with}\quad     \ga =     \inf \{ \ga (q):\,\,  E (q)  \in  \He \}. 
\]    
We want to show there exists 
\begin{align} 
\label{eqn13.8}    
\breve q  =  (\breve q_1, \dots, \breve q_m),\, \,  \breve  q_i  > 0 \, \, \mbox{for}\, \, 1 \leq i  \leq m \, \,  \mbox{with}\, \,  \ga ( \breve q ) = \ga\,\,  \mbox{and} \,\,  \mbox{Cap}_{\mathcal{A}} (   E (\breve q  )  )= 1.
\end{align} 
Once \eqref{eqn13.8}  is proved we can use the same argument as in \cite[Section 5]{J} or \cite[Section 6]{CNSXYZ} to  get  \eqref{eqn13.7}.

 To begin the proof of \eqref{eqn13.8}   we first note  that  if  $ E (q)\in \He, $  then  $ E(q) $  is   a closed convex set.   Also  we note from   \eqref{eqn13.5} that 
 \[   
 \int_{ \mathbb{S}^{n-1}}   \lan \tau, \xi \ran^+  d \mu ( \xi )           =  \int_{ \mathbb{S}^{n-1}}   \lan \tau, \xi \ran^-  d \mu ( \xi )  \quad \mbox{whenever}\quad    \tau \in \mathbb{S}^{n-1}
  \]   
  where $ a^+ =  \max{(a, 0) } $  and  $ a^-   = \max{(- a, 0)} . $   From this note and  \eqref{eqn13.4} 
 we see that   for some $ \ph >  0, $ 
 \begin{align}  
 \label{eqn13.9}
  \ph  <    \int_{ \mathbb{S}^{n-1}}   \lan \tau, \xi \ran^+  d \mu ( \xi )   \quad \mbox{for all} \quad  \tau  \in \mathbb{S}^{n-1}.
 \end{align}
If $ r \tau  \in  E (q ) , $  it  follows from \eqref{eqn13.9}  that  $ r  \leq  \ga (q ) /\ph $ so 
\begin{align}
\label{eqn13.10}  
 E(q)   \subset   \{ x:\,\,  |x| \leq  \ga (q) /\ph \}. 
 \end{align}  
From \eqref{eqn13.10}    we deduce the existence of       $ q^l  = (q_1^l,  \dots, q_m^l ),  q_i^l \geq 0,   1 \leq i \leq m, $  such that  $  E_l = 
E (q^l)$ for $l = 3, 4,  \dots,  $ is  a  sequence of  uniformly bounded compact convex sets in   $ \He,  $   with    
\[
\hat  q  =  \lim_{l\to \infty} q^l \quad \mbox{and}\quad  \lim_{l \to \infty}  \ga ( q^l )  =  \ga =  \ga ( \hat q ). 
\]   
From   finiteness of  $ \ga $   we also may assume (by taking further subsequences if necessary) that       $   E_l \to  E (\hat q )=  E_1, $  a  compact  convex set  containing $0$,  uniformly in the Hausdorff distance sense. From  Lemma  \ref{lemma13.1} we observe that   
\begin{align}
\label{eqn13.11} \lim_{l \to \infty}   \mbox{Cap}_{\mathcal{A}} (E_l ) =  \mbox{Cap}_{\mathcal{A}} (E_1 ). 
\end{align} 
It follows that   $ \mbox{Cap}_\mathcal{A}(  E_1 )  \geq  1$  and   $  E_1 \in \He. $   In fact    $ \mbox{Cap}_{\mathcal{A}} (   E_1 )  = 1, $  since otherwise we  would  have 
$ \ga (\ti q) < \ga(\hat q)$  for  $  \ti E  = \ti E (\ti q) \in \He $   where  for 
$ j \in \{1,2, \dots, m \}, $    
\[   
\ti q_j  =  \frac{\hat q_j }{  \mbox{Cap}_{\mathcal{A}} (  E_1 )^{1/(n-p)}}. 
\] 

Next we consider two cases.  
 
\noindent {\bf Case 1:} First  suppose  that  $ z $ is an  interior point of $  E_1.  $   Then  $  \breve E   =    E_1  -  z   \in   \He, $  since  the distance   from  $0$ to  each plane composing the boundary of  $ \breve E $ is positive. Also,      $ \breve  E  $  has  $ \mathcal{A}$-capacity $1$  and   if  $  \breve  E  =  E (\breve  q ), $ then from  \eqref{eqn13.5}  we see that  $ \ga (\breve q ) = \ga.      $   Thus,      \eqref{eqn13.8}  is  valid  in this case. 
      
\noindent {\bf Case 2:} If $   E_1 $  has empty interior, then  from convexity of  $  E_1 $  and \eqref{eqn13.6} we find that 
      $    E_1 $ is  contained in a  $ k < n - 1 $  dimensional plane  and  $  0  < \mathcal{H}^k (E_1) 
<  \infty. $   Moreover we must have $ p  >  n  - k $ since otherwise as mentioned in  Lemma \ref{lemma13.1} we have          
  $  \mbox{Cap}_{\mathcal{A}} (  E_1 )  = 0 $.  Also,  we may assume $0$ is an interior point of  $  E_1 $   relative to the  $ k$-dimensional plane containing  $   E_1 $  since otherwise we  consider $  E_1 - z $ for some $ z $ having this property and  argue as above.       In this case from the definition of  $ \He $  and  \eqref{eqn13.4},   we see that  there exists  a subset, say  $  \La $  of  $ \{ 1, \dots, m \} $  with $ \hat q_i = 0 $ when 
$ i  \in \La. $   From  \eqref{eqn13.6}  we deduce that  $  \La  $  has  cardinality at  least $3$.  Also since a point has $\mathcal{A}$-capacity zero  we see that   $  \{1, \dots, m \} \sem \La $  contains  at least two points.  
  Moreover,  since  $    E_1 $  is  a minimizer we observe  that  if  $  s    \not \in  \La, $  then  $ \hat q_s \not = 0$   and  
\[    
\{ x:\,\,  \lan x, \xi_s  \ran = \hat q_s  \}  \cap   E_1   \neq  \es.
\]   
Let     $ a   =   \frac{1}{4}    \min \{  \hat q_i  :   i   \not  \in  \La  \} $  and  for   small $ t > 0 $   let  
\begin{align} 
\label{eqn13.12}  
\begin{split}
\ti E  (t)  &=   {\ds \bigcap_{i = 1}^m  \{ x:\,\, \lan x,  \xi_i \ran    \leq  \hat q_i + a t  \} } 
\\
E_2   &=  {\ds \bigcap_{i = 1 }^m }  \{ x: \,\,\lan x,  \xi_i \ran  \leq   a   \}  . 
\end{split}
\end{align}  
 Put  
\begin{align}
\label{eqn13.13}   
E_t     =       \frac{ \ti  E  ( t )  }{ \mbox{Cap}_{\mathcal{A}} ( \ti E  (t) )^{ 1/(n-p)}}.   
\end{align}
We note that, in view of \eqref{eqn13.12},  $  E_t  =   E ( q (t) )  $  where $q(t)=(q_1(t), \ldots, q_m(t))$ and 
\begin{align}    
\label {eqn13.14}     
q_j (t) =  \frac{\hat q_j +  a t}{  \mbox{Cap}_{\mathcal{A}} ( \ti  E (t) )^{1/(n -p)}}\quad  \mbox{for}\quad  1 \leq j  \leq m. 
\end{align}  
 From properties of  $ \mathcal{A}$-capacity, we have $  \mbox{Cap}_{\mathcal{A}}  (E_t ) = 1 $  so  $ E_t \in \He. $ To get a  contradiction to our assumption that $   E_1 $ has empty interior  we show that  
\begin{align}  
\label{eqn13.15}  \ga ( q (t) )  <   \ga  \quad \mbox{for some  small}\, \,  t > 0. 
\end{align} 
 To prove \eqref{eqn13.15},   we first note that     
 $   E_1   +  t   E_2   \subset    \ti E (t )  $  for  $ t \in (0, 1) $ so    
 \[
    \mbox{Cap}_{\mathcal{A}} ( E_1 + t E_2  ) \leq    \mbox{Cap}_{\mathcal{A}} ( \ti  E (t ) ). 
\]      
    From this  inequality and    \eqref{eqn13.13},  \eqref{eqn13.14},    we conclude  that  to prove     \eqref{eqn13.15}  it suffices  to show  if     
    \[
      k ( t )  =  \mbox{Cap}_{\mathcal{A}} ( E_1  +  t  E_2  )^{- 1/(n -p)}   \sum_{i = 1 }^m     c_i  ( \hat q_i  + at )     
      \]  
   then  
\begin{align}   
    \label{eqn13.16}  k (t)  <  \ga \quad  \mbox{for}\, \,    t  > 0 \, \, \mbox{near  0}.  
\end{align}  
     
To  prove   \eqref{eqn13.16}, we   let, as in section \ref{section12}, $ u ( \cdot, t )  $  be the  $ \mathcal{A} = \nabla f$-capacitary function for    
$E_1  + t E_2    $  and let   $\mathbf{g} ( \cdot,  E_1 + t E_2  )     $ be the   Gauss map for  $ \ar (E_1 + t E_2 )  $    while   $ h_1, h_2  $  are   the support  functions for  $  E_1,  E_2, $  respectively.  Then  from   Remark \ref{remark12.2}  and  Proposition \ref{proposition12.1}  we have    for  $ t  \in   (0, 1), $  
   
\begin{align}
\label{eqn13.17}   
\frac{d}{dt}   \mbox{Cap}_{\mathcal{A}} (  E_1  +  t E_2 ) =    (p-1)  \int_{  \ar  ( E_1 +  t E_2  ) }   h_2  ( \mathbf{g}(x, E_1  +  t E_2)  )   f ( \nabla u (x, t ) ) d\mathcal{H}^{n-1} . 
\end{align}  
     Next we prove    
\begin{proposition}  
\label{proposition13.2}  
\begin{align} 
\label{eqn13.18} 
 \lim_{\tau \to 0}   \int_{\ar  (E_1 + \tau E_2 )}   h_2   ( \mathbf{g}( x ,  E_1  + \tau E_2 ) )  f ( \nabla u (x, \tau) ) d\mathcal{H}^{n-1}  \, = \, \infty.   \end{align}
 \end{proposition} 
Assuming Proposition  \ref{proposition13.2} we get  \eqref{eqn13.15} and so    a contradiction to our assumption that 
$ E_1 $  has empty interior as follows.  First observe from \eqref{eqn13.17}  that  
\begin{align}
\label{eqn13.19} 
\begin{split}
\left. (n-p)  [\mbox{Cap}_{\mathcal{A}} (  E_1 + t E_2 )]^{1\, + \, 1/(n -p)} {\ds  \frac{d}{dt} } k  (t) \right|_{t=\tau}  =   (n-p)  \mbox{Cap}_{\mathcal{A}} (  E_1 + \tau  E_2 )\sum_{i   = 1}^m  c_i  a\\ 
 -  (p -  1)   [  \sum_{i  = 1}^m     c_i ( \hat q_i   +  a \tau  )]  \int_{  \ar ( E_1  + \tau  E_2 ) } 
 h_2   ( \mathbf{g}(x,  E_1 + \tau  E_2  ) )  f ( \nabla u (x, \tau) ) d\mathcal{H}^{n-1}. 
\end{split}
\end{align}
Now    $  E_1 + \tau  E_2   \to   E_1 $ as $ \tau  \to 0 $  in  the sense of  Hausdorff distance so by  Lemma   \ref{lemma13.1},   we have      
\begin{align} 
\label{eqn13.20}  
\lim_{\tau \to 0}   \mbox{Cap}_{\mathcal{A}} (  E_1 + \tau  E_2   )  = \mbox{ Cap}_{\mathcal{A}} (   E_1  ) =1.  
\end{align}   
Clearly,   \eqref{eqn13.18}-\eqref{eqn13.20}           
   imply for  some  $ t_0 > 0  $  small   that   
\begin{align}     
\label{eqn13.21}    
\left.\frac{d}{dt}  k (t)\right|_{t=\tau} < 0 \quad   \mbox{ for } \, \, \tau \in ( 0, t_0].  
\end{align}
Also, from  \eqref{eqn13.20} we see that  
\[
\lim_{\tau \to 0} k ( \tau ) = \ga.
\]     
From this observation, the mean value theorem from calculus,  and   \eqref{eqn13.21} we conclude that \eqref{eqn13.16}  holds  so   $  E_1 $  has interior points.     \eqref{eqn13.8}  now follows from our earlier remarks.  
 \begin{proof}[Proof of   Proposition  \ref{proposition13.2}]
 Recall  that  $  E_1 $ is  contained in  a  $ k  <  n - 1 $ dimensional plane  and  $ n - k <  p  <  n. $  We assume as we may that    
\begin{align}
\label{eqn13.22}  
E_1    \subset  \{ x = ( x' ,  x'' ) :  x' = (x_1, \dots, x_k )   \quad   \mbox{and} \quad x''  =  (x_{k+1}, \dots  x_n )   = (0, \dots, 0)  \}  = \mathbb{R}^k.  
\end{align}   
Indeed, otherwise  we  can   rotate  our coordinate system to get  \eqref{eqn13.22}   and  corresponding  $ \hat{\mathcal{A}}$-capacitary functions, say   $  \bar u ( \cdot, t )  .   $ 
    Proving    Proposition \ref{proposition13.2}  for  $ \bar  u ( \cdot, t ) $   and transferring back we obtain 
   Proposition  \ref{proposition13.2}.     
   
   We   also note   that
\begin{align}    
\label{eqn13.23}     
\bar B (0,  4 a )   \cap  \mathbb{R}^k    \subset   E_1   \subset  \bar B (0, \rho )    
\end{align}  
which follows from our choice of  $a$ and for some $\rho$ large (depending only on the data).    We shall need the following lemma.

 \begin{lemma}  
 \label{lemma13.3}
 There exists  $   C_1 \geq 1,  $    such   that 
   if   $ \psi = ( p  -  n  + k )/(p-1) $  and $ x \in B (0 , \rho ), $  then   
\begin{align}   
\label{eqn13.24}    
|x''|^\psi   \leq  C_1 ( 1 -  u ( x, t ) )  \quad \mbox{whenever}  \quad    C_1  t  \leq  |x''|
\end{align}  
 where  $  C_1 \geq  1 $  depends on various  quantities but is independent of  $ x \in B ( 0, 2 \rho)$  and $t$.      
 \end{lemma}    
 \begin{proof}[Proof of Lemma \ref{lemma13.3}]  To prove this  Lemma  we  note from   Lemma  5.3  in    \cite{LN4} that  there  exists   an   $ \ti{\mathcal{A}}(\eta ) = - (\nabla f) ( - \eta )$-harmonic  function $  \hat  V $  on    $  \mathbb{R}^n  \sem  \mathbb{R}^k   $ with continuous boundary value $0$   on  $\mathbb{R}^k$  and  $ \hat V   (x)  \approx  | x'' |^{\psi}$ for  $x   \in      \mathbb{R}^n$  where constants depend only on $ p, n, k $  and the structure constants for $ f $.  For  fixed $ t \in (0, 1 )$,  let  $ v =  \max ( \hat V - C_2 t, 0 )$. Then  $  v $ is  $  \ti{\mathcal{A}}(\eta )$-harmonic  in       $\mathbb{R}^n \sem W  $ and  continuous on  $\mathbb{R}^n$   with  $ v \equiv 0 $ on 
     $ W = \{ x :  \hat V (x) \leq  C_2 t \}. $  
    From  the definition of  $ E_1 + t E_2 $ and $ v $  we see  for $ C_2 $ large enough,  depending on $ p, n, k, $  the structure constants for  $ f , $  and  $  \rho,  $  that 
        $ v  = 0 $ on  $ E_1 + t E_2 .$  Also  from   \eqref{eqn13.23}, \eqref{eqn3.4} of  Lemma  \ref{lemma3.2}, Harnack's inequality,  and  the fact that     $ E_1 + t E_2 $  has  $\mathcal{A} $-capacity  $  \geq 1 $  we  find      that  $  C_3  ( 1 -  u  ( \cdot, t  ) )  \geq  v $  on  $  \ar  B ( 0, 2 \rho ) $   where  $ C_3$  has the same dependence as  $ C_2$ and $u(x,t)$ is the $\tilde{\mathcal{A}}$-capacitary function for $ E_1 +  t  E_2$.   Using the maximum principle for  $  \ti{\mathcal{A}}$-harmonic functions it now follows that  $   v  \leq  C_4  ( 1 -  u ( \cdot, t ) )$  in    $  B ( 0, 2 \rho ). $    From this fact and our knowledge of  $  \hat  V $   we get   Lemma  \ref{lemma13.3}.   
\end{proof}  
      
To begin the proof of  Proposition \ref{proposition13.2}   we  assume  $ 0  < t \leq \ti t_0,  $  where $ \ti t_0   < <  a$.  We  also  observe    
 that     $ E_1 +  t  E_2  $   is  a compact  convex set   with  nonempty interior so  from  Corollary \ref{corollary9.6a} and Proposition \ref{proposition8.9}  we find that            for  $ \mathcal{H}^{n-1} $  almost every  $ \hat x  \in  \ar  ( E_1 + t E_2) $  
 \[
\nabla u ( y, t ) \to  \nabla u ( \hat x, t )\quad \mbox{as} \quad  y \to \hat x
\]
 non-tangentially  in  $   \mathbb{R}^n \sem (E_1 + t E_2 ). $  Moreover,  there exists  $ \ti c $   such that     $  B ( \hat x , 4t/\ti c ) \cap  \ar ( E_1 + t  E_2 ) $ is  the graph of  a  Lipschitz  function whenever 
 \[ 
 \hat x \in     B (0, 2a) \cap  \ar ( E_1 + t  E_2 )\quad   \mbox{and} \quad   0  < t \leq  \ti t_0  
 \]
 with Lipschitz constant independent of $ \hat x, t. $    It then    follows  from \eqref{eqn8.18}, \eqref{eqn8.29}$(a)$, and  \eqref{eqn8.49a}-\eqref{eqn8.49b} that   
 \begin{align}  
 \label{eqn13.26}  
\begin{split}
  c   \int_{B ( \hat x, t/\ti c ) \cap \ar (E_1 + t E_2)   }  f ( \nabla u ( \cdot, t ) ) d\mathcal{H}^{n-1} & \geq (  1 - u ( w, t ) )^p   t^{ n - 1 -p}  \\
   & \geq   c^{-1}   \int_{B ( \hat x, t/\ti c ) \cap  \ar ( E_1 + t  E_2)    }  f ( \nabla u ( \cdot, t ) ) d\mathcal{H}^{n-1} 
   \end{split}
\end{align}
 where  $ c$  depends only on the data 
and   $ w = w (\hat x, t ) $ denotes a point in  $ B (\hat  x, t/\ti c ) \cap (\mathbb{R}^{n}\setminus (E_1+t E_2))$ whose distance from  $ \ar (E_1 
+ t E_2  ) $ is   $ \geq  t/ c^2.$   Using   Harnack's inequality in  a  chain of  balls of radius $ \approx t $  connecting $ w $  to a  point 
 $ x \in  B (0, a )  $  with    $ 2 C_1  t  = |x''|   $   we deduce from  \eqref{eqn13.24} of  Lemma 
 \ref{lemma13.3} that  
\begin{align}   
\label{eqn13.26a}  
1  -  u (w, t ) \geq  C^{-1}  t^{\psi}
\end{align}
where $ C $ is independent of $ t \in (0,1). $   
 Using   \eqref{eqn13.26a} in  \eqref{eqn13.26} we obtain   for some $ C'   \geq 1, $ independent of  $ t, 0 < t \leq  \ti t_0,   $     that  
\begin{align}
\label{eqn13.27}    
   C' \int_{B ( \hat x, t/\ti c ) \cap  \ar (E_1  + t   E_2) }  f ( \nabla u ( \cdot , t ) )  d\mathcal{H}^{n-1} \,   \geq  t^{p( \psi - 1) + n-1}. 
\end{align}    
Now since  $  \ar (E_1 +  t E_2 )  \cap  B (0, 2a )  $ projects onto a  set  containing   $  B (0, 2a )  \cap  \rn{k} $   for $ 0 < t \leq  \ti t_0$,    we see there is  a  disjoint collection of  balls   $  B ( \hat x,  t/\ti c ) $  for $ \hat  x \in  \ar ( E_1 + t E_2) $   of  cardinality  approximately  $ t^{-k}  $   for which  \eqref{eqn13.27}  holds.    
            Since  
   \[  
   p  ( \psi - 1)  +  ( n - 1) -  k  =  (k + 1 - n ) / (p-1) <  0 
   \]   
   we conclude from  
    \eqref{eqn13.27}    that  for some $ C^*  $  independent of  small positive $ t $       
 \begin{align} 
 \label{eqn13.28}    
 C^*  \int_{\ar (E_1 + t E_2 ) \cap B(0,2a)}  f ( \nabla u ( \cdot, t ) )    d\mathcal{H}^{n-1}   \,  \geq   t^{ (k + 1 - n)/ (p-1) } \to \infty  \quad \mbox{as} \quad t  \to 0.  
\end{align}    
      Finally note that for  $ 0 < t  \leq \ti t_0, $   
      \[
       \mathbf{g}(x, E_1 + t  E_2   )  \in  \{ \xi_i : i \in \La \}
       \]  
       for   $\mathcal{H}^{n-1}$ almost  every  $ x  \in \ar ( E_1 + t E_2) \cap  B (0, 2a )   $   and      $ h_2  ( \xi_i  )  \equiv   a  $  whenever $ \xi_i \in \La. $    
  From this note and  \eqref{eqn13.28},  we   obtain first the validity  of \eqref{eqn13.18}  in  Proposition  \ref{proposition13.2} and thereupon that  \eqref{eqn13.8} is true.  
   \end{proof}
  Armed with   \eqref{eqn13.8}, we now  complete the  proof  of   existence in  Theorem  \ref{mink} in the discrete case.  Given   $   q^*    = ( q^*_1,  \dots, q^*_m ) \in \mathbb{R}^m $   with  $ q^*_i  > 0$ for $1 \leq i  \leq m, $ we note from   \eqref{eqn13.8}  that   for $ \bar t_0   > 0 $ sufficiently small,   as in the remark following  \eqref{eqn13.15}  that     
$  E ( q^* (t) ) \in   \He $   for  $  0 < t  \leq  t_0, $     where  
\[   
q^* (t )   =   \frac{  (1-t) \hat q  +  t q^* }{\mbox{Cap}_{\mathcal{A}}( (1-t) E (\hat  q)  + t E (q^*))^{1/(n-p)}  } .  
\] 
Also,    $  \ga (q^* (t)) \geq  \ga  $  for  $ 0 \leq t  \leq  \bar{t}_0 $ thanks to  \eqref{eqn13.8}.     Now  as in  \eqref{eqn13.19}, we have           
for  $ \tau > 0 $ small, 
\begin{align}
\label{eqn13.29} 
\begin{split}
  (n-p)  [&\mbox{Cap}_{\mathcal{A}} ( (1-t)   E(\hat q ) + t  E(q^*)  )]^{1\, + \, 1/(n -p)} \left.   \frac{d \ga (q^*(t)) }{dt} \right|_{t=\tau}  \\
  &=   (n-p)  \mbox{Cap}_{\mathcal{A}} ( (1 - \tau)   E (\hat q )  + \tau  E( q^*) )   \sum_{i   = 1}^m  c_i ( q^*_i  -  \hat q_i )    \\
&\hs{.4in}  -    \left[\sum_{i  = 1}^m     c_i ( (1-\tau) \hat q_i   +  \tau  q_i^*  ) \right]  \frac{d}{dt}
\left.  \mbox{Cap}_{\mathcal{A}}   ( (1-t) E ( \hat q ) + t  E ( q^* ) ) \right|_{t=\tau}.
\end{split}
\end{align}
Moreover, using  
 \[
       \mbox{Cap}_{\mathcal{A}}   ( (1-t) E ( \hat q ) + t  E ( q^* )) =   (1- t)^{n-p}    \mbox{Cap}_{\mathcal{A}}   ( E ( \hat q ) + s  E ( q^* ) )  
       \]  
       where  
$  s  =  t/(1-t)$, Proposition \ref{proposition12.1},        and   the  chain rule    we  have  with  $ \hat h, h^* $ the support functions for 
$  E ( \hat q ),  E (q^*), $  and  $ u^* ( \cdot,  s  ) $  the $\mathcal{A}$-capacitary function for  $  E (\hat q ) + s E ( q^*),  $   
\begin{align}
 \label{eqn13.30}  
 \begin{split}
( 1 - t )^{p+2-n} & \frac{d}{dt}
\left.  \mbox{Cap}_{\mathcal{A}}   (  (1-t) E ( \hat q ) + t  E ( q^* )  ) \right|_{t=\tau}   =   -  (n-p)  ( 1 - \tau)  \mbox{Cap}_{\mathcal{A}}   (  E ( \hat q ) + \tau  E ( q^* ) )  \\
& \hs{.4in}+ (p-1)   \int_{\ar ( E(\hat q ) + \tau  E (q^*) ) }   h^* ( \mathbf{g} ( \cdot,  E (\hat q )  +  \tau  E 
( q^* )  )  ) \, f ( \nabla u^* ( \cdot, \tau ) )   d\mathcal{H}^{n-1}   . 
\end{split}
\end{align}
 Using  \eqref{eqn12.30}  with $ E_0  =  E ( \hat q )  $  in  \eqref{eqn13.30} and   letting  $ t \rar 0, $   we conclude from  \eqref{eqn13.29},  \eqref{eqn13.11},  and   the mean value theorem in elementary calculus that    
\begin{align}
 \label{eqn13.31} 
\begin{split} 
  0   &\leq   ( n - p)  \lim_{\tau \rar 0} \frac{ d \ga ( q^* (\tau) )}{d\tau}  \\
  &=  (n-p)    \sum_{i=1}^m  c_i  ( q^*_i -  \hat q_i )  -   (p-1) \ga  \, {\ds  \int_{\partial E (\hat q) }  
 ( h^*  -  \hat h)  ( \mathbf{g} ( x, E ( \hat q) ) )   \,  f  ( \nabla  u^*  ( x, 0 ) )  d \mathcal{H}^{n-1} }  \\
 &=    { \ds (n-p)  \sum_{i=1}^m  c_i  ( q_i^* - \hat q_i ) -  (p-1) \ga  \,  \sum_{i=1}^m  ( q_i^* - \hat q_i ) \int_{ \mathbf{g}^{-1} ( \xi_i ,  E ( \hat q )  ) } f ( \nabla u^* ( x, 0 ) ) d \mathcal{H}^{n-1} }  
\end{split}
\end{align}
provided  $ q^*  $  is  near enough $ \hat q. $   Clearly,  the  possible choices of   $ q^* -  \hat q $  contain an  open neighborhood of  $0$.   Thus  
\begin{align}
 \label{eqn13.32} 
\begin{split} 
 c_i  =\left(\frac{p-1}{n-p}\right) \,  \ga  \,   \int_{ \mathbf{g}^{-1} ( \xi_i ,  E ( \hat q )  ) } f ( \nabla u^* ( \cdot, 0) )  \, d \mathcal{H}^{n-1} \quad  \mbox{for}\, \,   1 \leq i \leq m. 
 \end{split}
 \end{align}
From  \eqref{eqn13.32}  and  $ p$-homogeneity of  $ f $   we  find that if  $ p \not = n - 1, $  and 
$  E  =  \phi   E ( \hat q ) $   where  $  \phi ^{n  -  p - 1 } =  { \ts (\frac{p-1}{n-p}) } \,  \ga    , $  and  $ U  $  is  the  $  \mathcal{A}  =  \nabla f$-capacitary function corresponding to $ E, $    then  \eqref{eqn13.7} holds. If  $ p = n-1, $  put $  b  =  { \ts (\frac{p-1}{n-p}) } \,  \ga , $   $  E = E (\hat q), $ and $ U = u^* ( \cdot, 0). $    This completes the proof of existence in the discrete case when     \eqref{eqn13.4}-\eqref{eqn13.6} hold. 
\begin{remark} 
\label{rmk13.4}  
We  note for later use from \eqref{eqn13.32}, the definition  of $ \ph, $    \eqref{eqn12.30} with $ E_0 = E ,   $  and  \eqref{eqn12.32}  that if $ r $ denotes the radius of  a  ball with  $ \mathcal{A}$-capacity $1$ and  $ h $ is the support function for $ E $  as in  \eqref{eqn13.7}  when $ 1 < p < n, p \not = n - 1, $\  or its amended form when  $ p  = n-1,$   then  
\begin{align}
\begin{split}
 \label{eqn13.33} {\ds \int_{ \mathbb{S}^{n-1} }} h ( \xi   ) d \mu  ( \xi  ) & = \frac{n-p}{p-1}  \, \mbox{Cap}_{\mathcal{A}}  (E)\\
& \begin{cases}
    =\frac{n-p}{p-1}  \, \mbox{Cap}_{\mathcal{A}}  (E) \leq c    \mu  (E)^{\frac{n-p}{n  - p - 1} } &  \mbox{when} \, \,  1 < p < n  - 1, \\
     \leq \ga \, =  b/(n-2)     \leq r   \mu  ( E ) & \mbox{when}\, \,  p = n - 1, \\
     =    \frac{n-p}{p-1}  \, \mbox{Cap}_{\mathcal{A}}  (E)  \leq  c    (\mbox{diam}(E))^{n - p} & \mbox{when}\, \,  n - 1<p<n.
\end{cases}
\end{split}
\end{align}
\end{remark} 
\subsection{Existence in Theorem \ref{mink} in   the continuous case}  
Armed with existence in Theorem  \ref{mink} in the discrete case, we now  consider existence when  $ \mu  $  is  a  finite positive measure on $ \mathbb{S}^{n - 1}$  satisfying  \eqref{eqn7.1}.  We  choose  a  sequence of  discrete measures     
$ \{\mu_{j}\}_{j\geq 1}  $  satisfying  \eqref{eqn13.4}-\eqref{eqn13.6} when $ p $ is fixed $ 1 < p < n  $  with  
\[
\mu_j  \rightharpoonup \mu\quad   \mbox{weakly as}\quad   j  \to  \infty.
\]   

If  $ p  \neq n - 1$, we let  $  E_j, j = 1, 2, \dots,  $  be a  corresponding sequence of  compact convex sets with nonempty interiors and $ \mathcal{A}$-capacitary  functions  $ U_j $   for which    \eqref{eqn13.7}  holds  at support points of  $ \mu_j. $   

If $ p = n - 1 $, we choose  $ E_j   $  and  corresponding capacitary function $ U_j $  with   $ \mbox{Cap}_{\mathcal{A}} (E_j)   =  1 $  satisfying   the  amended form  of   \eqref{eqn13.7} for $ j = 1, 2,  \dots $  Then  $  \mu_j  =   b_j  \ti \mu_j $  for some 
$ b_j > 0, $ where  $ \ti \mu_j $  is the measure in    \eqref{eqn7.6} $(b)$ with  $ u,  \mu  $  replaced by  
   $ U_j,  \ti \mu_j $  for  $ j  = 1, 2,  \dots $

From the  definition of  weak convergence we  may  assume  for  some  $ C \geq 1 $ that  
\begin{align} 
\label{eqn13.34}   
C^{-1}  \leq   \mu_j  (  \mathbb{S}^{n-1} )  \leq   C  \quad  \mbox{for}\, \, j = 1, 2, \dots 
\end{align}  
and  
thanks to  \eqref{eqn13.4} that  for some $ \hat C  \geq 1, $  \begin{align} 
\label{eqn13.35}   
\hat C^{-1}  \leq    \int_{\mathbb{S}^{n-1} }    |  \lan \he,  \xi   \ran  |\,  d \mu_j (\xi ) \quad 
\mbox{whenever}\quad  \he \in  \mathbb{S}^{n-1} \, \, \mbox{and}\,\, j = 1, 2, \dots. 
\end{align}   
Following  \cite{J}  or  \cite{CNSXYZ},  we  claim that we may also assume  
\begin{align} 
\label{eqn13.36}   
  E_j  \subset  \bar B ( 0,  \rho )  \mbox{ for $ j = 1, 2, \dots, $ and some }  \rho  <  \infty.
\end{align} 
 To prove   \eqref{eqn13.36}   we  first note from  \eqref{eqn13.5} that we may translate $  E_j $ if necessary so that if $ d_j = \mbox{diam}(E_j)  $ then  the  line segment from  $  - d_j y_j/2 $  to  $ d_j y_j/2 $ is contained in  $ E_j $ for some $ y_j  \in 
  \mathbb{S}^{n-1} $ when $ j = 1, 2,  \dots $.  If  $ h_j $  denotes the support function for $ E_j, $  then from the definition of  support function it  follows that  
\begin{align} 
\label{eqn13.37}   
h_j ( \xi  )   \geq    {\ts \frac{1}{2}}  |  \lan d_j  y_j ,  \xi   \ran | \quad 
\mbox{whenever} \quad  \xi  \in   \mathbb{S}^{n-1}.
\end{align}   
Using   \eqref{eqn13.37},  \eqref{eqn13.35},    \eqref{eqn13.33},  and  \eqref{eqn13.34},   we deduce that       
\begin{align}
 \label{eqn13.38} 
 \begin{split}
 (2 \hat  C)^{-1} \,   d_j    &\leq    \frac{d_j}{2} {\ds \int_{\mathbb{S}^{n-1}} }    |  \lan    y_j ,  \xi   \ran  |\,  d \mu_j  (\xi)    \leq    {\ds  \int_{\mathbb{S}^{n-1}} }    h_j ( \xi )  d\mu_j ( \xi )  \\ 
 &=  \left(\frac{n-p}{p-1}\right) \, \, \mbox{Cap}_{\mathcal{A}}  ( E_j ) \\
&\leq  \begin{cases}
 \ti C  &   \mbox{for}\, \,  1 < p   <  n - 1, \\
   b_j /(n-2)   \leq    \ti C   & \mbox{for}\, \,  p = n,  \\ 
 c  d_j^{n-p}  \leq  \ti C &   \mbox{for}\, \,   n - 1 < p  < n. 
 \end{cases} 
 \end{split}
\end{align}
  where $ \ti  C  $  is    a positive constant that does not depend on $ j$.    Thus, claim  \eqref{eqn13.36}  is true.

From  \eqref{eqn13.36}  we   see   that  a  subsequence of  
$ \{E_j\}_{j\geq 1}$  (also denoted $\{E_j\}$)  converges to a  compact convex set  $  E   \subset  \bar B ( 0, \rho ) $ in the sense  of  Hausdorff distance.  

We proceed by considering the following two cases. \\

\noindent{\bf Case A: $E$ has nonempty interior.} In this case, if $ p \neq n  - 1, $   it  follows    from  weak convergence of measures in Proposition  \ref{proposition11.1}  that  Theorem \ref{mink}   is true.  

To handle the possibility that   $  p = n - 1, $  and for later use, fix  $ p,  1 < p < n, $ and if $ j  = 1, 2,  \dots, $ let  
$  \nu_j$ denote  the  capacitary measure  corresponding to  $ U_j $  as  in \eqref{eqn3.4} $(a)$ of   Lemma \ref{lemma3.2}.   Then    from Lemma \ref{lemma8.6} 
  and   \eqref{eqn8.49a} of Proposition \ref{proposition8.9} we  see that $ \nu_j $ is absolutely continuous  with respect to  $ \mathcal{H}^{n-1} $ on  $  \ar  E_j $  and   
\begin{align} 
\label{eqn13.40}   
   \frac{d \nu_j }{ d \mathcal{H}^{n-1}} (y)   =   p   \frac{ f  ( \nabla U_j (y) )}{| \nabla U_j ( y ) |} \quad \mbox{for}\quad  \mathcal{ H}^{n-1} \, \, \mbox{almost every}\, \, y  \in  \ar  E_j.
   \end{align} 
Let  $  \hat r  =  \sup\{s  :   B ( x, s) \subset E \} $ be the inner radius of  $ E. $   Since $ E_j $ converges to $ E $ in the sense of Hausdorff distance it   follows that  $ E_j $  has inner radius at least $ \hat{r}/2 $  and from  \eqref{eqn13.36}  that  $ E_j \subset  \bar B (0,   \rho ) $ for $ j \geq j_0, $ provided  $ j_0 $ is large enough.  Using  these  facts  and convexity of  $ E_j $,  it follows  from  basic geometry that  if    $ \hat x   \in   \ar   E_j , $   then 
after a  possible rotation, 
\[  
B ( \hat x, \hat r/100 ) \cap  E_j  =  \{ x = (x' , x_n ) : x_n >  \hat{\ph} (x') \}  
\]   
where  $ \hat{\ph}$ is a  Lipschitz function on  $ \mathbb{R}^{n-1} $  with  $  \|  \hat{\ph} \hat  \|_{\mathbb{R}^{n-1}}  \leq   c (\rho/\hat{r})$.   Moreover, $  \ar  E_j $ can be covered by at most $ c (\frac{\rho}{\hat r})^{n-1} $ of radius $\hat{r}/1000$ where  $ c $ depends only on the data. From these observations,  \eqref{eqn13.40},  as well as the reverse H\"{o}lder  inequality in  \eqref{eqn8.29} $(a) $ with $ p  = q $ and our discovery in  Corollary \ref{corollary9.6a} that for $j\geq j_0$  sufficiently large $j_0$, there exists  
$\breve C $   depending only on the data and  $ \hat r,   \rho, $   such that      
\begin{align}
 \label{eqn13.42}  
 \begin{split} 
\mu_j^*  ( \mathbb{S}^{n-1} )  &=  \int_{\ar E_j}  f  ( \nabla U_j )  d \mathcal{H}^{n-1}  \\ 
&\leq  \breve C   (\hat{r})^{(1-n)/(p-1)} \,  
	 \,   \left( \int_{\ar E_j}  p  \,   \frac{ f  ( \nabla U_j ) }{ | \nabla U_j |}   d \mathcal{H}^{n-1} \right)^{p/(p-1)} \\    
	 &=\breve C   	 (\hat r)^{(1-n) /(p-1)} \,   ( \nu_j  (\ar E_j ) )^{p/(p-1)}
\end{split}
\end{align}      
 where $ \mu_j^* = \mu_j $ if  $ p \neq n - 1 $ and 
$ \mu_j^* = \ti \mu_j $ if  $ p = n - 1$.  

Next as in   the discussion following \eqref{eqn13.2}, we see that  if   
$  \mbox{Cap}_{\mathcal{A}} (E)  \not  = 0 $  then a subsequence of  $ \{\nu_j \}$  (also denoted  $ \{\nu_j\}$)  satisfies,  
\begin{align}
\label{eqn13.43}  
\lim_{j \to \infty}   \nu_j  =  \nu  \quad \mbox{weakly  where $\nu$ is the capacitary measure for  $  E.$ } 
\end{align}  
We now finish the proof of  Theorem \ref{mink} when  $ E $ has nonempty interior.  From \eqref{eqn2.5} we deduce  $\mbox{Cap}_{\mathcal{A}} (E) \neq 0. $   
Then as in  \eqref{eqn13.1}  and  \eqref{eqn13.2} we observe that     
\begin{align}  
\label{eqn13.44}     
\lim_{j\to \infty}  \nu_j  ( \ar E_j )  =  \lim_{j \to \infty} \mbox{Cap}_{\mathcal{A}} ( E_j )  =   \mbox{Cap}_{\mathcal{A}} ( E )  =  \nu (E)  = 1.    
\end{align} 
Using   \eqref{eqn13.44}  in  \eqref{eqn13.42} we conclude that 
$  \{\ti \mu_j \}$  is  uniformly bounded for $j\geq j_0$.  In view of  \eqref{eqn13.38}  and  Proposition 
\ref{proposition11.1}  we can choose subsequences of   $ \{b_j\}$ and $\{\ti \mu_j \}$  (also  denoted  $ \{b_j\}$  and $\{\ti \mu_j\}$) so that  
\begin{align}
\label{eqn13.45}   
\lim_{j \to \infty}  b_j = b, \, \,0 < b  < \infty,   \quad \mbox{and}\quad     \lim_{ j \to \infty}  \ti \mu_j =  \ti \mu  \, \, \mbox{weakly}
\end{align}
where $ \ti \mu $  is the measure in  
   \eqref{eqn7.6} $(b)$ with  $ u,  \mu  $  replaced by  
$ U,  \ti  \mu. $   Also  $ U $ is the capacitary function for $  E. $  Then, using \eqref{eqn13.45}, we have  $ \mu  = b \, \ti \mu $ so the proof of  Theorem \ref{mink} is complete when 
 $  1 < p  < n  $ and $  E  $ has nonempty interior.   \\

\noindent {\bf Case B: $  E $  has empty interior.} In this case, we assume that 
\begin{align}  
\label{eqn13.46}   
\mbox{Cap}_{\mathcal{A}} (E) \not = 0.
\end{align} 
At the end of this subsection, we show that \eqref{eqn13.46} holds.  Since sets with finite $ \mathcal{H}^{n-p}  $ measure have  zero    $ \mathcal{A}$-capacity,  
it   then follows from  \eqref{eqn13.36}   as  in the discrete case  that there is a $k$-dimensional plane $ P $ with
$ n - p  < k  \leq n - 1$  and      
\begin{align}
\label{eqn13.46a} 
E  \subset  P  \cap B (0, \rho) \quad \mbox{ with }    \quad     0  <  \mathcal{H}^k ( E )<\infty.  
\end{align}    
By considering two cases, $ n-p< k  < n  - 1 $ and $k=n-1$, we show that \eqref{eqn13.46a} is not possible and this will finish the proof that $E$ has nonempty interior. \\

\noindent {\bf Case B1: Suppose that  $ n-p< k  < n  - 1. $}  Translating  and rotating  $ E $ if  necessary,  we  may assume   \eqref{eqn13.22}-\eqref{eqn13.23}  hold for some $ a > 0$  
with $  E_1 $ replaced by $ E $ and  $ \rho $  by  $ 2 \rho.$  
          
  Let
  \[  
  t_j  =     d_{\mathcal{H}} (  E_j,  E ) \quad \mbox{for}\, \, j  = 1, 2, \dots.    
\]  
Then for $  j $  large  enough we can argue as in  Lemma  \ref{lemma13.3}  with 
$  t $  replaced by $ t_j,     $     and  $  u(\cdot, t ), E_1 + t E_2, $   by  $ U_j, E_j . $ 
We  obtain from the analogue of  \eqref{eqn13.24}  for $ j \geq j_0,  $  that  
\begin{align}  
\label{eqn13.47a}   
1 - U_j  ( x  )    \geq  C_1^{-1}   \, |x''| ^{\psi }  \quad    \mbox{for}\, \,    x = ( x', x'' ) \in  B (0, 4 \rho ) \, \,   \mbox{and}\, \,     C_1 \,  t_j  \leq  |x''|
\end{align}  
where $ x' \in \mathbb{R}^k $  and  $ \psi   = (p-n+k)/(p-1). $   Fix  $ j  \geq  j_0, $   and given $  y   \in \ar E_j  \cap  B ( 0,  a)  $,  let $ T_j (y),  $     be  a  supporting hyperplane  to  $  \ar E_j $  at  $ y. $   Let  $\hat H_j $  be the open half space  with  
 $\hat H_j \cap E_j = \es $  and  $ \ar \hat H_j  =  T_j( y ). $  
  Let   $  y^* $  denote the point  in  $ \hat H_j $  which lies on the normal line through $ y  $  with  $  | y - y^*  |  =  2  C_1  t_j $ where $  C_1  $ is as in  \eqref{eqn13.47a}. Note that  for $ j $ sufficiently large it is true that  $ d ( y^*, P )  >  C_1 t_j $.  since otherwise  it  would follow from the triangle inequality that there exists  $ z  \in B  (0, 2 a ) \cap E $ 
  with  $ d ( z, E_j ) >  t_j . $  Thus  \eqref{eqn13.47a} holds with $ x = y^* . $     Let   $ \ph $  be the  $ \mathcal{A}$-harmonic function in $ \hat  H_j \cap  B ( y,  8    C_1  t_j )  \sem 
  \bar B ( y^*,  C_1 t_j)   $  with continuous boundary values
  \[
\phi\equiv 
\left\{
\begin{array}{ll}
1 - U_j & \mbox{on} \,\,\, \ar B ( y^*,   C_1   \, t_j), \\
0 & \mbox{on}\, \,\, \ar ( \hat H_j  \cap B ( y,  8 C_1  t_j )).
\end{array}
\right.  
  \]
  Then  from the maximum principle for 
     $ \ti {\mathcal{A}}$-harmonic functions we have  $ \ph \leq  1 -  U_j$ on $ \hat  H_j \cap  B ( y,  8    C_1  t_j )  \setminus B ( y^*,   C_1   \, t_j)$.    Comparing $ \ph $ to a linear function and using a boundary Harnack inequality from  \cite{LLN} in  
  $\hat  H_j  \cap  B ( y,  8 C_1   t_j) $  we deduce for some $ c^* $ depending only on the data   and  $\rho $ that     
  \[ 
  ( 1 -   U_j (  y^{*} ) ) /t_j  \, \leq \,  c^*  ( 1 -  U_j  (  \hat  z ) ) / d (  \hat z,  T_j (y)) 
  \]
when $  \hat z  \in      \hat H_j \cap B ( y,   C_1   \,  t_j ) . $  Letting  $ \hat z  \to y $ non-tangentially,   we conclude from      this  inequality and  \eqref{eqn13.47a} with $ x  = y^* $  that         \begin{align}
\label{eqn13.48a}   
\, t_j^{\psi  - 1}\,    \leq  C^{**}  |   \nabla U_j ( y ) |  \quad   \mbox{for}\, \,  \mathcal{H}^{n-1} \mbox{ almost  every }  y \in \ar E_j \cap B (0, a) 
  \end{align}
  and  $ j  \geq j_0.  $ Here  $ C^{**} $  has the same dependence as $ C_1. $  From  \eqref{eqn13.48a}, Theorem \ref{mink} in the discrete case, \eqref{eqn13.40} with $ \nu $  replaced by $ \nu_j $, and the structural assumptions on  $ f $  we see that  
\begin{align}
\label{eqn13.49a}  
t_j^{\psi  - 1}\,  \nu_j ( \ar E_j \cap B (0,a) )   \leq  C'  \mu^*_j  ( \mathbf{g}_{j}( \ar  E_j  \cap  B (0, a ) ))  
  \end{align}
    where $ \mu_j^*  = \mu_j $    when  $p \neq n - 1, $  and  $ \mu_j^* =  \ti  \mu_j $ when $ p  = n - 1$ for $ 1 <  p < n$.  Here $\mathbf{g}_j$ is the Gauss map for $\partial E_j$.
  We note that  $ \psi  -  1  =    ( 1 - n + k )/(p-1) < 0. $  Also         from  \eqref{eqn13.43} we  
  deduce  
     \[  
     \liminf_{j  \to  \infty}  \nu_j (\ar E_j  \cap B (0, a )  )  \geq  \nu  ( \ar E  \cap B (0, a/2)) > 0 
     \]  
     where  the last inequality follows from the fact that otherwise  $ U $  extends to an  $\mathcal{A}$-harmonic function in  $  B (0, a/2) $  which would then imply  by  Harnack's inequality  that  $ U  \equiv  1 $  .   Using this inequality in  \eqref{eqn13.49a} we see that  if  $ p  \neq n - 1, $ then   $  \mu_j  ( \mathbb{S}^{n-1})  \to \infty $  in  contradiction  to   \eqref{eqn13.34}.  If  $  p  = n - 1, $  then from 
  \eqref{eqn13.38}  we see  that   
     \[ 
     \inf \{ b_j, \, \,  j  = 1, 2,  \dots \} > 0 
     \]   
     since otherwise it would  follow that $  E $  consists of  a  single point  -  a  set with zero 
     $\mathcal{A}$-capacity.  Using  this  observation and  arguing as above, we get once again that  $  \mu_j (\mathbb{S}^{n-1})  \to  \infty $  which again  contradicts   \eqref{eqn13.34}.           Thus, we first conclude that $E$ can not be contained in a $k$-dimensional plane for $n-p<k<n-1$ under the assumption \eqref{eqn13.46}. \\

     

\noindent {\bf Case B2: Suppose that  $ k  = n  - 1. $} In this case, we continue our proof under the assumption that \eqref{eqn13.46} holds. We also assume, as  we may,   that $ P  = \{ x: x_n  = 0  \}  $   
and 
\begin{align}  
\label{eqn13.47}   
B  ( 0,  4a)  \cap  P  \subset     E   \subset  B ( 0, \rho ) \cap P.
\end{align}
   Let   $  U $  be  the  $ \mathcal{A}$-capacitary function for  $E$  and as previously, 
$ U_j $  is  the  $ \mathcal{A}$-capacitary function for  $  E_j.$    Translating  $  E_j $ slightly upward  if  necessary  we  may assume  that   
\[   
\lim_{j \to  \infty}    d_{\mathcal{H}} ( E_j,  E ) = 0   \quad \mbox{and}\quad   E_j  \subset \{  x: x_n > 0  \}.   
\]  
 Let  $ \nabla U_{+} (x) $  denote  the  limit (whenever it exists)  as   $y \to  x$     non-tangentially  through values with $  y_n  > 0 .$    We  prove   
 \begin{proposition}  
 \label{proposition13.5}    
  There exists $ C \geq 1 $  such that    
\begin{align}
\label{eqn13.48}   
C  \liminf_{j \to \infty }   \int_{\ar  E_j }  f ( \nabla U_j )  d\mathcal{H}^{n-1}  \geq    
    \int_{\ar  E }  f ( \nabla U_+)  d\mathcal{H}^{n-1}   
     -  C^2    \mathcal{H}^{n-1} ( E ).  
\end{align}
 \end{proposition}   
\begin{proof} 
 Given  
    $ \ep  >  0  $  choose  $ j_1 $ so large that  
   $  d_{\mathcal{H}} ( E_j,  E)  \leq  \ep $ for 
    $ j \geq j_1. $                    Comparing   boundary values of   $ U,  U_j  $ in  
$ B (0,  2 \rho )  \sem  E_j   $    and  using   Lemmas  \ref{lemma2.3},  \ref{lemma3.3}, \ref{lemma3.4},   we  deduce   the existence of  $ 0 < \alpha \leq  1/2,  \hat C \geq  1,  $  
  such that      
\begin{align}
\label{eqn13.49}  
1 -  U  \leq  \hat  C   (   1  -  U_j     +  \ep^\al ) 
    \end{align}
    for  $ j  \geq j_1$.     Next  we  divide  the interior of   $ E $ into   $(n  - 1)$-dimensional closed  Whitney cubes   $ \{ Q_k  \} . $   Let    $ \ar'   E  $  denote  
the  boundary of  $  E  $  considered  as  a  set in  $ P.$    
Then  the cubes in  $  \{ Q_k \} $  have disjoint interiors with  side length  $ s ( Q_k ) $  and the property that considered as  sets in  $  
    P,   $  the distance say  $ d' ( Q_k,  \ar'  E  ) $  from $ Q_k $ to  the boundary of  $ E $    satisfies  
\begin{align}
\label{eqn13.50} 
10^{-n} s(Q_k)  \leq   d' ( Q_k, \ar'  E  )  \leq  10^n  s ( Q_k). 
\end{align}  
Let         
  $  Q  \in  \{ Q_k \}  $  with  $ s ( Q )  \geq   \ep^\al   $  and put    
  $   Q_+  =  Q \times (0, s(Q)) . $
Suppose  $  y = (y_1, \dots, y_n ) = y_Q $  is   a  point in    $  Q_+ \sem E_j $  with    
 $ d( y,  \ar'  E )   \geq y_n/ 2  \geq    s(Q)/4 . $  

 We  consider two possibilities.  If   $  ( 1 - U ) ( y )  \geq 2 \hat C \ep^{\al} $ ($\hat C $  as in 
\eqref{eqn13.49}),  then from \eqref{eqn13.49} we have    $     1 -  U(y)     \leq 2 \hat C  ( 1 - 
U_j ( y)  ) $  and arguing  as in  \eqref{eqn13.26}   for   $ U_j,  U  $   we  get 
\begin{align}  
\label{eqn13.51}          
\begin{split}    
    \bar C^3 \int_{\ar  E_j  \cap  Q_+ }  f ( \nabla U_j )  d\mathcal{H}^{n-1} &\geq  \bar C^2  (1-U_j )^p (y) \,   s ( Q )^{ n-1- p}  \\
    & \geq  \bar C  (1-U )^p (y)  \, s ( Q )^{ n-1 - p}  \\
    &\geq { \ds \int_{ Q  }  f ( \nabla U_+ )  d\mathcal{H}^{n-1} } . 
  \end{split}
  \end{align}
      If   $  1 -  U ( y )  <   2 \hat C   \ep^\al, $ then since  $ s ( Q )  \geq \ep^\al ,$  an argument similar to the above gives  
\begin{align}
\label{eqn13.52}   
\int_{ Q  }  f ( \nabla U_+ )  d\mathcal{H}^{n-1}   \leq  C_+ \,  s ( Q )^{n-1} 
\end{align}
 where  $ C_+  $ is independent of $ j \geq j_2 \geq j_1 $  provided $ j_2 $  is  large enough. 
  Combining   \eqref{eqn13.51},  \eqref{eqn13.52}  and using    \eqref{eqn13.50}    we find  after summing over  
  $  Q  \in  \{ Q_k \} $  that  for some  $ \breve C  \geq 1, $ independent of $ j  \geq j_2, $     
\begin{align} 
\label{eqn13.53}     
\breve C \int_{\ar E_j }  f ( \nabla U_j )  d\mathcal{H}^{n-1}  \geq  \int_{ \{ x \in  E :\,  d' (x, \ar' E) \geq \breve C  \ep^{\al} \} }   f  ( \nabla  U_+ )  d\mathcal{H}^{n-1}  -  
    \breve C^2  \mathcal{H}^{n-1}  ( E ) . 
    \end{align}

 Letting first $ j \to \infty $  and after that  $ \ep \to 0 $   we obtain from \eqref{eqn13.53}   and  the monotone convergence theorem or  Fatou's Lemma that  \eqref{eqn13.48} is true. This finishes the proof of   Proposition \ref{proposition13.5}.    
   \end{proof}  
   
Next we prove   

  \begin{proposition}  
  \label{proposition13.6} 
\begin{align}   
\label{eqn13.54}     
\int_{ E }  f ( \nabla U_+)  d\mathcal{H}^{n-1} =  \infty.    
    \end{align}
     \end{proposition}   
   We  note that   Propositions  \ref{proposition13.5}, \ref{proposition13.6}  give a contradiction to our  assumption that   
\eqref{eqn13.47}  is true,   since the total mass of  the measures in   \eqref{eqn13.34} are uniformly bounded and  invariant under translation. Hence, we also conclude in this case that $E$ can not be contained in $k$-dimensional plane when $k=n-1$ under the assumption \eqref{eqn13.46}. This finishes the proof of existence in Theorem \ref{mink} under the assumption  \eqref{eqn13.46}. 

 \begin{proof}[Proof of Proposition \ref{proposition13.6}]
 For  the readers  benefit  we  first  outline a     ``simple''  proof  of \eqref{eqn13.54} in Proposition  \ref{proposition13.6} for    the  $p$-Laplace equation (i.e., when
 $ (f(\eta) =  p^{-1}  |\eta|^p)$).       We use the same  notation as in  Proposition \ref{proposition13.5}. 
   Let    $ \{ Q_k \}  $  be   a  Whitney decomposition of  the interior of $ E $  considered as  a  subset of  $ P. $   Let  $ Q  \in  \{  Q_k \},   $   and  let $ z $ be a  point in  $ \ar'  E $  with  $ d'  ( z,    Q )  \approx  s ( Q). $   From convexity of  $  E  $  we see that  there is  a  $(n - 2)$-dimensional  plane, say  $ P_1  $  containing $ z $ with the property that  $  E $  is contained in the closure of  
one of the components  of    $ P \sem  P_1 .  $    Rotating $  P_1 $ if necessary we may assume that  
\[   
E  \subset  \Om = \{   x \in \mathbb{R}^n  : \,\,  x_1  - z_1 \leq  0  \, \, \mbox{ and  } \, \,  x_n = 0   \}.   
\]   
We note that Krol in  \cite{Kr}   has  constructed  a   homogeneous $ 1 -  1/p  $   solution, say $v'$,  to the $p$-Laplace equation in   $  \{  (x_1, x_n ) \} \sem \{ (x_1,  0 ) : x_1 \leq 0 \}$      which  vanishes  continuously on      $ (-\infty,  0]  \times  \{0\}.  $       We   extend $v' $  continuously to  
$ \mathbb{R}^n $ (also denoted $v'$)  by  defining  this  function to   be constant in the other  $(n - 2)$ coordinate directions.  Then  $ v (x) = v' (x-z)$ for   $x  \in \rn{n}$ is $p$-harmonic in $\mathbb{R}^n  \sem \Om. $       
Comparing boundary values  and using the maximum  principle, as  well as   Lemmas \ref{lemma3.2}, \ref{lemma3.3},    we  deduce that 
\begin{align}
\label{eqn13.55}     
c (1 - U  ( x ) ) \geq  v ( x )   \quad \mbox{whenever}\, \,   x  \in    B (0,  2 \rho  ) ,  
\end{align}   
 where  $ c $  depends only on  the data and  the  $p$-capacity  of  $  E. $    As  in  Proposition \ref{proposition13.5}    we see   that    
\begin{align} 
\label{eqn13.56} 
c'   \int_Q    |\nabla  U_+ |^p   \, d\mathcal{H}^{n-1}  \geq   ( 1 - U (y_Q) )^p  s (Q)^{n-1-p}.
\end{align} 
Now from  Krol's construction, we also deduce that 
\begin{align} 
\label{eqn13.57}  
c'' v  (y_Q) \geq  s (Q)^{1-1/p} .  
\end{align}  
Combining  \eqref{eqn13.55}-\eqref{eqn13.57}  we conclude that  for some  $ \ti c  $  with the same dependence as the above constants, 
\begin{align} 
\label{eqn13.58}    \ti c   \int_Q    |\nabla  U_+ |^p   \, d\mathcal{H}^{n-1}  \geq   
  s (Q)^{n-2}.
  \end{align}
Now since  $  B (0, 4a)  \cap P \subset  E  $  we see that  for  $ l  $ large there are  at least $ \approx  2^{ l (n -2)}  $  members of  $ \{Q_k \} $ whose side length lies between  $  2^{- l-1} a $ and  $ 2^{-l} a. $     Using this fact  and  summing \eqref{eqn13.58}   we get   Proposition \ref{proposition13.6} in this special case.   

To  get  Proposition \ref{proposition13.6}  for a general $p$-homogeneous   $f$  satisfying our structure conditions,  we  use a  clever  idea of  Venouziou  and   Verchota in \cite{VV}.  To simplify matters we make a 
   further  translation, scaling  and rotation, if necessary, so that   $ E  $  becomes  $  E'  $  with  
\begin{align}
  \label{eqn13.59}  
 \begin{split}   
 & (a) \hs{.2in}  E'  \subset  B (0, \rho'  ) \cap  P \mbox{  for  some  }   \rho' >1, \\
 & (b) \hs{.2in}   0   \in \ar'  E',  \\ 
 & (c) \hs{.2in}    E'  \cap \ar' (B (0, 1 ) \cap  P)  \not  = \es,  \\
 &  (d)  \hs{.2in}     E'  \subset  \{  x \in \mathbb{R}^{n-1} :\,\,  x_{1}  \leq 0  \}. 
 \end{split}
 \end{align} 
 Then  $ f, U, $ become  $ f', U' $ under this transformation.   Since $ f' $ satisfies the same structure  assumptions as  $ f $  it clearly  suffices to prove  Proposition \ref{proposition13.6}   for  $  E' , U' ,  $  the  $\mathcal{A'} = \nabla  f'$  capacitary function  for  $  E' $.    For ease of notation we  drop the primes  in \eqref{eqn13.59} and just write   $ U,  E,   f. $      
  
Let    $  D =   B (0, 1)  \sem E$ and  given  $ t > 0$ for  $0 < t  \leq  1/8, $  we set   $  z_t   =  \frac{1}{4} e_{1} +  t e_n  $. Let  
\[   
H_t  =  \{  y +  s  (  y   -  z_t  )  :  s  \geq  0\, \, \mbox{and}\, \,  y   \in E \} \cap  \bar B (0, 1 ) \quad \mbox{and}\quad  \,  D_t  = B (0, 1)  \sem  H_t. 
\]     
We note  that   $ H_t $ is  the union of  line segments   with one endpoint on   $ E $  and the other endpoint on $ \ar B(0, 1). $  Also, each line segment lies  on a  
    ray  beginning at  
$ (1/4, 0, \dots, t ) $  and   $  D_t  $  is   a  starlike Lipschitz  domain with respect to  
$ z_t . $  That is, $ \ar D_t $  is  the  union  of   graphs of  a  finite number of  Lipschitz functions defined on  $(n- 1)$-dimensional  planes and if     $ y \in   D_t,  $  then  the  ray joining   $   z_t  $   to  $ y $ is also in  $ D_t.   $  
   
Let  $  \ti f ( \eta ) =  f ( - \eta )$ whenever $\eta  \in  \mathbb{R}^n$  and let  
\[
  \mathcal{G}_0 = \mathcal{G}_0 (  \cdot, z_t ),\quad \mathcal{G}_1 =  \mathcal{G}_1 ( \cdot,  z_t ), \quad \mbox{and}  \quad \mathcal{G}_2  =  \mathcal{G}_2  (\cdot, z_t ) 
  \]
denote the  
$  \ti{\mathcal{A}} = \nabla   \ti f$-harmonic  Green's functions for  $  B (0, 1  ),   D, D_t $  respectively with pole at  $ z_t.  $  Also,  let  $ F =  F  ( \cdot,  z_t )  $  be the  $ \ti {\mathcal{
A}}$-harmonic  fundamental    solution on  $ \rn{n} $  with pole at  $ z_t.  $      Put        
\[  
\left\{ 
\begin{array}{l}
\ze_0 (\cdot,  z_t )  =  F ( \cdot ,  z_t )  -  \mathcal{G}_0 ( \cdot, z_t ), \\    
\ze_1 (\cdot,  z_t )  =  F ( \cdot ,  z_t )  -  \mathcal{G}_1 ( \cdot, z_t ),\\     
   \ze_2  (\cdot,  z_t )  =  F ( \cdot ,  z_t )  -  \mathcal{G}_2  ( \cdot, z_t ).  
\end{array}
\right.   
   \]   
From  $(e) $  of    Lemma  \ref{lemma9.1}   with $ \mathcal{A} $  replaced by  $   \ti { \mathcal{A}}, $   we see that  
$\ze_i  = \ze_i  ( \cdot, z_t )   $  has  a  H{\"o}lder continuous extension to    a neighborhood  of   $  z_t $  and  so  is locally  H{\"o}lder continuous in 
its  respective domain whenever    $  i   \in  \{ 0, 1, 2 \}.  $      
\\ 

\noi  To  complete the proof  of  Proposition \ref{proposition13.6},  we shall need the following  Rellich-type identity. 
\begin{lemma}  
\label{lemma13.7}  
With the  above notation, for $i=0,2$, 
\begin{align}
\label{eqn13.60}  
\begin{split}
{\ds  \int_{\ar  O }   \lan x  -  z_t,  \nu  \ran \,  \ti f (\nabla  \mathcal{G}_i )   d \mathcal{H}^{n-1}  = \frac{(n-p)}{p(p-1)}  \,  \ze_i ( z_t ) } 
\end{split}
\end{align}
 where $ O = B(0,1)  $ when $ i = 0 $  and $ O = D_t $  when $ i = 2$ with $ \nu $ is the  outer unit  normal  to  $O$. 
\end{lemma}   
\begin{proof} 
 To  start the  proof  of  Lemma  \ref{lemma13.7}  we write  $ \mathcal{G} $  for $ \mathcal{G}_i $ and   use the Gauss-Green Theorem  as in  \eqref{eqn11.10}-\eqref{eqn11.14} with   $  B  =   \bar B ( z_t,  \ep),  0  < \ep $  small,  to get   
 \begin{align}
 \label{eqn13.61}  
 \begin{split} 
 I &=  \int_{O\sem B}   \nabla \cdot (  (x - z_t )\ti
 f ( \nabla \mathcal{G} ) )  \, dx  \\
 &=    \int_{\ar  O}\lan x - z_t,  \nu  \ran \ti
 f  ( \nabla \mathcal{G} )  d \mathcal{H}^{n-1}   + \int_{\ar B }  \lan x - z_t,  \nu  \ran  \ti f  ( \nabla \mathcal{G} )   d \mathcal{H}^{n-1}  . 
\end{split}
\end{align}
 Moreover,  
 \begin{align}
  \label{eqn13.62}   
  \begin{split}
  I   &=   n  \int_{O\sem B}  
\ti   f ( \nabla \mathcal{G}  )  dx   +    \sum_{k, j  = 1}^n  \int_{O\sem B}     ( x_k - (z_t)_k ) \ti f_{\eta_j } (\nabla \mathcal{G}  )     \mathcal{G} _{x_k x_j}  dx 
    \\    &= n  \int_{O\sem B} \ti f ( \nabla \mathcal{G} )  dx   + I_1.
 \end{split}
 \end{align}
Integrating $  I_1 $  by parts, using  $p$-homogeneity of  $ \ti
f $, as well as  $\ti
 {\mathcal{A}} =  \nabla \ti
 f$-harmonicity of  $ \mathcal{G}   $ in $ O  \sem  \bar B $,  we  deduce that 
\begin{align} 
\label{eqn13.63} 
\begin{split} 
 I_1  &=  \int_{\ar O } \lan x - z_t ,  \nabla \mathcal{G}  \ran  \, \lan \nabla \ti
f (\nabla \mathcal{G} ) 
, \nu   \ran  \, d\mathcal{H}^{n-1} +   \int_{\ar B } \lan x - z_t ,  \nabla  \mathcal{G}  \ran  \, \lan \nabla \ti
f (\nabla \mathcal{G} ) 
, \nu   \ran  \, d\mathcal{H}^{n-1} \\
&\hs{.4in}-   p  \int_{O\sem B} \ti  f ( \nabla \mathcal{G}  )  dx .  
 \end{split}
 \end{align}     
Combining  \eqref{eqn13.61}-\eqref{eqn13.63},   we  find after some juggling  that  
\begin{align}
 \label{eqn13.64} 
 \begin{split}
       (n-p)      \int_{O\setminus B} \ti
 f ( \nabla \mathcal{G}  )  dx   &=    \int_{\ar  O}  \lan x - z_t,  \nu  \ran \ti  f  ( \nabla \mathcal{G})  d \mathcal{H}^{n-1}     -  \int_{\ar O } \lan x - z_t ,  \nabla \mathcal{G}  \ran  \, \lan \nabla \ti f(\nabla \mathcal{G} ) , \nu   \ran  \, d \mathcal{H}^{n-1}                                                                  \\
 & + \int_{\ar  B} \lan x - z_t,  \nu  \ran \ti
 f  ( \nabla \mathcal{G}  )  d \mathcal{H}^{n-1}      -  \int_{\ar B } \lan x - z_t ,  \nabla \mathcal{G}   \ran  \, \lan \nabla \ti
f (\nabla \mathcal{G} ) 
, \nu   \ran  \, d \mathcal{H}^{n-1}. 
\end{split}
\end{align}
     We  note that  
$  \nu  =  -  \frac{ \nabla \mathcal{G} }{|\nabla \mathcal{G}  |}  $ on  $  \ar O$ for    $\mathcal{H}^{n-1}    $  almost  everywhere.   Using   this  fact  and  $ p$-homogeneity  of  $ \ti
f $,  we  obtain     
\begin{align}
  \label{eqn13.65}      
  \begin{split}   
   \int_{\ar  O}  \lan x - z_t,  \nu  \ran  \ti f  ( 
\nabla \mathcal{G} ) d \mathcal{H}^{n-1}    -  \int_{\ar O } \lan x - z_t ,  &\nabla \mathcal{G}   \ran  \, \lan \nabla \ti
 f(\nabla \mathcal{G} ) 
, \nu   \ran  \, d \mathcal{H}^{n-1}                      
     \\
     &=   - (p - 1)  \int_{\ar  O}  \lan x - z_t,  \nu    \ran \ti
 f  ( \nabla \mathcal{G}  )  d \mathcal{H}^{n-1} .  
\end{split}
 \end{align}
       Now on  $  \ar B $ we have  $  \nu =   - \frac{ x - z_t }{| x -  z_t |} $ so               
\begin{align} 
\label{eqn13.66} 
\begin{split}
  \int_{\ar  B} &\lan
 x - z_t,  \nu  \ran \ti
 f  ( \nabla \mathcal{G}
 )     d\mathcal{H}^{n-1}    -  \int_{\ar B } \lan x - z_t ,  \nabla \mathcal{G}  \ran  \, \lan \nabla  \ti
f(\nabla \mathcal{G}) 
, \nu   \ran  \, d \mathcal{H}^{n-1} \\
  & =  
  -   \int_{\ar  B} | x - z_t | \ti
 f  ( \nabla \mathcal{G}
 ) dx +    \int_{\ar B } \lan x - z_t ,  \nabla \mathcal{G}
  \ran  \, \lan \nabla \ti
f(\nabla \mathcal{G}
),  {\ts\frac{x - z_t }{ | x - z_t |}}\ran  d \mathcal{H}^{n-1}.   
\end{split}
\end{align}
    Again, from  $ p$-homogeneity of  $ \ti
f$, and  $  \ti {\mathcal{A}} = \nabla   \ti
f$-harmonicity  of   
$ \mathcal{G}$,  and  the Gauss-Green  Theorem  we see that     
\begin{align}
\label{eqn13.67}  
\begin{split}
p  \int_{O\sem B }    \ti
 f (\nabla \mathcal{G}
 )  dx  =      \int_{O\sem B }  \nabla \cdot ( \mathcal{G}
  \nabla   \ti
f (\nabla \mathcal{G}
 ) )  dx       =   -  \int_{\ar B }  \mathcal{G}
  \,   \lan  \nabla  \ti
 f (\nabla \mathcal{G}
 ),  
 {\ts \frac{x - z_t }{ | x - z_t |}}\ran   d  \mathcal{H}^{n-1}.
  \end{split}
  \end{align} 
         Using \eqref{eqn13.65}-\eqref{eqn13.67}    in \eqref{eqn13.64} we arrive after  some  more juggling at       
\begin{align}
 \label{eqn13.68} 
 \begin{split}
    (n/p - 1 ) \int_{\ar B } \mathcal{G}
  \,  & \lan \nabla \ti
f (\nabla \mathcal{G}
 ),  {\ts\frac{x - z_t }{ | x - z_t |}} \ran  \, d \mathcal{H}^{n-1}    +    \int_{\ar B } \lan x - z_t ,  \nabla \mathcal{G}
  \ran  \, \lan \nabla  \ti f (\nabla \mathcal{G}
 ),   {\ts  \frac{x - z_t }{ | x - z_t |}} \ran  \, d \mathcal{H}^{n-1}  \\ 
 &  -   \int_{\ar  B} | x - z_t | \ti
 f  ( \nabla \mathcal{G}
 ) \, d  \mathcal{H}^{n-1}
   =     (p - 1)  \int_{\ar  O}  \lan x - z_t,  \nu    \ran \ti
 f  ( \nabla \mathcal{G}
 )  d \mathcal{H}^{n-1}.   
 \end{split}
 \end{align}     
    We intend  to  let  $  \ep  \to 0$  in  the left-hand side of  \eqref{eqn13.68}.    To  study the  asymptotics as $ \ep \to  0$, we  note from  Lemma \ref{lemma9.1}  with   $  w  = z_t$ and    $\ti {\mathcal{A}} =  \mathcal{A}$     that   $  F  -  \mathcal{G}
  =  \ze  $  in   $ O $  where  
$ \ze $   is    H\"{o}lder  continuous in   $  B  ( z_t , 1/8 ) $   for some  exponent  $  \al  \in (0, 1 ) $  depending  only on the data.   Moreover,  using   \eqref{eqn10.41}, \eqref{eqn10.47}-\eqref{eqn10.48},  elliptic regularity theory, and  arguing   as in \eqref{eqn12.3}-\eqref{eqn12.6}  we see for some constant  $ c  \geq  1, $  depending only on the data that  
\begin{align}
\label{eqn13.69}     
|  \nabla  \ze (x) |   \leq  c  \,  \ep^{\al - 1}  \max_{\ar B (  z_t,   1/8 ) } \ze \quad  \mbox{whenever} \, \,  x  \in  \ar  B  ( z_t ,  \ep )\, \, \mbox{and}\, \, 0 < \ep \leq  1/100.  
\end{align}      
 
Also from the structure and  regularity assumptions on  $ \ti f $, we see that 
if   $  \eta' ,  \eta^*  \in   B  (  \eta' ,  |\eta' |/2  )$ with $  \eta' \neq 0$  then    
\begin{align} 
\label{eqn13.70}
\begin{split}
 & (\al) \hs{.2in}     |\ti f (\eta' ) -  \ti f (\eta^*) | \leq  c | \eta' |^{p-1} 
| \eta' -  \eta^* |,   \\
& (\be) \hs{.2in}  |  \nabla\ti f (\eta') - \nabla \ti f ( \eta^* )|
\leq  c   |\eta'|^{p-2}  |\eta'  - \eta^* |.
\end{split}
\end{align}  
Moreover,   from    \eqref{eqn10.41} we have  
\begin{align} 
\label{eqn13.71} 
| z - z_t|^{-1}   | F( z ) |  +    | \nabla  F  ( z  ) |  \approx    | z  -  z_t |^{(1-n)/(p-1) }  \quad \mbox{for}\, \, 
 z   \in   \rn{n} \sem \{z_t \}.  
 \end{align}
 where proportionality constants and $ c $  in \eqref{eqn13.71} depend only on the data.  
      
If  $  x  \in \ar B ( z_t,  \ep )$ then from \eqref{eqn13.69}-\eqref{eqn13.71}  with $ z =  x$ and $\eta^*  = 
\nabla \mathcal{G} ( x$),  $\eta'  =  \nabla  F ( x )$,  we obtain    
\begin{align}
 \label{eqn13.72} 
 \begin{split}
 & (\al') \hs{.2in}       |\ti f (\nabla  F ( x ) )   - \ti f (\nabla \mathcal{G} ( x ) ) | \leq  \ti c   | \nabla  F  (x)  |^{p-1} \,  |  \nabla \zeta (x)  | \, 
  \leq \,  \ti c^2    |x - z_t |^{\al - n},  \\  
&  (\be')     \hs{.2in}      | \nabla \ti f (\nabla  F ( x ) )   -  \nabla \ti f (\nabla \mathcal{G} ( x ) ) | \leq  \ti c   | \nabla  F  (x)  |^{p-2} \,  |  \nabla \zeta (x) |    
 \leq  \ti c^2  \,  | x - z_t |^{1-n +\al  + \frac{(n-p)}{p-1}}.
 \end{split}
 \end{align}  
   Using  \eqref{eqn13.72}, we find  for the  integrands on the left-hand side of \eqref{eqn13.68}   that  
\begin{align}
 \label{eqn13.73} 
 \begin{split}
 (  n/p - 1 ) \mathcal{G}  \,   \lan \nabla \ti f (\nabla \mathcal{G}  ),   {\ts  \frac{x - z_t }{ | x - z_t |}} \ran  \, &=  (n/p-1) ( F - \ze ) \lan \nabla \ti f (\nabla F ),   {\ts  \frac{x - z_t }{ | x - z_t |}} \ran   +  O ( | x - z_t | )^{1-n + \al},   \\
    \lan x - z_t  ,  \nabla \mathcal{G}   \ran  \, \lan \nabla \ti f (\nabla \mathcal{G} ),   {\ts  \frac{x - z_t }{ | x - z_t |}} \ran  \, &=   \lan x - z_t  ,  \nabla F  \ran  \, \lan \nabla  \ti f (\nabla F ),   {\ts  \frac{x - z_t }{ | x - z_t |}} \ran   + O ( |x-z_t| )^{1-n + \al} ,  \\
   -     | x - z_t | \ti  f  ( \nabla \mathcal{G} ) \, &=    -   |x-z_t |  \ti f (\nabla F )  +  O(|x-z_t|^{1 - n + \al}) 
\end{split}
\end{align}   
 where we have used standard  big $ O $ notation.  Replacing the  sum of  the left-hand side of  these equations in \eqref{eqn13.73} by the sum of  the  right-hand sides  and using   $ (p-n)/(p-1)$-homogeneity of  $F$ in $  x - z_t$,  we  get  for $ x  \in  \ar B ( z_t, \ep) $ that        
\begin{align}
\label{eqn13.74} 
\begin{split}
  J_1 \, &= \, \int_{\ar B }  \left\{  {\ts  \frac{p - n}{p (p-1)} } \, F (x) \,   \,  \lan \nabla \ti f (\nabla F (x)),   {\ts  \frac{x - z_t }{ | x - z_t |}} \ran  \, -  | x - z_t | \ti f  ( \nabla F (x)) \, \right\}  d  \mathcal{H}^{n-1}      \\ 
   &=  {\ds  \int_{\ar B }}    {\ts \frac{n - p }{p}} \, \ze(x)  \, \,   \lan \nabla \ti f (\nabla F(x)),   {\ts  \frac{x - z_t }{ | x - z_t |}\,  \ran  d \mathcal{H}^{n-1}  }  +    (p - 1) {\ds  \int_{\ar  O}  \lan x - z_t,  \nu   \, \ran \ti f  ( \nabla \mathcal{G} (x))  d \mathcal{H}^{n-1} } + O (\ep^\al)   \\
   &  = J_2 +   (p - 1) {\ds  \int_{\ar  O}  \lan x - z_t,  \nu   \, \ran \ti f  ( \nabla \mathcal{G}  (x) )  d\mathcal{H}^{n-1}} + O (\ep^\al) 
\end{split}
   \end{align}
    where   the  last  integral (integral over $\partial O$) is  $ p - 1 $  times the one we want to compute.     We  claim  that      
    \begin{align}
       \label{eqn13.75}  
       J_1  \equiv   0 .  
       \end{align}    
    To prove  this claim we note from  section \ref{appendix}  with $ f $  replaced by $ \ti f $ (see the sentence  above  \eqref{eqn6.4}) that  
   if  $ \ti  f (\eta ) =   k (- \eta )^p  $  and 
   \[
      h ( x )  =  \max \left\{ \frac{  \lan x, y  \ran}{  k ( y)}, \, \, y  \in \rn{n} \sem \{0\}  \right\}, 
      \]
   then     $  h  $ has  continuous  second partials and 
\begin{align} 
\label{eqn13.76}  
\nabla   k ( \nabla  h (x)  ) = x/ h(x)  \quad   \mbox{ while } \quad   k ( \nabla  h ) = 1.
\end{align}      
  Also  it followed  (see  Remark \ref{rmk7.1})   that   $   \ti F ( x)  =  \he  \, (h (x))^{(p-n)/(p-1)}$  is  the fundamental solution for 
  $  \ti { \mathcal{A}} = \nabla \ti f$-harmonic functions with pole at  zero where    
    \[   
     \he^{p-1}   =  p^{-1} \left( \frac{ n - p}{p-1} \right)^{1-p}  \left( \int_{\mathbb{S}^{n-1}}  h ( \om )^{-n}  d \om  \right)^{-1}.  
    \]
   Let $   F ( x )  =  \ti F ( x - z_t ),  $ whenever $ x \in \rn{n} \sem 
\{z_t\}.  $   
   Using  these facts and \eqref{eqn13.76} in $ J_1 $ and the  homogeneity of the various  functions,  we see that   
   \begin{align}
   \label{eqn13.77}
    -  |x - z_t | \, \ti  f  (\nabla F (x)  ) =  - \he^p \,  |x - z_t |  \left ( \frac{n - p}{p-1}  \right)^p \,  h(x - z_t )^{\frac{(1-n)p}{p-1} } .  
   \end{align} 
	Moreover,  using \eqref{eqn13.76}, it follows that
	\begin{align}
	  \label{eqn13.78}  
	  \begin{split}
 F (x)  \lan \nabla \ti  f  ( \nabla  F ( x ) ),  
{\ts  \frac{x - z_t }{ | x - z_t |} } \ran & =  -  \he^p  h(x - z_t )^{\frac{p-n}{p-1} }  p \left( \frac{ n - p}{p-1} \right)^{p-1} 
 h( x - z_t )^{1-n} \lan    \nabla  k  (  \nabla h (x - z_t ) ) ,  {\ts \frac{x  - z_t}{| x - z_t |} }\ran
 \\ 
 &=   - \he^p  {\ds  p \,  \left( \frac{ n - p}{p-1} \right)^{p-1}   h(x - z_t )^{\frac{ (1-n) p}{p-1}}   | x - z_t |  }  .    
 \end{split}
 \end{align}   
 Multiplying  the right-hand side of  \eqref{eqn13.78}  by 
 $ \frac{p-n}{p(p-1)} $  and  adding to   \eqref{eqn13.77}   we obtain  claim \eqref{eqn13.75}. 
 
Next  arguing as in  \eqref{eqn13.77},  we get  
\begin{align} 
\label{eqn13.79}   
 J_2   = -  \int_{ \ar B ( z_t, \ep )}  \he^{p-1}  (n-p)   \left( \frac{ n - p}{p-1} \right)^{p-1} \ze ( x )  |x-z_t |  h(x - z_t)^{-n}  d \mathcal{H}^{n-1}.     
 \end{align}
  Using  one  homogeneity of  $ h $  we  see that \eqref{eqn13.79} can  be rewritten in spherical coordinates, $   \ep \, \om  =  x - z_t,    \,  \om \in \mathbb{S}^{n-1},  $ as
\begin{align}
\label{13.80}
   J_2   = -  \int_{ \mathbb{S}^{n-1} } \he^{p-1}    ( n - p)   \left( \frac{ n - p}{p-1} \right)^{p-1} 
   \ze (  z_t + \ep  \om  )  h(\om )^{-n}  d \mathcal{H}^{n-1}.  
   \end{align}
Letting  $  \ep \to 0 $ and using continuity of  $ \zeta$  at $ z_t $  we conclude 
from  \eqref{eqn13.74},  \eqref{eqn13.75}, and  \eqref{13.80}  that \eqref{eqn13.60} in Lemma  \ref{lemma13.7}   is true  since    
\begin{align}
\label{13.81}   
\he^{p-1}      \left( \frac{ n - p}{p-1} \right)^{p} \int_{ \mathbb{S}^{n-1} }  h(\om )^{-n}  d \mathcal{H}^{n-1}   =   \frac{n-p}{p(p-1)}. 
 \end{align}
 This finishes the proof of Lemma \ref{lemma13.7}.
\end{proof}  
\noi {\bf  Proof of  Proposition 13.6}.   We  shall  apply  Lemma  \ref{lemma13.7}    with  $ O  =  D_t. $ Before doing this  we note that if  $  \nu $ is the outer unit normal to  $ D_t $  then  
    $  \lan x  -  z_t , \nu (x)  \ran  = 0  $  when   $  x  \in  \ar D_t \sem \mathbb{S}^{n-1}   $  and  $  x_n  < 0 $  while  $  \lan x  -  z_t , \nu (x)  \ran  = t   $  when   $  x  \in  E  \cap   B (0, 1)$ for $\mathcal{H}^{n-1}$-almost everywhere.   Using these facts and  Lemma  \ref{lemma13.7}    we  obtain   for 
\[
\mathcal{G}_2 = \mathcal{G}_2  (\cdot, z_t )\quad \mbox{and}\quad  \ze_2  =  \ze_2 ( \cdot, z_t ) 
\]
  that 
\begin{align}
 \label{eqn13.82} 
 \begin{split}
 \ga \,  \ze_2 ( z_t )  &=   \int_{ \ar D_t  } \lan x - z_t,  \nu \ran    \ti  f (\nabla \mathcal{G}_2  ) d  \mathcal{H}^{n-1}   \\
 &=    t     \int_{ E \cap B (0, 1 ) }       \ti f ( \nabla \mathcal{G}_2   )  d  \mathcal{H}^{n-1}  + {\ds  \int_{ \ar  D_t \cap  \mathbb{S}^{n-1} }  \lan x - z_t,  \nu \ran  \ti f (\nabla \mathcal{G}_2   ) d  \mathcal{H}^{n-1}}.
 \end{split}
 \end{align}
where $\gamma:=\frac{n-p}{p(p-1)}$ is the constant in \eqref{13.81}.  Now   from the maximum principle for  $ \ti{\mathcal{A}}$-harmonic functions we  
 have for 
 \[
  \mathcal{G}_i = \mathcal{G}_i ( \cdot, z_t ), \quad \mbox{and} \quad \ze_i = \ze_i (\cdot, z_t ) \quad \mbox{for}\, \,  i = 0, 1, 2,  
  \]
  that
\begin{align}
\label{eqn13.83}   
\mathcal{G}_2      \leq  \mathcal{G}_1  \leq \mathcal{G}_0 \, \, \mbox{in}\, \,   D_t  \quad 
\mbox{ so } \quad  \ze_0 ( z_t ) \leq \ze_1 ( z_t ) \leq \ze_2 ( z_t). 
\end{align}   
Using  \eqref{eqn13.83}, the mean value theorem,  the fact that  $  \nabla \mathcal{G}_i ,  i  = 1, 2, $  has non-tangential  limits from  above  $\mathcal{H}^{n-1}$-almost everywhere on  $ E $  and that all limits have the same direction,   we conclude that 
 \[   
 \ti f  (  \nabla \mathcal{G}_2    )    \leq  \ti  f  (  (\nabla \mathcal{G}_1)_+  )  \quad \mbox{on}\, \,    E \cap  B ( 0, 1).  
 \]       
 Likewise,     
 \[   
 \ti  f  (  \nabla \mathcal{G}_2    )    \leq   \ti  f  (  \nabla \mathcal{G}_0  ) \quad \mbox{for}\, \,  
   \mathcal{H}^{n-1}\, \,   \mbox{almost every}\, \,    x \in  \mathbb{S}^{n-1} \cap \ar D_t. 
  \] 
  Using these facts and \eqref{eqn13.82}-\eqref{eqn13.83},  and Lemma \ref{lemma13.7}  with $O = B(0,1)$, we get 
\begin{align}  
\label{eqn13.84} 
\begin{split}
  t   {  \ds \int_{ E \cap B (0, 1 ) }  } 
   \ti f ( (\nabla \mathcal{G}_1)_+  )  d  \mathcal{H}^{n-1}   & \geq  \ga \ze_2 ( z_t) 
   -   {\ds  \int_{ \mathbb{S}^{n-1} }  \lan x - z_t,  \nu \ran \ti  f (\nabla \mathcal{G}_0 ) d \mathcal{H}^{n-1} }  \\ 
   & = \ga (\zeta_2 ( z_t ) - \ze_0 ( z_t ))  \\
   &\geq  \ga (\zeta_1 ( z_t ) -  \zeta_0 (z_t) ).   
\end{split}
\end{align}
 Letting  $ t \to 0$  in  \eqref{eqn13.84}  we  assert   that to  complete the  proof  of   Proposition \ref{proposition13.6} it suffices   to show   
\begin{align} 
\label{eqn13.85}   
c \, [  \zeta_1 ( z_t ) - \ze_0 ( z_t )] \geq 1
\end{align}
  where  $ c \geq 1 $ is a positive constant  depending on  the data    and  $ \mbox{Cap}_{\mathcal{A}} (E)  $  but    independent of  $ t.  $  Indeed,    from  \eqref{eqn3.4} $(b)$ of      Lemma \ref{lemma3.2} and  \eqref{eqn10.36} $(c) $ of  Lemma  \ref{lemma9.1}, as well as the maximum principle  for $  \ti{\mathcal{A}}$-harmonic functions , we  find  $ c \geq  1,  $  independent of  $ t , $    with 
$  c \,  (  1 - \, U  )  \geq   \mathcal{G}_1    $   on    $  D_t   \setminus  B (e_1/4, 1/8)$ when $  0 <   t  <   1/100 $.  Then as in the displays above  \eqref{eqn13.84} if  follows that 
\[ 
c  \ti f ( -  (\nabla U_+ ) ) =  c   f (  \nabla U_+ )     \geq  \ti f  ( (\nabla \mathcal{G}_1)_+ ) 
\] 
for  $  \mathcal{H}^{n-1} $  almost every  $  x   \in    E \cap  B (0, 1) . $   Using this inequality, \eqref{eqn13.85}, and  letting  $ t  \to  0$  in  \eqref{eqn13.84}  we  get   Proposition \ref{proposition13.6}.   

To   prove  \eqref{eqn13.85}  we  note   as  in  Lemma \ref{lemma2.3} 
  that for some  $  \ti \al \in (0, 1), c' \geq 1, $ 
  we have   
  \[  
  \max_{ B (0, s)} \mathcal{G}_1    \leq c'  s^{ \ti \al}   \mathcal{G}_1 ( e_1 /16 )  \quad \mbox{for}\,\,   0 < s 
  \leq   1/16  
  \] 
  where 
  $ c' $  depends only on the data and $ \mbox{Cap}_{\mathcal{A}} (E)$.  
  Also from  Lemma \ref{lemma2.1}   we see that  $  c''  \min_{\bar B(0, 1/16)} \mathcal{G}_0 \geq 1 $  where $ c'' $  
  depends only on the data.  Thus there exists  $ \hat \rho,  0 <  \hat \rho  <  1/16 $  with the same dependence as  $ c' $ such that 
\begin{align}  
\label{eqn13.86}    
\mathcal{G}_0 - \mathcal{G}_1 \geq  \hat \rho \quad \mbox{in}\, \,   \bar B (0, \hat \rho)  \sem  E. 
\end{align}
  We  claim  that  
\begin{align}  
\label{eqn13.87}    
|\nabla   \mathcal{G}_0 (x) |  \approx   \frac{\mathcal{G}_0 (x)}{ | x - z_t | } \approx   | x - z_t |^{\frac{1-n}{p-1} }\quad   \mbox{for}\, \,   x  \in B (0, 1/2). 
  \end{align}
The left-hand inequality  in  \eqref{eqn13.87}  follows from \eqref{eqn9.14} $(\be) $  while the right-hand inequality is a  consequence of   Lemma \ref{lemma9.1} $(e)$ and our knowledge of  $ F. $  
   
Armed  with  \eqref{eqn13.86},  \eqref{eqn13.87},   we now  prove  \eqref{eqn13.85}, and so also  Proposition \ref{proposition13.6}.  
  From  \eqref{eqn13.87}   it  follows  in a now well-known way  that  $  \mathcal{G}_0  - \mathcal{G}_1 $  
	  satisfies a    locally  uniformly elliptic  PDE in  $ B (0, 1/2 )  \sem  E $  similar to     \eqref{eqn10.42},  \eqref{eqn10.43}.    From Harnack's inequality for this  PDE  and  \eqref{eqn13.86} we deduce for some   $ \ti  \rho  > 0 $  that  
	  \[
	   \mathcal{G}_0 - \mathcal{G}_1  \geq  \ti  \rho \quad \mbox{on} \, \,    \ar  B (  e_1/4 , 1/4 -  \hat \rho ). 
	   \]
	   Using the same argument as in the proof of  Lemma 
  \ref{lemma9.1}  $ (e), $    it  now  follows that  $    \,   \ze_1 (z_t) -   \ze_0 (z_t)  \geq  \ti \rho $  whenever  $  t \in   (0,1/8). $   
  We  conclude that    \eqref{eqn13.85}  and  Proposition \ref{proposition13.6}   are true. 
  
 We next show that if $E$ has empty interior, the assumption  \eqref{eqn13.46} holds.  To this end, we consider following three cases depending on $p$.  For   $ 1< p  <  n - 1 $, we note from \eqref{eqn13.38}  and  \eqref{eqn1.6} that 
      for  some  $ C_+ \geq  1 $ independent of  $ j,  $  
\begin{align}
 \label{13.50b} 
 C_+^{-1}  \,  \mbox{diam}(E_j)  \leq   \mbox{Cap}_{\mathcal{A}} (E_j )      \leq       C_+  \, (\mbox{diam}(E_j) )^{n-p}.
 \end{align}
 Thus, \eqref{13.50b} implies that $\mbox{diam}(E_j)$ is bounded below independently of $ j , $  which in view of \eqref{eqn13.43}   implies that      \eqref{eqn13.46} always holds when $ 1 < p < n - 1. $  When  $ p  = n - 1 $, then from \eqref{eqn13.43} we deduce that  $ \mbox{ Cap}_{\mathcal{A}} (E ) = 1 $ so again 
     \eqref{eqn13.46}  always holds. Finally, when $  n - 1  < p  < n$, then  a  line  segment  of  length $ l$     has  capacity  $  \approx  l^{n-p} $  so  from  \eqref{eqn13.43} we deduce that if       \eqref{eqn13.46} is false then    $ \mbox{diam}(E_j)\to 0$.   For $j=1,2,\ldots,$ choose  $  s_j, \hat  E_j,   $  so  that  
        $ s_j  \hat E_j  =  E_j $   and  $\mbox{Cap}_{\mathcal{A}} ( \hat E_j )  = 1. $  Then from the above discussion we  find that  $ s_j  \to 0 $  as  $ j \to  \infty. $   
        Let  $  \hat \mu_j $  denote the measure  in    Theorem  \ref{mink}  defined  relative  to   $  \hat E_j . $  Then from the usual dilation argument  we  have    $  \hat  \mu_j  =  s_j^{p + 1  -  n}  \mu_j . $   From    \eqref{eqn13.34}              we see that  $  \{\hat \mu_j \}_{j\geq 1}$  converges weakly to  zero and  a  subsequence of  
          $ \{\hat E_j\} $  converges  to  $ \hat E $  a  set of  $ \mathcal{A}$-capacity one.   We can now argue  as  previously with  $  \hat  E  $  replacing  $ E $ to  get a  contradiction.  Using just uniform boundedness of  $ \{\hat \mu_j\}_{j\geq 1}$  and our earlier work  it  follows that $ E $ has nonempty interior.    From   weak convergence of  measures in  Proposition \ref{proposition11.1} we now  get a contradiction since  $ \hat \mu_j \to 0 $  weakly as  $ j  \to \infty. $  Thus,    assumption  \eqref{eqn13.46}  holds  when $ 1 < p < n$.  The proof of  existence in  Theorem  \ref{mink}   is now complete.   
 \end{proof}

\subsection{Uniqueness of Minkowski problem}
\label{uniq}
Uniqueness in Theorem \ref{mink}  can be shown  using   the equality result in the Brunn-Minkowski inequality(Theorem \ref{theorem1.4}) as in \cite{CNSXYZ} or \cite{CJL}: 

\begin{proof}[Proof of $(c), (e) $  in Theorem \ref{mink}] To  prove uniqueness  in  Theorem  
\ref{mink},  suppose  $  \mu $  
is a positive  finite   Borel measure on  $  \mathbb{S}^{n-1}$  satisfying   \eqref{eqn7.1}  and  let  $ E_0,  E_1$ be two compact convex sets  with nonempty interiors satisfying \eqref{eqn7.6} in Theorem \ref{mink}  relative to $ \mu. $       Let $h_0$ and $h_1$ be the support functions of $E_0$ and $E_1$ respectively. For $t\in[0,1]$ we let $ E_t=(1-t) E_0+t  E_1$.  Using    Proposition \ref{proposition12.1} and \eqref{eqn12.30} we deduce as in   \eqref{eqn13.30} and \eqref{eqn13.31}  that   if   $ p  \not  =  n - 1, $  then    
\begin{align}
\label{eqn13.88}
\begin{split}
\left.\frac{d}{dt} \mbox{Cap}_{\mathcal{A}}(E_t) \right|_{ t = 0 } \,  &= (p-1) \int_{\mathbb{S}^{n-1}} (h_1(\xi)-h_0(\xi)) d\mu(\xi) \\
&=(n-p) [\mbox{Cap}_{\mathcal{A}}(E_1)-\mbox{Cap}_{\mathcal{A}}(E_0)].
\end{split}
\end{align}
We  define
\[
\mathbf{m}(t)=\mbox{Cap}_{\mathcal{A}}(E_t)^{\frac{1}{n-p}}. 
\]
Then basic calculus  and   \eqref{eqn13.88}   gives us that   
\begin{align}
\label{eqn13.89}
\begin{split}
\mathbf{m}'(0)   &=\mbox{Cap}_{\mathcal{A}}(E_0)^{\frac{1}{n-p}-1}[\mbox{Cap}_{\mathcal{A}}(E_1)-\mbox{Cap}_{\mathcal{A}}( E_0)]\\
&=\mathbf{m}(0)^{1-n+p}[\mathbf{m}(1)^{n-p}-\mathbf{m}(0)^{n-p}].
\end{split}
\end{align}
From \eqref{eqn1.7} in Theorem \ref{theorem1.4} with $ E_1, E_2, \la $ replaced by $ E_0,  E_1, t $  we find that  $\mathbf{m}$ is a concave function   on  $[0,1]$ and therefore 
\begin{align}  
\label{eqn13.90}  \mathbf{m}'(0)\geq  \mathbf{m}(1) -  \mathbf{m}(0) 
\end{align}  
with  strict inequality unless    $ \mathbf{m}$ is  linear  in  $ t ,$   which implies  equality holds in the  Brunn Minkowski  inequality  for $ t \in [0,1]. $   Let      
\[  
l  = \left( \frac {\mbox{Cap}_{\mathcal{A}}(E_1)}{\mbox{Cap}_{\mathcal{A}}(E_0)}\right)^{\frac{1}{n-p}}.       
\] 
Using  \eqref{eqn13.90}  in  \eqref{eqn13.89} we see that   
\[      
l^{n-p}  - 1 \geq l -   1 .  
\]  
Reversing the roles of  $ E_0,  E_1 $  we also get  
  \[  
  l^{p-n} -  1  \geq  l^{-1} - 1.  
  \]   
  Clearly,  both these inequalities can only hold if  $ l = 1. $ 
Thus  $  \mbox{Cap}_{\mathcal{A}}(E_0) =  \mbox{Cap}_{\mathcal{A}}(E_1)  $  and  equality holds in   \eqref{eqn1.7} for $ t \in [0,1]. $  From  Theorem \ref{theorem1.4} we conclude that $ E_0$ is a translate and dilate of $ E_1$.   From  \eqref{eqn1.5} it follows that honest dilations  are not possible    when $p\neq n-1$.   

If $ p  = n - 1 $,  let  $ b_0, b_1 $ correspond to  $ E_0,  E_1, $ respectively as in  \eqref{eqn7.6} 
$(d). $   Then  $ \mbox{Cap}_{\mathcal{A}} ( E_i ) = 1,  $ for $ i = 0, 1 $ and   arguing as in  
\eqref{eqn13.88}  we  see that   
\begin{align}
 \label{eqn13.91}
\left.  b_0   \, \frac{d}{dt} \mbox{Cap}_{\mathcal{A}}(E_t) \right|_{ t = 0 } \,  = (p-1)  \int_{\mathbb{S}^{n-1}} (h_1(\xi)-h_0(\xi)) d\mu(\xi)   =  \left.  b_1   \, \frac{d}{dt} \mbox{Cap}_{\mathcal{A}}(E_t) \right|_{ t = 1 }.
\end{align} 
From  concavity of  $ \mathbf{m} (t) $  as  above we see that 
 $ \mathbf{m}' (0)  \geq  \mathbf{m}' (1) $  so   \eqref{eqn13.91} implies   $ b_0   \leq  b_1 $  with strict inequality unless  equality holds in  \eqref{eqn1.7}   of   Theorem 
\ref{theorem1.4}   for  $ E_t, \, t \in [0,1]. $ Reversing the roles of  $ E_0,  E_1 $ we get that 
$ b_0 = b_1 $  so  from  Theorem  \ref{theorem1.4},  $ E_1 $ is  homothetic to  $ E_0. $           

This finishes the proof of (c), (e).  in Theorem \ref{mink} and so also of  Theorem \ref{mink}.
\end{proof}

\section*{Acknowledgment}          
This material is based upon work supported by National Science Foundation under Grant No. DMS-1440140 while the first author were in residence at the MSRI in Berkeley, California, during the Spring 2017 semester. The first author was also supported by ICMAT Severo Ochoa project SEV-2015-0554, and also acknowledges that the research leading to these
results has received funding from the European Research Council under the European Union’s Seventh Framework Programme (FP7/2007-2013)/ERC agreement no. 615112 HAPDEGMT. 
The fourth author was partially supported by NSF DMS-1265996.

\newcommand{\etalchar}[1]{$^{#1}$}


\begin{thebibliography}{BGMX11}

\bibitem[A]{A}
Murat Akman.
\newblock On the dimension of a certain measure in the plane.
\newblock {\em Ann. Acad. Sci. Fenn. Math.}, 39(1):187--209, 2014.

\bibitem[A1]{A1}
A.~D. Aleksandrov.
\newblock On the theory of mixed volumes. iii. extension of two theorems of
  minkowski on convex polyhedra to arbitrary convex bodie*s.
\newblock {\em Mat. Sb. (N.S.)}, 3:27--46, 1938.

\bibitem[A2]{A2}
A.~D. Aleksandrov.
\newblock On the surface area measure of convex bodies.
\newblock {\em Mat. Sb. (N.S.)}, 6:167--174, 1939.

\bibitem[ALV]{ALV}
Murat Akman, John Lewis, and Andrew Vogel.
\newblock {$\sigma$}-finiteness of elliptic measures for quasilinear elliptic
  {PDE} in space.
\newblock {\em Adv. Math.}, 309:512--557, 2017.

\bibitem[B1]{B1}
Christer Borell.
\newblock Capacitary inequalities of the {B}runn-{M}inkowski type.
\newblock {\em Math. Ann.}, 263(2):179--184, 1983.

\bibitem[B2]{B2}
Christer Borell.
\newblock Hitting probabilities of killed {B}rownian motion: a study on
  geometric regularity.
\newblock {\em Ann. Sci. \'Ecole Norm. Sup. (4)}, 17(3):451--467, 1984.

\bibitem[BGMX]{BGMX}
Baojun Bian, Pengfei Guan, Xi-Nan Ma, and Lu~Xu.
\newblock A constant rank theorem for quasiconcave solutions of fully nonlinear
  partial differential equations.
\newblock {\em Indiana Univ. Math. J.}, 60(1):101--119, 2011.

\bibitem[BLS]{BLS}
Chiara Bianchini, Marco Longinetti, and Paolo Salani.
\newblock Quasiconcave solutions to elliptic problems in convex rings.
\newblock {\em Indiana Univ. Math. J.}, 58(4):1565--1589, 2009.

\bibitem[C]{C}
Andrea Colesanti.
\newblock Brunn-{M}inkowski inequalities for variational functionals and
  related problems.
\newblock {\em Adv. Math.}, 194(1):105--140, 2005.

\bibitem[CC]{CC}
Andrea Colesanti and Paola Cuoghi.
\newblock The {B}runn-{M}inkowski inequality for the {$n$}-dimensional
  logarithmic capacity of convex bodies.
\newblock {\em Potential Anal.}, 22(3):289--304, 2005.

\bibitem[CF]{CF}
R.~R. Coifman and C.~Fefferman.
\newblock Weighted norm inequalities for maximal functions and singular
  integrals.
\newblock {\em Studia Math.}, 51:241--250, 1974.

\bibitem[CFMS]{CFMS}
L.~Caffarelli, E.~Fabes, S.~Mortola, and S.~Salsa.
\newblock Boundary behavior of nonnegative solutions of elliptic operators in
  divergence form.
\newblock {\em Indiana Univ. Math. J.}, 30(4):621--640, 1981.

\bibitem[CJL]{CJL}
Luis~A. Caffarelli, David Jerison, and Elliott~H. Lieb.
\newblock On the case of equality in the {B}runn-{M}inkowski inequality for
  capacity.
\newblock {\em Adv. Math.}, 117(2):193--207, 1996.

\bibitem[CNSXYZ]{CNSXYZ}
A.~Colesanti, K.~Nystr{\"o}m, P.~Salani, J.~Xiao, D.~Yang, and G.~Zhang.
\newblock The {H}adamard variational formula and the {M}inkowski problem for
  {$p$}-capacity.
\newblock {\em Adv. Math.}, 285:1511--1588, 2015.

\bibitem[CS]{CS}
Andrea Colesanti and Paolo Salani.
\newblock The {B}runn-{M}inkowski inequality for {$p$}-capacity of convex
  bodies.
\newblock {\em Math. Ann.}, 327(3):459--479, 2003.

\bibitem[CS1]{CS1}
Andrea Cianchi and Paolo Salani.
\newblock Overdetermined anisotropic elliptic problems.
\newblock {\em Math. Ann.}, 345(4):859--881, 2009.

\bibitem[D]{D}
Bj{\"o}rn E.~J. Dahlberg.
\newblock Estimates of harmonic measure.
\newblock {\em Arch. Rational Mech. Anal.}, 65(3):275--288, 1977.

\bibitem[EG]{EG}
Lawrence~C. Evans and Ronald~F. Gariepy.
\newblock {\em Measure theory and fine properties of functions}.
\newblock Studies in Advanced Mathematics. CRC Press, Boca Raton, FL, 1992.

\bibitem[FJ]{FJ}
W.~Fenchel and B.~Jessen.
\newblock Mengenfunktionen und konvexe körper, danske vid. selsk.
\newblock {\em Mat.-Fys. Medd.}, 16:1--31, 1938.

\bibitem[G]{G}
R.~J. Gardner.
\newblock The {B}runn-{M}inkowski inequality.
\newblock {\em Bull. Amer. Math. Soc. (N.S.)}, 39(3):355--405, 2002.

\bibitem[GL]{GL}
Nicola Garofalo and Fang-Hua Lin.
\newblock Unique continuation for elliptic operators: a geometric-variational
  approach.
\newblock {\em Comm. Pure Appl. Math.}, 40(3):347--366, 1987.

\bibitem[GT]{GT}
David Gilbarg and Neil~S. Trudinger.
\newblock {\em Elliptic partial differential equations of second order}.
\newblock Classics in Mathematics. Springer-Verlag, Berlin, 2001.
\newblock Reprint of the 1998 edition.

\bibitem[Ga]{Ga}
R.~M. Gabriel.
\newblock An extended principle of the maximum for harmonic functions in
  {$3$}-dimensions.
\newblock {\em J. London Math. Soc.}, 30:388--401, 1955.

\bibitem[HKM]{HKM}
Juha Heinonen, Tero Kilpel{\"a}inen, and Olli Martio.
\newblock {\em Nonlinear Potential Theory of Degenerate Elliptic Equations}.
\newblock Dover Publications Inc., 2006.

\bibitem[J]{J}
David Jerison.
\newblock A {M}inkowski problem for electrostatic capacity.
\newblock {\em Acta Math.}, 176(1):1--47, 1996.

\bibitem[K]{K}
Nicholas~J. Korevaar.
\newblock Convexity of level sets for solutions to elliptic ring problems.
\newblock {\em Comm. Partial Differential Equations}, 15(4):541--556, 1990.

\bibitem[KP]{KP}
Carlos~E. Kenig and Jill Pipher.
\newblock The {D}irichlet problem for elliptic equations with drift terms.
\newblock {\em Publ. Mat.}, 45(1):199--217, 2001.

\bibitem[KZ]{KZ}
Tero Kilpel{\"a}inen and Xiao Zhong.
\newblock Growth of entire {$\mathcal{A}$}-subharmonic functions.
\newblock {\em Ann. Acad. Sci. Fenn. Math.}, 28(1):181--192, 2003.

\bibitem[Kr]{Kr}
I.~N. Krol.
\newblock The behavior of the solutions of a certain quasilinear equation near
  zero cusps of the boundary.
\newblock {\em Trudy Mat. Inst. Steklov.}, 125:140--146, 233, 1973.
\newblock Boundary value problems of mathematical physics, 8.

\bibitem[L]{L}
John Lewis.
\newblock Capacitary functions in convex rings.
\newblock {\em Arch. Rational Mech. Anal.}, 66(3):201--224, 1977.

\bibitem[LLN]{LLN}
John Lewis, Niklas Lundstr{\"o}m, and Kaj Nystr{\"o}m.
\newblock Boundary {H}arnack inequalities for operators of {$p$}-{L}aplace type
  in {R}eifenberg flat domains.
\newblock In {\em Perspectives in partial differential equations, harmonic
  analysis and applications}, volume~79 of {\em Proc. Sympos. Pure Math.},
  pages 229--266. Amer. Math. Soc., Providence, RI, 2008.

\bibitem[LN]{LN}
John Lewis and Kaj Nystr{\"o}m.
\newblock Boundary behaviour for {$p$} harmonic functions in {L}ipschitz and
  starlike {L}ipschitz ring domains.
\newblock {\em Ann. Sci. \'Ecole Norm. Sup. (4)}, 40(5):765--813, 2007.

\bibitem[LN1]{LN1}
John Lewis and Kaj Nystr{\"o}m.
\newblock Boundary behavior and the {M}artin boundary problem for {$p$}
  harmonic functions in {L}ipschitz domains.
\newblock {\em Ann. of Math. (2)}, 172(3):1907--1948, 2010.

\bibitem[LN2]{LN2}
John Lewis and Kaj Nystr{\"o}m.
\newblock Regularity and free boundary regularity for the {$p$} {L}aplacian in
  {L}ipschitz and {$C^1$} domains.
\newblock {\em Ann. Acad. Sci. Fenn. Math.}, 33(2):523--548, 2008.

\bibitem[LN3]{LN3}
John Lewis and Kaj Nystr{\"o}m.
\newblock Regularity and free boundary regularity for the {$p$}-{L}aplace
  operator in {R}eifenberg flat and {A}hlfors regular domains.
\newblock {\em J. Amer. Math. Soc.}, 25(3):827--862, 2012.

\bibitem[LN4]{LN4}
John Lewis and Kaj Nystr\"om.
\newblock Quasi-linear {PDE}s and low-dimensional sets.
\newblock {\em J. Eur. Math. Soc. (JEMS)}, 20(7):1689--1746, 2018.

\bibitem[LSW]{LSW}
W.~Littman, G.~Stampacchia, and H.~F. Weinberger.
\newblock Regular points for elliptic equations with discontinuous
  coefficients.
\newblock {\em Ann. Scuola Norm. Sup. Pisa (3)}, 17:43--77, 1963.

\bibitem[Li]{Li}
Gary~M. Lieberman.
\newblock Boundary regularity for solutions of degenerate elliptic equations.
\newblock {\em Nonlinear Anal.}, 12(11):1203--1219, 1988.

\bibitem[Lo]{Lo}
Marco Longinetti.
\newblock Some isoperimetric inequalities for the level curves of capacity and
  {G}reen's functions on convex plane domains.
\newblock {\em SIAM J. Math. Anal.}, 19(2):377--389, 1988.

\bibitem[M1]{M1}
Hermann Minkowski.
\newblock Volumen und {O}berfl\"ache.
\newblock {\em Math. Ann.}, 57(4):447--495, 1903.

\bibitem[M2]{M2}
Hermann Minkowski.
\newblock Allgemeine lehrsätze über die convexen polyeder.
\newblock {\em Nachrichten von der Gesellschaft der Wissenschaften zu
  Göttingen, Mathematisch-Physikalische Klasse}, 1897:198--220, 1897.

\bibitem[S]{S}
James Serrin.
\newblock Local behavior of solutions of quasi-linear equations.
\newblock {\em Acta Mathematica}, 111:247--302, 1964.
\newblock 10.1007/BF02391014.

\bibitem[Sc]{Sc}
Rolf Schneider.
\newblock {\em Convex bodies: the {B}runn-{M}inkowski theory}, volume~44 of
  {\em Encyclopedia of Mathematics and its Applications}.
\newblock Cambridge University Press, Cambridge, 1993.

\bibitem[St]{St}
Elias~M. Stein.
\newblock {\em Singular integrals and differentiability properties of
  functions}.
\newblock Princeton Mathematical Series, No. 30. Princeton University Press,
  Princeton, N.J., 1970.

\bibitem[T]{T}
Peter Tolksdorf.
\newblock Regularity for a more general class of quasilinear elliptic
  equations.
\newblock {\em J. Differential Equations}, 51(1):126--150, 1984.

\bibitem[VV]{VV}
Moises Venouziou and Gregory~C. Verchota.
\newblock The mixed problem for harmonic functions in polyhedra of
  {$\mathbb{R}^3$}.
\newblock In {\em Perspectives in partial differential equations, harmonic
  analysis and applications}, volume~79 of {\em Proc. Sympos. Pure Math.},
  pages 407--423. Amer. Math. Soc., Providence, RI, 2008.

\end{thebibliography}
    \end{document}